\documentclass[final,1p,times]{elsarticle}
\usepackage{lineno,hyperref}
\usepackage{graphicx}
\usepackage{amsmath, amssymb, amsfonts, amsthm}
\usepackage{mathrsfs}
\usepackage{url}
\usepackage{latexsym}
\usepackage{color}
\usepackage{subfigure}
\usepackage{multirow}
\usepackage{epstopdf}
\usepackage{float}
\journal{}

\usepackage{adjustbox}
\usepackage{diagbox}
\usepackage{lipsum}
\usepackage{amsfonts}
\usepackage{graphicx,slashbox}
\usepackage{algpseudocode}
\usepackage{amsopn}
\usepackage{algorithm}
\usepackage{algpseudocode}
\usepackage{listings}
\usepackage{verbatim}
\usepackage{lipsum}
\usepackage{amsfonts}
\usepackage{graphicx}
\usepackage{subfigure}
\usepackage{epstopdf}
\usepackage{amsmath,amssymb,amsthm}
\usepackage{setspace}
\usepackage{multirow,array,booktabs}% packages for table
\theoremstyle{plain}

 \ifpdf
 \DeclareGraphicsExtensions{.eps,.pdf,.png,.jpg}
 \else
 \DeclareGraphicsExtensions{.eps}
 \fi
\counterwithin{algorithm}{section}
\newtheorem{theorem}{Theorem}[section]
\newtheorem{lemma}[theorem]{Lemma}
\newtheorem{remark}{Remark}[section]

\newcommand{\bm}{\boldsymbol}

\newcommand{\Grad}[1]{\nabla #1}

\def\bx{\bm{x}}
\def\e{{\rm{e}}}

\numberwithin{equation}{section}

\graphicspath{{./figs/}}

\usepackage{amsopn}

\begin{document}
\graphicspath{{figures_sq/}{fig_ArbiD/}{figures_Dbnd/}}
\begin{frontmatter}

\title{Computing optimal partition problems via Lagrange multiplier approach}

\author[a]{Qing Cheng\corref{cor1}}
\ead{qingcheng@tongji.edu.cn}
\author[b,c]{Jing Guo}
\ead{jingguo@cuhk.edu.cn}
\author[b,c]{Dong Wang}
\ead{wangdong@cuhk.edu.cn}

\cortext[cor1]{Corresponding author}

\address[a]{Department of Mathematics, Tongji University, Shanghai 200092, China  $\&$ Key Laboratory of Intelligent Computing and Applications (Tongji University), Ministry of Education, China}
\address[b]{School of Science and Engineering, The Chinese University of Hong Kong, Shenzhen, Guangdong 518172, China}
\address[c]{Shenzhen International Center for Industrial and Applied Mathematics, Shenzhen Research Institute of Big Data, Guangdong 518172, China}

	% REQUIRED
	\begin{abstract}
		In this paper, we consider numerical approximations for the optimal partition problem using Lagrange multipliers. By rewriting it into   constrained gradient flows, three and four steps numerical schemes based on the Lagrange multiplier approach \cite{ChSh22,ChSh_II22} are proposed to solve the constrained gradient system.  Numerical schemes  proposed for the constrained gradient flows satisfy the nice  properties of orthogonality-preserving, norm-preserving, positivity-preserving and energy dissipating. The proposed schemes are very efficient  in which only linear  Poisson equations are solved at each time step.  Extensive numerical results in 2D and 3D for optimal partition problem are presented to validate the effectiveness and accuracy of the proposed numerical schemes. 
	\end{abstract}
	
	% REQUIRED
	% 	\begin{AMS}
	% 68Q25, 68R10, 68U05
	% \end{AMS}
 \end{frontmatter}
	% REQUIRED

	\section{Introduction}
	In this paper, we consider an optimal partition problem \cite{Wang22,WangOs19,BucBut98}  involving a bounded and smooth geometric domain $\Omega \in \mathbb{R}^d$, minimizing the functional $E=\int_{\Omega}\frac 12 \sum\limits_{i=1}^k |\Grad u_i|^2 d\bx$ and subject to physical constraints over admissible set. This constrained minimization problem is  governed by the following system
	\begin{equation}\label{optimal:1}
		\begin{split}
			&\min \int_{\Omega}\frac 12 \sum\limits_{i=1}^k |\Grad u_i|^2 d\bx,\\
			&u_i u_j =0, \quad \forall i\neq j,\quad i,j\in[k],\\
			& u_i\ge 0,\quad i\in[k],\\
			&\|u_i\|=1, \quad i\in[k].
		\end{split}
	\end{equation}
	Here, the positive integer $k$ in \eqref{optimal:1} represents the number of partition domains, $[k]={1,2,\ldots, k}$ and $\|\cdot\|$ denotes the $l^2$ norm. The suitable boundary conditions for \eqref{optimal:1} shall be  periodic, Dirichlet, or Neumann conditions.  In recent years, the problem of minimizing a functional over physical constrained set has received much  attention due to its huge applications in the study of multicomponent Bose-Einstein condensates \cite{Bao04,BaoDu04,ChangLin04,liu2021normalized}, Schrodinger equations \cite{antoine2021scalar}, vesicle membrane equations \cite{WaJuDu16}, anisotropically conducting materials \cite{RosWan09,RosWan13}  and systems with large interactions \cite{ConTerVer02,CyBaHo05,CyHo08}.

	In past decades,  there exist some theoretical results  for optimal partition problem \eqref{optimal:1}.  The proof of existence results  was considered  in \cite{BucBut18,CafLin07,ConterVer03}, with regularity presented in \cite{HelfThom09} and consistency in \cite{OsRe17}. The asymptotic behavior of \eqref{optimal:1} as $k\rightarrow\infty$ has been explored in \cite{CafLin07,BouOud10,NoelHelfVial10}.

	The optimal partition problem \eqref{optimal:1} has also attracted much interest in developing physical constraint-preserving algorithms.  Since  the original minimization problem \eqref{optimal:1} satisfies three physical constraints, then a highly desirable property of numerical algorithms for the optimal partition problem \eqref{optimal:1}  should preserve physical constraints: (1) norm conserving, (2) orthogonality-preserving, and (3) positivity-preserving, which  turns out to be very challenging.  There exist some important numerical approaches to solve various  constrained optimization in current research and applications,  such as Riemannian optimization methods \cite{YinZh22}, penalty methods \cite{JiWa20,WaJuDu16}, projection techniques \cite{AnGaSu21,WangOs19,WaGaE01}, and Lagrange multipliers approach \cite{CaLin08,ChSh22,ChSh_II22,ChSh23}.
	For the optimal partition problem  \eqref{optimal:1}, two popular approaches to enforce constraints   include:
	\begin{itemize}
		\item Penalty approach: 
		To solve the optimal partition problem numerically, some reformulations have been proposed in \cite{DuLin08,Wang22,Bo18,BoVe16}. In \cite{DuLin08}, a penalty term
		\[F^\varepsilon(u)=\frac{1}{4\varepsilon^2}\sum\limits_{j=1,j\ne i}^k u_i^2 u_j^2,\]
		was introduced to the objective functional in \eqref{optimal:1} to penalize the non-overlapping of functions $u_i$ and $u_j$ for $i\neq j$. A semi-implicit operator splitting scheme was designed to solve the resulting nonlinear norm-preserving gradient flow system. The diffusion-generated method presented in \cite{WangOs19} also diffuses in the first step, as in \cite{DuLin08}, but uses a projection to satisfy the disjoint constraints. In \cite{Wang22}, Wang further reformulated the optimal partition problem as a bilevel optimization
		
		\[\min\limits_{\varphi}\min\limits_{u} \frac{k}{\tau}-\sum\limits_{\ell\in[k]}\frac{1}{\tau}\int_{\Omega}\varphi_\ell |\e^{\frac{\tau}{2}}u_\ell|^2\, d\bx,\]
		and proposed an unconditionally stable iterative approximation scheme that is easy to implement and more efficient than existing schemes derived from level set methods \cite{ChuLe21}, optimization approaches \cite{BouOud10}, etc.
		
		\item Lagrange multiplier approach: Introduce Lagrange multipliers to enforce exactly the constraints \cite{cheng2018multiple,cheng2020global}. The main advantage is that the constraints can be satisfied exactly, while its drawback is that it may lead to difficulty in solving nonlinear systems at each time step. In this paper, by introducing Lagrange multipliers to enforce the constraint of orthogonality and positivity, we obtain a  constrained gradient flow which leads to a new approach to construct numerical algorithms for the original minimization problem \eqref{optimal:1}. 
		
	\end{itemize}

	The goal of this paper is  to develop  efficient numerical algorithms for solving the minimization problem \eqref{optimal:1} with multiple constraints. Instead of introducing  penalty terms, we introduce Lagrange multipliers to enforce the constraints of orthogonality, positivity, $L^2$ norm conservation exactly in the gradient flow \eqref{norm:1}-\eqref{norm:5}.  To be more specific, we introduce the well-known Karush-Kuhn-Tucker (KKT) conditions \eqref{norm:4}  to enforce the positivity and obtain an expanded system. The  constraints of orthogonality \eqref{norm:2} and norm conservation \eqref{norm:3} are also enforced exactly by different Lagrange multipliers. To avoid solving constrained  problems at each time step, we adopt in this paper an operator-splitting approach to develop efficient and accurate algorithms for the gradient flow.  We propose two types  of operator-splitting schemes to for \eqref{norm:1}-\eqref{norm:5}. Our numerical schemes enjoy the following advantages:
	\begin{itemize}
		\item  It enjoys low computational cost: we  only need to solve several linear  Poisson equations at each time step for the first type of numerical schemes and a nonlinear algebraic equation plus several linear Poisson equations for the second type of numerical schemes ;
		
		\item   It inherits good structure-preserving properties:  both types of numerical schemes  can satisfy the nice properties of positivity-preserving, norm conserving and orthogonality-preserving  exactly at each time step; 
		
		\item   It has good stability property: the second type of numerical schemes  can be unconditionally energy stable and also satisfy a  discrete energy dissipative law. 
	\end{itemize}

	The remainder of this paper is organized as follows. In Section \ref{sec:gradflow}, we rewrite the optimal partition problem into a gradient flow. In Section \ref{sec:num_sch}, we develop several  efficient  algorithms to solve the constrained gradient flow. In Section \ref{subsec:num_result}, we  apply our numerical schemes to optimal partition problems with different parts and  present numerical experiments in 2D and 3D for our numerical algorithms  proposed in Section \ref{sec:num_sch}. Finally,  we give a concluding remark in Section~\ref{sec:con}. 
	
	\section{Gradient flow}\label{sec:gradflow}
	To solve constrained minimization problem \eqref{optimal:1}, the key idea of our approach is to rewrite the minimization problem \eqref{optimal:1} into the  gradient flow \eqref{gradient_sys} . Hence, we shall  develop several classes of numerical schemes based on the gradient flow inspired by the recently Lagrange multiplier approach \cite{ChSh22,ChSh_II22}.  By introducing Lagrange multipliers,  the constrained minimization problem \eqref{optimal:1} has  the following  gradient structure  taking  on the form
	\begin{eqnarray} \label{gradient_sys}
		&&\partial_t u_i(\bx,t) =\Delta u_i(\bx,t) +\xi_i(t)u_i(\bx,t) + \sum\limits_{j=1,j\ne i}^k\eta_{ij}(\bx,t)u_j(\bx,t)+\lambda_i(\bx,t) ,\label{norm:1}\\
		&&u_i u_j =0, \quad \forall i\neq j\in [k],\label{norm:2}\\[.5em]
		&&\|u_i\|=1, \quad i\in [k],\label{norm:3}\\[.5em]
		&& u_i\ge 0,\quad  \lambda_i \ge 0,\quad \lambda_i u_i=0,\quad     i\in [k],\label{norm:4}
	\end{eqnarray}
	along with   the initial condition
	\begin{equation}\label{norm:5}
		u_i(0,\bx)  = u_i^0(\bx), \quad \forall \bx \in \Omega.
	\end{equation}
	In system \eqref{norm:1}-\eqref{norm:5},  orthogonal constraints in \eqref{norm:2}  are enforced by Lagrange multipliers $\eta_{ij}(\bx,t)$. Lagrange multipliers $\xi_i(t)$ with $i=1,2,\cdots, k$ are introduced to satisfy the property of $L^2$ norm-preserving. Equation \eqref{norm:4} is the  Karush–Kuhn–Tucker(KKT) conditions \cite{ChSh22,ChSh_II22} where $\lambda_i$ for $i=1,2,\cdots, k$ are time spatial Lagrange multipliers for the constraint of positivity-preserving.  In \eqref{norm:2}, there are \(k(k-1)/2\) constraints enforcing the orthogonality, and thus \(k(k-1)/2\) corresponding Lagrange multipliers \(\eta_{ij}\) are introduced in \eqref{norm:1}.

	\begin{remark}
		Both the Lagrange multipliers  $\eta_{ij}(\bx,t)$ and $\lambda_i(\bx,t)$ $i,j =1,2,\cdots, k$ depend on space $\bx$ and time $t$, since the constraints in \eqref{norm:2} and \eqref{norm:4} are defined pointwisely in space \cite{ChSh22,ChSh_II22}. While the  physical constraints \eqref{norm:3} are defined globally, then the scalar auxiliary variables $\xi_i(t)$  are  introduced  in \eqref{norm:1} for $i=1,2,\cdots, k$, see \cite{cheng2018multiple,cheng2020global}.  
	\end{remark}
	
	In fact, we can show that the constrained  system \eqref{norm:1}-\eqref{norm:5} is a gradient flow and  satisfies the following energy dissipative law.
	\begin{theorem}
	For any integer  $k\ge 1$, the constrained system \eqref{norm:1}-\eqref{norm:5} is energy dissipative, in the sense that
		\begin{equation}
			\frac{d}{dt}E(t) = -\sum\limits_{i=1}^k\|\partial_t u_i(\bx,t)\|^2,
		\end{equation}
		where the energy is defined by
		\begin{equation}
			E(t) = \sum\limits_{i=1}^k \int_{\Omega}\frac 12|\Grad u_i|^2 d\bx. 
		\end{equation}
	\end{theorem}
	\begin{proof}
	For any  $ i\in[k]$, taking  inner product of \eqref{norm:1} with $-\partial_t u_i(\bx,t)$ and taking integration by parts, we obtain
		\begin{equation}\label{n:diss:1}
			-\|\partial_t u_i(\bx,t)\|^2 = -\frac{d}{dt}\int_{\Omega}\frac 12|\Grad u_i|^2 d\bx - (\lambda_i,\partial_t u_i)-  \left(\sum\limits_{j=1,j\ne i}^k\eta_{ij}(\bx,t)u_j,   \partial_t u_i\right),
		\end{equation}
		where we used the equality from \eqref{norm:3}
		\begin{equation}
			(\partial_t u_i, u_i)=0.
		\end{equation}
		From the KKT-condition \eqref{norm:4}, we have $\lambda_i u_i=0$ for any  $i\in[k]$ which implies that 
		\begin{equation}\label{eq}
			\frac{d}{dt}(\lambda_i u_i)=\lambda_i\partial_t u_i + \partial_t\lambda_i  u_i =0. 
		\end{equation}
		If $\lambda_i=0$, then we have  $(\lambda_i,\partial_t u_i)=0$. Similarly, if $\lambda_i\neq 0$, then from KKT-condition \eqref{norm:4}, we have $u_i=0$ which also implies $(\lambda_i,\partial_t u_i)=0$ from \eqref{eq}. 
		
		From orthogonal constraints $u_i u_j=0$ if $i \neq j$, we have
		\begin{equation}\label{eq:2}
			u_j\partial_t u_i  + u_i\partial_t u_j = 0, \quad \mbox{if} \quad  i\neq j. 
		\end{equation}
		Combining \eqref{eq:2} with $\eta_{ij}=\eta_{ji}$,  we can  derive that
		\begin{equation}
			\sum\limits_{i=1}^k\left(\sum\limits_{j=1,j\ne i}^k\eta_{ij}(\bx,t)u_j,   \partial_t u_i\right) = 0. 
		\end{equation}
		Combining all results above, the equation \eqref{n:diss:1} can be rewritten into 
		\begin{equation}\label{n:diss:2}
			-\sum\limits_{i=1}^k\|\partial_t u_i(\bx,t)\|^2 = \frac{d}{dt}\int_{\Omega}\sum\limits_{i=1}^k\frac 12|\Grad u_i|^2 d\bx. 
		\end{equation}
		
	\end{proof}

	\section{Numerical approximation}\label{sec:num_sch}
	In this section, we focus on numerical approximations for the gradient flow system \eqref{norm:1}-\eqref{norm:5}. Denote $\tau$ as  the time step, and $t_n=n\tau$ for $n=0,1,2,\cdots, \frac{T}{\tau}$ where $T$ is the final computational time. Two different types of  operator splitting schemes are proposed to solve  \eqref{norm:1}-\eqref{norm:5} efficiently.
	\subsection{A four-step operator splitting scheme}
	We now present efficient numerical approximations for the system \eqref{norm:1}-\eqref{norm:5} employing an operator-splitting method. Starting with a given state $u_i^n=u_i(\bx,t_n)$ for $ i\in[k]$, we calculate the next time step $u_i^{n+1}$ through the following four steps:
	
	{\bf Step-1:} (Diffusion step) Given $u_i^n$ for $ i\in[k]$, we  solve the  heat equation
	\begin{eqnarray} \label{diff:1}
		\partial_t u_i(\bx,t) =\Delta u_i(\bx,t),\quad \bx\in \Omega, \quad t\in(t_n, t_{n+1}),
	\end{eqnarray}
	with  $u_i(\bx,t_n)=u_i^n(\bx)$. The solution at $t_{n+1}$ can be expressed exactly  by, see \cite{Wang22,WangOs19}
	\begin{equation*}
		\tilde u^{n+1}_i = \e^{\tau\Delta} u^n_i,\quad i\in[k].
	\end{equation*}

	{\bf Step-2:} (Positivity-preserving  step) Given $\tilde u_i^{n+1}$ for $ i\in[k]$,
	motivated from  the following system 
		\begin{equation}\label{step2:eq}
			\begin{split}
				&\partial_t u_i(\bx,t)=\lambda_i^{n+1},\quad \bx\in \Omega, \quad t\in(t_n, t_{n+1}),\\[.5em]
				&u_i(\bx,t_{n})=\tilde u_i^{n+1}(\bx),\quad \bx\in \Omega,\\[.5em]
				& \lambda_i(\bx,t) \ge0;\; u_i(\bx,t) \ge 0;\;\lambda_i(\bx,t)u_i(\bx,t)=0,\quad i\in[k],
			\end{split}
		\end{equation}
		at $t=t_{n+1}$, see \cite{ChSh22,ChSh_II22},  we  solve
	\begin{equation}\label{app:1:step:3}
		\begin{split}
			&\frac{\bar u_i^{n+1}-\tilde u_i^{n+1}}{\tau}=\lambda_i^{n+1}, \quad \quad i\in[k],\\[.5em]
			&\lambda_i^{n+1} \ge0 ;\;\bar u_i^{n+1} \ge 0;\;\lambda_i^{n+1}\bar u_i^{n+1}=0,\quad i\in[k],
		\end{split}
	\end{equation}
	at $t=t_{n+1}$. 
	
	{\bf Step-3:} (Orthogonality-preserving  step)  Given $\bar u_i^{n+1}$, and apply the forward Euler method to 
	\begin{equation}\label{app:1:projection:1}
		\begin{split}
			&	\partial_t u_i(\bx,t)=\sum\limits_{j\in [k],j\ne i} \eta_{ij}(\bx,t)u_j(\bx,t),\quad \bx\in \Omega, \quad t\in(t_n, t_{n+1}),\\[.5em]
			&u_i (\bx,t_n)=\bar u_i^{n+1}(\bx), \quad u_i(\bx,t)u_j(\bx,t)=0, \quad i\in[k],\quad \forall j\in [k], \quad j\ne i, 
		\end{split}
	\end{equation}
	and solve $\hat u_i^{n+1}$ from
	\begin{equation}\label{disj_ForwardEuler}
		\begin{split}
			\frac{\hat u_i^{n+1}-\bar u_i^{n+1}}{\tau}&=\sum\limits_{j\in [k],j\ne i}\eta_{ij}^{n+1}\bar u_j^{n+1},\quad i\in [k],\\[.5em]
			\hat u_i^{n+1}\ge 0,\quad  \hat u_i^{n+1}\hat u_j^{n+1}&=0, \quad i\in[k],\quad \forall j\in [k], \quad j\ne i.
		\end{split}
	\end{equation}
	
	{\bf Step-4:} (Norm-preserving step) Given $\hat u_i^{n+1}$,  apply the   backward Euler method to 
	\begin{equation}\label{projection:step:4}
		\begin{split}
			&	\partial_t u_i(\bx,t)=\xi_i(t)u_i(\bx,t),\quad \bx\in \Omega, \quad t\in(t_n, t_{n+1}),\\
			&u_i(\bx,t_n)=\hat u_i^{n+1}(\bx),\quad \|u_i(\bx,t)\|=1,\quad i\in [k],
		\end{split}
	\end{equation}
	see also in  \cite{ChSh23}   and   solve $u_i^{n+1}$ from
	\begin{equation}\label{pro:eq}
		\begin{split}
			&	\frac{u_i^{n+1}-\hat u_i^{n+1}}{\tau}=\xi_i^{n+1} u_i^{n+1},\quad \bx\in \Omega, \quad t\in(t_n, t_{n+1}),\\
			&\|u_i^{n+1}\|=1,\quad i\in [k].
		\end{split}
	\end{equation}
	
	Rewriting \eqref{pro:eq} into $(1-\tau \xi_i^{n+1}) u_i^{n+1}=\hat u_i^{n+1}$ and applying $\|u_i^{n+1}\|=1$, we can solve  $\xi_i^{n+1} = \frac{1-\|\hat u_i^{n+1}\|}{\tau} $ and $	u_i^{n+1}=\frac{\hat u_i^{n+1}}{\|\hat u_i^{n+1}\|}$, $i\in[k]$.  Actually, we can find that the norm-preserving step  can be regarded as a  projection step, see Remark \ref{re:o}. Taking integral for the first equation of \eqref{step2:eq} in time, we  find that  the solution at $t=t_{n+1}$ should be 
		\[\bar u_i^{n+1}=\int_{t_n}^{t_{n+1}}\lambda_i^{n+1}\, dt+\tilde u_i^{n+1}(\bx). \]
		Applying the KKT  condition
		\[\lambda_i(\bx,t) \ge0;\;\bar u_i(\bx,t) \ge 0;\;\lambda_i(\bx,t)\bar u_i(\bx,t)=0,\quad i\in[k],\]
		we can  easily derive  that 
		$\bar u_i^{n+1}=\max(\tilde u_i^{n+1},0)$, $i\in [k]$.  
	
	\begin{remark}\label{re:o}
		The norm-preserving step is equivalent to  the following optimization problem
		\begin{equation}\label{opt:0}
			\begin{cases}
				u^{n+1}_i =     \min\limits_{u_i} \int_\Omega ( u_i - \hat{u}_i^{n+1})^2/(2\tau)d\bx, \quad  i \in [k], \\[1 em]
				\| {u}_i^{n+1}\| =1, \quad  i \in [k],\forall j \in [k], j \ne i.
			\end{cases}
		\end{equation}
		From \eqref{opt:0},  we observe that this step can be regarded as projecting $\hat u_i^{n+1}$ into constrained set $\|u\|=1$ in the $L^2$ norm sense. 
	\end{remark}

	\begin{remark}
		Actually,  the the  positvity-preserving step \eqref{app:1:step:3} can be rewritten into the optimization step 
		\begin{equation}\label{opt:1}
			\begin{cases}
				\bar u^{n+1}_i =     \min\limits_{\bar u_i} \int_\Omega (\bar{u}_i - \tilde{u}_i^{n+1})^2/(2\tau)d\bx, \quad  i \in [k], \\[1 em]
				\bar {u}_i^{n+1} \ge 0, \quad  i \in [k],\forall j \in [k], j \ne i.
			\end{cases}
		\end{equation}
		Then the unique solvability of \eqref{app:1:step:3} can be easily obtained since there exist a unique solution for the convex  constrained  optimization problem \eqref{opt:1}. 
		
	\end{remark}

	\begin{remark}
		Similarly, the orthogonality-preserving step \eqref{disj_ForwardEuler} is also equivalent to the following optimization problem
		\begin{equation}
			\begin{cases}
				u_i^{n+1} =  \min\limits_{u_i(\bx)} \int_\Omega ({u}_i(\bx) - \bar{u}_i^{n+1}(\bx))^2/(2\tau)-\sum_{j\in [k],j\ne i} \eta_{ij}^{n+1} \bar{u}_j^{n+1}{u}_i(\bx)\, d\bx, \quad  i \in [k], \\[1 em]
			\quad {u}_i^{n+1} {u}_j^{n+1} = 0, \quad  i \in [k],\forall j \in [k], j \ne i.
			\end{cases}
		\end{equation}
		Since the orthogonal constraint ${u}_i^{n+1} {u}_j^{n+1} = 0$ will not  satisfy the property of convexity,   then there exists multiple solutions for this step.  
		
	\end{remark}

	Next we show how to efficiently solve the scheme proposed above. Obviously, the diffusion step and positivity-preserving step can be solved uniquely.  To present the algorithm, it is necessary to elucidate the solution to \eqref{disj_ForwardEuler} in the third step (Orthogonality-preserving step). The property of  non-convexity of the constraint $u_iu_j=0$ with  $i\ne j$ for  $ i,j \in [k]$  indicates  the existence of multiple solutions for the equation \eqref{disj_ForwardEuler}, in the subsequent lemma,  we shall  illustrate one such solution.                
	\begin{lemma}\label{lem:disj}
		Given $\bar u_i^{n+1}\ge0$, $i\in[k]$,  $\eta_{ij}^{n+1}=\eta_{ji}^{n+1}$, $n\ge 0$, one solution of the scheme \eqref{disj_ForwardEuler}  is 
		\begin{equation}\label{disj_sol}
			\begin{split}
				&\hat u_i^{n+1}=\left\{
				\begin{array}{lcl}
					\frac{(\bar u_i^{n+1})^2-\max\limits_{j\in[k],j\ne i}\left\{(\bar u_j^{n+1})^2\right\}}{\bar{u}_i^{n+1}},  &      & {\mbox{if} \quad \bar u_i^{n+1}> \max\limits_{j\in[k],j\ne i}\{\bar u_j^{n+1}\}},\\[.5em]
					0,   &    & {\mbox{otherwise} }.
				\end{array} \right.
			\end{split}
		\end{equation}
	\end{lemma}
	
	\begin{proof}
For \( k = 2 \), using \(\eta_{12}^{n+1} = \eta_{21}^{n+1}\), we have
\begin{align}
    \label{4step:dis_p2_u1}
    \hat{u}_1^{n+1} = \bar{u}_1^{n+1} + \tau \eta_{12}^{n+1} \bar{u}_2^{n+1}, \\
    \label{4step:dis_p2_u2}
    \hat{u}_2^{n+1} = \bar{u}_2^{n+1} + \tau \eta_{12}^{n+1} \bar{u}_1^{n+1}.
\end{align}
Since \(\hat{u}_1^{n+1} \hat{u}_2^{n+1} = 0\), we derive
\begin{equation}\label{Sysk2Cons}
    (\bar{u}_1^{n+1} + \tau \eta_{12}^{n+1} \bar{u}_2^{n+1})(\bar{u}_2^{n+1} + \tau \eta_{12}^{n+1} \bar{u}_1^{n+1}) = 0.
\end{equation}

\textbf{Case 1:} If \(\bar{u}_1^{n+1} \bar{u}_2^{n+1} = 0\), we set \(\eta_{12}^{n+1} = 0\) and get
\[
\hat{u}_1^{n+1} = \bar{u}_1^{n+1}, \quad \hat{u}_2^{n+1} = \bar{u}_2^{n+1},
\]
which satisfies the orthogonal condition and agrees with \eqref{disj_sol}.

\textbf{Case 2:} If \(\bar{u}_1^{n+1} \bar{u}_2^{n+1} \ne 0\), \(\bar{u}_1^{n+1} \ge \bar{u}_2^{n+1} > 0\), combining \(\hat{u}_1^{n+1} \ge 0\), we obtain
\begin{equation}\label{Sysk2}
\left\{
\begin{array}{l}
    \bar{u}_1^{n+1} + \tau \eta_{12}^{n+1} \bar{u}_2^{n+1} \ge 0, \\[.5em]
    \bar{u}_2^{n+1} + \tau \eta_{12}^{n+1} \bar{u}_1^{n+1} = 0.
\end{array}
\right.
\end{equation}
The unique solution of \eqref{Sysk2} is \(\eta_{12}^{n+1} = -\frac{\bar{u}_2^{n+1}}{\tau \bar{u}_1^{n+1}}\). Using \eqref{4step:dis_p2_u1} and \eqref{4step:dis_p2_u2}, the solution of \eqref{disj_ForwardEuler} for \(k = 2\) can be expressed as
\[
\hat{u}_1^{n+1} = \frac{(\bar{u}_1^{n+1})^2 - (\bar{u}_2^{n+1})^2}{\bar{u}_1^{n+1}}, \quad \hat{u}_2^{n+1} = 0,
\]
for \(\bar{u}_1^{n+1} \ge \bar{u}_2^{n+1}\). Similarly, for \(0 < \bar{u}_1^{n+1} \le \bar{u}_2^{n+1}\), the solution of \eqref{disj_ForwardEuler} with \(k = 2\) is
\[
\hat{u}_1^{n+1} = 0, \quad \hat{u}_2^{n+1} = \frac{(\bar{u}_2^{n+1})^2 - (\bar{u}_1^{n+1})^2}{\bar{u}_2^{n+1}}.
\]
Therefore, the solution of \eqref{disj_ForwardEuler} for the case of \(\bar{u}_1^{n+1} \bar{u}_2^{n+1} \ne 0\) can be written as
\begin{equation}\label{sol:al:1}
\left(\hat{u}_1^{n+1}, \hat{u}_2^{n+1}\right) = \left\{
\begin{array}{lcl}
\left(\frac{(\bar{u}_1^{n+1})^2 - (\bar{u}_2^{n+1})^2}{\bar{u}_1^{n+1}}, 0\right), & \text{if} & \bar{u}_1^{n+1} \ge \bar{u}_2^{n+1}, \\[1em]
\left(0, \frac{(\bar{u}_2^{n+1})^2 - (\bar{u}_1^{n+1})^2}{\bar{u}_2^{n+1}}\right), & \text{if} & \bar{u}_1^{n+1} \le \bar{u}_2^{n+1},
\end{array}
\right.
\end{equation}
which is also consistent with \eqref{disj_sol}.

For \(k = 3\), we have
\begin{equation}\label{SysAlg1:k3}
\begin{split}
    \hat{u}_1^{n+1} &= \bar{u}_1^{n+1} + \tau \eta_{12}^{n+1} \bar{u}_2^{n+1} + \tau \eta_{13}^{n+1} \bar{u}_3^{n+1}, \\[.5em]
    \hat{u}_2^{n+1} &= \bar{u}_2^{n+1} + \tau \eta_{12}^{n+1} \bar{u}_1^{n+1} + \tau \eta_{23}^{n+1} \bar{u}_3^{n+1}, \\[.5em]
    \hat{u}_3^{n+1} &= \bar{u}_3^{n+1} + \tau \eta_{13}^{n+1} \bar{u}_1^{n+1} + \tau \eta_{23}^{n+1} \bar{u}_2^{n+1}.
\end{split}
\end{equation}

\textbf{Case 1:} If \(\bar{u}_\ell^{n+1} = 0\) for some \(\ell \in [3]\), we set \(\eta_{\ell j}^{n+1} = \eta_{j \ell}^{n+1} = 0\) for \(j \in [3]\), \(j \neq \ell\), which implies \(\hat{u}_\ell^{n+1} = \bar{u}_\ell^{n+1} = 0\). Following a similar derivation as in the case of \(k = 2\), we can find that the solution to \eqref{SysAlg1:k3} for this case corresponds to \eqref{disj_sol}.

\textbf{Case 2:} If for all \(\ell \in [3]\), \(\bar{u}_\ell^{n+1} > 0\), imposing the constraints
\[
\hat{u}_i^{n+1} \hat{u}_j^{n+1} = 0, \quad i \in [3], \quad \forall j \in [3], \quad j \neq i,
\]
on \eqref{SysAlg1:k3}, we obtain
\begin{equation}\label{Sysk3Cons}
\begin{split}
\left(\bar{u}_1^{n+1} + \tau \eta_{12}^{n+1} \bar{u}_2^{n+1} + \tau \eta_{13}^{n+1} \bar{u}_3^{n+1}\right) \left(\bar{u}_2^{n+1} + \tau \eta_{12}^{n+1} \bar{u}_1^{n+1} + \tau \eta_{23}^{n+1} \bar{u}_3^{n+1}\right) &= 0, \\[.5em]
\left(\bar{u}_1^{n+1} + \tau \eta_{12}^{n+1} \bar{u}_2^{n+1} + \tau \eta_{13}^{n+1} \bar{u}_3^{n+1}\right) \left(\bar{u}_3^{n+1} + \tau \eta_{13}^{n+1} \bar{u}_1^{n+1} + \tau \eta_{23}^{n+1} \bar{u}_2^{n+1}\right) &= 0, \\[.5em]
\left(\bar{u}_2^{n+1} + \tau \eta_{12}^{n+1} \bar{u}_1^{n+1} + \tau \eta_{23}^{n+1} \bar{u}_3^{n+1}\right) \left(\bar{u}_3^{n+1} + \tau \eta_{13}^{n+1} \bar{u}_1^{n+1} + \tau \eta_{23}^{n+1} \bar{u}_2^{n+1}\right) &= 0.
\end{split}
\end{equation}
Assuming \(\bar{u}_m^{n+1} \ge \bar{u}_p^{n+1} \ge \bar{u}_q^{n+1}\) for \(m \ne p \ne q \in [3]\), \eqref{Sysk3Cons} can be written as
\begin{equation}\label{Sysk3}
\left\{
\begin{array}{l}
\bar{u}_m^{n+1} + \tau \eta_{mp}^{n+1} \bar{u}_p^{n+1} + \tau \eta_{mq}^{n+1} \bar{u}_q^{n+1} \ge 0, \\[.5em]
\bar{u}_p^{n+1} + \tau \eta_{pm}^{n+1} \bar{u}_m^{n+1} + \tau \eta_{pq}^{n+1} \bar{u}_q^{n+1} = 0, \\[.5em]
\bar{u}_q^{n+1} + \tau \eta_{qm}^{n+1} \bar{u}_m^{n+1} + \tau \eta_{qp}^{n+1} \bar{u}_p^{n+1} = 0.
\end{array}
\right.
\end{equation}
Note that, in contrast to the system \eqref{Sysk2} for \(k = 2\), the system \eqref{Sysk3} admits multiple valid choices for \(\eta_{ij}^{n+1}\). A specific solution to \eqref{Sysk3} is
\begin{equation}\label{Lagdisjk3}
\left\{
\begin{array}{l}
\eta_{mq}^{n+1} = \eta_{qm}^{n+1} = -\frac{\bar{u}_q^{n+1}}{2\tau \bar{u}_m^{n+1}}, \\[.5em]
\eta_{pq}^{n+1} = \eta_{qp}^{n+1} = -\frac{\bar{u}_q^{n+1}}{2\tau \bar{u}_p^{n+1}}, \\[.5em]
\eta_{mp}^{n+1} = \eta_{pm}^{n+1} = -\frac{2(\bar{u}_p^{n+1})^2 - (\bar{u}_q^{n+1})^2}{2\tau \bar{u}_m^{n+1} \bar{u}_p^{n+1}}.
\end{array}
\right.
\end{equation}
Substituting \eqref{Lagdisjk3} into \eqref{SysAlg1:k3}, we obtain the solution \eqref{disj_sol}. 
The solution for the case \(k \ge 4\) can be derived in a manner similar to the case of \(k = 3\). The details are omitted here. Thus, the proof is complete.

	\end{proof}
From Lemma \ref{lem:disj}, we shall  propose  Algorithm \ref{al:1} to  solve  the system \eqref{norm:1}-\eqref{norm:5}.

	\begin{theorem}
	The four-step operator  splitting Algorithm \ref{al:1} for the system \eqref{norm:1}-\eqref{norm:5}  satisfies  the following nice properties 
		\begin{equation}
			u_i^{n+1} \ge 0, \quad \|u_i^{n+1}\|=1, \quad u_i^{n+1}u_j^{n+1}=0,  \quad i \in [k],\; \forall j \in [k],\; j \ne i. 
		\end{equation}
	\end{theorem}
	
	For the four step operator-splitting scheme \eqref{diff:1}-\eqref{pro:eq}, the property of  energy dissipation  can not be  satisfied  unconditionally.  For sufficiently small $\tau$, the  four-step operator-splitting scheme \eqref{diff:1}-\eqref{pro:eq} should satisfy the property of energy dissipation, we leave detailed discussion in future work.

Based on Lemma \ref{lem:disj}, we implement the four step operator-splitting scheme \eqref{diff:1}-\eqref{pro:eq}  as the  Algorithm~\ref{Alg:4step}.

	\begin{algorithm}[H]\label{al:1}
		\begin{algorithmic}
			\State {\bf Input:} {Let $\Omega$ be a given domain, $\tau > 0$,  $N_{\max}$  be the  maximum number of iteration  and $u^0 \in H_0^1(\Omega)$ be the initial condition.} 
			\State {\bf Output:} {$u_i^{n+1}$, $i=1,\cdots, k$.}
			\State Initialize $n = 0$.
			\While{$n< N_{\max}$ \ }
			
			{\bf 1.  Diffusion step.} Compute 
			\begin{equation}\label{Alg1_kPartitionS1}
				\tilde u^{n+1}_i = \e^{\tau\Delta} u^n_i.
			\end{equation}

			{\bf 2. Positivity-preserving step.} Compute
			\begin{equation}\label{Alg1_kPartitionS2}
				\bar{u}_i^{n+1}=\max\{	\tilde{u}_i^{n+1},0\},\quad i=1,\cdots,k.
			\end{equation}
			
			{\bf 3. Orthogonality-preserving step.} 
			\begin{equation}\label{Alg1_kPartitionS3}
				\begin{split}
					&\hat u_i^{n+1}=\left\{
					\begin{array}{lcl}
						\frac{(\bar u_i^{n+1})^2-\max\limits_{j\in[k],j\ne i}\left\{(\bar u_j^{n+1})^2\right\}}{\bar{u}_i^{n+1}},  &      & {\mbox{if} \quad \bar u_i^{n+1}> \max\limits_{j\in[k],j\ne i}\{\bar u_j^{n+1}\}},\\[.5em]
						0,   &    & {\mbox{otherwise} }.
					\end{array} \right.
				\end{split}
			\end{equation}

			{\bf 4. Norm-preserving step.} Update $u_i^{n+1}$ by
			\begin{equation}\label{Alg1_kPartitionS4}
				u_i^{n+1}=\frac{\bar u_i^{n+1}}{\|\bar u_i^{n+1}\|},\quad i=1,\cdots,k.
			\end{equation}
			
			\If{\eqref{stop_cond} is true}
			
			Stop the iteration.
			
			\EndIf

			Set $n = n+1$.

			\EndWhile
		\end{algorithmic}
		\caption{A four-step operator splitting scheme. } \label{Alg:4step}
	\end{algorithm}
	
		In the implementation, we stop the iteration if the partition remains unchanged between two consecutive iterations,  which can be expressed by
		\begin{equation}\label{stop_cond}
			\sum\limits_{i=1}^k\Big\|\psi(u_i^{n+1}-\max\limits_{j\in[k]}\{u_j^{n+1}\})-\psi(u_i^{n}-\max\limits_{j\in[k]}\{u_j^{n}\})\Big\|=0,
		\end{equation}
		where
		\begin{align*}
			\psi(x)=\left\{\begin{array}{lll}
				1,&&x=0,\\
				0, &&x\ne 0.
			\end{array}\right.
		\end{align*}

	\subsection{Three-step operator splitting schemes}
	Actually, we can combine positivity-preserving step with orthogonality-preserving step and propose the following three-step numerical algorithms.
	
	Given $u_i^n$, $i\in[k]$,  $n\ge0$, we now present  two three-step numerical algorithm for  solving  \eqref{norm:1}-\eqref{norm:4}. 
	
	{\bf Step-1:} (Diffusion step) Given $u_i^n$ for $ i\in[k]$, we  solve the  heat equation
	\begin{eqnarray} \label{Alg2_diff:1}
		\partial_t u_i(\bx,t) =\Delta u_i(\bx,t),\quad \bx\in \Omega, \quad t\in(t_n, t_{n+1}),
	\end{eqnarray}
	with  $u_i(\bx,t_n)=u_i^n(\bx)$. The solution at $t_{n+1}$ can be expressed exactly  by
	\begin{equation*}
		\tilde u^{n+1}_i = \e^{\tau\Delta} u^n_i,\quad i\in[k].
	\end{equation*}
	
	{\bf Step-2:} (Positivity-preserving and orthogonality-preserving step) Given $\tilde u_i^{n+1}$ for $ i\in[k]$,
	motivated from  the following system 
	\begin{equation}\label{3step_gradient_disj}
		\begin{split}
			&	\partial_t u_i(\bx,t)=\lambda_i(\bx,t)+\sum\limits_{j\in [k],j\ne i} \eta_{ij}(\bx,t)u_j(\bx,t),\quad \bx\in \Omega, \quad t\in(t_n, t_{n+1}),\\
			&u_i (\bx,t_n)=\tilde u_i^{n+1}(\bx), \quad u_i(\bx,t)u_j(\bx,t)=0, \quad i\in[k],\quad \forall j\in [k], \quad j\ne i.
		\end{split}
	\end{equation}
We solve $\hat u_i^{n+1}$ from
	\begin{equation}\label{3step_disj_ExplicitEuler}
		\begin{split}
			\frac{\hat u_i^{n+1}-\tilde u_i^{n+1}}{\tau}&=\lambda_i^{n+1}+\sum\limits_{j\in [k],j\ne i}\eta_{ij}^{n+1}\tilde u_j^{n+1},\quad i\in [k],\\
			\lambda_i^{n+1}\ge 0,\quad	\lambda_i^{n+1}	\hat u_i^{n+1}=0,\quad  
			&\hat u_i^{n+1}\ge 0,\quad \hat u_i^{n+1}\hat u_j^{n+1}=0, \quad i\in[k],\quad \forall j\in [k], \quad j\ne i.
		\end{split}
	\end{equation}

	{\bf Step-3:} (Norm-preserving step) Given $\hat u_i^{n+1}$ for $ i\in[k]$, we solve $u_i^{n+1}$ from
	\begin{equation}\label{3step_disj_norm}
		u_i^{n+1}=\frac{\hat u_i^{n+1}}{\|\hat u_i^{n+1}\|},\quad i\in[k].
	\end{equation}	
	Since {\bf Step-1} and {\bf Step-3} can be solved efficiently and uniquely,  below, we mainly  show how to efficiently solve \eqref{3step_disj_ExplicitEuler} in   {\bf Step-2}. Similar to four-steps operator-splitting scheme proposed above,  there exist multiple solutions for  the equation \eqref{3step_disj_ExplicitEuler}. Below,  we shall give two solutions for \eqref{3step_disj_ExplicitEuler} and propose two new  types of three-steps algorithms. 
	
	\begin{lemma}
		Given $\bar u_i^{n+1}\ge0$, $i\in[k]$, $k\ge 2$, $\eta_{ij}^{n+1}=\eta_{ji}^{n+1}$, $n\ge 0$, one solution of the scheme  \eqref{3step_disj_ExplicitEuler}  is 
		\begin{equation}\label{Alg2Disj_sol}
			\begin{split}
				&\hat u_i^{n+1}=\left\{
				\begin{array}{lcl}
					\tilde  u_i^{n+1}- \max\limits_{j\in[k],j\ne i}\{\tilde u_j^{n+1}\},	  &    & {\mbox{if} \quad  \tilde  u_i^{n+1}>0 \quad\mbox{and}\quad \tilde  u_i^{n+1}> \max\limits_{j\in[k],j\ne i}\{\tilde u_j^{n+1}\}},\\[.5em]
					0,   &    & {\mbox{otherwise} }.
				\end{array} \right.
			\end{split}
		\end{equation}
	\end{lemma}
	
	\begin{proof}
If there exists \(\tilde u_i^{n+1} \le 0\) for some \(i \in [k]\),  choose
\[(\lambda_i^{n+1}, \eta_{ij}^{n+1}) = \left(-\frac{\tilde u_i^{n+1}}{\tau}, 0\right),\] 
for all \(j \in [k]\) with \(j \ne i\). Then, we have \(\hat u_i^{n+1} = 0\), which satisfies the conditions in \eqref{3step_disj_ExplicitEuler} and is consistent with \eqref{Alg2Disj_sol}. It remains to prove that \eqref{3step_disj_ExplicitEuler} admits the solution \eqref{Alg2Disj_sol} when \(\tilde u_i^{n+1} > 0\) for \(i \in [k]\).

We start with the case of \(k=2\). For \(k=2\) in \eqref{3step_disj_ExplicitEuler}, we have
\begin{equation}\label{Sysk2Alg2Cons}
    \begin{split}
        \hat u_1^{n+1} &= \tilde u_1^{n+1} + \tau \lambda_1^{n+1} + \tau \eta_{12}^{n+1} \tilde u_2^{n+1},\\[.5em]
        \hat u_2^{n+1} &= \tilde u_2^{n+1} + \tau \lambda_2^{n+1} + \tau \eta_{12}^{n+1} \tilde u_1^{n+1},\\[.5em]
        \lambda_1^{n+1} &\ge 0, \quad \lambda_1^{n+1} \hat u_1^{n+1} = 0,\\[.5em]
        \lambda_2^{n+1} &\ge 0, \quad \lambda_2^{n+1} \hat u_2^{n+1} = 0,\\[.5em]
        \hat u_1^{n+1} &\ge 0, \quad \hat u_2^{n+1} \ge 0, \quad \hat u_1^{n+1} \hat u_2^{n+1} = 0.
    \end{split}
\end{equation}
If \(\tilde{u}_1^{n+1} \ge \tilde{u}_2^{n+1}\), imposing the conditions
\[
\hat{u}_1^{n+1} \ge 0, \quad \hat{u}_2^{n+1} \ge 0, \quad \hat{u}_1^{n+1} \hat{u}_2^{n+1} = 0,
\]
in \eqref{Sysk2Alg2Cons}, we get
\begin{equation}\label{Sysk2Alg2}
\left\{
\begin{array}{l}
\tilde{u}_1^{n+1} + \tau \lambda_1^{n+1} + \tau \eta_{12}^{n+1} \tilde{u}_2^{n+1} \ge 0, \\[.5em]
\tilde{u}_2^{n+1} + \tau \lambda_2^{n+1} + \tau \eta_{12}^{n+1} \tilde{u}_1^{n+1} = 0, \\[.5em]
\lambda_1^{n+1} \ge 0, \quad \lambda_1^{n+1} \hat{u}_1^{n+1} = 0, \quad \lambda_2^{n+1} \ge 0, \quad \lambda_2^{n+1} \hat{u}_2^{n+1} = 0.
\end{array}
\right.
\end{equation}
For \(\tilde u_1^{n+1} \ge \tilde u_2^{n+1}\), one solution to \eqref{Sysk2Alg2} can be expressed as
\begin{equation}\label{Sysk2Alg2_Lagsol}
\left(\eta_{12}^{n+1}, \lambda_1^{n+1}, \lambda_2^{n+1}\right) = \left(-\frac{1}{\tau}, 0, \frac{\tilde u_1^{n+1} - \tilde u_2^{n+1}}{\tau}\right).
\end{equation}
Substituting \eqref{Sysk2Alg2_Lagsol} into \eqref{Sysk2Alg2Cons} gives
\begin{equation}
\left(\hat u_1^{n+1}, \hat u_2^{n+1}\right) = \left(\tilde u_1^{n+1} - \tilde u_2^{n+1}, 0\right).
\end{equation}
Similarly, for \(\tilde u_1^{n+1} \le \tilde u_2^{n+1}\), we can prove that one solution of \eqref{Sysk2Alg2Cons} is
\begin{equation}
\left(\hat u_1^{n+1}, \hat u_2^{n+1}\right) = \left(0, \tilde u_2^{n+1} - \tilde u_1^{n+1}\right).
\end{equation}

For the case of \(k \ge 3\), let \(\tilde u_m^{n+1} = \max_{j \in [k]} \{\tilde u_j^{n+1}\}\) and let \(q\) be the smallest index such that \(\tilde u_q^{n+1} = \max_{j \in [k], j \ne m} \{\tilde u_j^{n+1}\}\). Applying the conditions
\[ \hat u_i^{n+1} \ge 0, \quad \hat u_i^{n+1} \hat u_j^{n+1} = 0, \quad \forall i \in [k], \quad \forall j \in [k], \quad j \ne i, \]
in \eqref{3step_disj_ExplicitEuler}, we obtain
\begin{equation}\label{Syskge3Alg2}
\left\{
\begin{array}{l}
\tilde{u}_m^{n+1} + \tau \lambda_m^{n+1} + \tau \sum_{j \in [k], j \ne m} \eta_{mj}^{n+1} \tilde{u}_j^{n+1} \ge 0, \\[1em]
\tilde{u}_\ell^{n+1} + \tau \lambda_\ell^{n+1} + \tau \sum_{j \in [k], j \ne \ell} \eta_{\ell j}^{n+1} \tilde{u}_j^{n+1} = 0, \quad \forall \ell \in [k], \quad \ell \ne m, \\[1em]
\lambda_i^{n+1} \ge 0, \quad \lambda_i^{n+1} \hat{u}_i^{n+1} = 0, \quad \forall i \in [k].
\end{array}
\right.
\end{equation}
One solution to \eqref{Syskge3Alg2} can be expressed as
\begin{equation}\label{LagDisAlg2}
\eta_{ij}^{n+1} = -\frac{1}{\tau} \left\{
\begin{array}{lll}
1, & & \text{if } (i,j) = (m,q) \text{ or } (i,j) = (q,m), \\[.5em]
0, & & \text{if } i = m \text{ and } j \ne q \text{ or } j = m \text{ and } i \ne q, \\[.5em]
\max \left( \frac{\tilde u_i^{n+1}}{\tilde u_j^{n+1}}, \frac{\tilde u_j^{n+1}}{\tilde u_i^{n+1}} \right), & & \text{otherwise},
\end{array} \right.
\end{equation}
and
\begin{equation}\label{LagPosAlg2}
\lambda_i^{n+1} = \left\{
\begin{array}{lll}
0, & & \text{if } i=m, \\[.5em]
-\frac{\tilde u_i^{n+1}}{\tau} - \sum_{j \in [k], j \ne i} \eta_{ij}^{n+1} \tilde u_j^{n+1}, & & \text{otherwise}.
\end{array} \right.
\end{equation}
Using the expression of  Lagrange multipliers \(\eta_{ij}^{n+1}\) in \eqref{LagDisAlg2}, we can obtain
\begin{equation*}
 \lambda_i^{n+1} = -\frac{\tilde u_i^{n+1}}{\tau} - \sum_{j \in [k], j \ne i} \eta_{ij}^{n+1} \tilde u_j^{n+1} \ge -\frac{\tilde u_i^{n+1}}{\tau} + \frac{\tilde u_i^{n+1}}{\tau} = 0, \quad \text{for} \quad i \ne m,
\end{equation*}
which implies the nonnegativity of \(\lambda_i^{n+1}\).
Substituting \eqref{LagDisAlg2} and \eqref{LagPosAlg2} into \eqref{3step_disj_ExplicitEuler} yields
\begin{equation*}
\hat u_i^{n+1} = \left\{
\begin{array}{lcl}
\tilde u_i^{n+1} - \max_{j \in [k], j \ne i} \{\tilde u_j^{n+1}\}, & & \text{if } \tilde u_i^{n+1} > 0 \text{ and } \tilde u_i^{n+1} > \max_{j \in [k], j \ne i} \{\tilde u_j^{n+1}\}, \\[.5em]
0, & & \text{otherwise}.
\end{array} \right.
\end{equation*}
The proof is complete.

	\end{proof}

	\begin{remark}
		Similarly, the orthogonality-preserving step \eqref{3step_disj_ExplicitEuler} is also equivalent to the following optimization problem
		\begin{equation}
			\begin{cases}
				u_i^{n+1} =  \min\limits_{u_i(\bx)} \int_\Omega ({u}_i(\bx) - \bar{u}_i^{n+1}(\bx))^2/(2\tau)-\sum_{j\in [k],j\ne i} \eta_{ij}^{n+1} \bar{u}_j^{n+1}{u}_i(\bx)\, d\bx, \quad  i \in [k], \\[1 em]
				{u}_i^{n+1} \ge 0, \quad {u}_i^{n+1} {u}_j^{n+1} = 0, \quad  i \in [k],\forall j \in [k], j \ne i.
			\end{cases}
		\end{equation}
		Since the orthogonal constraint ${u}_i^{n+1} {u}_j^{n+1} = 0$ will not  satisfy the property of convexity,   then there exists multiple solutions for this step.  
		
	\end{remark}

	The steps above are summarized in Algorithm~\ref{Alg:3step_Type1}.
	\begin{algorithm}[ht]
		\begin{algorithmic}
			\State {\bf Input:} {Let $\Omega$ be a given domain, $\tau > 0$,  $N_{\max}$  be the  maximum number of iteration  and $u^0 \in H_0^1(\Omega)$ be the initial condition.} 
			\State {\bf Output:} {$u_i^{n+1}$, $i=1,\cdots, k$.}
			\State Initialize $n = 0$.
			\While{$n< N_{\max}$ \ }
			
			{\bf 1.  Diffusion Step.} Compute 
			\begin{equation}\label{Alg_3stepT1_S1}
				\tilde u^{n+1}_i = \e^{\tau\Delta} u^n_i.
			\end{equation}

			{\bf 2. Positivity and disjoint-preserving step.} 
			\begin{equation}\label{Alg_3stepT1_S2}
				\begin{split}
					&\hat u_i^{n+1}=\left\{
					\begin{array}{lcl}
						\tilde  u_i^{n+1}- \max\limits_{j\in[k],j\ne i}\{\tilde u_j^{n+1}\},	  &    & {\mbox{if} \quad  \tilde  u_i^{n+1}>0 \quad\mbox{and}\quad \tilde  u_i^{n+1}> \max\limits_{j\in[k],j\ne i}\{\tilde u_j^{n+1}\}},\\[.5em]
						0,   &    & {\mbox{otherwise} }.
					\end{array} \right.
				\end{split}
			\end{equation}

			{\bf 3. Norm-preserving step.} Update $u_i^{n+1}$ by
			\begin{equation}\label{Alg_3stepT1_S3}
				u_i^{n+1}=\frac{\hat u_i^{n+1}}{\|\hat u_i^{n+1}\|},\quad i=1,\cdots,k.
			\end{equation}
			
			\If{\eqref{stop_cond} is true}
			
			Stop the iteration.
			
			\EndIf

			Set $n = n+1$.

			\EndWhile
		\end{algorithmic}
		\caption{First type of three-step  numerical algorithm. } \label{Alg:3step_Type1}
	\end{algorithm}

We shall modify the {\bf Step-2} from the first type three-step numerical scheme \eqref{Alg2_diff:1}-\eqref{3step_disj_norm} and propose 
	the following  second type of three-step numerical scheme  for solving \eqref{norm:1}-\eqref{norm:4}.
	
	{\bf Step-1}  and  {\bf Step-3}	 are exactly the same with the first type of the first type three-step numerical scheme \eqref{Alg2_diff:1}-\eqref{3step_disj_norm}. We implement the new step-2 as following:
	
	{\bf Step-2:} (Positivity-preserving and orthogonality-preserving step) Given $\tilde u_i^{n+1}$ for $ i\in[k]$, 
	we solve
	\begin{equation}\label{3step_disj_BackwEuler}
		\begin{split}
			\hat u_i^{n+1}-\tilde u_i^{n+1}=\tau&\lambda_i^{n+1}+\tau\sum\limits_{j\in [k],j\ne i}\eta_{ij}^{n+1}\frac{\hat u_j^{n+1}+\tilde u_j^{n+1}}{2},\quad i\in [k],\\
			\lambda_i^{n+1}\ge 0,\quad	\lambda_i^{n+1}	\hat u_i^{n+1}=0,\quad  
			&\hat u_i^{n+1}\ge 0,\quad \hat u_i^{n+1}\hat u_j^{n+1}=0, \quad i\in[k],\quad \forall j\in [k], \quad j\ne i.
		\end{split}
	\end{equation}

	\begin{lemma}\label{lem:3stepT2sol}
Given $\tilde{u}_i^{n+1}$ for $i \in [k]$, denote $\tilde{u}_m^{n+1} = \max\limits_{j \in [k]} \{\tilde{u}_j^{n+1}\}$ with $m = \min\left\{\arg\max\limits_{j \in [k]} \{\tilde{u}_j^{n+1}\}\right\}$ and $\tilde{u}_q^{n+1} = \max\limits_{j \in [k], j \ne m} \{\tilde{u}_j^{n+1}\}$. Then, for $i \in [k]$, one solution of the scheme \eqref{3step_disj_BackwEuler} is
		\begin{equation}\label{3stepT2_Sol}
			\begin{split}
				&\hat u_i^{n+1}=\left\{
				\begin{array}{lcl}
					\tilde u_m^{n+1}-\left(\tilde u_m^{n+1}\tilde u_q^{n+1}\right)^\frac{1}{2},	  &    & {\mbox{if} \quad i=m},\\[.5em]
					0,   &    & {\mbox{otherwise} }.
				\end{array} \right.
			\end{split}
		\end{equation}
	\end{lemma}
	\begin{proof}
If $\tilde{u}_i^{n+1} \le 0$, we choose 
\[(\lambda_i^{n+1}, \eta_{ij}^{n+1}) = \left( -\tilde{u}_i^{n+1}/\tau, 0 \right),\]
for all $j \in [k]$, $j \ne i$, and get $\hat{u}_i^{n+1} = 0$, which agrees with \eqref{3stepT2_Sol}. Below, we demonstrate that \eqref{3stepT2_Sol} is the solution of \eqref{3step_disj_BackwEuler} for $\tilde{u}_i^{n+1} >0$ for $i \in [k]$. 

For the case of $k = 2$, we have
\begin{equation}\label{Sysk2Alg3Cons}
    \begin{split}
        \hat{u}_1^{n+1} &= \tilde{u}_1^{n+1} + \tau \lambda_1^{n+1} + \tau \eta_{12}^{n+1} \frac{\hat{u}_2^{n+1} + \tilde{u}_2^{n+1}}{2}, \\[.5em]
        \hat{u}_2^{n+1} &= \tilde{u}_2^{n+1} + \tau \lambda_2^{n+1} + \tau \eta_{12}^{n+1} \frac{\hat{u}_1^{n+1} + \tilde{u}_1^{n+1}}{2}, \\[.5em]
        \lambda_1^{n+1} &\ge 0, \quad \lambda_1^{n+1} \hat{u}_1^{n+1} = 0, \\[.5em]
        \lambda_2^{n+1} &\ge 0, \quad \lambda_2^{n+1} \hat{u}_2^{n+1} = 0, \\[.5em]
        \hat{u}_1^{n+1} &\ge 0, \quad \hat{u}_2^{n+1} \ge 0, \quad \hat{u}_1^{n+1} \hat{u}_2^{n+1} = 0.
    \end{split}
\end{equation}
If $\tilde{u}_1^{n+1} \ge \tilde{u}_2^{n+1}>0$, considering the condition
\[ \hat{u}_1^{n+1} \ge 0, \quad \hat{u}_2^{n+1} \ge 0, \quad \hat{u}_1^{n+1} \hat{u}_2^{n+1} = 0, \]
we can set $\hat{u}_1^{n+1} \ge 0$ and $\hat{u}_2^{n+1} = 0$, which implies
\begin{equation*}
\left\{
\begin{array}{l}
    \tilde{u}_1^{n+1} + \tau \lambda_1^{n+1} + \tau \eta_{12}^{n+1} \frac{\tilde{u}_2^{n+1}}{2} \ge 0, \\[.5em]
    \tilde{u}_2^{n+1} + \tau \lambda_2^{n+1} + \tau \eta_{12}^{n+1} \frac{\hat{u}_1^{n+1} + \tilde{u}_1^{n+1}}{2} = 0.
\end{array}
\right.
\end{equation*}
Combining with
\[ \lambda_1^{n+1} \ge 0, \quad \lambda_1^{n+1} \hat{u}_1^{n+1} = 0, \]
we have
\begin{equation*}
\left\{
\begin{array}{l}
    \tilde{u}_1^{n+1} + \tau \eta_{12}^{n+1} \frac{\tilde{u}_2^{n+1}}{2} \ge 0, \\[.5em]
    \tilde{u}_2^{n+1} + \tau \lambda_2^{n+1} + \tau \eta_{12}^{n+1} \frac{\hat{u}_1^{n+1} + \tilde{u}_1^{n+1}}{2} = 0.
\end{array}
\right.
\end{equation*}
One solution of the above system is 
\begin{equation}\label{Sysk2Alg3_Lagsol}
\left(\eta_{12}^{n+1}, \lambda_1^{n+1}, \lambda_2^{n+1}\right) = \left(-\frac{2}{\tau} \left(\frac{\tilde{u}_1^{n+1}}{\tilde{u}_2^{n+1}}\right)^{\frac{1}{2}}, 0, -\tilde{u}_2^{n+1} + \left(\frac{\tilde{u}_1^{n+1}}{\tilde{u}_2^{n+1}}\right)^{\frac{1}{2}} \left(2\tilde{u}_1^{n+1} - \left(\tilde{u}_1^{n+1} \tilde{u}_2^{n+1}\right)^{\frac{1}{2}}\right)\right).
\end{equation}
Substituting \eqref{Sysk2Alg3_Lagsol} into \eqref{Sysk2Alg3Cons} gives 
\begin{equation}
\left(\hat{u}_1^{n+1}, \hat{u}_2^{n+1}\right) = \left(\tilde{u}_1^{n+1} - \left(\tilde{u}_1^{n+1}\tilde{u}_2^{n+1}\right)^{\frac{1}{2}}, 0\right).
\end{equation}
If $0<\tilde{u}_1^{n+1} \le \tilde{u}_2^{n+1}$, following the same way, we obtain
\begin{equation*}
\left(\eta_{12}^{n+1}, \lambda_1^{n+1}, \lambda_2^{n+1}\right) = \left(-\frac{2}{\tau} \left(\frac{\tilde{u}_2^{n+1}}{\tilde{u}_1^{n+1}}\right)^{\frac{1}{2}}, -\tilde{u}_1^{n+1} + \left(\frac{\tilde{u}_2^{n+1}}{\tilde{u}_1^{n+1}}\right)^{\frac{1}{2}} \left(2\tilde{u}_2^{n+1} - \left(\tilde{u}_1^{n+1} \tilde{u}_2^{n+1}\right)^{\frac{1}{2}}\right),0\right),
\end{equation*}
which yields the solution 
\begin{equation*}
\left(\hat{u}_1^{n+1}, \hat{u}_2^{n+1}\right) = \left(0, \tilde{u}_2^{n+1} - \left(\tilde{u}_1^{n+1}\tilde{u}_2^{n+1}\right)^{\frac{1}{2}}\right).
\end{equation*}

Similarly, for the case where \( k \ge 3 \), we can determine the Lagrange multipliers as follows:
\begin{equation}\label{Alg3_Algdisj}
    \begin{split}
        \eta_{ij}^{n+1} = -\frac{2}{\tau} \left\{
        \begin{array}{lll}
            \left(\frac{\tilde{u}_m^{n+1}}{\tilde{u}_q^{n+1}}\right)^{\frac{1}{2}}, & \text{if } (i,j) = (m,q) \text{ or } (i,j) = (q,m), \\[1em]
            0, & \text{if } i = m \text{ and } j \ne q \text{ or } j = m \text{ and } i \ne q, \\[1em]
            \max \left(\frac{\tilde{u}_i^{n+1}}{\tilde{u}_j^{n+1}}, \frac{\tilde{u}_j^{n+1}}{\tilde{u}_i^{n+1}}\right), & \text{otherwise},
        \end{array} \right.
    \end{split}
\end{equation}
and
\begin{equation}\label{Alg3_Algpos}
    \begin{split}
        \lambda_i^{n+1} = \left\{
        \begin{array}{lll}
            0, & \text{if } i = m, \\[.5em]
            -\tilde{u}_i^{n+1} - \frac{\tau}{2} \eta_{im}^{n+1} \left(\tilde{u}_m^{n+1} - \left(\tilde{u}_m^{n+1} \tilde{u}_q^{n+1}\right)^{\frac{1}{2}}\right) - \frac{\tau}{2} \sum_{j \in [k], j \ne i} \eta_{ij}^{n+1} \tilde{u}_j^{n+1}, & \text{otherwise}.
        \end{array} \right.
    \end{split}
\end{equation}
From \eqref{Alg3_Algdisj}, we find that 
\[
\lambda_i^{n+1} \ge -\tilde{u}_i^{n+1} + \tilde{u}_i^{n+1} \ge 0, \quad \text{for} \quad i \ne m,
\]
satisfying the KKT conditions in \eqref{3step_disj_BackwEuler}.
Substituting \eqref{Alg3_Algdisj} and \eqref{Alg3_Algpos} into \eqref{3step_disj_BackwEuler} yields the solution in \eqref{3stepT2_Sol}. Therefore, the proof is concluded.
\end{proof}
	
	\begin{algorithm}[H]
		\begin{algorithmic}
			\State {\bf Input:} {Let $\Omega$ be a given domain, $\tau > 0$,  $N_{\max}$  be the  maximum number of iteration  and $u^0 \in H_0^1(\Omega)$ be the initial condition.} 
			\State {\bf Output:} {$u_i^{n+1}$, $i=1,\cdots, k$.}
			\State Initialize $n = 0$.
			\While{$n< N_{\max}$ \ }
			
			{\bf 1.  Diffusion Step.} Compute 
			\begin{equation}\label{Alg_3stepT2_S1}
				\tilde u^{n+1}_i = \e^{\tau\Delta} u^n_i.
			\end{equation}

			{\bf 2. Orthogonality-preserving step.} 
			\begin{equation}\label{Alg_3stepT2_S2}
				\begin{split}
					&\hat u_i^{n+1}=\left\{
					\begin{array}{lcl}
						\tilde u_i^{n+1}-\left(\tilde u_i^{n+1}\max\limits_{\ell\in[k],\ell\ne i} \{\tilde u_\ell^{n+1}\}\right)^\frac{1}{2},  &    & {\mbox{if} \quad \tilde u_i^{n+1}>0\quad \mbox{and}  \quad i=\min\{\arg \max\limits_{\ell\in[k]} \tilde u_\ell^{n+1}\}},\\[.5em]
						0,   &    & {\mbox{otherwise} }.
					\end{array} \right.
				\end{split}
			\end{equation}
			
			{\bf 3. Norm-preserving step.} Update $u_i^{n+1}$ by
			\begin{equation}\label{Alg_3stepT2_S3}
				u_i^{n+1}=\frac{\hat u_i^{n+1}}{\|\hat u_i^{n+1}\|},\quad i=1,\cdots,k.
			\end{equation}
			
			\If{\eqref{stop_cond} is true}
			
			Stop the iteration.
			
			\EndIf

			Set $n = n+1$.

			\EndWhile
		\end{algorithmic}
		\caption{Second type of three-step numerical algorithm.. } \label{Alg:3step_Type2}
	\end{algorithm}
	Based on Lemma~\ref{lem:3stepT2sol}, we propose the three-step numerical algorithms described above.
\begin{theorem}
The three-step operator splitting Algorithm \ref{Alg:3step_Type1} and Algorithm \ref{Alg:3step_Type2} for the system \eqref{norm:1}-\eqref{norm:5} satisfy the following properties:
\begin{equation}\label{pros}
    u_i^{n+1} \ge 0, \quad \|u_i^{n+1}\| = 1, \quad u_i^{n+1} u_j^{n+1} = 0, \quad i \in [k], \; \forall j \in [k], \; j \ne i.
\end{equation}
\end{theorem}
\begin{proof}
The property \eqref{pros} can be easily observed, we omit the proof here.
\end{proof}

	\subsection{Energy decaying  scheme}
	The numerical algorithms proposed above do not satisfy the property of energy dissipation. To inherit this property for our gradient systems, it is natural to develop numerical algorithms satisfying a discrete energy dissipation law. The key idea of our numerical algorithms is  to regard  the property of  energy dissipation as a nonlinear global constraint, see \cite{ChSh23}. Then  by  introducing  an extra Lagrange multiplier $\sigma(t)$ which is independent of spatial variables,  the original gradient system can be formulated as follows
\begin{equation}\label{eng_disspSch}
    \begin{split}
        &\partial_t u_i = \Delta u_i + \xi_i u_i + \lambda_i + \sum\limits_{j=1}^{k-1} \eta_j u_j, \quad i \in [k-1], \\[.5em]
        &\partial_t u_k = \Delta u_k + \xi_k u_k + \lambda_k + \sum\limits_{j=2}^{k} \eta_j u_j, \\
        &u_i = \frac{u_i(\bx) + \Psi[u_i, \sigma](\bx,t) \sigma(t)}{\|u_i(\bx) + \Psi[u_i, \sigma](\bx,t) \sigma(t)\|}, \quad i \in [k], \\[.5em]
        &\frac{d}{dt} E(u) = -\sum\limits_{i=1}^k \|\partial_t u_i(\bx, t)\|^2, \\[.5em]
        &u_i u_j = 0, \quad j \in [k], \quad j \ne i, \quad i \in [k], \\[.5em]
        &u_i \ge 0; \quad \lambda_i \ge 0; \quad \lambda_i u_i = 0, \quad i \in [k],
    \end{split}
\end{equation}
	where the energy is defined by $E(u)=\frac 12\int_{\Omega} |\Grad u|^2d\bx$. We need to emphasis that, to preserve positvity and orthogonality,  function $\Psi$ is defined in  a very special form 
	\begin{equation}\label{varp}
		\Psi[u_i, \sigma](\bx,t) = \left\{
		\begin{array}{ll}
			1, & \text{if }\quad u_i(\bx, t) > 0 \quad\text{and }\quad u_i(\bx,t)+\sigma(t) > 0, \\[.5em]
			0, & \text{otherwise}.
		\end{array}
		\right.
	\end{equation}
%  \begin{equation}\label{varp}
%     \Psi[u_i, \sigma](\bx, t) = \mathbf{1}\left\{ u_i(\bx) > 0 \text{ and } u_i(\bx) + \sigma(t) > 0 \right\},
% \end{equation}
	Notice that when the Lagrange multiplier $\sigma(t) \equiv 0$, the above system reduces to the original gradient system \eqref{norm:1}-\eqref{norm:5}.

	We propose the following numerical algorithms for the system \eqref{eng_disspSch} :

	{\bf Step-1:} (Diffusion step) Given $u_i^n$ for $i\in[k]$,  solve $\tilde u_i^{n+1}$  from
	\begin{equation}\label{key}
		\partial_t \tilde u_i^{n+1}=\Delta \tilde u_i^{n+1}.
	\end{equation}
	The solution at $t_{n+1}$ can be expressed exactly  by
	\begin{equation*}
		\tilde u^{n+1}_i = \e^{\tau\Delta} u^n_i,\quad i\in[k].
	\end{equation*}
	
	{\bf Step-2:} (Positivity-preserving and orthogonality-preserving step) Given $\tilde u_i^{n+1}$ for $i\in[k]$,  solve $(\bar u_i^{n+1},\lambda_i^{n+1},\eta_i^{n+1})$  from
	\begin{equation}\label{en:step2}
		\begin{split}
			&\frac{\bar u_i^{n+1}-\tilde u_i^{n+1}}{\tau}=\lambda_i^{n+1}+\sum\limits_{j=1}^{k-1}\eta _j ^{n+1} u_j^{n},  \quad  i\in [k-1], \\[.5em]
			&  \bar u_i^{n+1} \bar u_j^{n+1}=0, \quad j\in[k], \quad j\ne i, \quad  i\in [k],\\[.5em]
			&\bar u_i^{n+1} \ge 0;\;\lambda_i^{n+1} \ge 0;\;\lambda_i^{n+1}\bar u_i^{n+1}=0,\quad i\in[k].
		\end{split}
	\end{equation}
	
	{\bf Step-3:} (Norm-preserving step) Given $\bar u_i^{n+1}$ for $ i\in[k]$, we solve $\hat u_i^{n+1}$ from
	\begin{equation}\label{en:step3}
		\hat u_i^{n+1}=\frac{\bar  u_i^{n+1}}{\|\bar  u_i^{n+1}\|},\quad i\in[k].
	\end{equation}	
	
	{\bf Step-4:} (Energy dissipating step) Given $\hat u_i^{n+1}$ for $i\in[k]$,  solve $ u_i^{n+1}$ and  $\sigma^{n+1}$  from
	\begin{equation}\label{S4:eneg}
		\begin{split}
			u_i^{n+1}&= \frac{\hat u_i^{n+1}+\Psi[\hat u_i^{n+1},\sigma^{n+1}]\sigma^{n+1}}{\Vert \hat u_i^{n+1}+\Psi[\hat u_i^{n+1},\sigma^{n+1}]\sigma^{n+1}\Vert },\quad i\in[k],\\ 
			&\frac{E^{n+1}-E^n}{\tau}=-\frac{1}{\tau^2}\sum\limits_{i=1}^k\| u_i^{n+1}- u_i^{n}\|^2.
		\end{split}
	\end{equation}	
	
	\begin{theorem}
		The numerical scheme \eqref{key}-\eqref{S4:eneg}  satisfies the property of positivity-preserving, orthogonality-preserving, i.e. \begin{equation}
			u_i^{n+1} \ge 0 ,\quad u_i^{n+1}u_j^{n+1}=0, \quad i,j \in[k],\quad  j\neq i. 
		\end{equation}
		It is also  energy dissipative with respective to time, in the sense that
		\begin{equation}
			\begin{split}
				\frac{E^{n+1}-E^n}{\tau}=-\frac{1}{\tau^2}\sum\limits_{i=1}^k\| u_i^{n+1}- u_i^{n}\|^2. 
			\end{split}
		\end{equation}	
	\end{theorem}
	\begin{proof}
		We can easily observe that the numerical scheme \eqref{key}-\eqref{S4:eneg} is energy dissipative and preserves the $L^2$ norm from equation \eqref{S4:eneg}.   From the definition of  \eqref{varp}, we can derive $u_i^{n+1}$ is non-negative for $\forall i \in [k]$.  From \eqref{en:step2} and \eqref{en:step2}, we can easily derive $\hat u_i^{n+1}\hat u_j^{n+1}=0$ for $j\ne i$.  Furthermore, if we have $\hat u_i^{n+1}=0$ for $\forall i \in [k]$, then  from \eqref{S4:eneg} and the definition of function $\Psi$ in  \eqref{varp}, then we can  obtain  $u_i^{n+1}=0$ for $\forall i \in [k]$. Then the property of orthogonality-preserving  can also be  satisfied 
		\begin{equation*}
			\quad u_i^{n+1}u_j^{n+1}=0, \quad i,j \in[k],\quad  j\neq i. 
		\end{equation*}
	\end{proof}

	The solutions for the first three steps are presented in Algorithms~\ref{Alg:3step_Type1} and Algorithms~\ref{Alg:3step_Type2}. For the last step,   the solution of   \eqref{S4:eneg} can be obtained by  solving a nonlinear algebraic equation for $\sigma^{n+1}$ taking on the following form:                    
	\begin{equation}\label{nonl_Alg2}
		F(\sigma^{n+1}):=E^{n+1}-E^n+\frac{1}{\tau}\sum\limits_{i=1}^k\| u_i^{n+1}- u_i^{n}\|^2=0,
	\end{equation}
	where $E^{n+1}$ is defined by 
	\begin{align}
		E^{n+1}=\frac 12\sum\limits_{i=1}^k\int_{\Omega}  |\Grad u_i ^{n+1}|^2 d\bx=\frac 12\sum\limits_{i=1}^k\int_{\Omega}  \Bigg |\nabla  \frac{\hat u_i^{n+1}+\Psi[\hat u_i^{n+1},\sigma^{n+1}]\sigma^{n+1}}{\Vert \hat u_i^{n+1}+\Psi[\hat u_i^{n+1},\sigma^{n+1}]\sigma^{n+1}\Vert }\Bigg|^2 d\bx.
	\end{align}
	To solve the  nonlinear algebraic equation \eqref{nonl_Alg2}, we shall use either the Newton's method or 
	the secant method  in the following form
	\begin{equation}\label{sigma_secant}
		\sigma_{s+1}=\sigma_s-\frac{F(\sigma_s)(\sigma_s-\sigma_{s-1})}{F(\sigma_s)-F(\sigma_{s-1})}, 
	\end{equation}
	with an  initial guess of $\sigma_0=-O(\tau^2)$ and $\sigma_1=0$, noticing that the solution $\sigma^{n+1}$ is an approximation to 0, see \cite{cheng2020global}. In Algorithms~\ref{Alg:3step_Type1Diss} and \ref{Alg:3step_Type2Diss}, we present the corresponding pseudocode.
	\begin{algorithm}[H]
		\begin{algorithmic}
			\State {\bf Input:} {Let $\Omega$ be a given domain, $\tau > 0$,  $N_{\max}$  be the  maximum number of iteration  and $u^0 \in H_0^1(\Omega)$ be the initial condition.} 
			\State {\bf Output:} {$u_i^{n+1}$, $i=1,\cdots, k$.}
			\State Set $n = 0$, $\sigma_0=-\tau^2$ and $\sigma_1=0$.
			\While{$n< N_{\max}$ \ }
			
			{\bf 1. } Compute $u^{n+1}_i$ by \eqref{Alg_3stepT1_S1}-\eqref{Alg_3stepT1_S3}. 
			
			{\bf 2. Energy decreasing step.}  Set $s=1$.
			
			\While{$\sum\limits_{i=1}^k\|\nabla u_i^{n+1}\|^2>\sum\limits_{i=1}^k\|\nabla u_i^{n}\|^2$}

			Compute $\sigma_{s+1}$ by \eqref{sigma_secant}.	Let $\hat u_i^{n+1}=u_i^{n+1}$ and recalculate $u_i^{n+1}$ by
			\begin{equation}
				\begin{split}
					u_i^{n+1} = \frac{\hat u_i^{n+1}+\Psi[\hat u_i^{n+1},\sigma_{s+1}]\sigma_{s+1}}{\Vert \hat u_i^{n+1}+\Psi[\hat u_i^{n+1},\sigma_{s+1}]\sigma_{s+1}\Vert }. 
				\end{split}
			\end{equation}

			Set $s=s+1$.
			\EndWhile
			
			\If{\eqref{stop_cond} is true}
			
			Stop the iteration.
			
			\EndIf

			Set $n = n+1$.

			\EndWhile
		\end{algorithmic}
		\caption{An energy decreasing scheme for Algorithm~\ref{Alg:3step_Type1}. }\label{Alg:3step_Type1Diss}
	\end{algorithm}
	
	\begin{algorithm}[H]
		\begin{algorithmic}
			\State {\bf Input:} {Let $\Omega$ be a given domain, $\tau > 0$,  $N_{\max}$  be the  maximum number of iteration  and $u^0 \in H_0^1(\Omega)$ be the initial condition.} 
			\State {\bf Output:} {$u_i^{n+1}$, $i=1,\cdots, k$.}
			\State Set $n = 0$, $\sigma_0=-\tau^2$ and $\sigma_1=0$.
			\While{$n< N_{\max}$ \ }
			
			{\bf 1. } Compute $u^{n+1}_i$ by \eqref{Alg_3stepT2_S1}-\eqref{Alg_3stepT2_S3}. 
			
			{\bf 2. Energy decreasing step.}  Set $s=1$.
			
			\While{$\sum\limits_{i=1}^k\|\nabla u_i^{n+1}\|^2>\sum\limits_{i=1}^k\|\nabla u_i^{n}\|^2$}

			Compute $\sigma_{s+1}$ by \eqref{sigma_secant}.	Let $\hat u_i^{n+1}=u_i^{n+1}$ and recalculate $u_i^{n+1}$ by
			\begin{equation}
				\begin{split}
					u_i^{n+1} = \frac{\hat u_i^{n+1}+\Psi[\hat u_i^{n+1},\sigma_{s+1}]\sigma_{s+1}}{\Vert \hat u_i^{n+1}+\Psi[\hat u_i^{n+1},\sigma_{s+1}]\sigma_{s+1}\Vert }. 
				\end{split}
			\end{equation}

			Set $s=s+1$.
			\EndWhile
			
			\If{\eqref{stop_cond} is true}
			
			Stop the iteration.
			
			\EndIf

			Set $n = n+1$.

			\EndWhile
		\end{algorithmic}
		\caption{An energy decreasing scheme for Algorithm~\ref{Alg:3step_Type2}. } \label{Alg:3step_Type2Diss}
	\end{algorithm}
 	\begin{remark}
		To obtain $\sigma^{n+1}$ from \eqref{nonl_Alg2}, we need to solve a nonlinear algebraic equation. The unique solvability of \eqref{nonl_Alg2} turns out to  be very challenging.  Hence, the unique solvability of  the numerical scheme \eqref{key}-\eqref{S4:eneg}  will be our future work. 
	\end{remark}
	\section{Numerical experiments}	\label{subsec:num_result}
	In this section, we present the numerical results of Algorithms~\ref{Alg:4step}-\ref{Alg:3step_Type2Diss} and test their efficiency for solving \eqref{gradient_sys}-\eqref{norm:4} with the periodic or Dirichlet boundary conditions.  The two-dimensional regular computational domain tested is \([-\pi, \pi]^2\). The initial guess \( u^0 \) for the \( k \)-partition problem \eqref{gradient_sys}-\eqref{norm:4} is generated as follows \cite{Wang22}:
	\begin{enumerate}
		\item Generate \( k \) random points (seeds) \( x_i \) for \( i = 1, \ldots, k \) within the computational domain.
		\item For each seed \( x_i \), compute its corresponding Voronoi cell \( A_i \). Define the initial guess \( u_i^0 \) for each partition as
		\begin{equation}\label{u0_guess}
			u_i^0 = \frac{\chi_{A_i}}{\left( \int_\Omega \chi_{A_i} \, dx \right)^{1/2}}, \quad i = 1, \ldots, k,   
		\end{equation}
		where \( \chi_{A_i} \) denotes the indicator function of the set \( A_i \).
	\end{enumerate}
\subsection{Partitions with periodic boundary conditions}
In this example, the effectiveness and robustness of Algorithms~\ref{Alg:4step}-\ref{Alg:3step_Type2Diss} for solving \eqref{gradient_sys}-\eqref{norm:4} with periodic boundary conditions are evaluated. We employ the fast Fourier transform method in the diffusion step. Partitions are tested on both 2-dimensional and 3-dimensional tori.
\subsubsection{Performance of Algorithms~\ref{Alg:4step}-\ref{Alg:3step_Type2}}
To test robustness, we first implement Algorithm~\ref{Alg:4step} with different time step sizes. The results in Figure~\ref{fig:EvolAlg_4stepdt001} and Figure~\ref{fig:EvolAlg_4stepdt01} reveal that Algorithm~\ref{Alg:4step} is robust against changes in time step size, requiring fewer iterations for larger $\tau$ to achieve a final partition approximation. Figure~\ref{fig:EvolAlg_3stepType1dt001} displays the evolution of partitions for Algorithm~\ref{Alg:3step_Type1}, showing similarities to the results in Figure~\ref{fig:EvolAlg_4stepdt001}. With the same initial conditions, the final partition obtained by Algorithm~\ref{Alg:3step_Type2} is identical, and thus, we omit its presentation here. To demonstrate the efficiency of Algorithms~\ref{Alg:4step}--\ref{Alg:3step_Type2}, we plot the numerical energy in Figure~\ref{fig:ECompAlg1_3}, which shows a rapid decrease in numerical energy during the initial steps. By selecting a large \(\tau\), the proposed algorithms can converge to a small neighborhood of a local minimum within a few dozen iterations.
	
	\begin{figure}[H]
		\centering
		\begin{tabular}{c|c|c|c|c|c|c|c}
			\hline 
			initial & 50 & 100  & 150 & 200  & 250& 300  & 500 \\
			\includegraphics[width = 0.09\textwidth, clip, trim = 4cm 1cm 3cm 1cm]{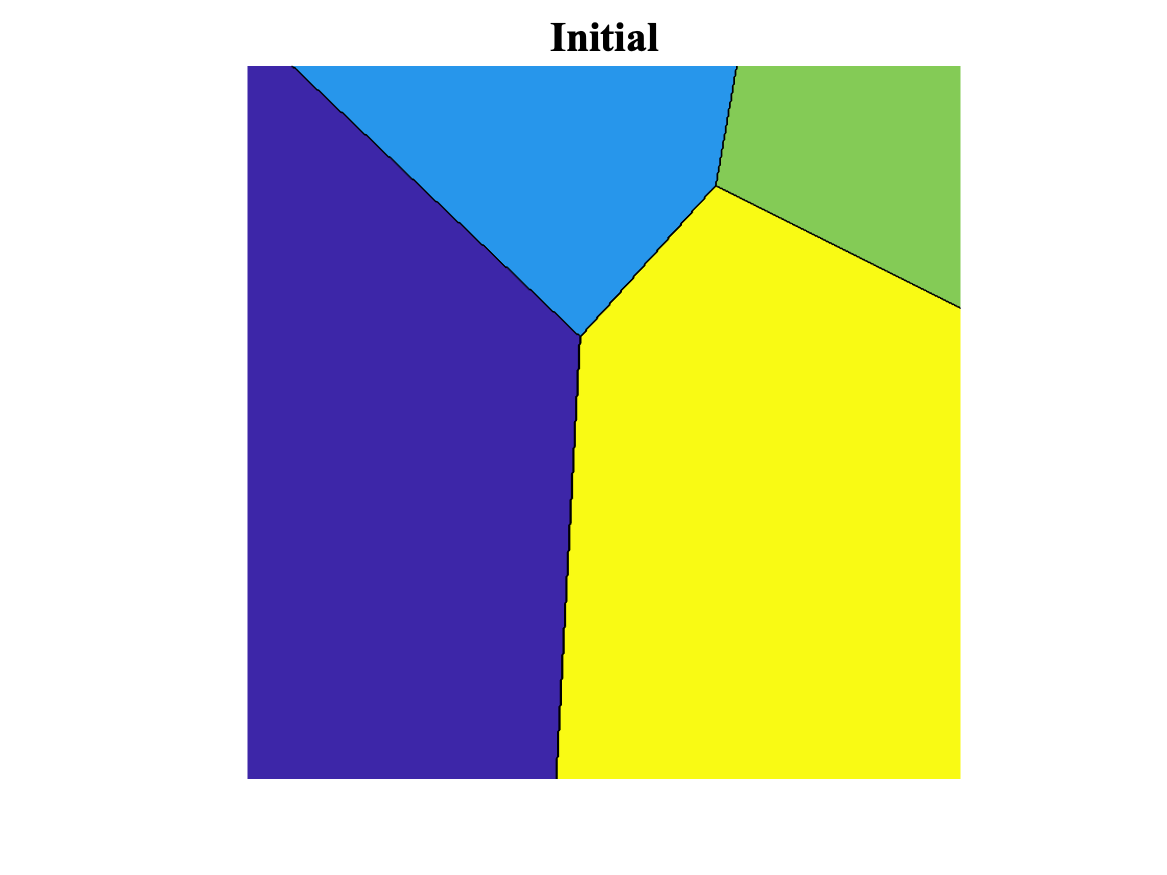}&  
			\includegraphics[width = 0.09\textwidth, clip, trim = 4cm 1cm 3cm 1cm]{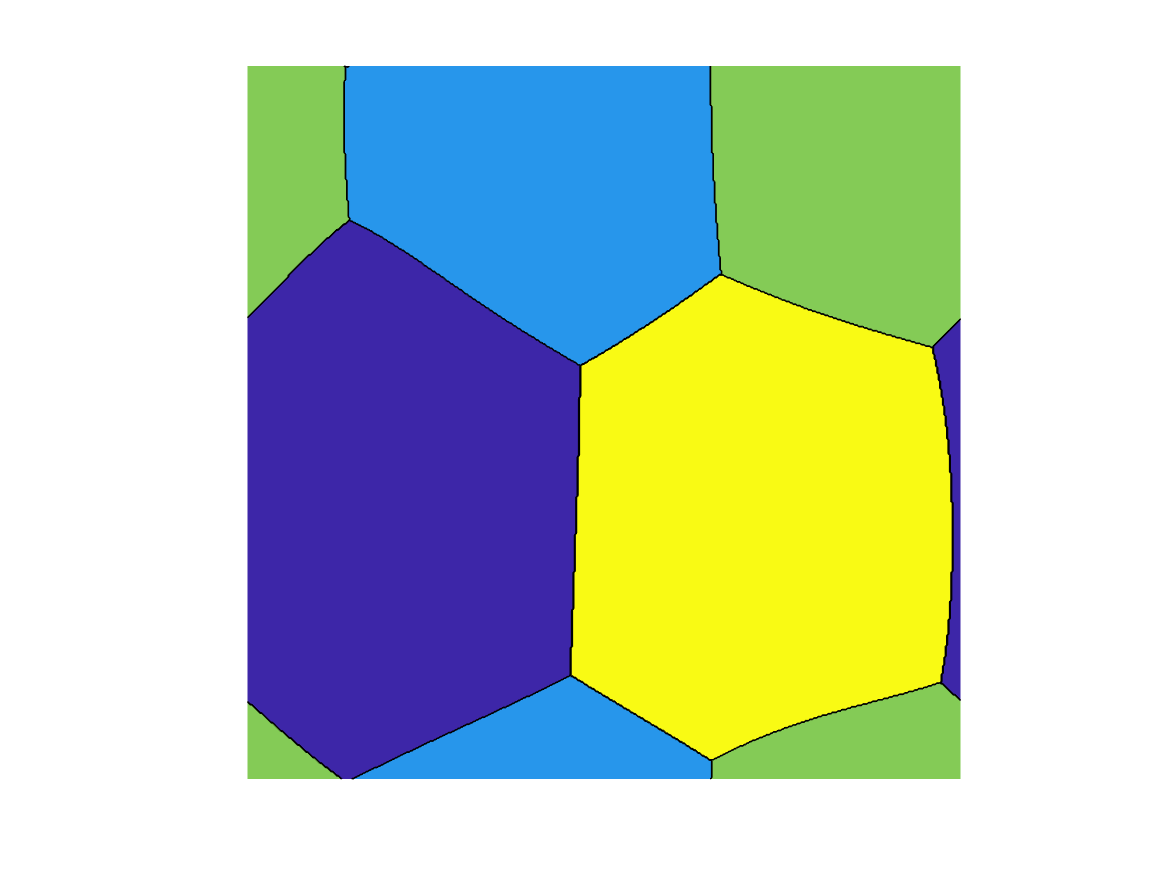}&
			\includegraphics[width = 0.09\textwidth, clip, trim = 4cm 1cm 3cm 1cm]{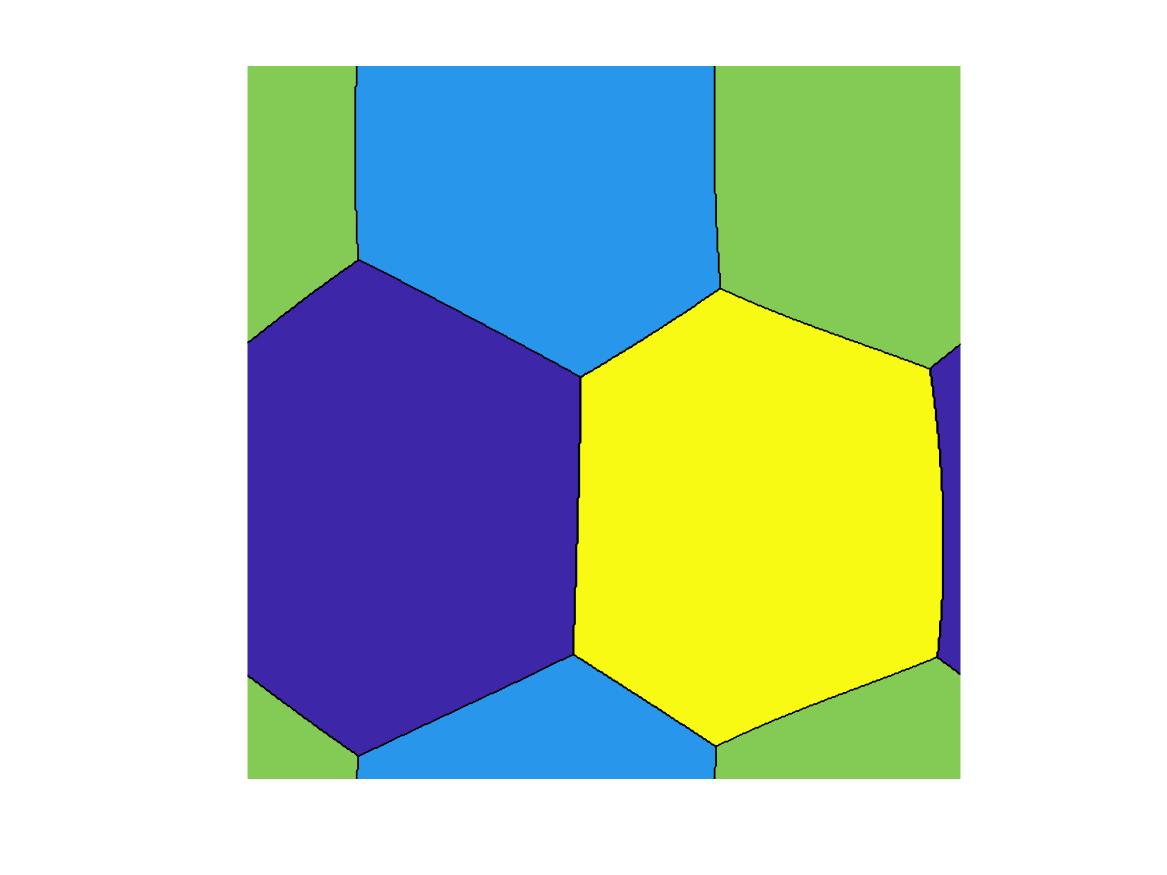}&
			\includegraphics[width = 0.09\textwidth, clip, trim = 4cm 1cm 3cm 1cm]{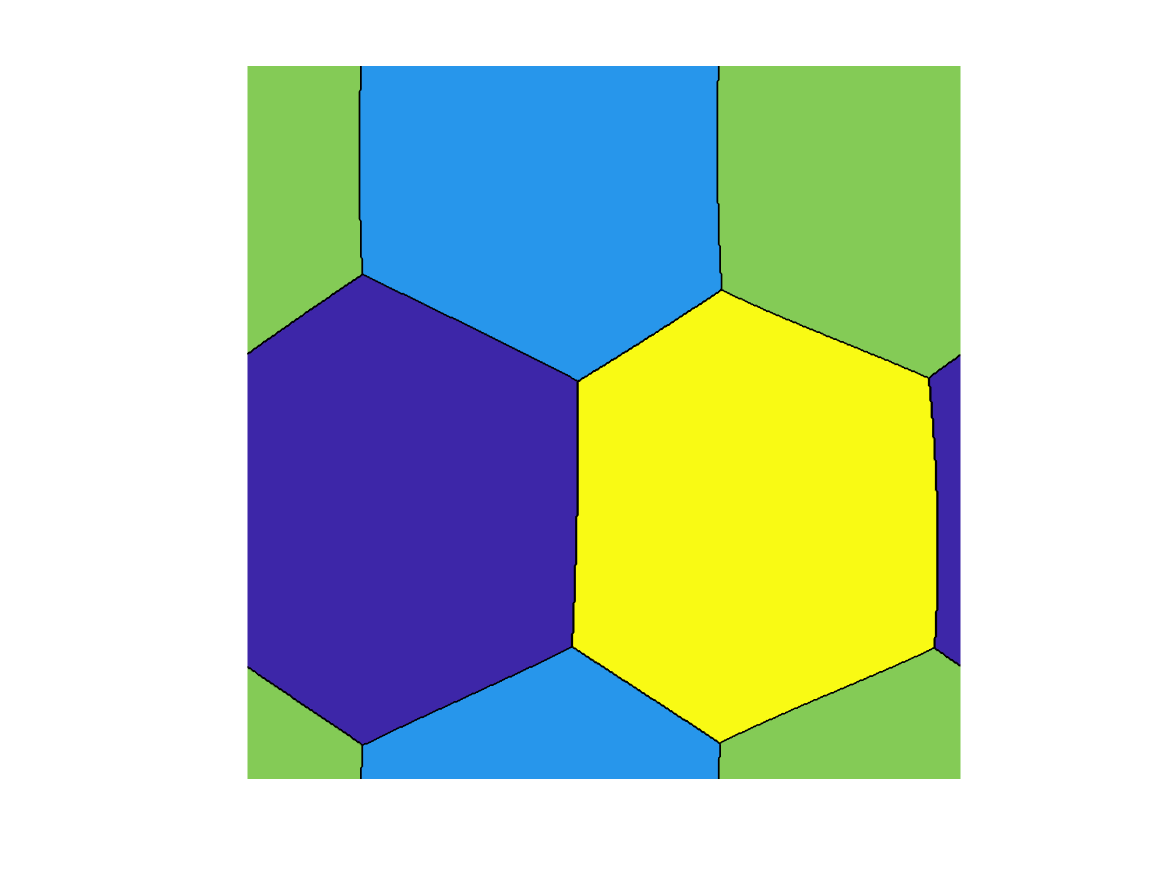}&
			\includegraphics[width = 0.09\textwidth, clip, trim = 4cm 1cm 3cm 1cm]{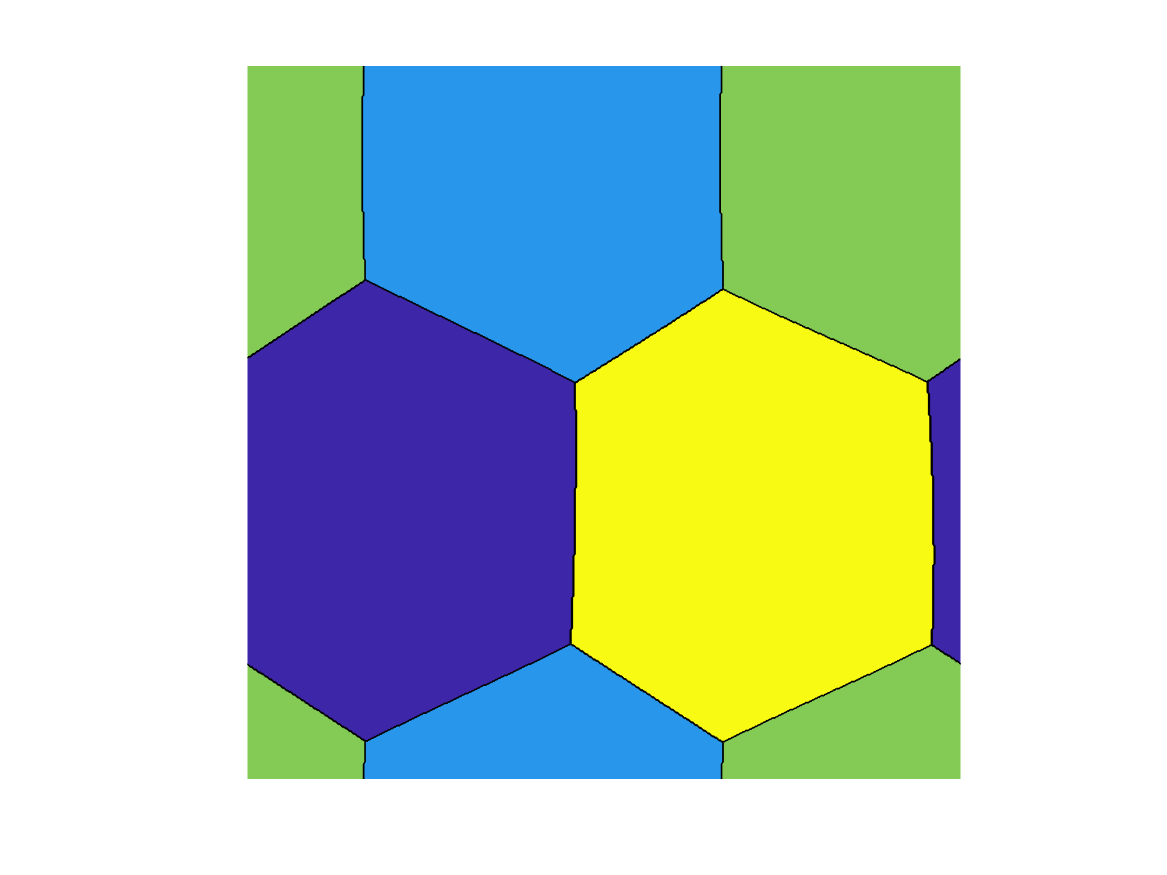}&
			\includegraphics[width = 0.09\textwidth, clip, trim = 4cm 1cm 3cm 1cm]{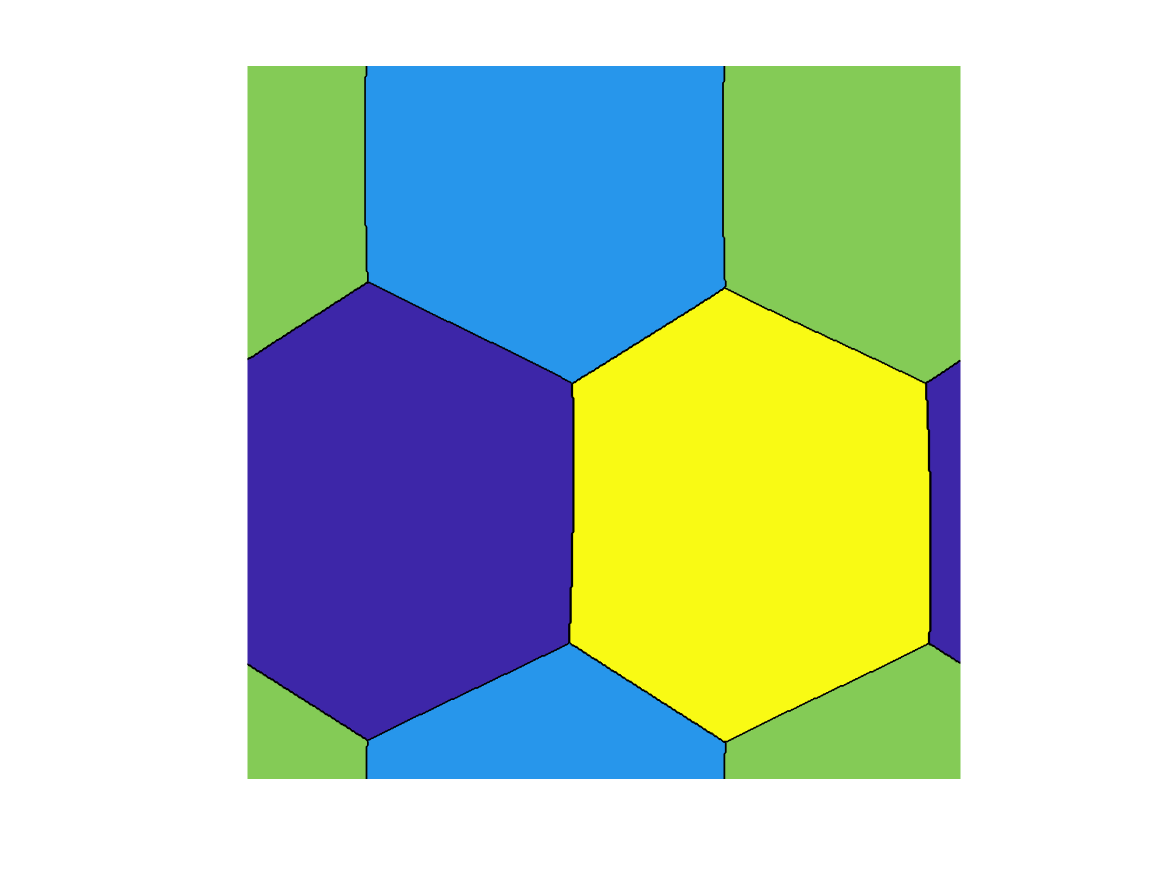}&
			\includegraphics[width = 0.09\textwidth, clip, trim = 4cm 1cm 3cm 1cm]{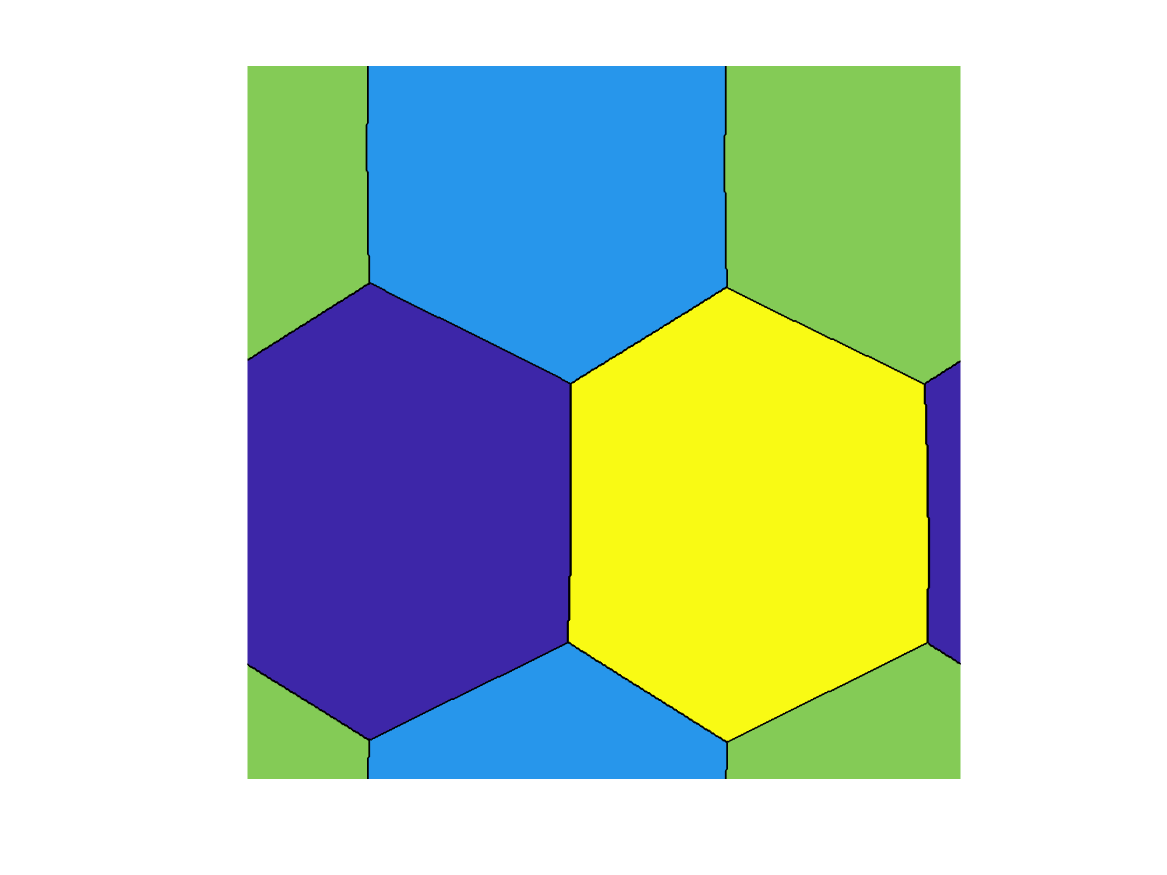}&
			\includegraphics[width = 0.09\textwidth, clip, trim = 4cm 1cm 3cm 1cm]{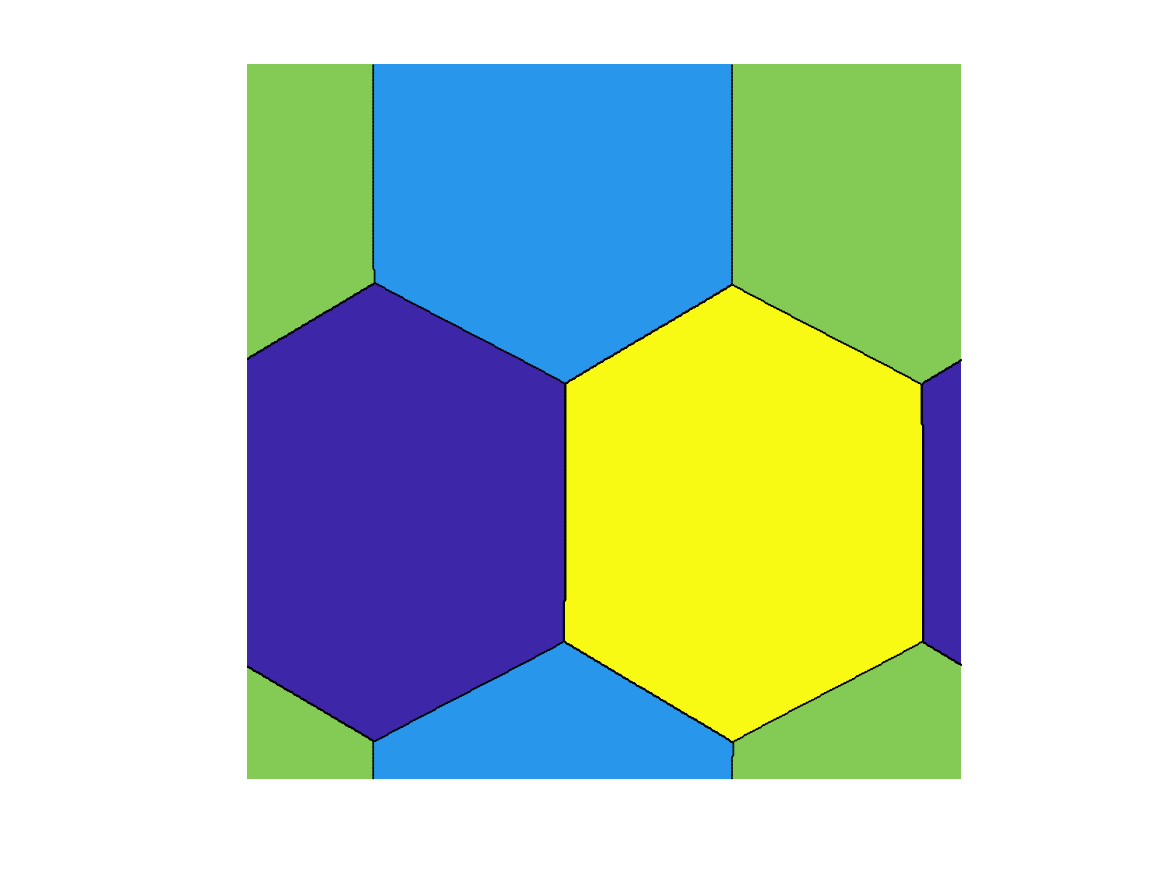} \\ 
			\hline
			\hline 
			initial & 100 & 200  & 300 & 400  & 500& 600  & 700 \\
			\includegraphics[width = 0.09\textwidth, clip, trim = 4cm 1cm 3cm 1cm]{initialm8.png}&  
			\includegraphics[width = 0.09\textwidth, clip, trim = 4cm 1cm 3cm 1cm]{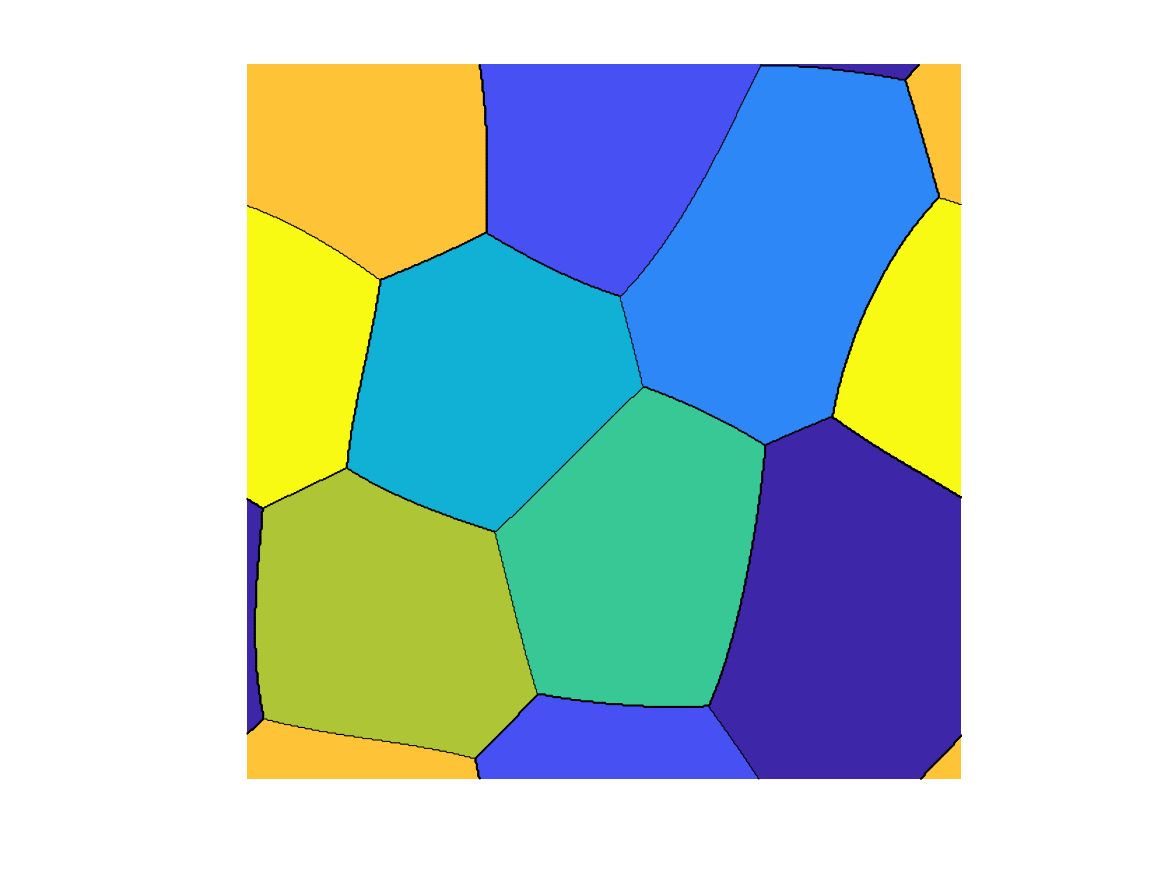}&
			\includegraphics[width = 0.09\textwidth, clip, trim = 4cm 1cm 3cm 1cm]{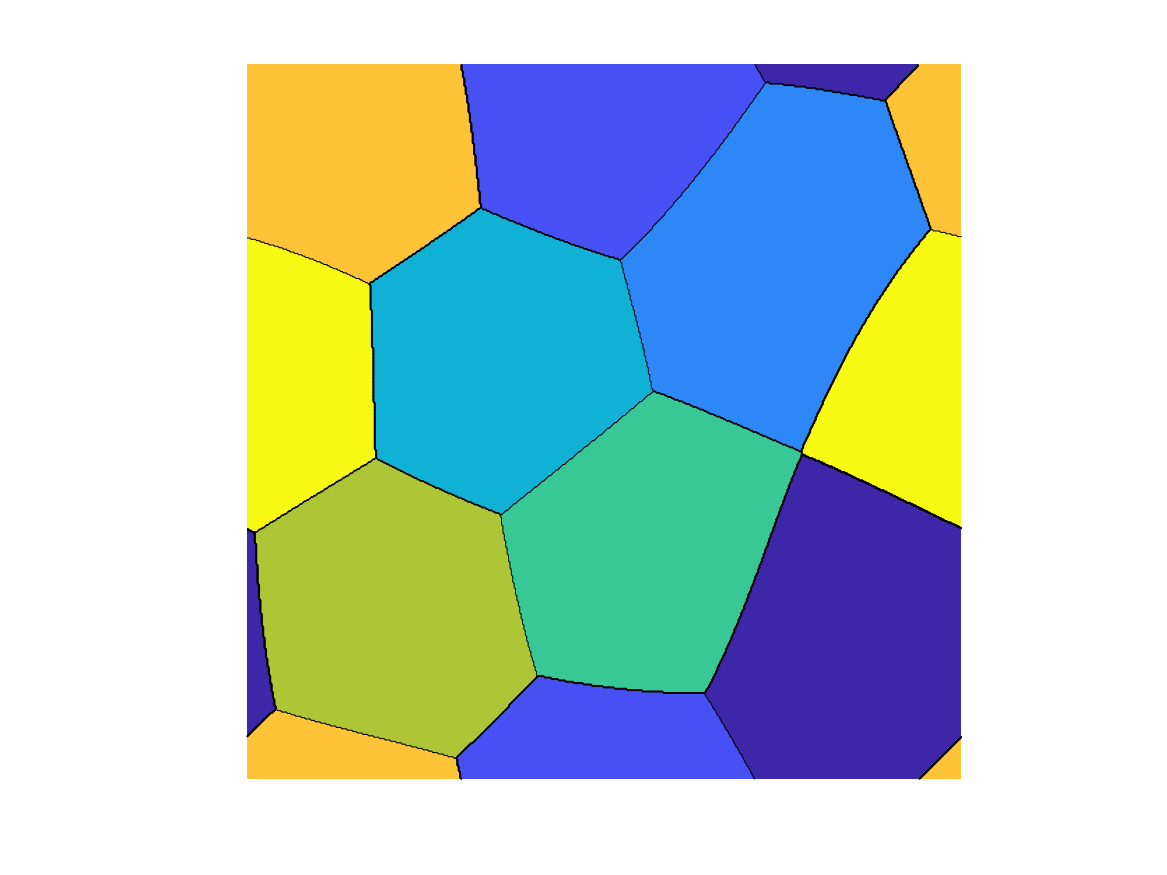}&
			\includegraphics[width = 0.09\textwidth, clip, trim = 4cm 1cm 3cm 1cm]{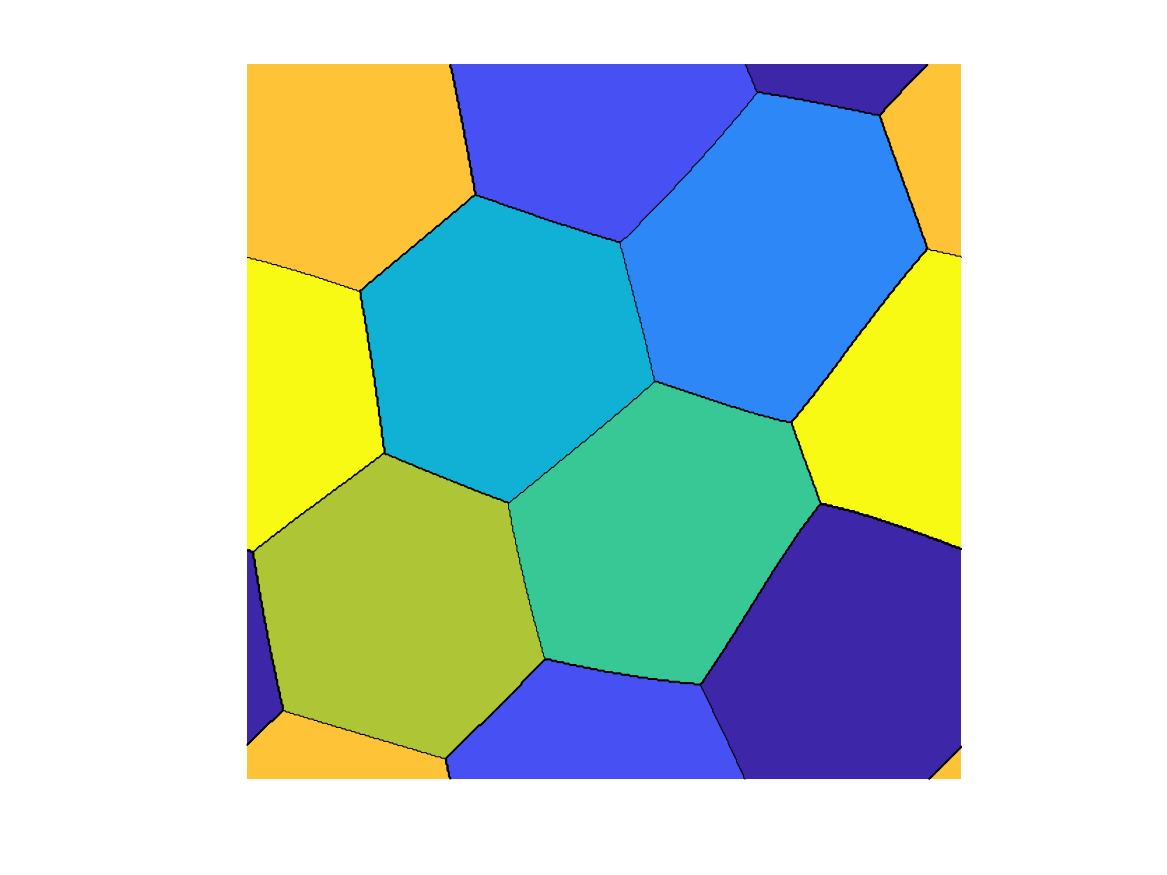}&
			\includegraphics[width = 0.09\textwidth, clip, trim = 4cm 1cm 3cm 1cm]{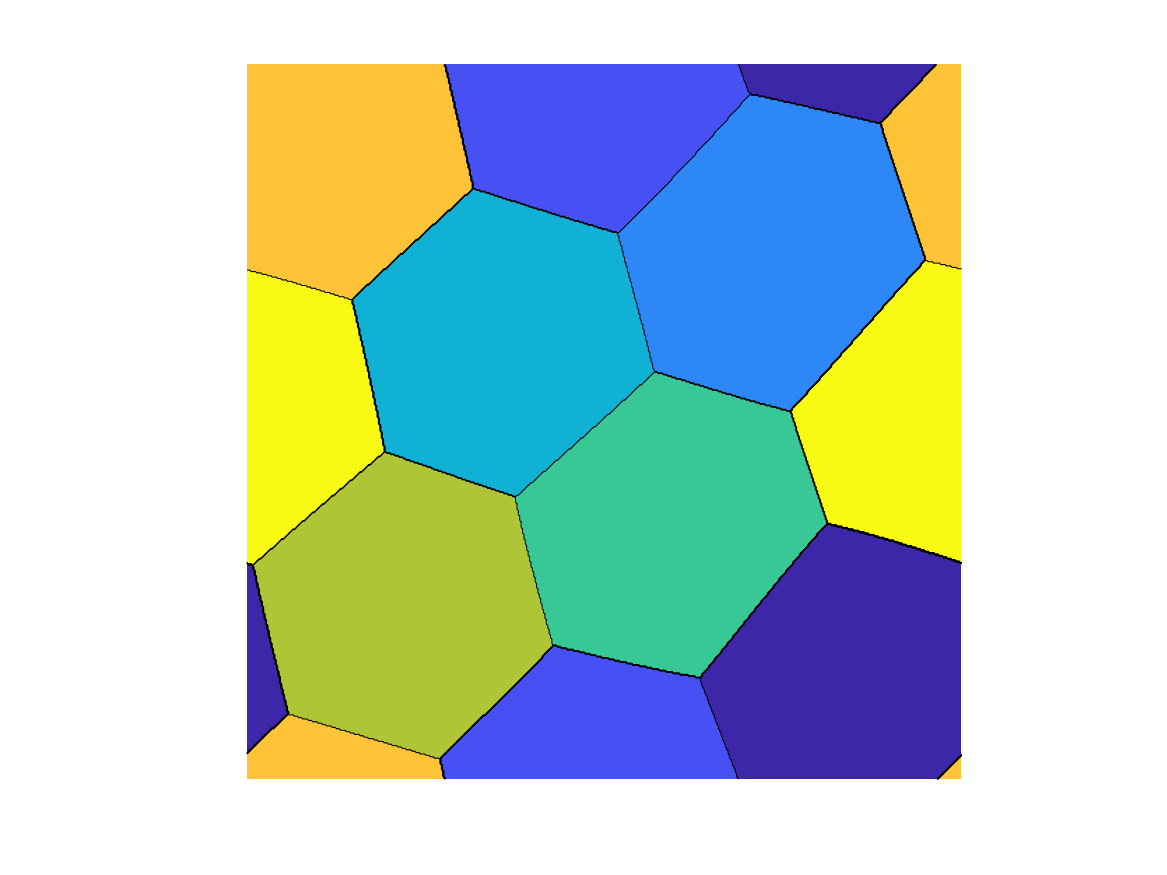}&
			\includegraphics[width = 0.09\textwidth, clip, trim = 4cm 1cm 3cm 1cm]{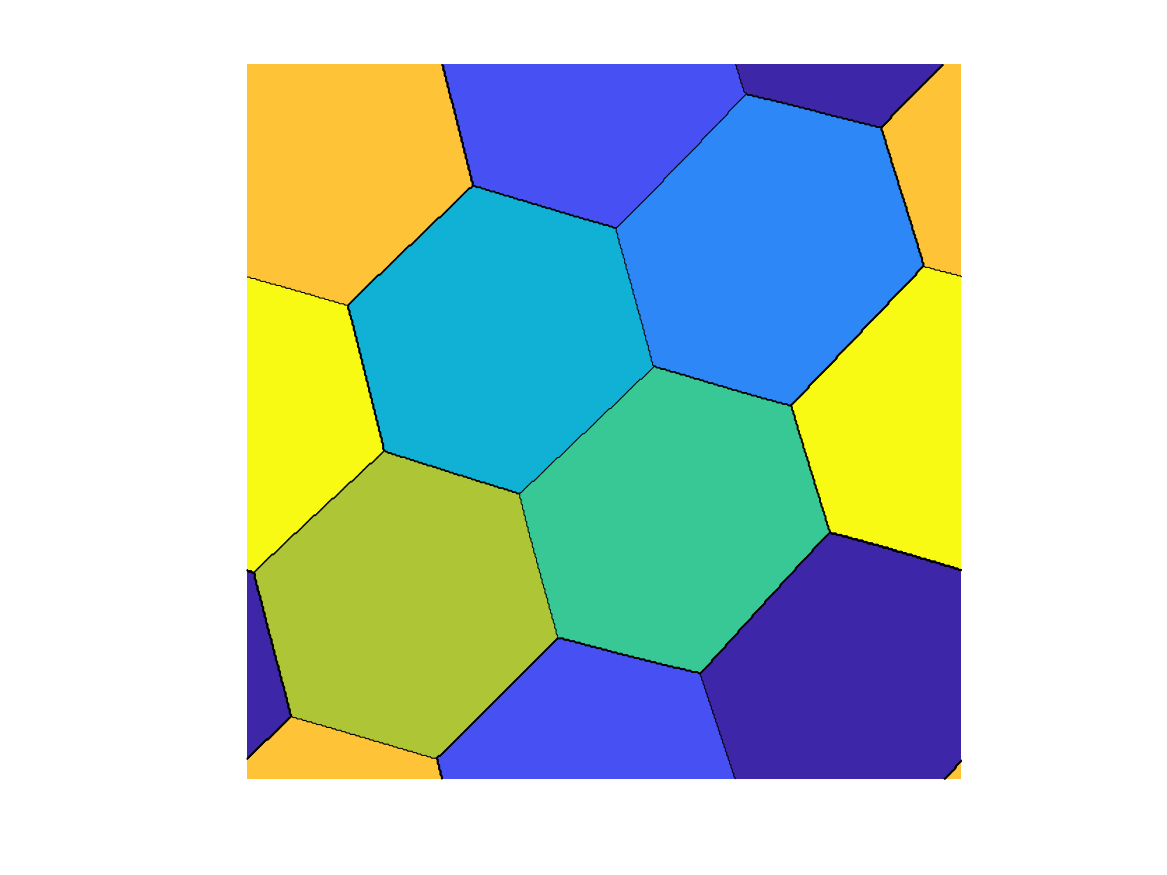}&
			\includegraphics[width = 0.09\textwidth, clip, trim = 4cm 1cm 3cm 1cm]{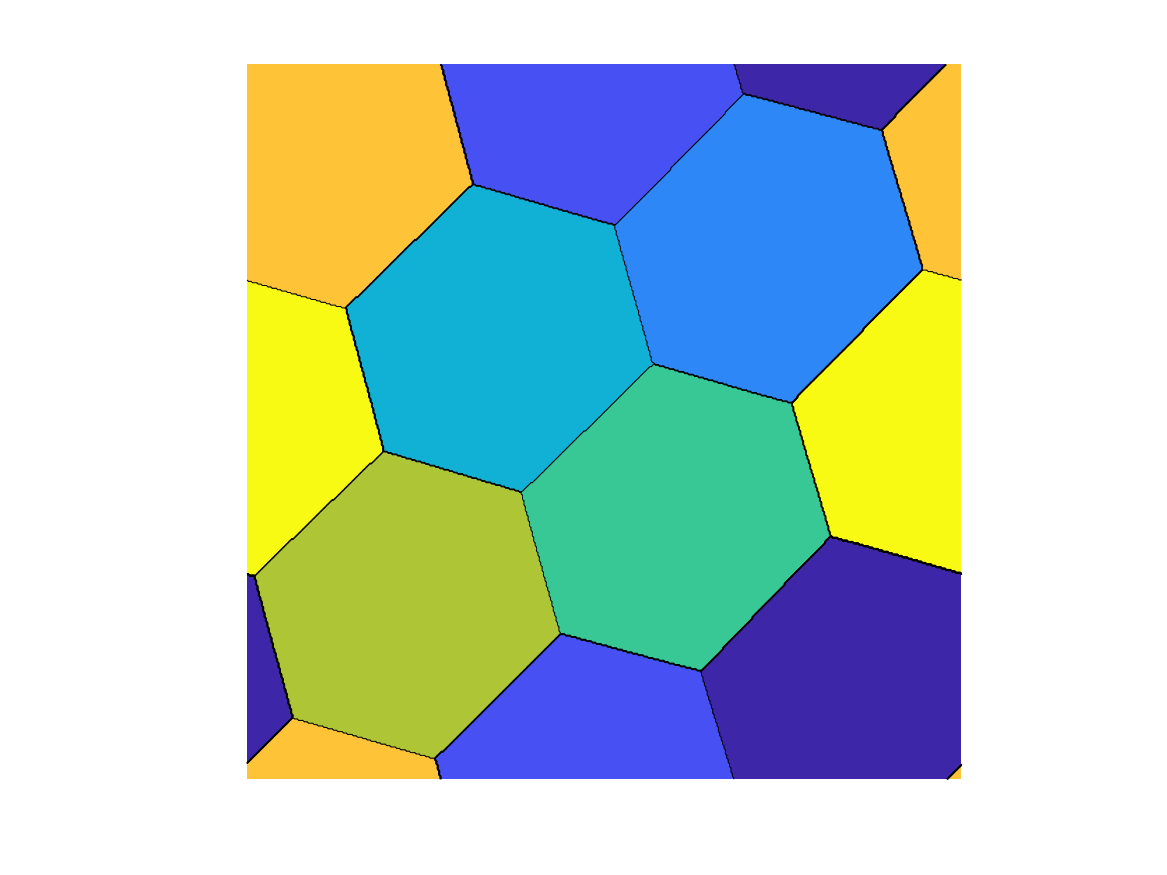}&
			\includegraphics[width = 0.09\textwidth, clip, trim = 4cm 1cm 3cm 1cm]{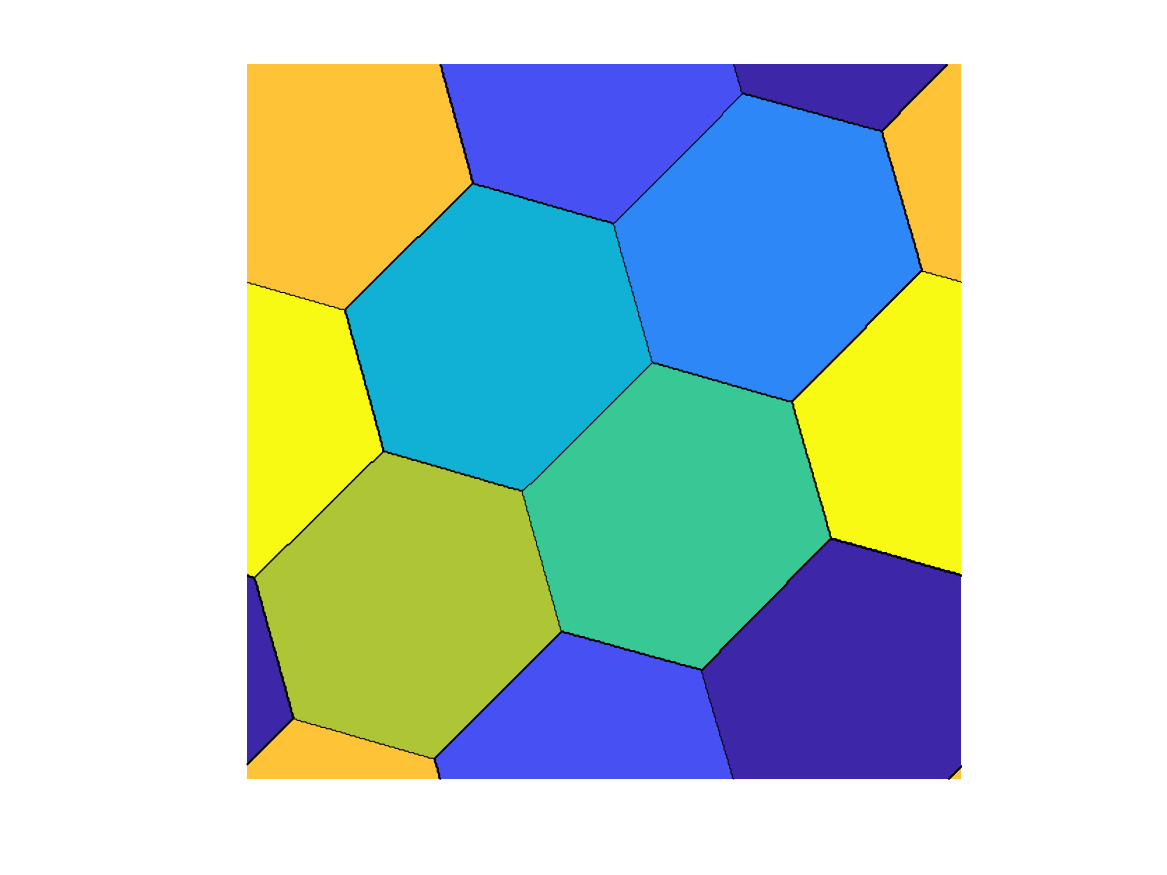} \\ 
			\hline
		\end{tabular}
		\caption{Snapshots of Algorithm~\ref{Alg:4step} at different iterations on a $512\times512$ discretized mesh with a time step size of $\tau=0.01$. The first row corresponds to $k=4$, while the second row corresponds to $k=8$.} \label{fig:EvolAlg_4stepdt001}
	\end{figure}
	\begin{figure}[H]
		\centering
		\begin{tabular}{c|c|c|c|c|c|c|c}
			\hline 
			initial & 5 & 10  & 15 & 20  & 25& 80  & 145 \\
			\includegraphics[width = 0.09\textwidth, clip, trim = 4cm 1cm 3cm 1cm]{figures_sq/initialm4.png}&  
			\includegraphics[width = 0.09\textwidth, clip, trim = 4cm 1cm 3cm 1cm]{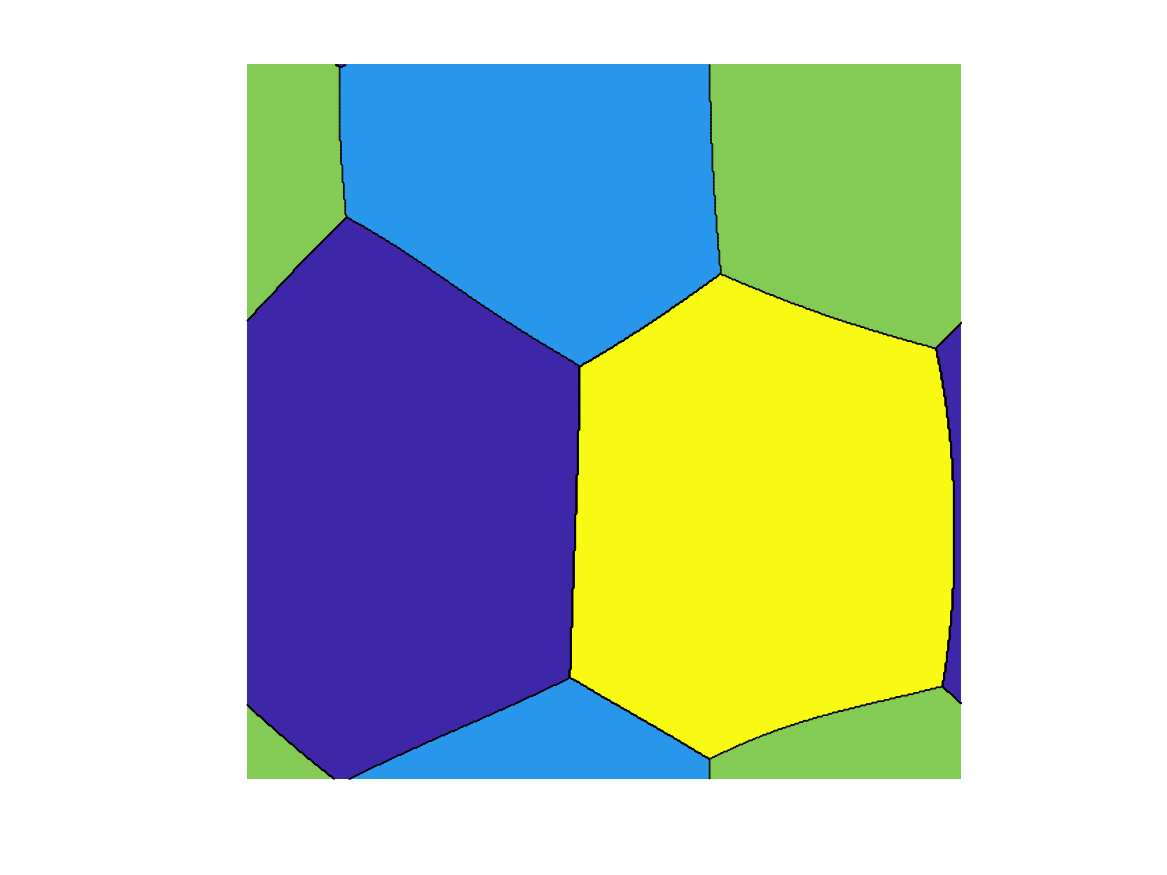}&
			\includegraphics[width = 0.09\textwidth, clip, trim = 4cm 1cm 3cm 1cm]{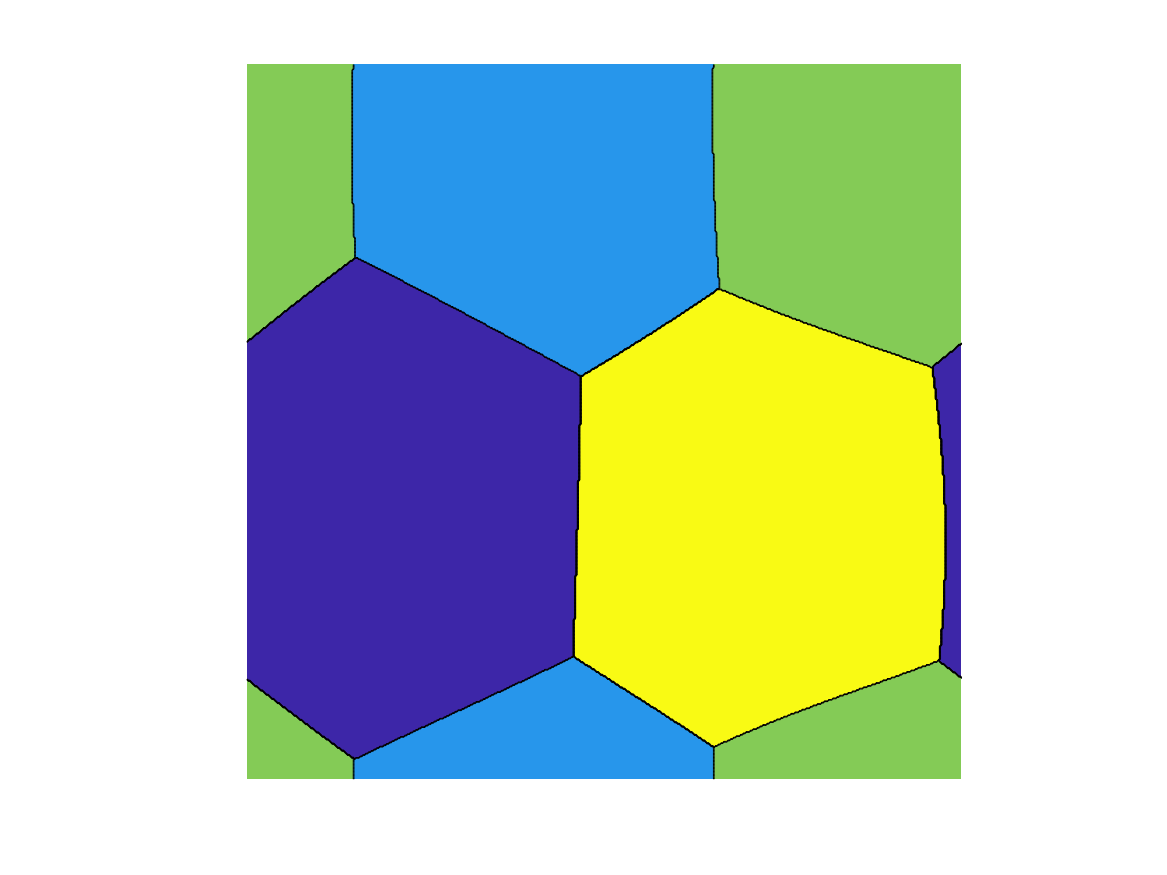}&
			\includegraphics[width = 0.09\textwidth, clip, trim = 4cm 1cm 3cm 1cm]{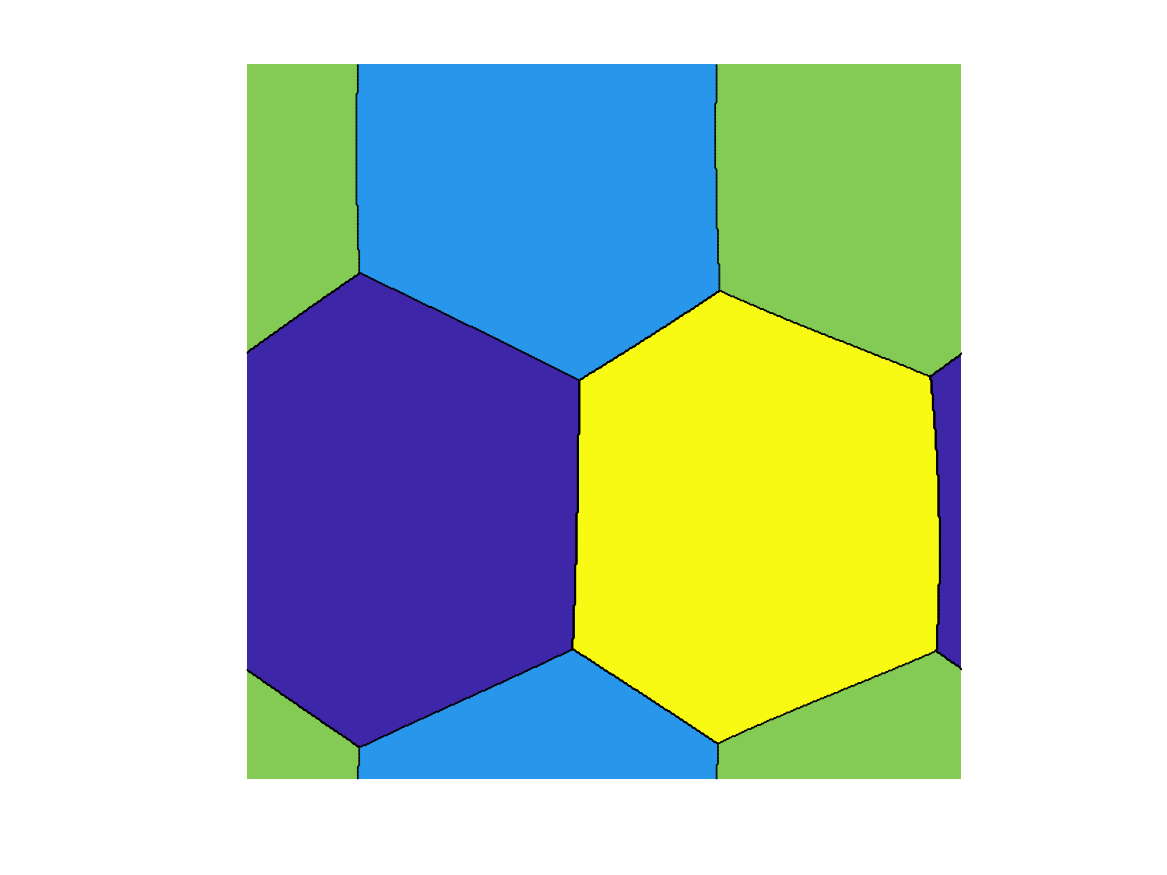}&
			\includegraphics[width = 0.09\textwidth, clip, trim = 4cm 1cm 3cm 1cm]{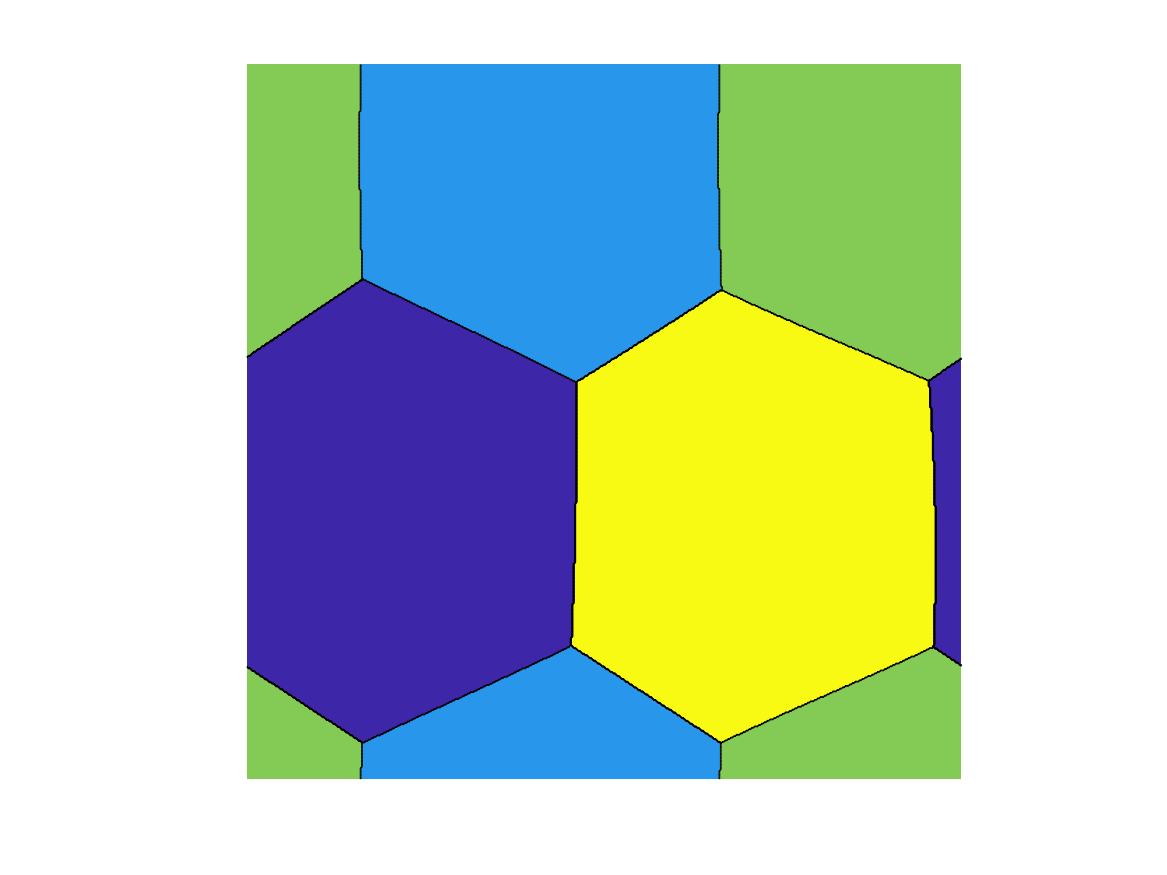}&
			\includegraphics[width = 0.09\textwidth, clip, trim = 4cm 1cm 3cm 1cm]{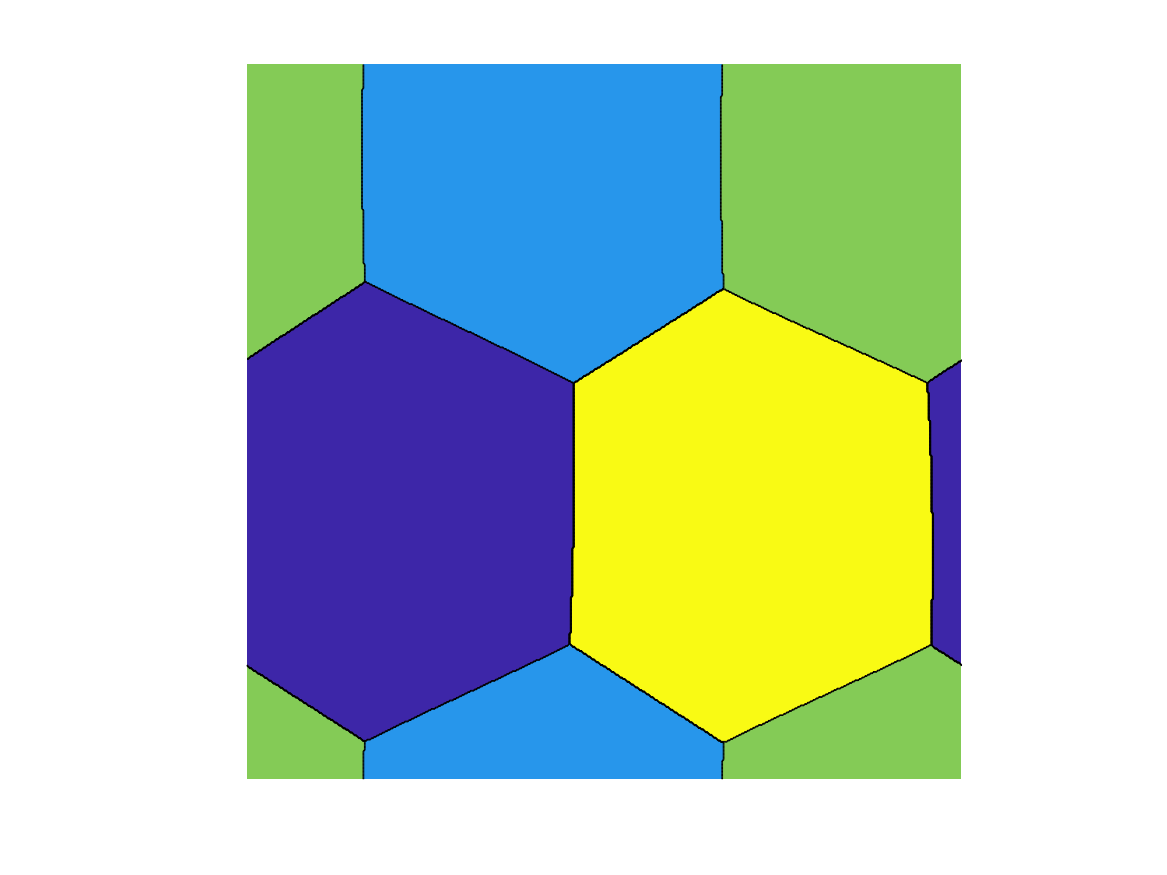}&
			\includegraphics[width = 0.09\textwidth, clip, trim = 4cm 1cm 3cm 1cm]{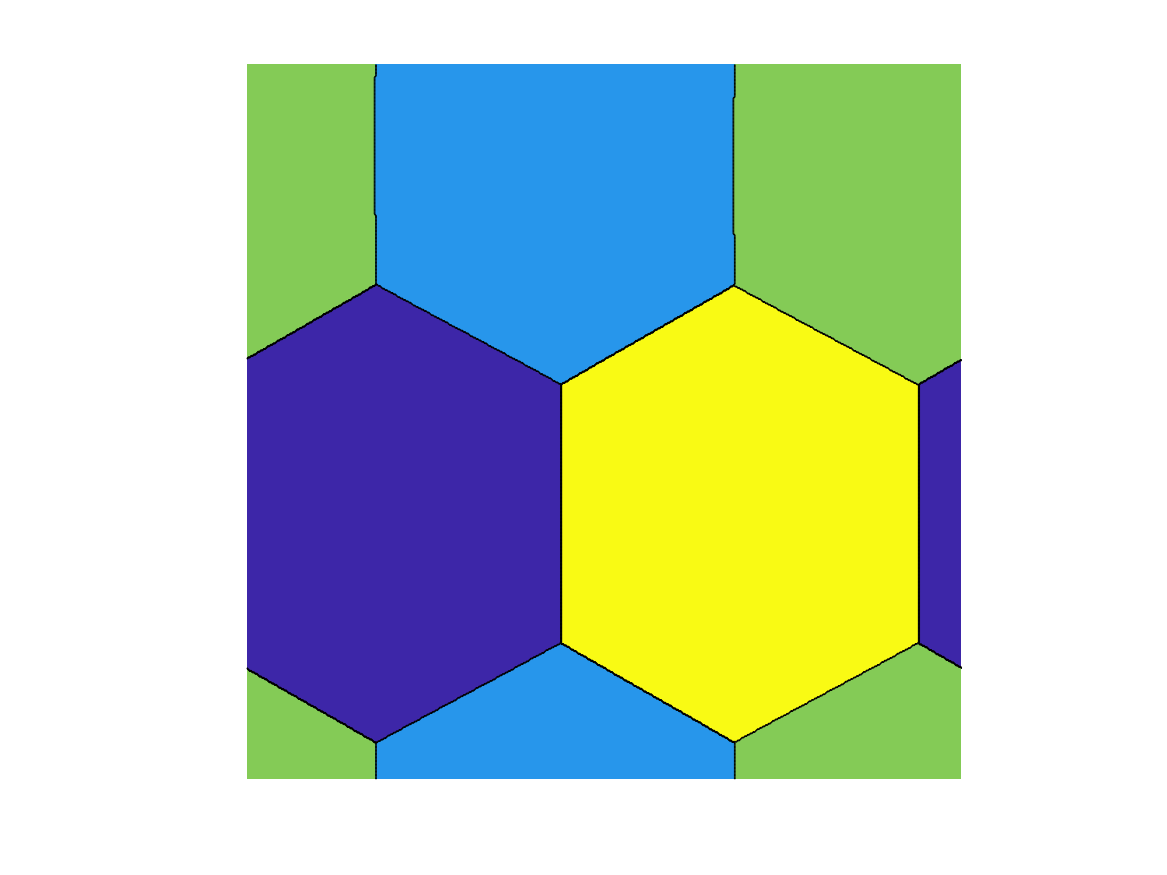}&
			\includegraphics[width = 0.09\textwidth, clip, trim = 4cm 1cm 3cm 1cm]{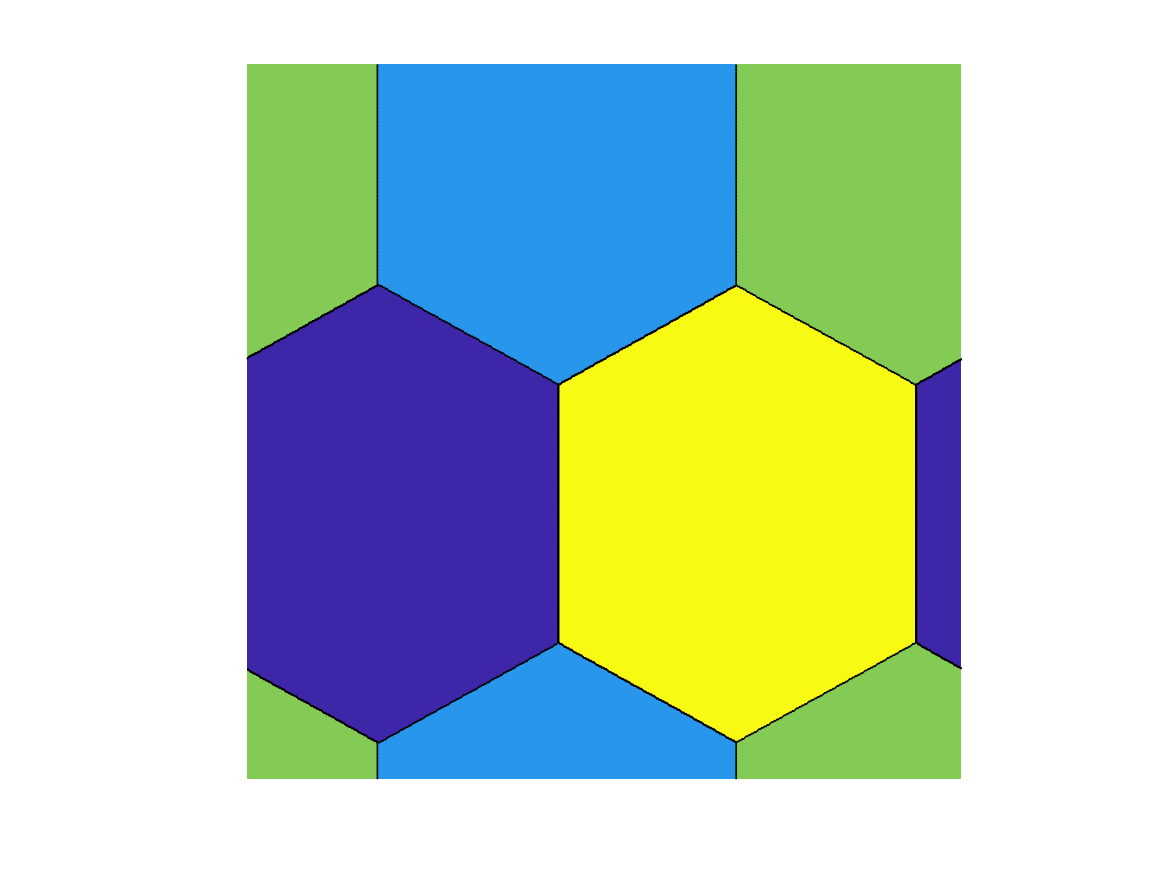} \\ 
			\hline
			\hline 
			initial & 5 & 15  & 25 & 35  & 45& 65  & 120 \\
			\includegraphics[width = 0.09\textwidth, clip, trim = 4cm 1cm 3cm 1cm]{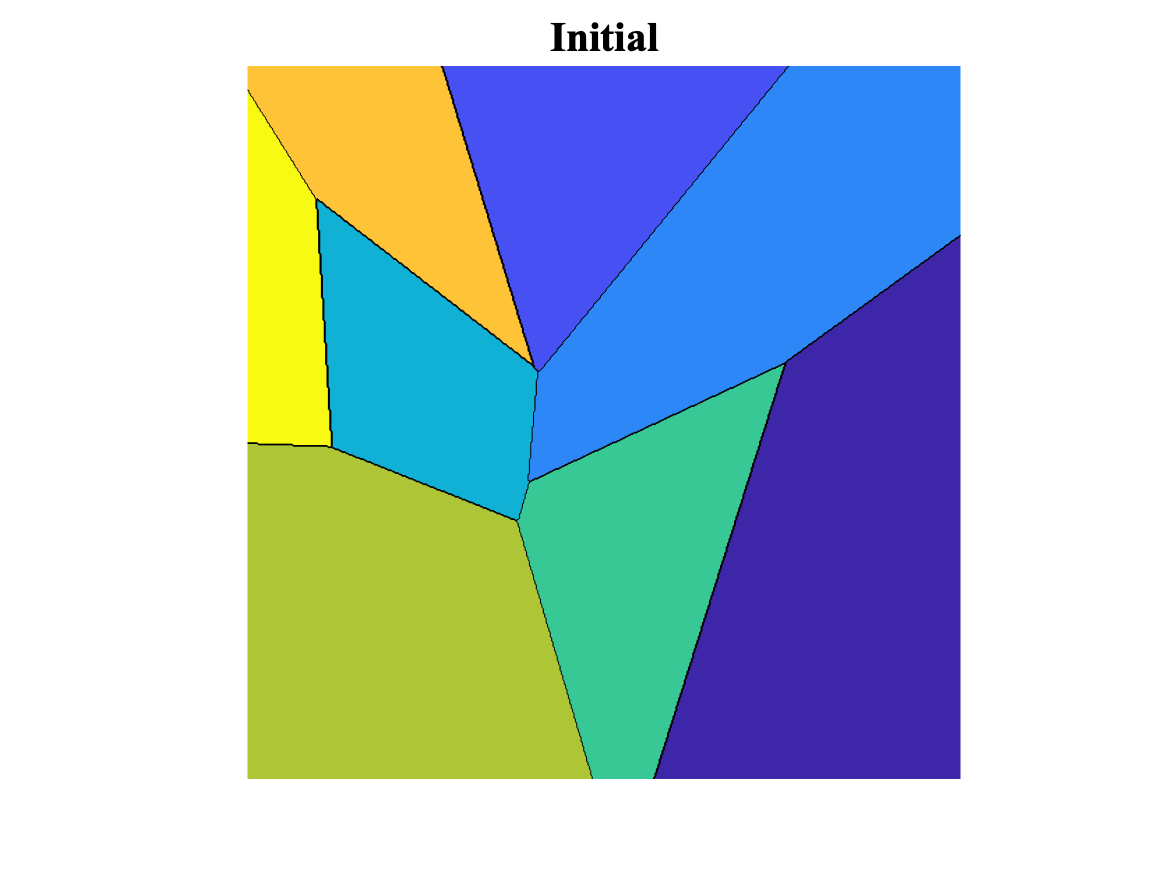}&  
			\includegraphics[width = 0.09\textwidth, clip, trim = 4cm 1cm 3cm 1cm]{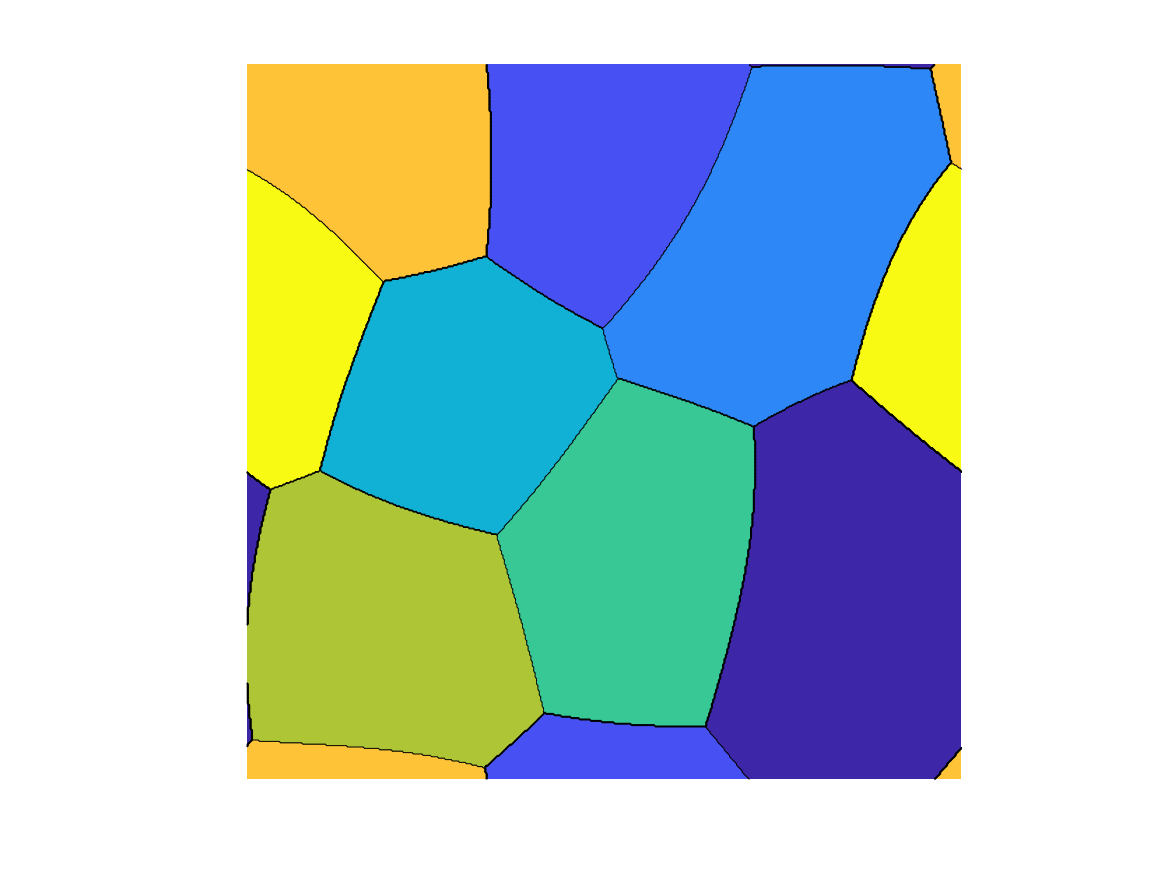}&
			\includegraphics[width = 0.09\textwidth, clip, trim = 4cm 1cm 3cm 1cm]{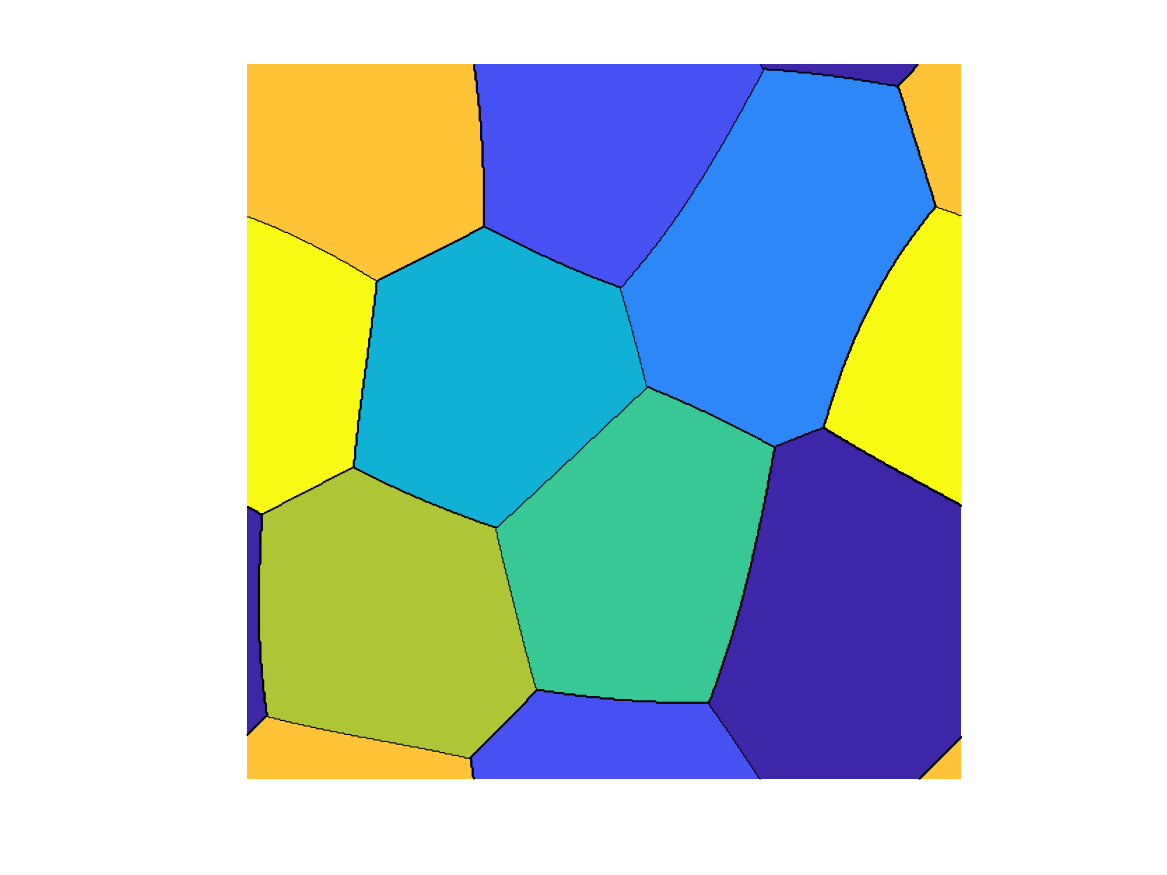}&
			\includegraphics[width = 0.09\textwidth, clip, trim = 4cm 1cm 3cm 1cm]{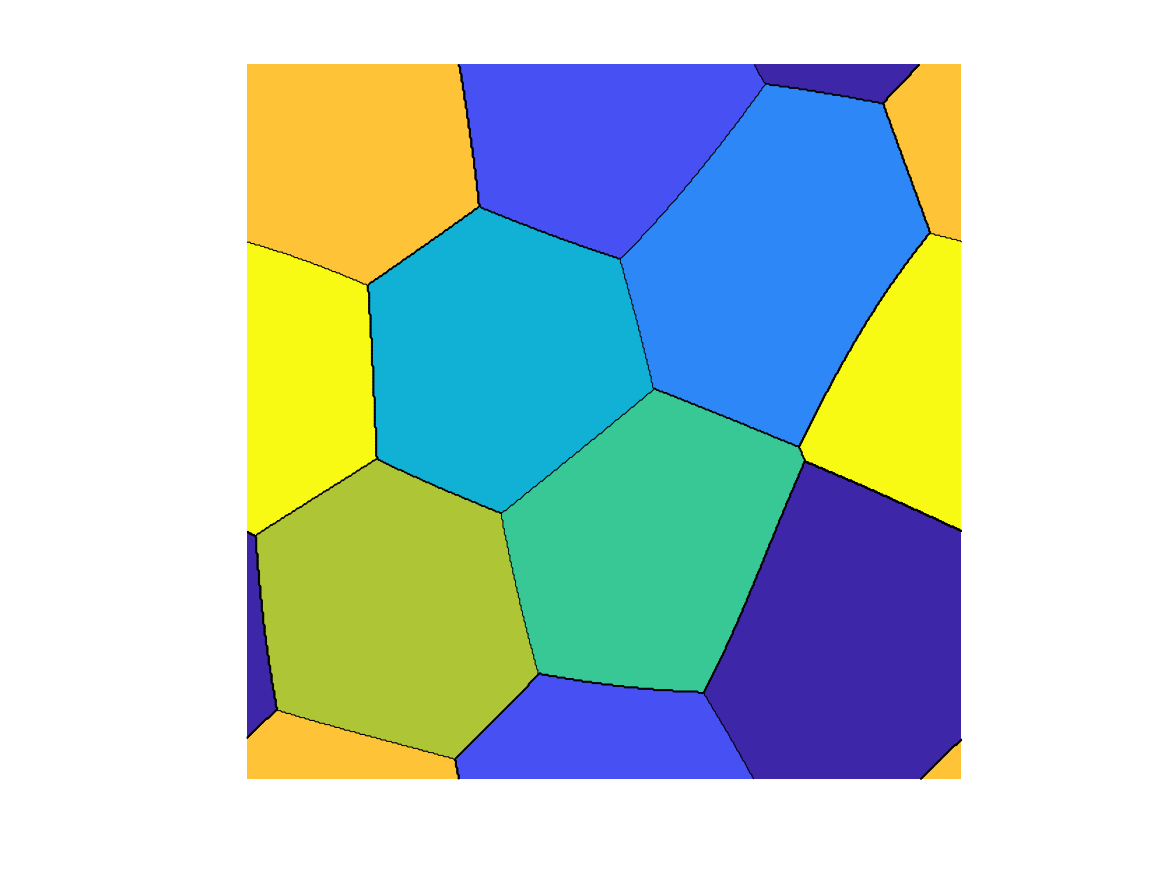}&
			\includegraphics[width = 0.09\textwidth, clip, trim = 4cm 1cm 3cm 1cm]{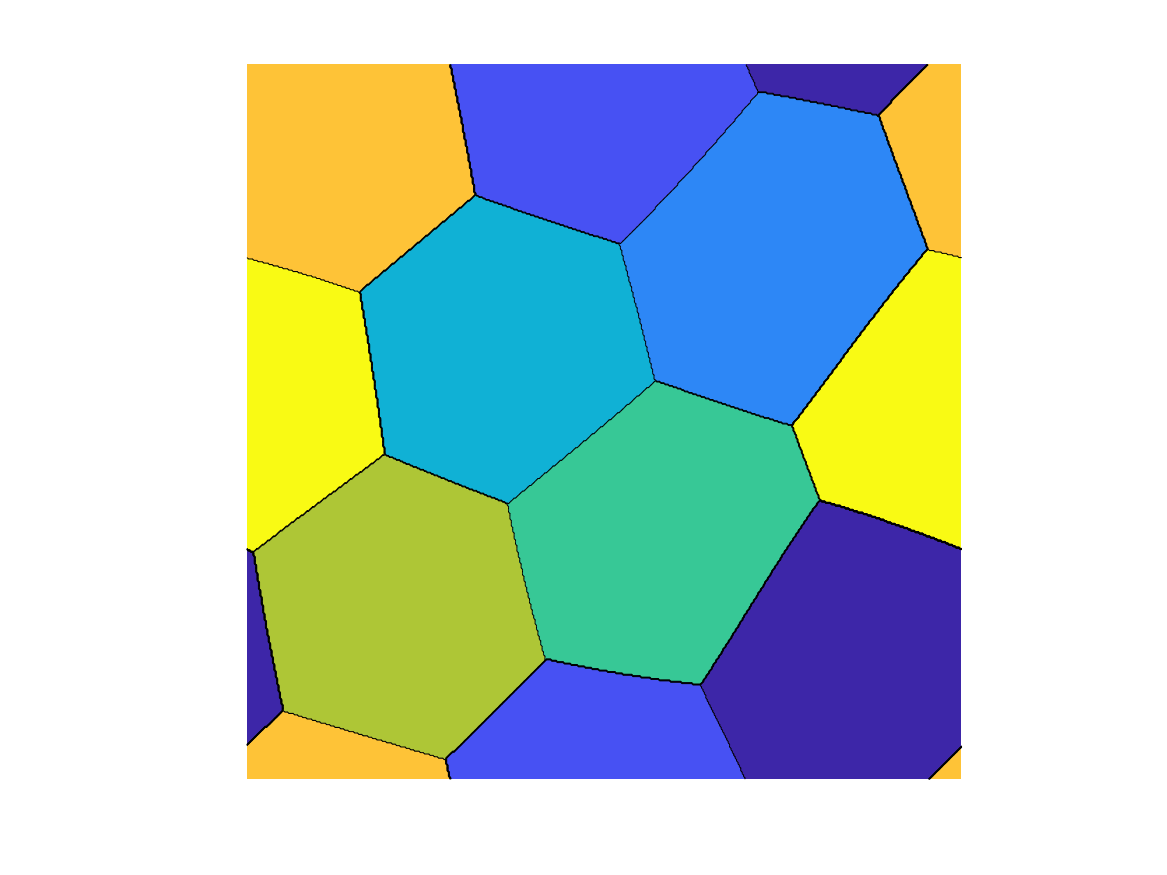}&
			\includegraphics[width = 0.09\textwidth, clip, trim = 4cm 1cm 3cm 1cm]{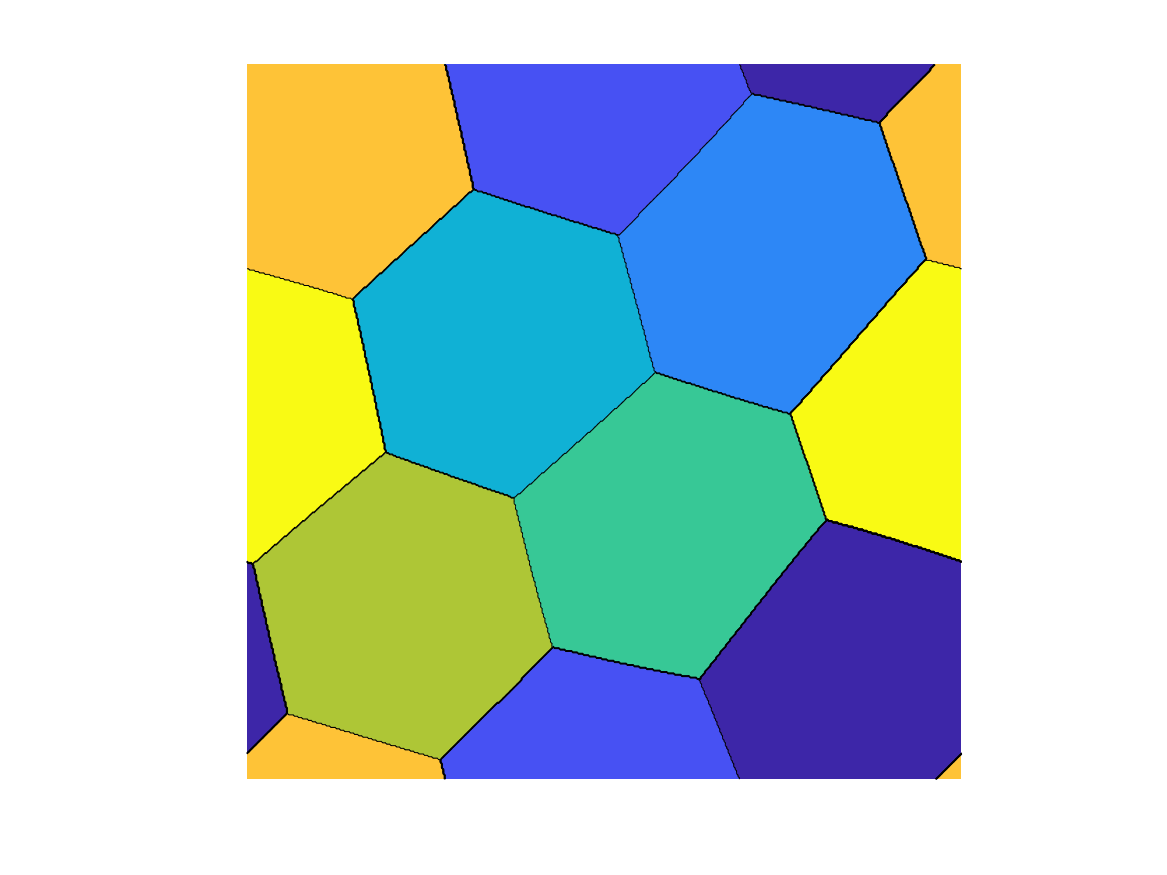}&
			\includegraphics[width = 0.09\textwidth, clip, trim = 4cm 1cm 3cm 1cm]{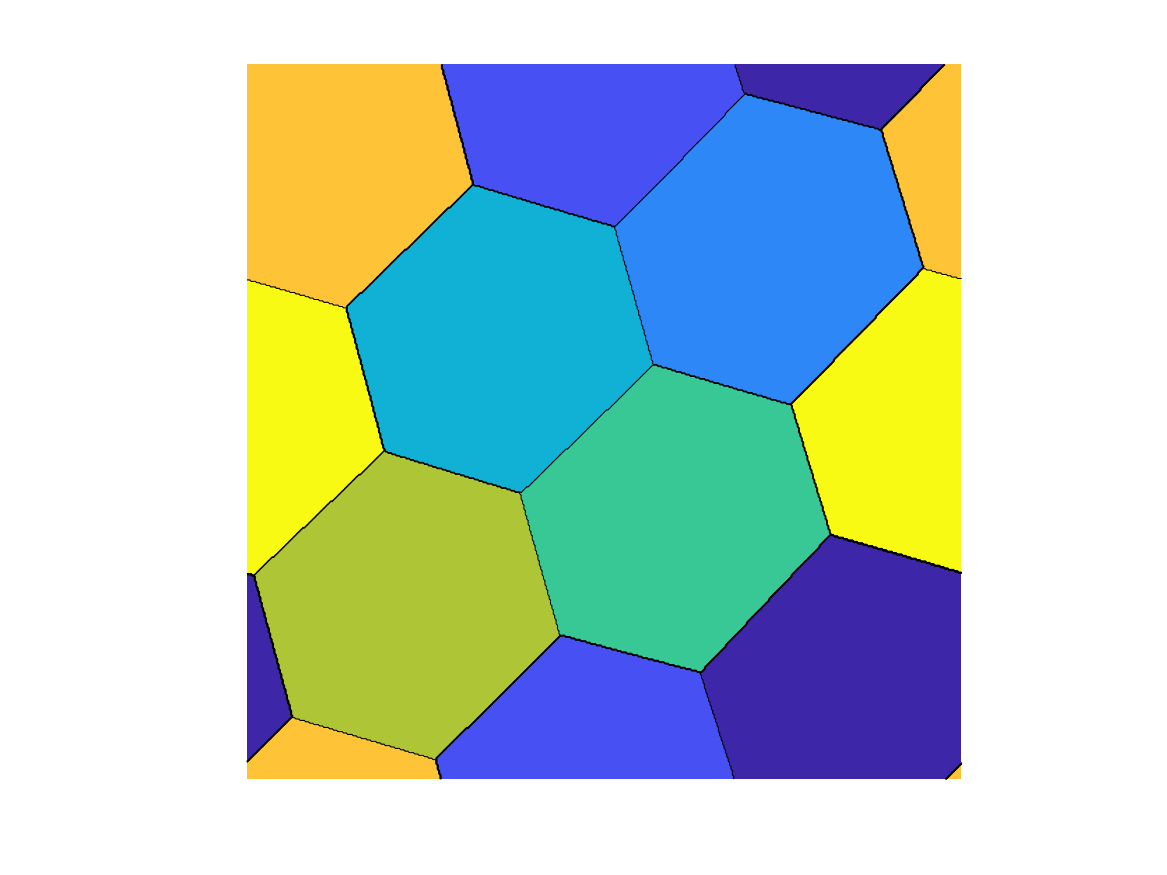}&
			\includegraphics[width = 0.09\textwidth, clip, trim = 4cm 1cm 3cm 1cm]{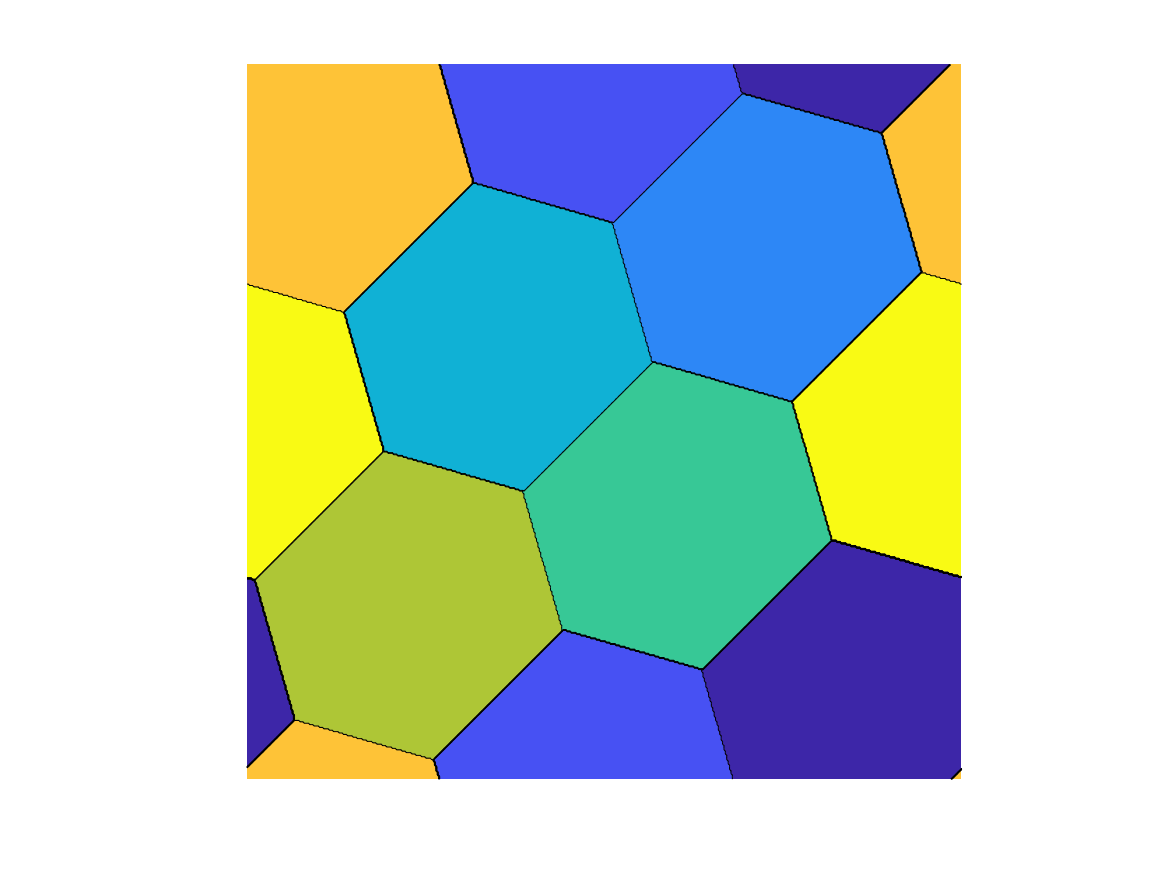} \\ 
			\hline
		\end{tabular}
		\caption{Snapshots of Algorithm~\ref{Alg:4step} at different iterations on a $512\times512$ discretized mesh with a time step size of $\tau=0.1$. The first row corresponds to $k=4$, while the second row corresponds to $k=8$.} \label{fig:EvolAlg_4stepdt01}
	\end{figure}
	\begin{figure}[H]
		\centering
		\begin{tabular}{c|c|c|c|c|c|c|c}
			\hline 
			initial & 50 & 100  & 150 & 200  & 250& 300  & 500 \\
			\includegraphics[width = 0.09\textwidth, clip, trim = 4cm 1cm 3cm 1cm]{figures_sq/initialm4.png}&  
			\includegraphics[width = 0.09\textwidth, clip, trim = 4cm 1cm 3cm 1cm]{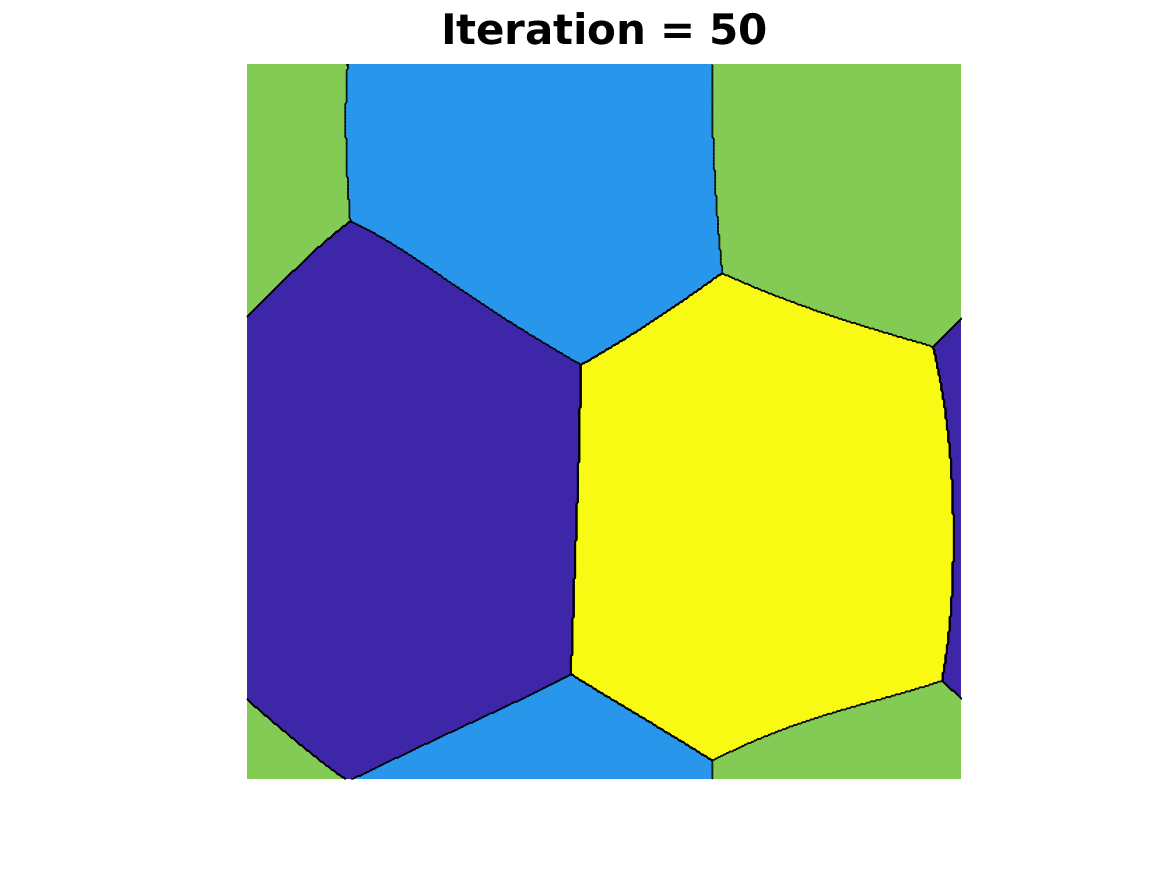}&
			\includegraphics[width = 0.09\textwidth, clip, trim = 4cm 1cm 3cm 1cm]{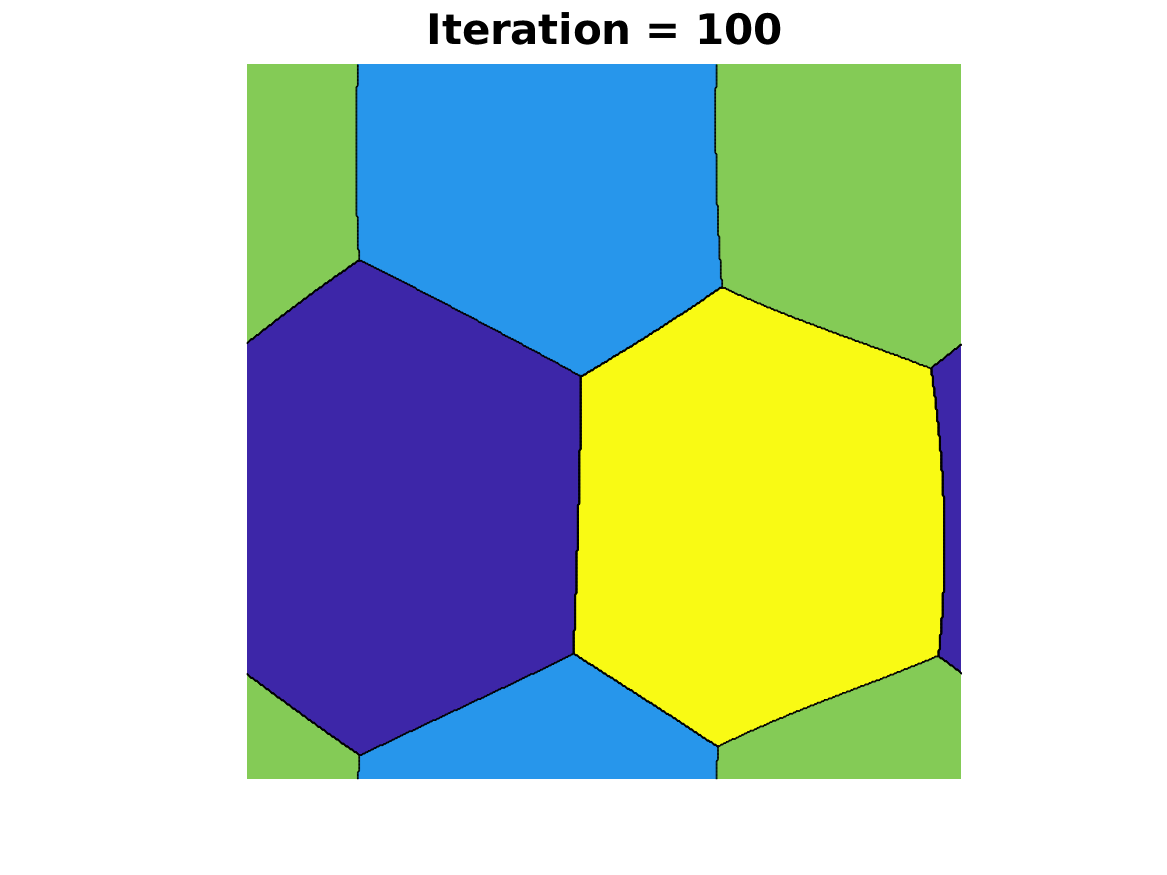}&
			\includegraphics[width = 0.09\textwidth, clip, trim = 4cm 1cm 3cm 1cm]{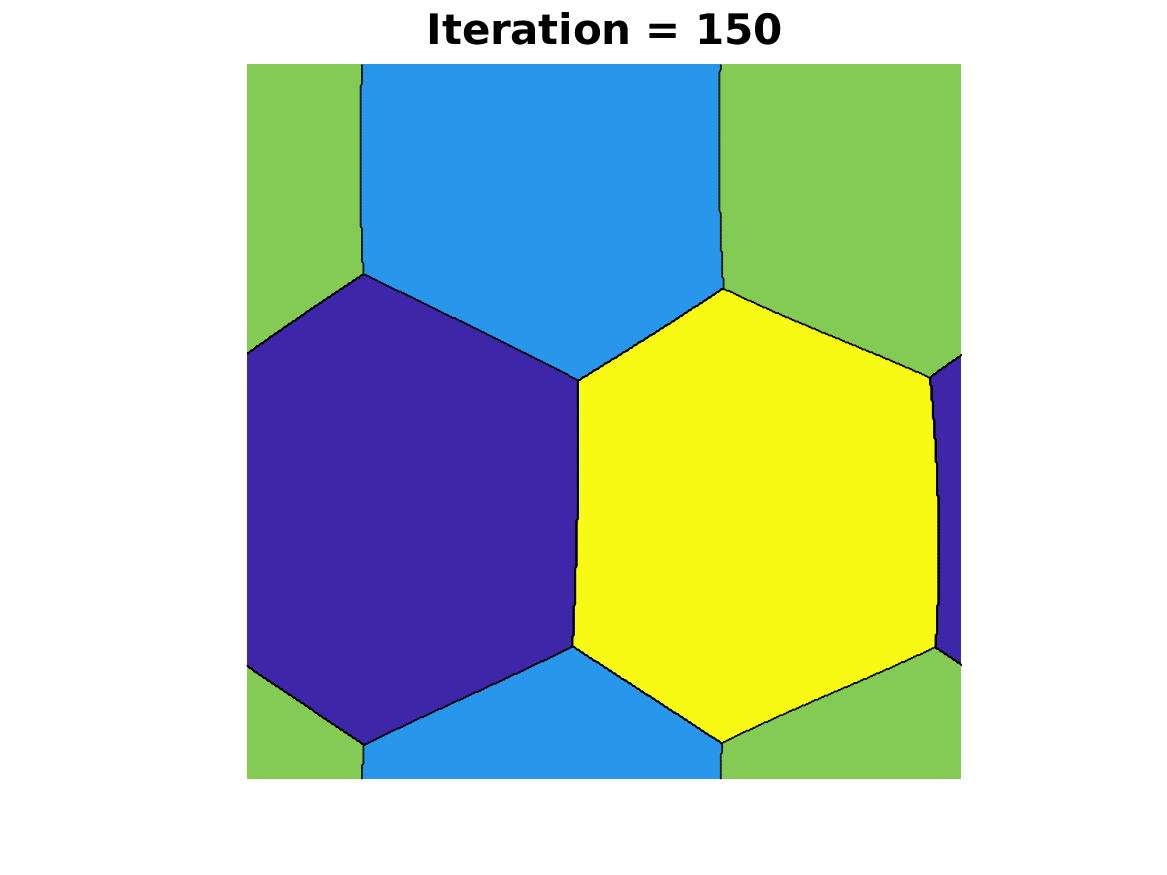}&
			\includegraphics[width = 0.09\textwidth, clip, trim = 4cm 1cm 3cm 1cm]{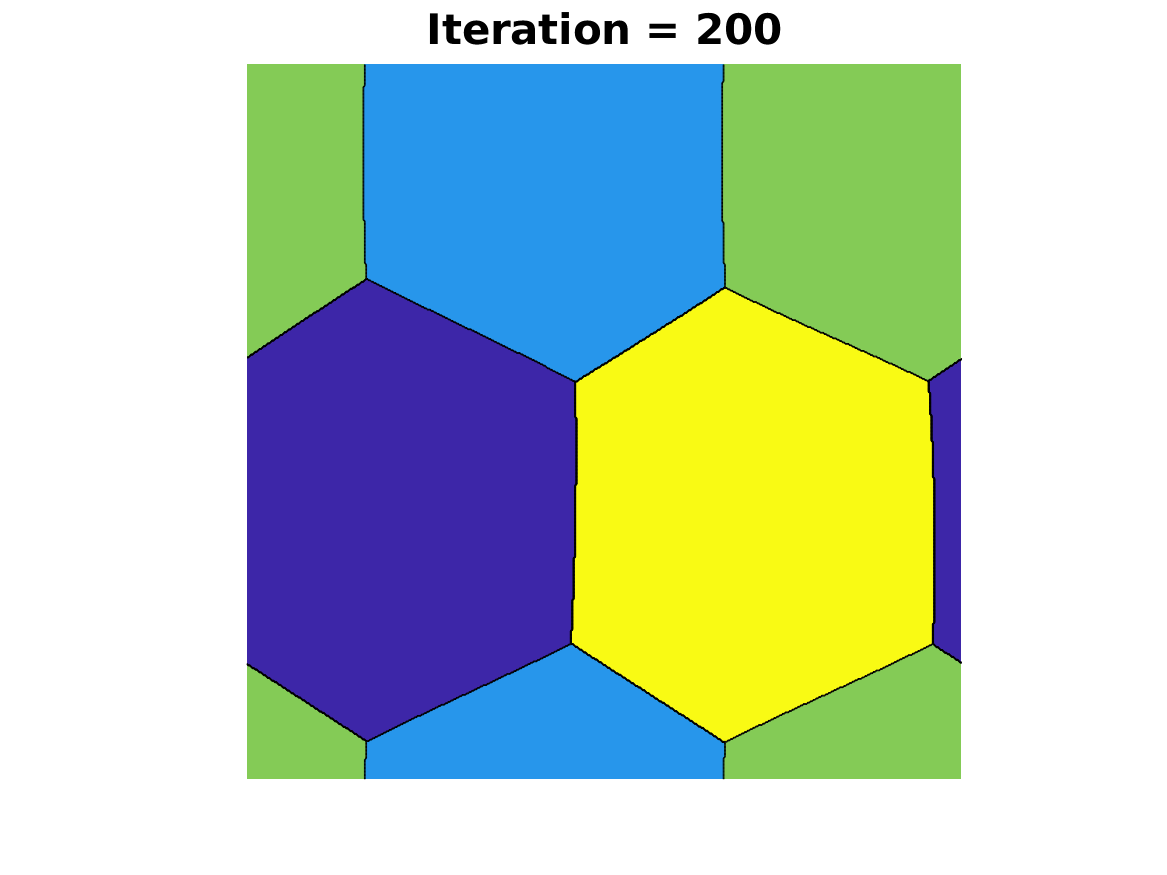}&
			\includegraphics[width = 0.09\textwidth, clip, trim = 4cm 1cm 3cm 1cm]{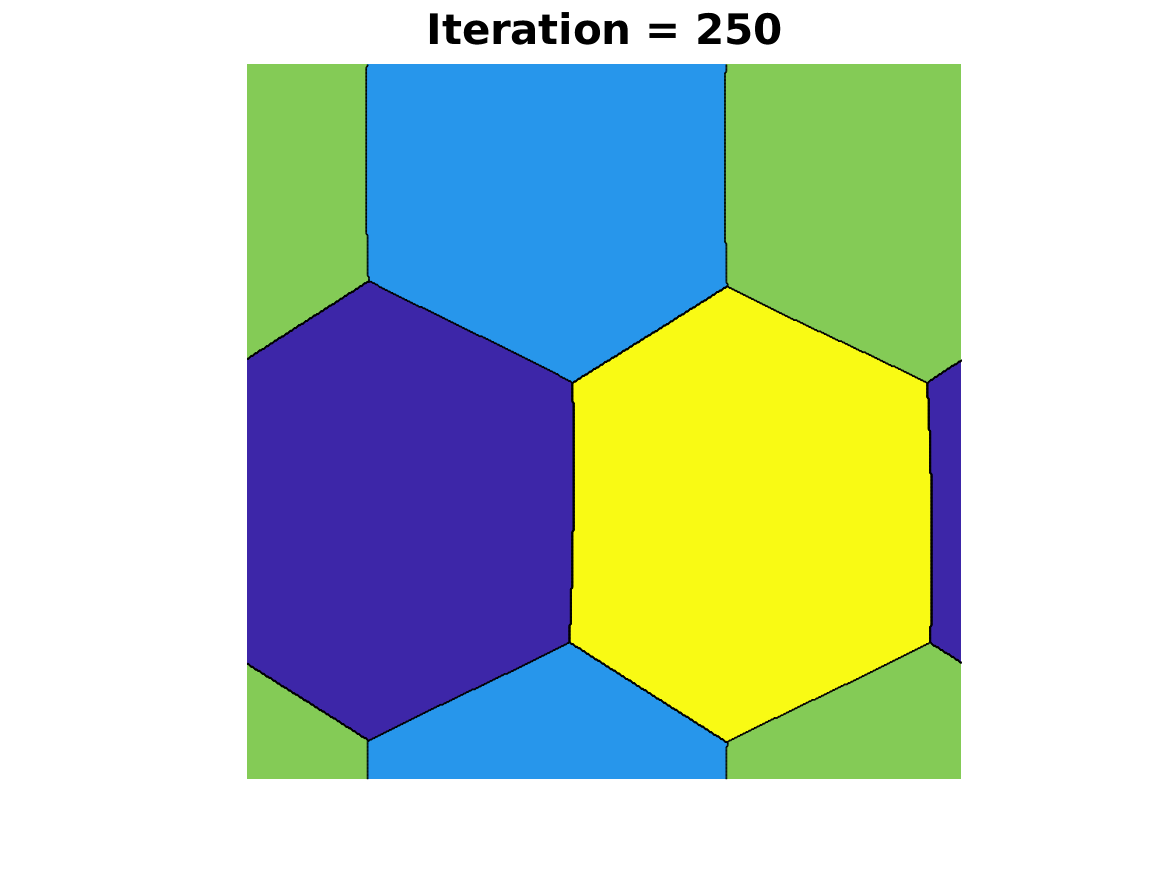}&
			\includegraphics[width = 0.09\textwidth, clip, trim = 4cm 1cm 3cm 1cm]{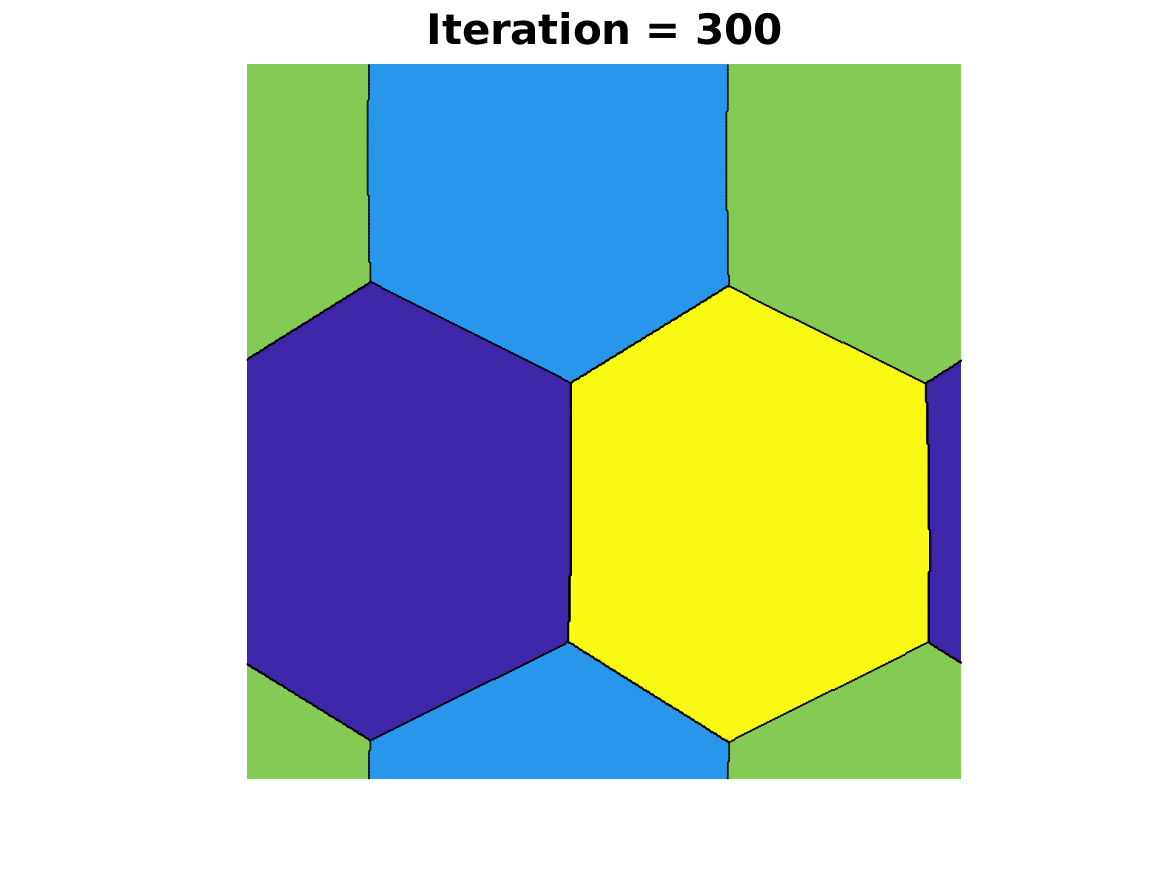}&
			\includegraphics[width = 0.09\textwidth, clip, trim = 4cm 1cm 3cm 1cm]{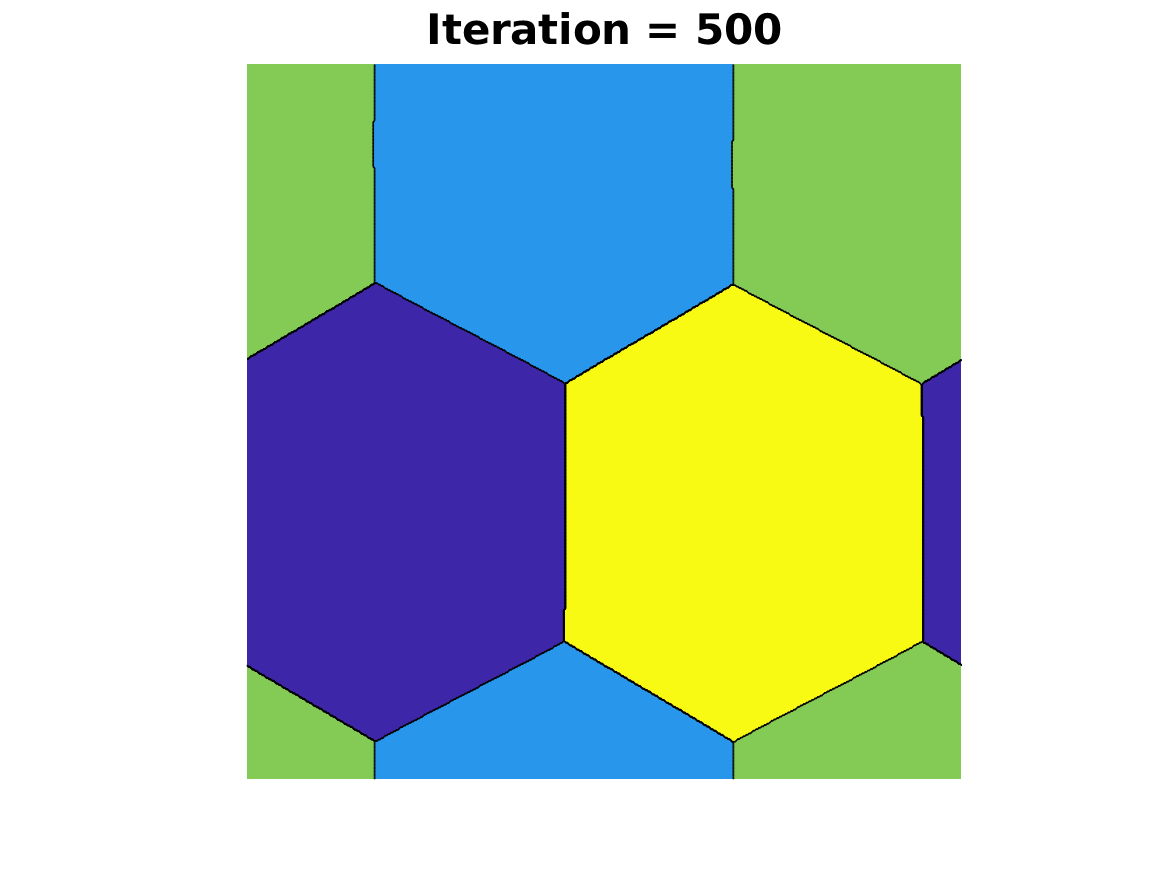} \\ 
			\hline
			\hline 
			initial & 100 & 200  & 300 & 400  & 500& 600  & 700 \\
			\includegraphics[width = 0.09\textwidth, clip, trim = 4cm 1cm 3cm 1cm]{figures_sq/initialm8.png}&  
			\includegraphics[width = 0.09\textwidth, clip, trim = 4cm 1cm 3cm 1cm]{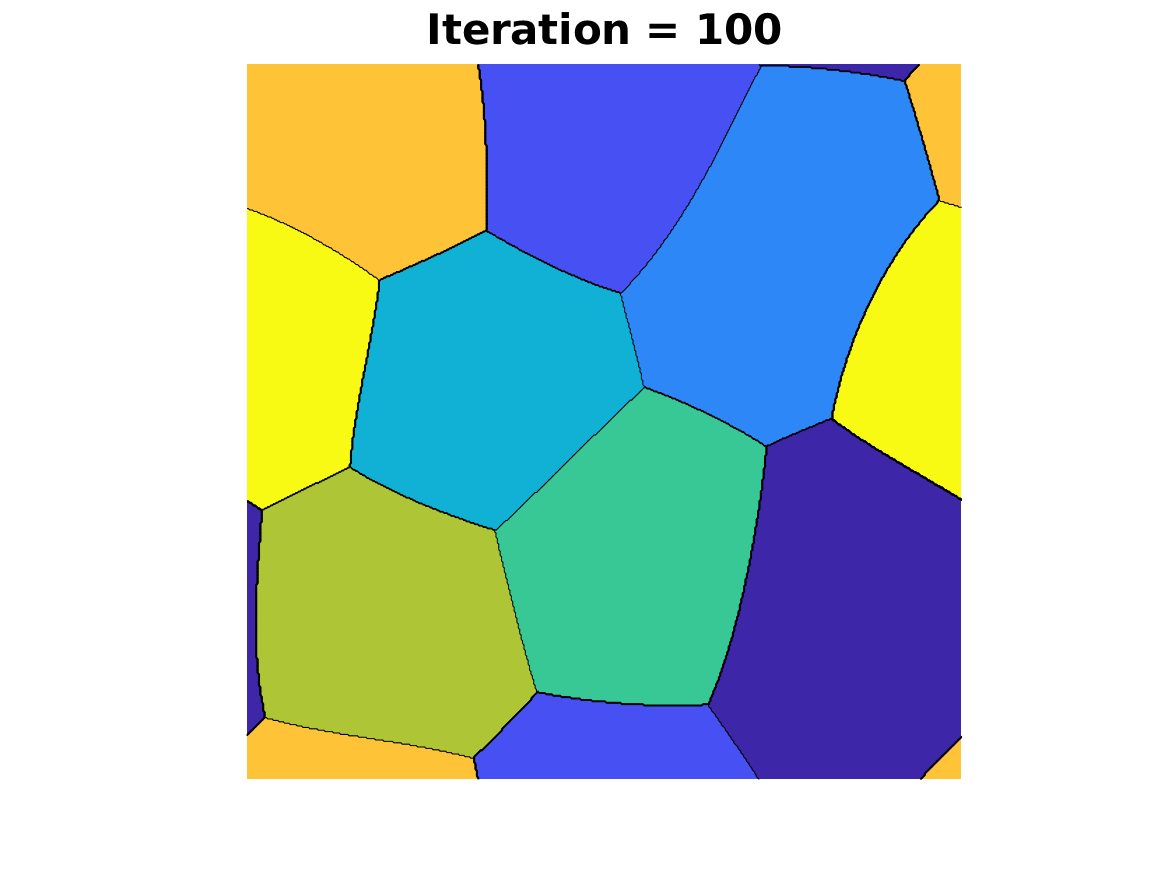}&
			\includegraphics[width = 0.09\textwidth, clip, trim = 4cm 1cm 3cm 1cm]{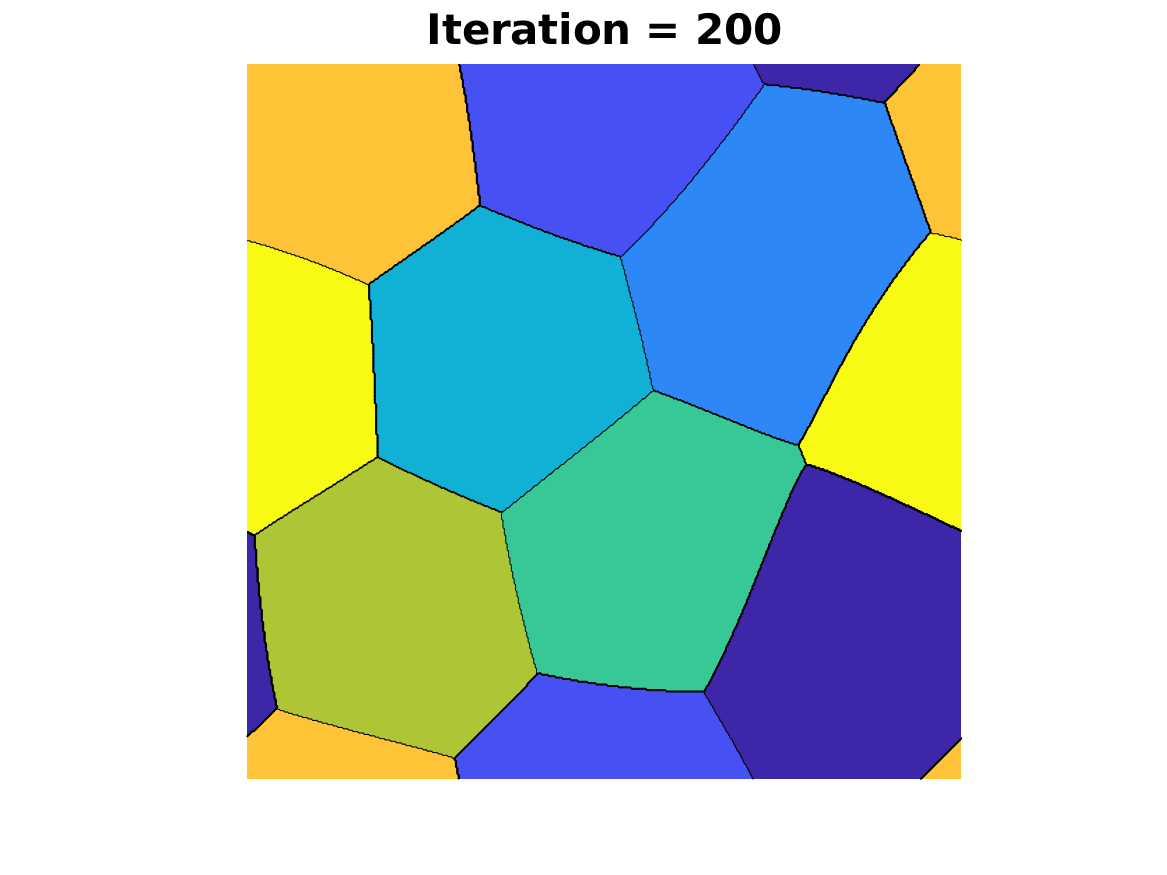}&
			\includegraphics[width = 0.09\textwidth, clip, trim = 4cm 1cm 3cm 1cm]{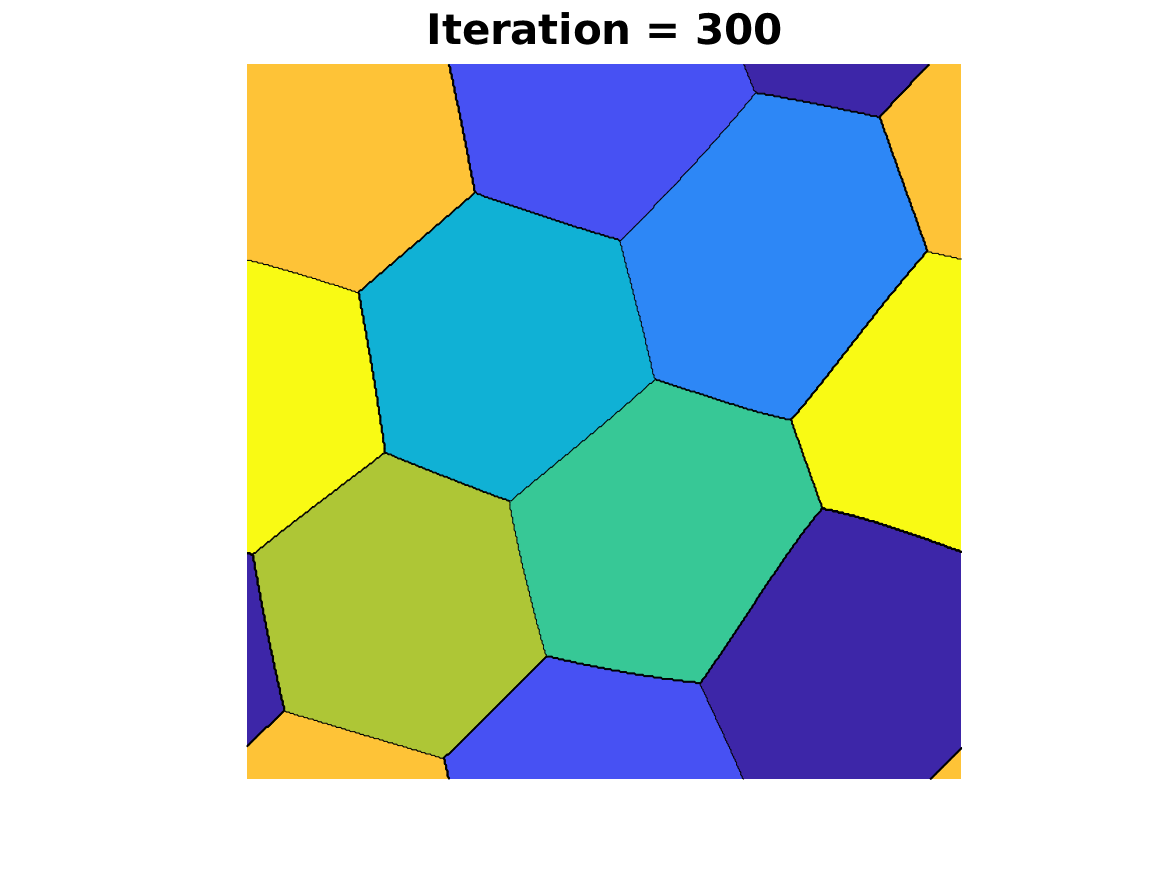}&
			\includegraphics[width = 0.09\textwidth, clip, trim = 4cm 1cm 3cm 1cm]{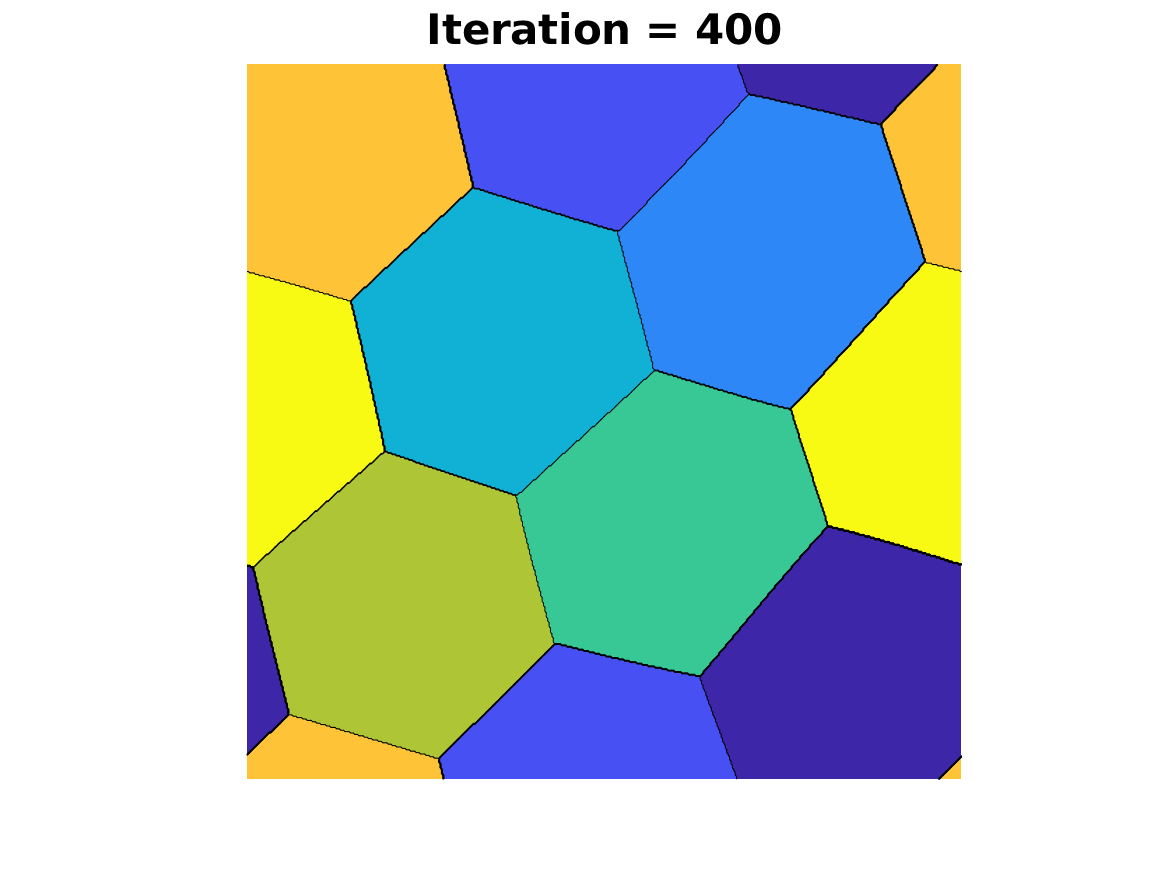}&
			\includegraphics[width = 0.09\textwidth, clip, trim = 4cm 1cm 3cm 1cm]{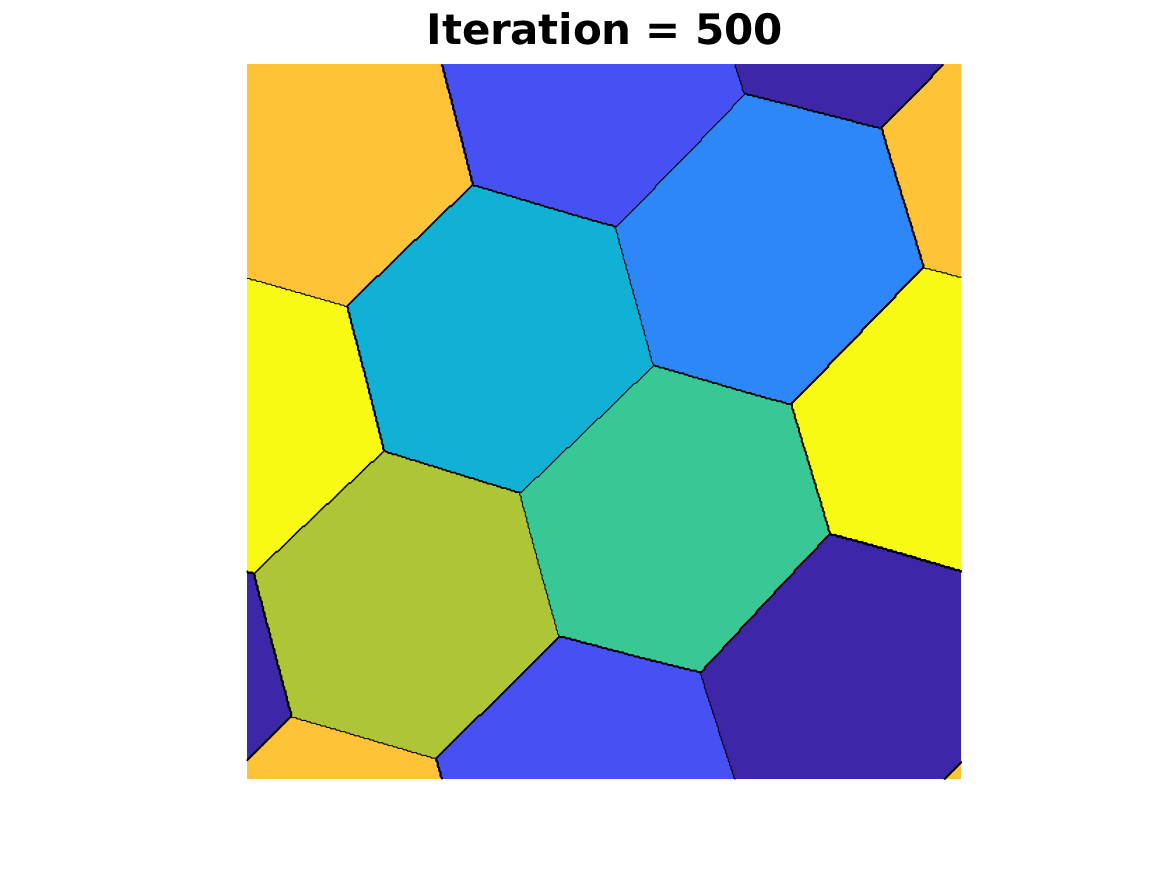}&
			\includegraphics[width = 0.09\textwidth, clip, trim = 4cm 1cm 3cm 1cm]{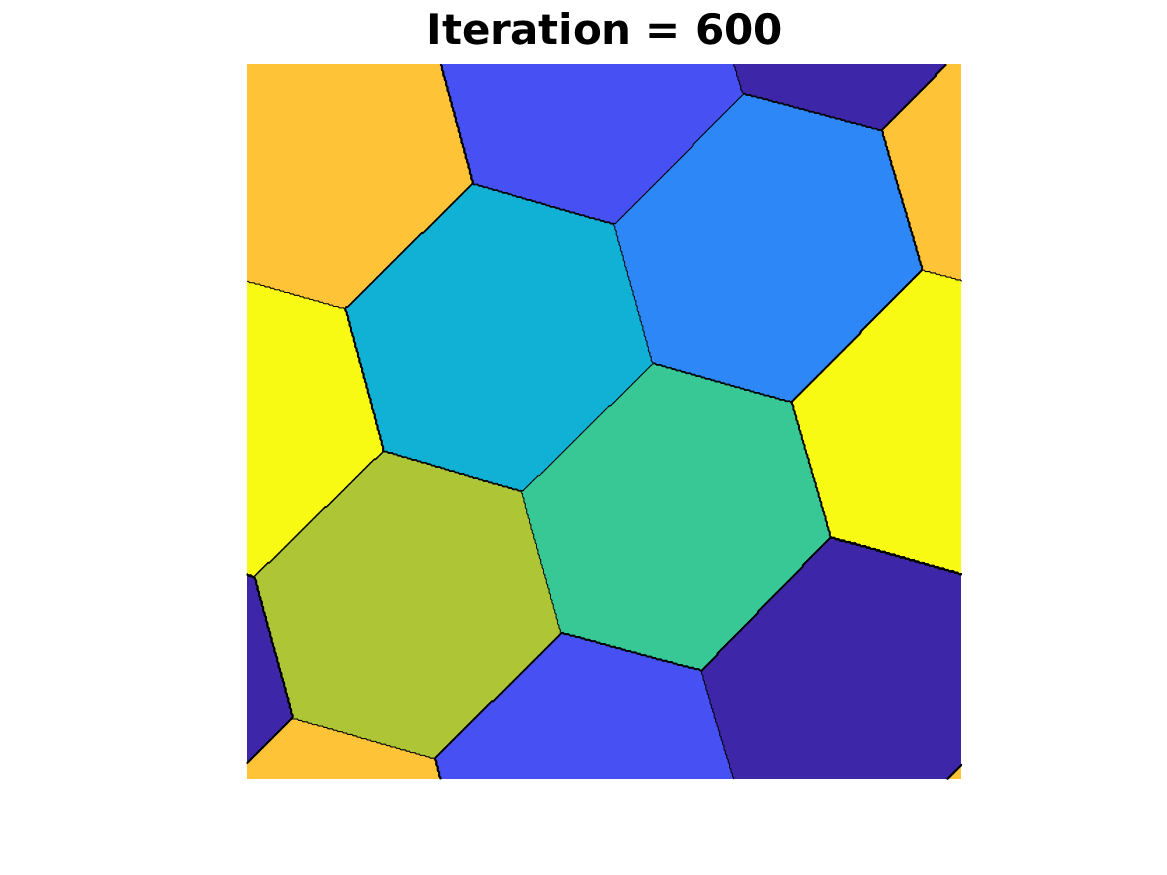}&
			\includegraphics[width = 0.09\textwidth, clip, trim = 4cm 1cm 3cm 1cm]{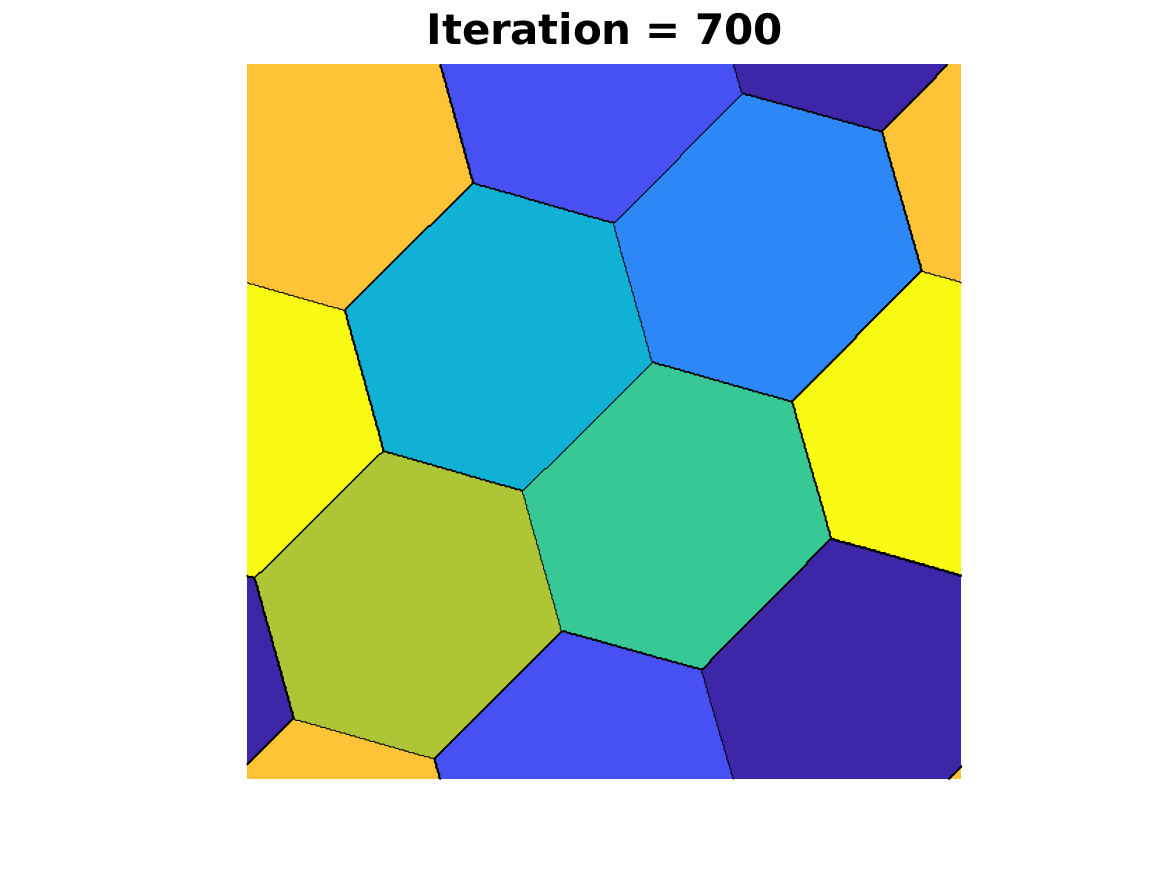} \\ 
			\hline
		\end{tabular}
		\caption{Snapshots of Algorithm~\ref{Alg:3step_Type1} at different iterations on a $512\times512$ discretized mesh with a time step size of $\tau=0.01$. The first row corresponds to $k=4$, while the second row corresponds to $k=8$.}\label{fig:EvolAlg_3stepType1dt001}
	\end{figure}
	
	\begin{figure}[H]
		\centering
		\includegraphics[width=0.45\textwidth]{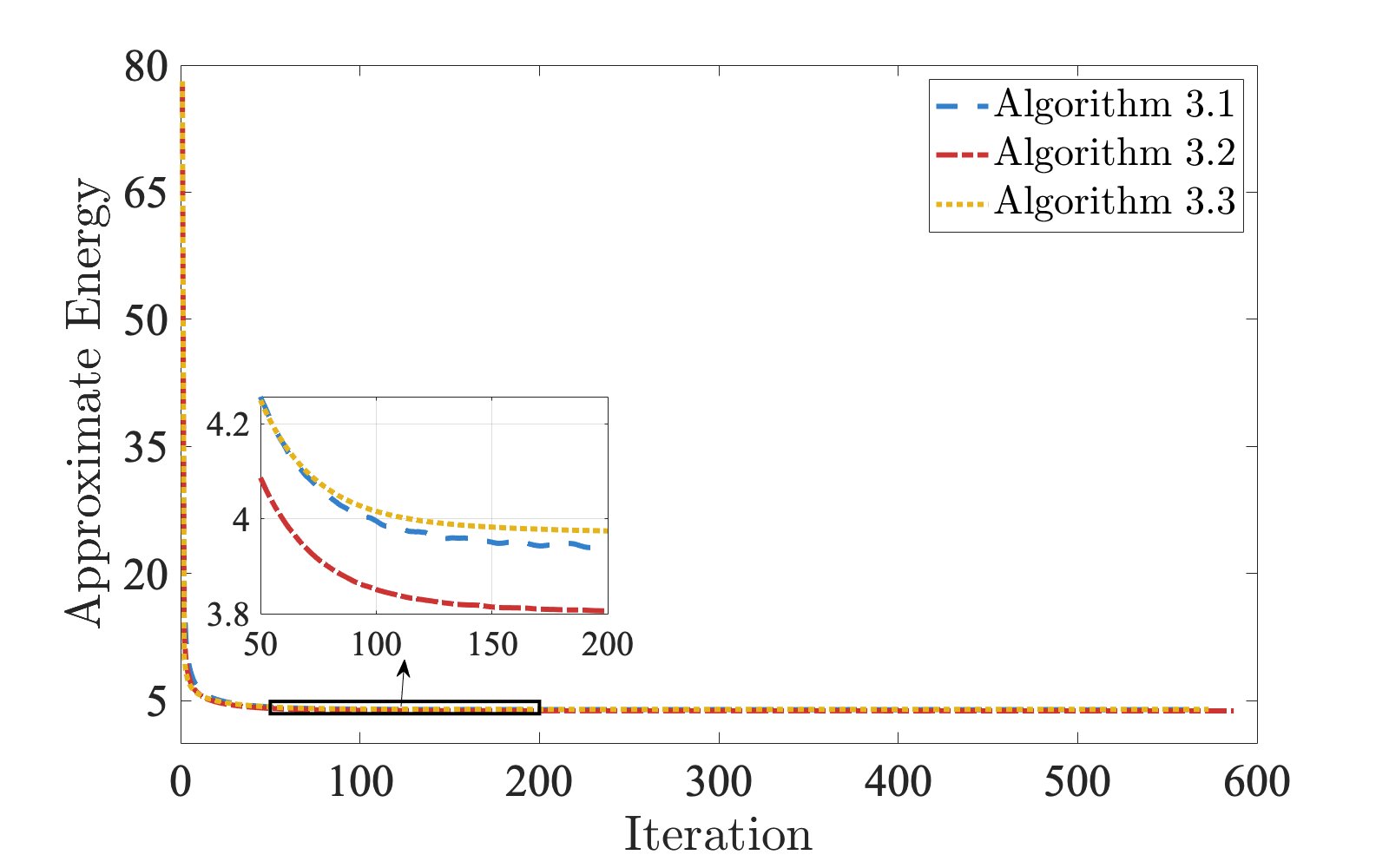}
		\includegraphics[width=0.45\textwidth]{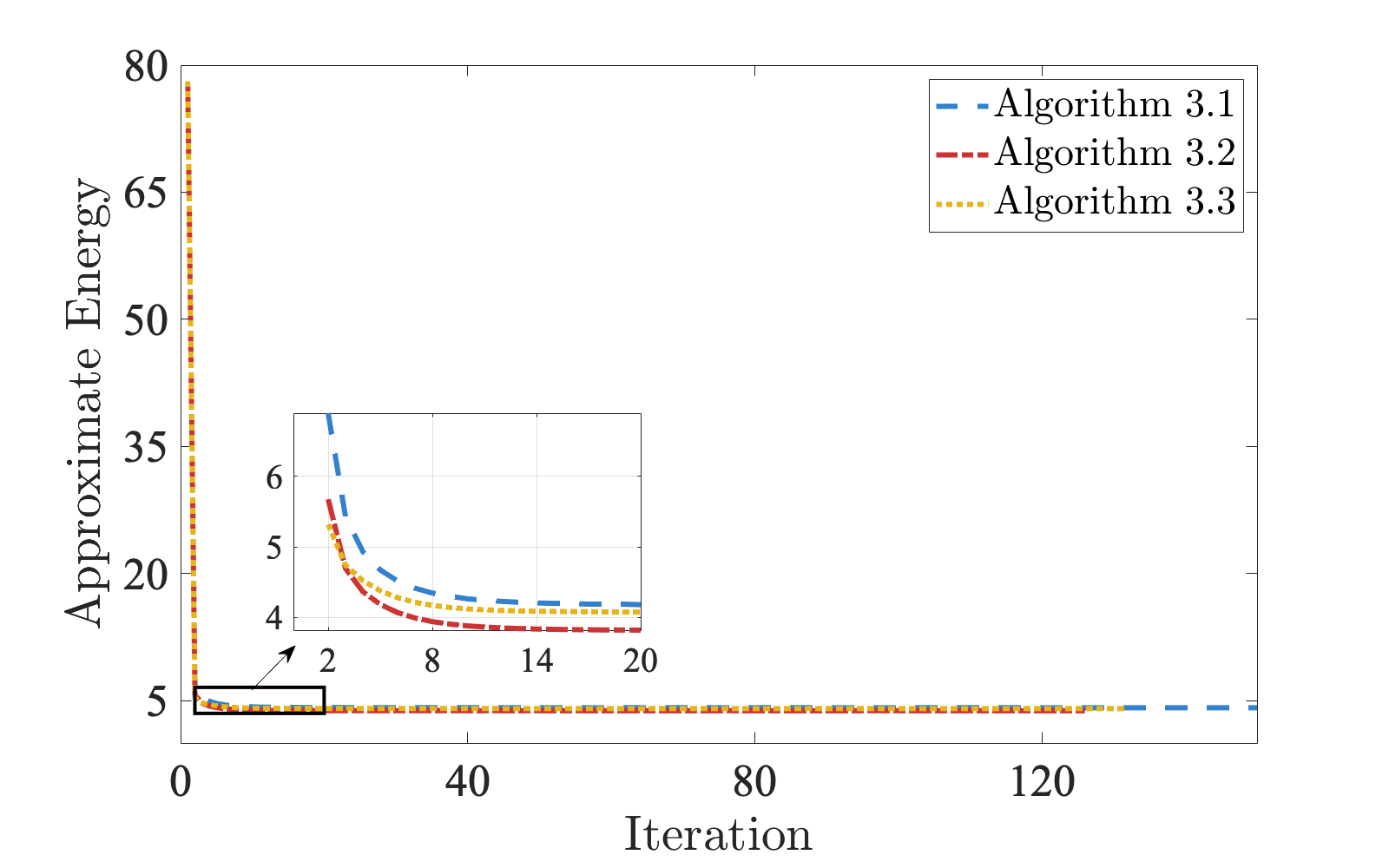}
		\caption{Approximate energy on a $512\times512$ discretized mesh with $k=4$. Left: $\tau=0.01$. Right: $\tau=0.1$.}\label{fig:ECompAlg1_3}
	\end{figure}
\subsubsection{Performance of the energy dissipative schemes}
To demonstrate the energy dissipative property of Algorithms~\ref{Alg:3step_Type1Diss} and \ref{Alg:3step_Type2Diss}, we plot the approximated energy computed by Algorithms~\ref{Alg:3step_Type1}-\ref{Alg:3step_Type2Diss} in Figure~\ref{fig:EnergyDissk4} and Figure~\ref{fig:EnergyDissk8}, with initial conditions given in Figure~\ref{fig:EvolAlg_4stepdt001}. As shown in these figures, the numerical energy computed by Algorithms~\ref{Alg:3step_Type1} and \ref{Alg:3step_Type2} oscillates as the partition approaches the final structure. However, the energy dissipation property is preserved by using the schemes presented in Algorithm~\ref{Alg:3step_Type1Diss} and Algorithm~\ref{Alg:3step_Type2Diss}.
	\begin{figure}[H]
		\centering
		\includegraphics[width=0.45\textwidth]{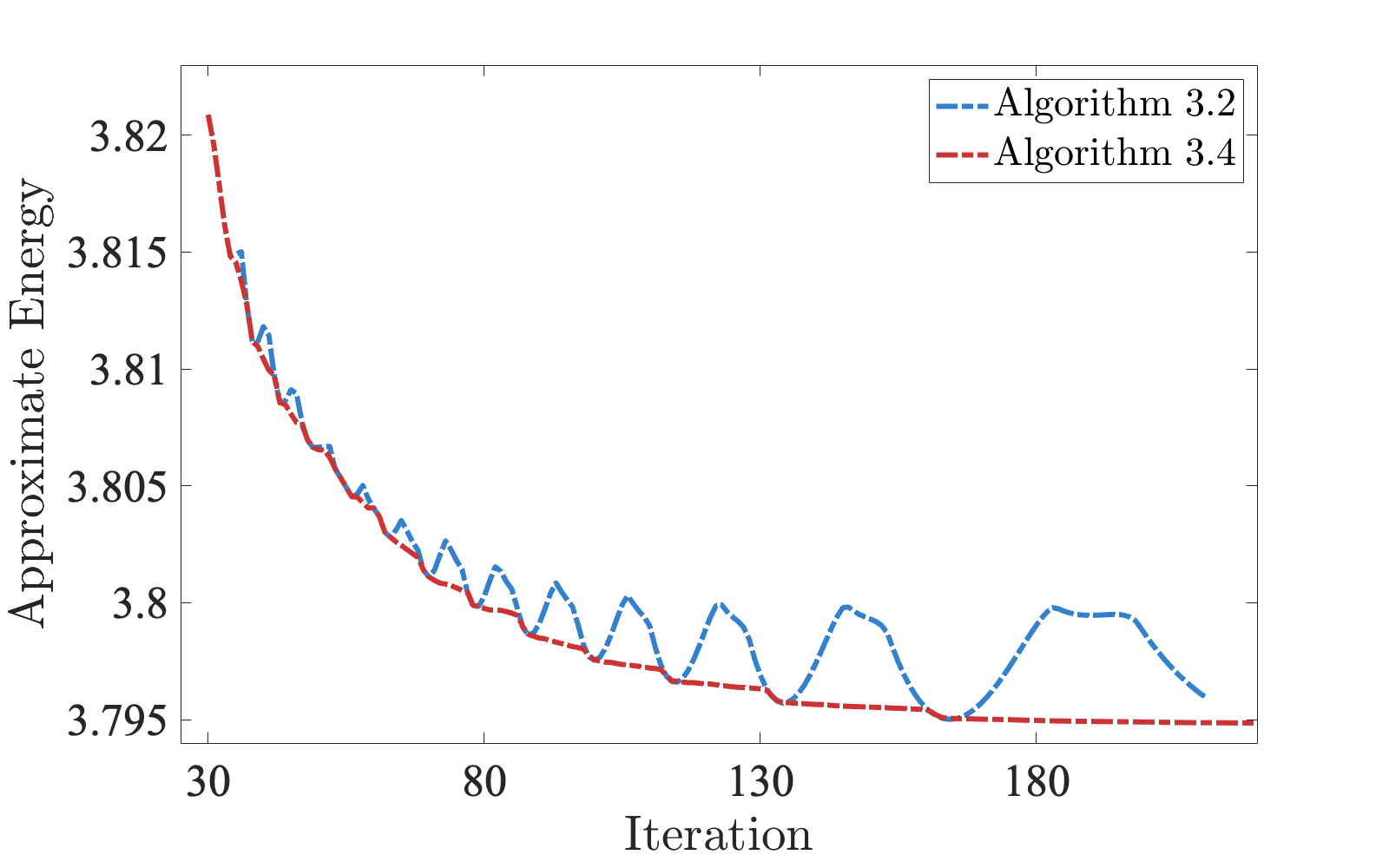}
		\includegraphics[width=0.45\textwidth]{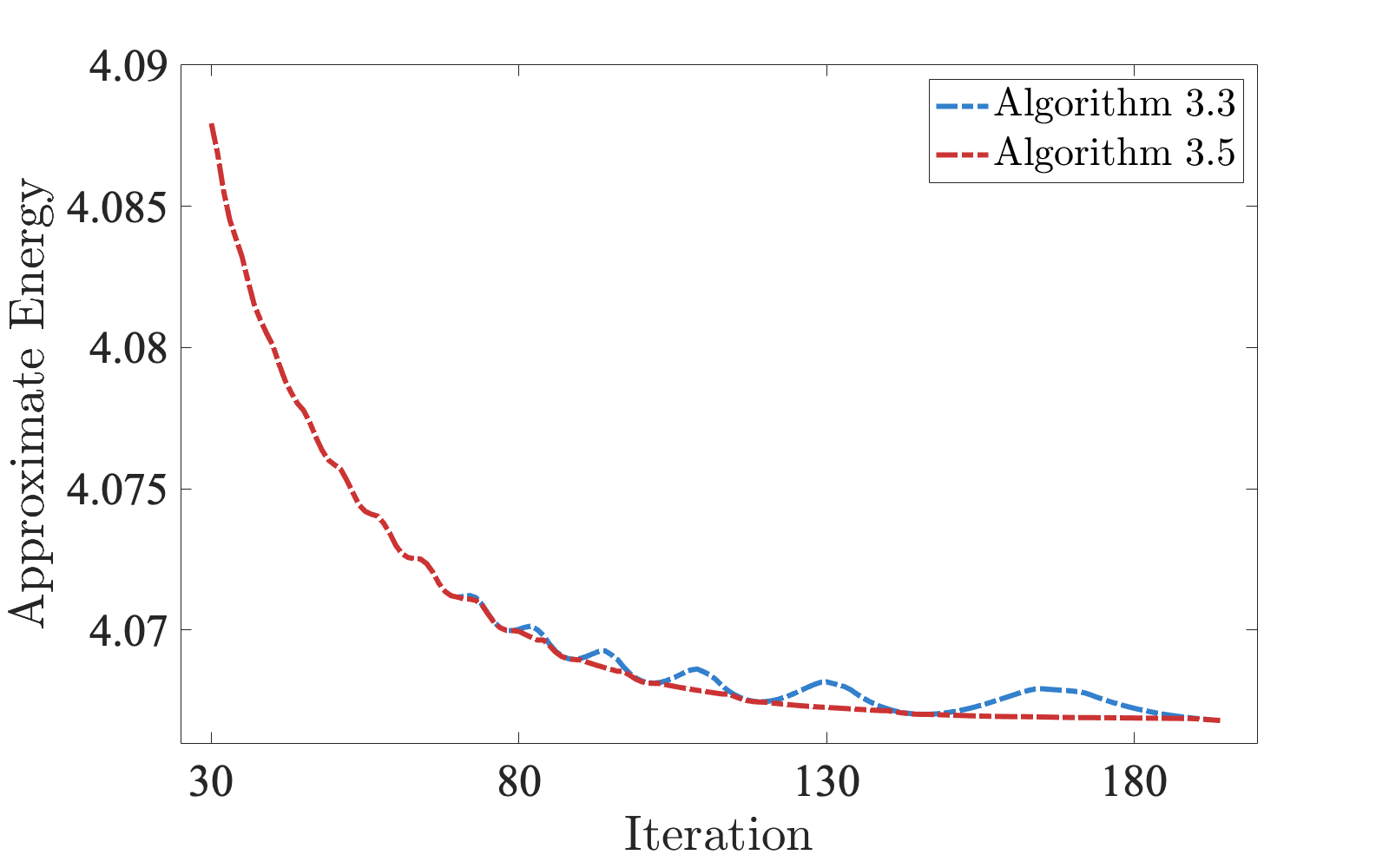}
		\caption{Approximate energy   of Algorithms~\ref{Alg:3step_Type1Diss} and \ref{Alg:3step_Type2Diss} on a $512\times512$ discretized mesh with  $k=4$ and $\tau=0.05$.}\label{fig:EnergyDissk4}
	\end{figure}
	\begin{figure}[H]
		\centering
		\includegraphics[width=0.45\textwidth]{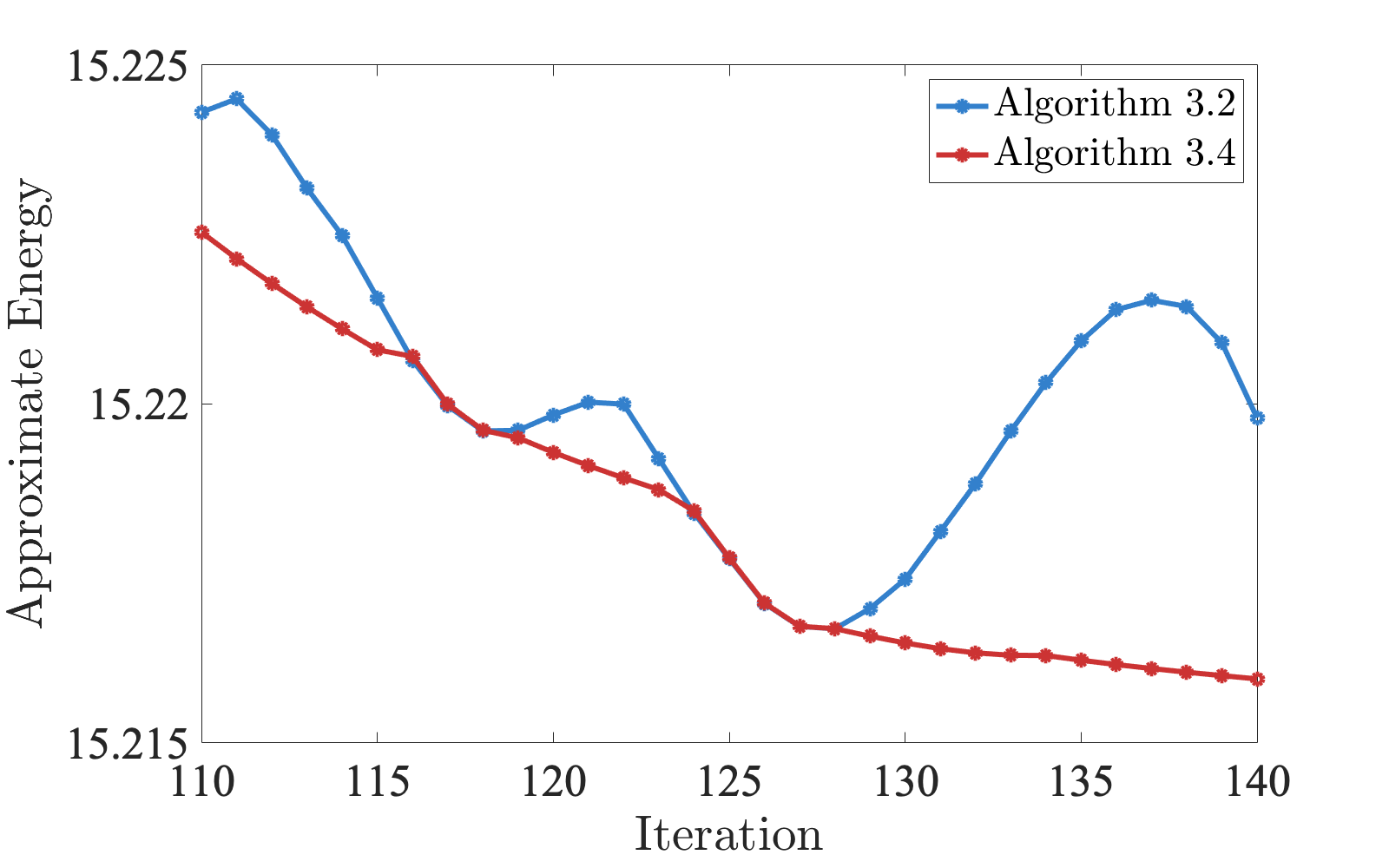}
		\includegraphics[width=0.45\textwidth]{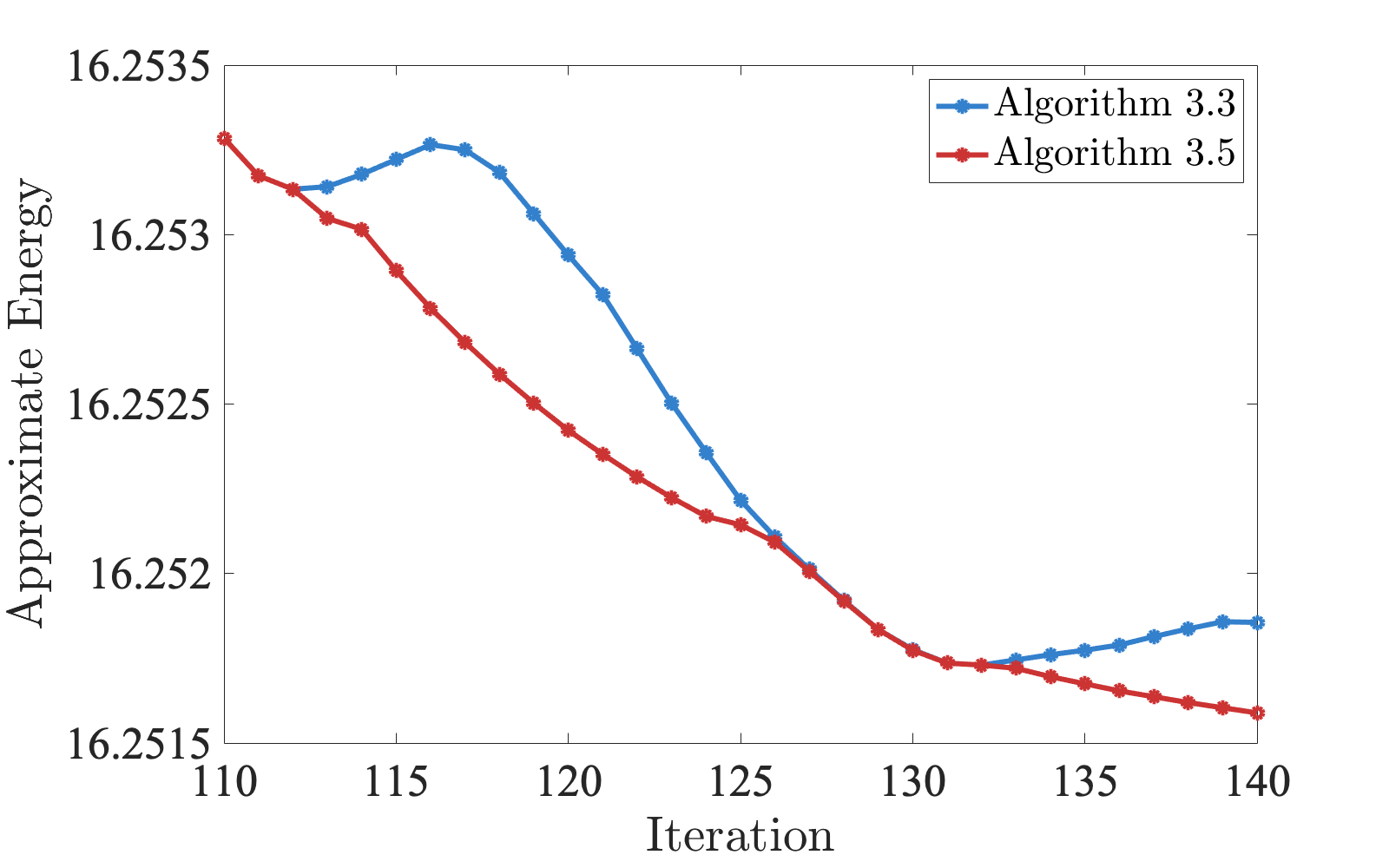}
		\caption{Approximate energy   of Algorithms~\ref{Alg:3step_Type1Diss} and \ref{Alg:3step_Type2Diss} on a $512\times512$ discretized mesh with  $k=8$ and $\tau=0.05$.}\label{fig:EnergyDissk8}
	\end{figure}	
\subsubsection{Various partitions for 2-dimensional flat tori} 
Using \(\tau = 0.1\), we evaluate the performance of Algorithm~\ref{Alg:4step} and Algorithm~\ref{Alg:3step_Type1} for various values of \(k\), starting with random initial guesses. As shown in Figures~\ref{fig:Alg_4stepkdt01} and \ref{fig:Alg_3stepType1kdt01}, both algorithms perform well even with the relatively large value of \(\tau\). These figures highlight the algorithms' robustness across different \(k\) values. Additionally, different local minima can be observed, such as those for \(k=3\) and \(k=5\), indicating the algorithms' capability to explore diverse candidates for the optimal partitions. Furthermore, the periodic extensions of the numerical results implemented by Algorithm~\ref{Alg:3step_Type2} are presented in Figure~\ref{fig:Alg3_extension}. This figure provides a comprehensive view of the algorithm's behavior over extended domains and illustrates the periodic nature of the results.
\begin{figure}[H]
		\centering
		\includegraphics[width = 0.13\textwidth, clip, trim = 4cm 1cm 3cm 1cm]{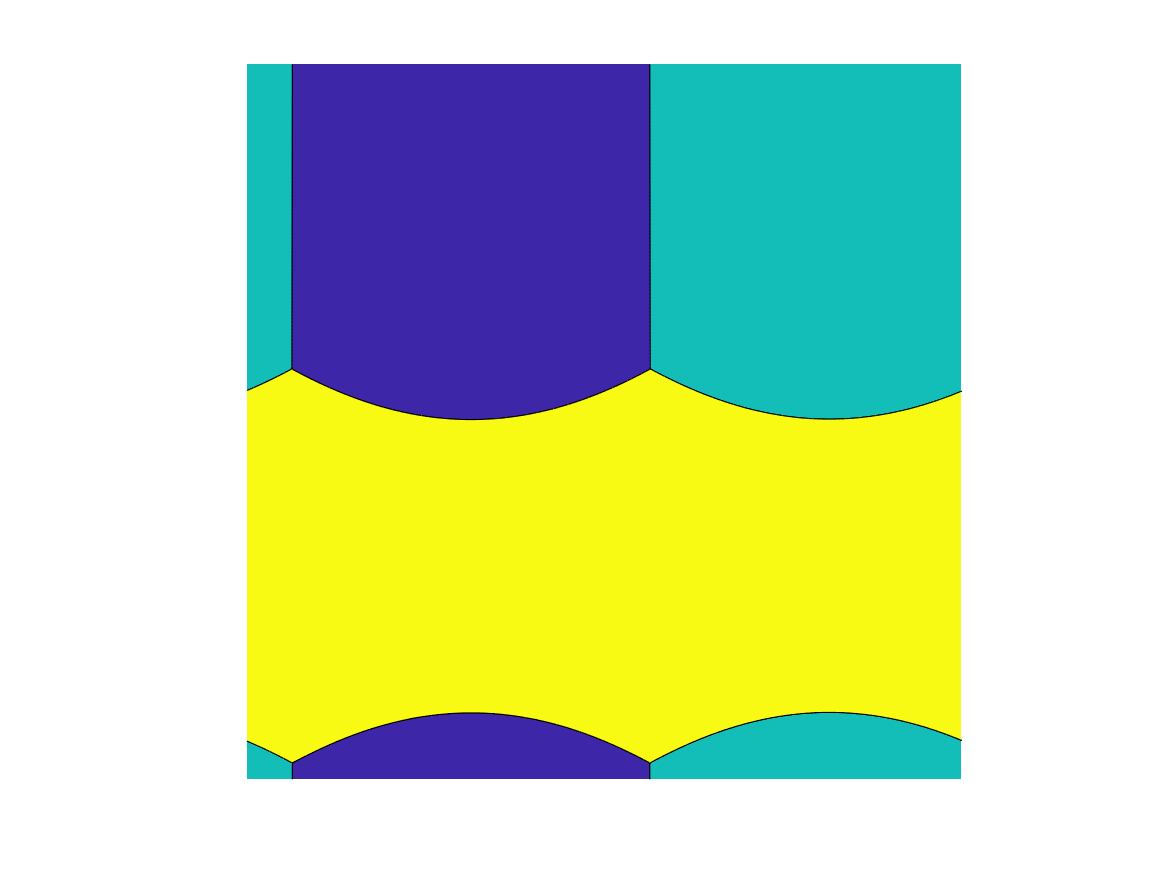}
		\includegraphics[width = 0.13\textwidth, clip, trim = 4cm 1cm 3cm 1cm]{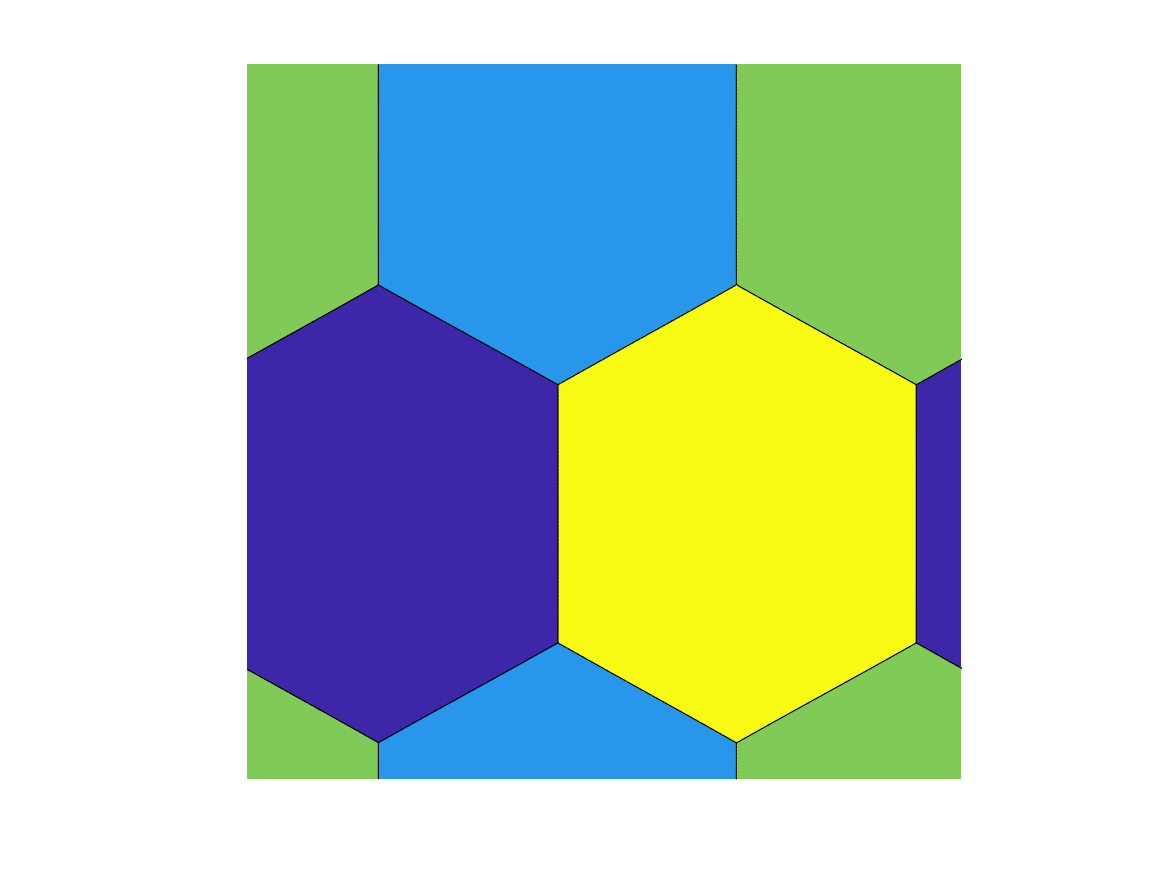}
		\includegraphics[width = 0.13\textwidth, clip, trim = 4cm 1cm 3cm 1cm]{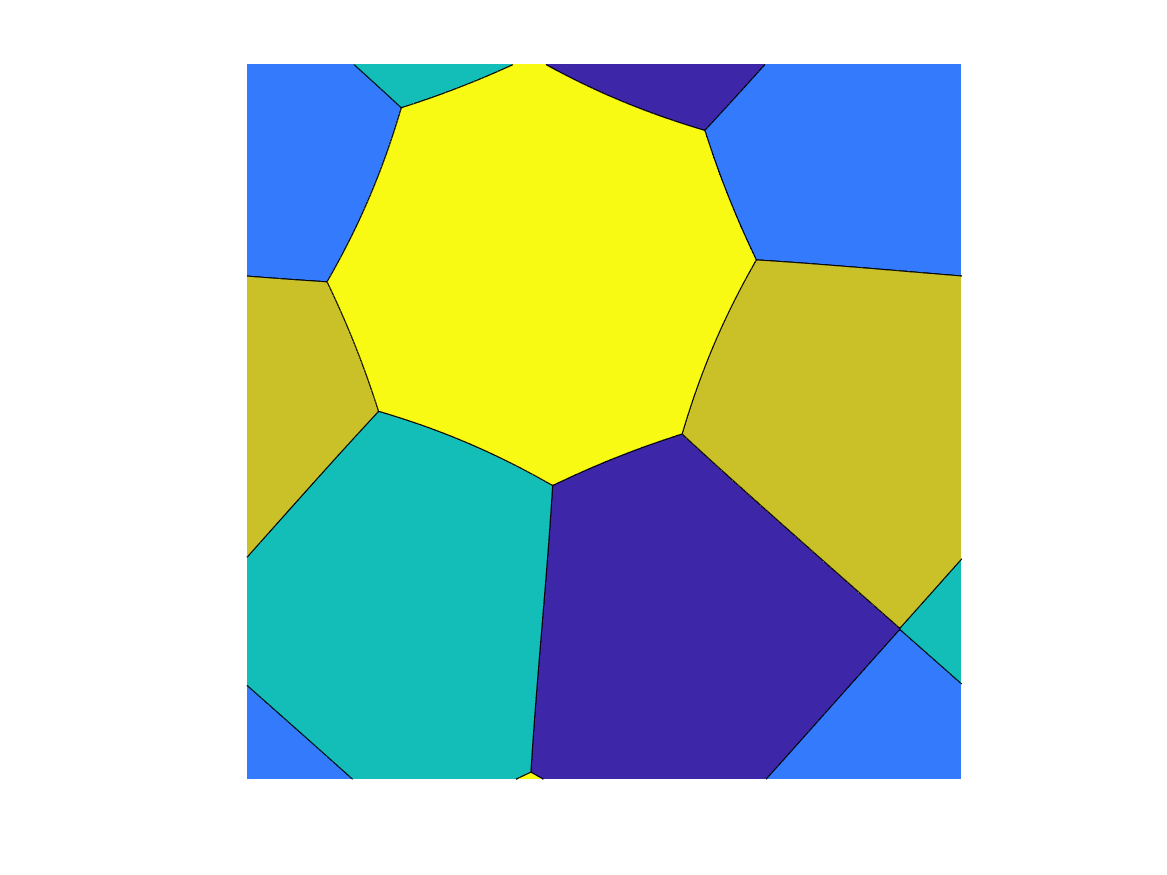}
		\includegraphics[width = 0.13\textwidth, clip, trim = 4cm 1cm 3cm 1cm]{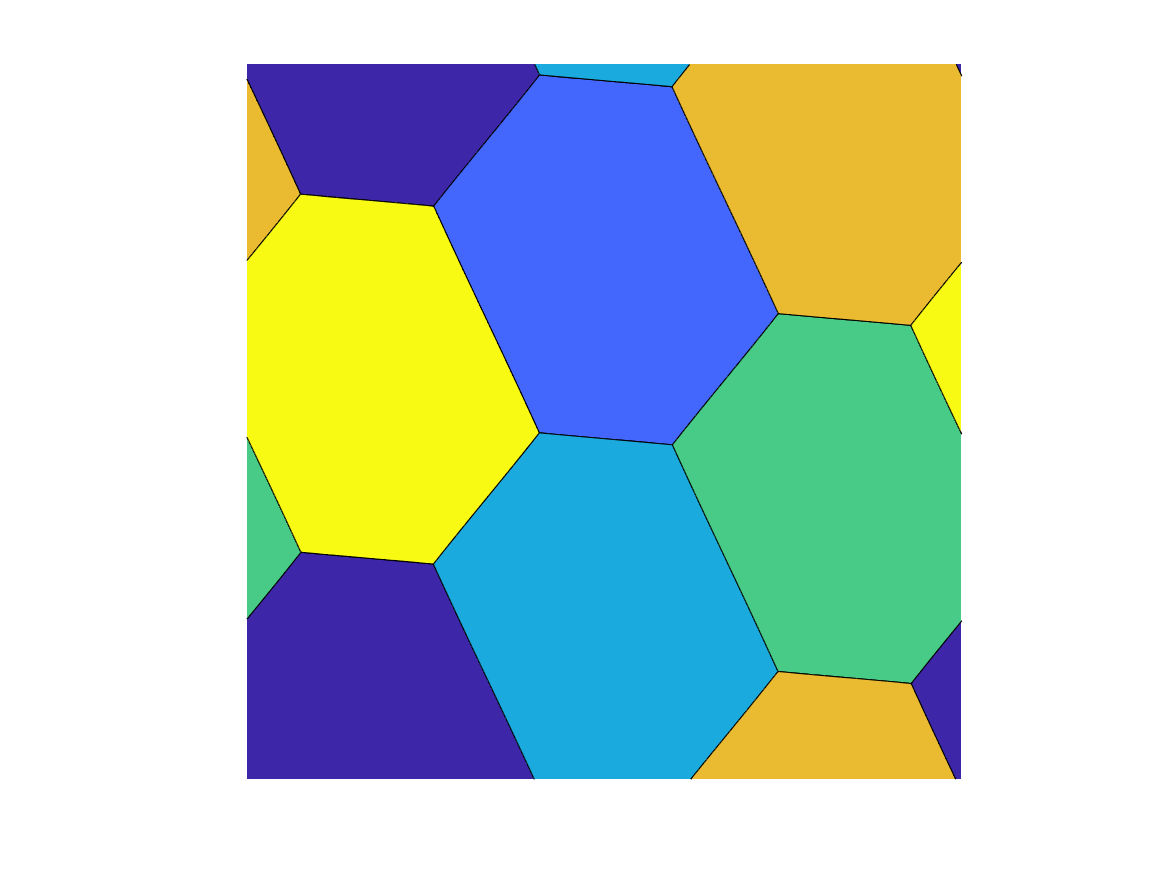}
		\includegraphics[width = 0.13\textwidth, clip, trim = 4cm 1cm 3cm 1cm]{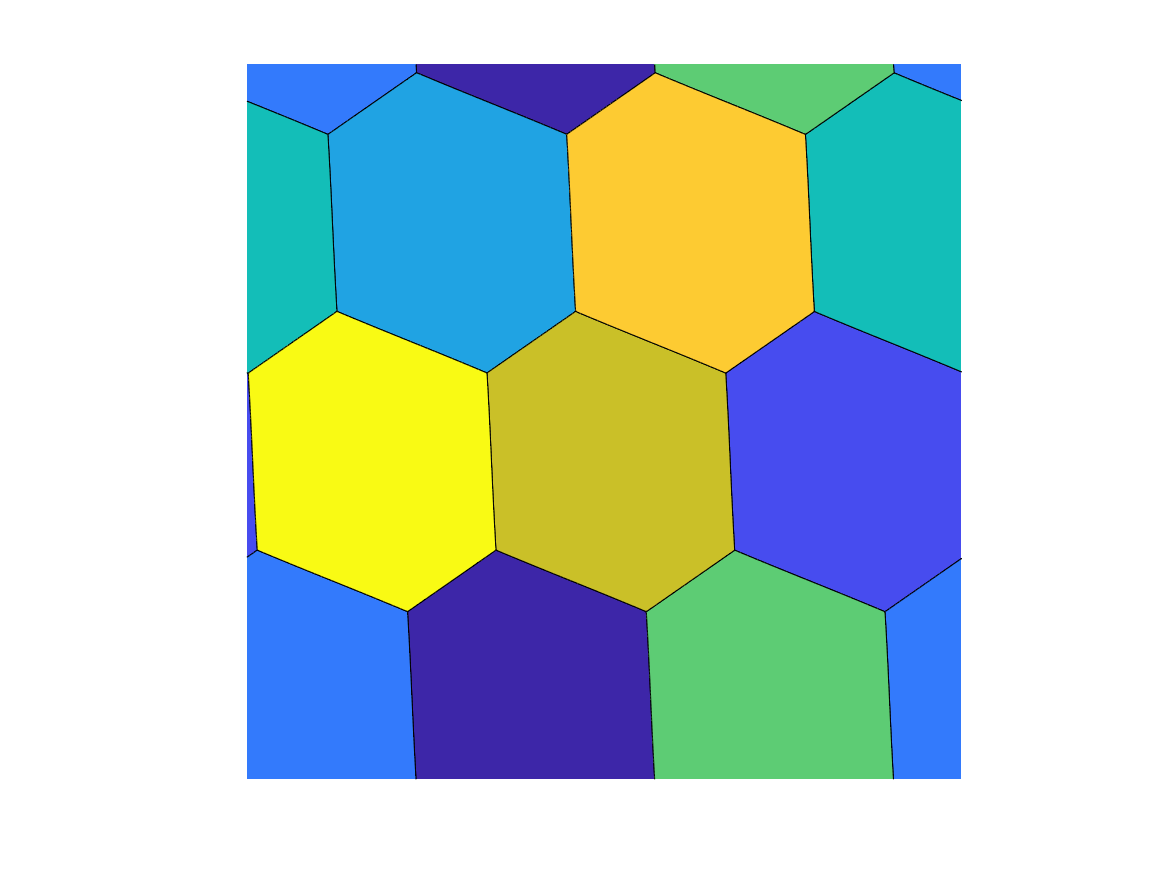}
		\includegraphics[width = 0.13\textwidth, clip, trim = 4cm 1cm 3cm 1cm]{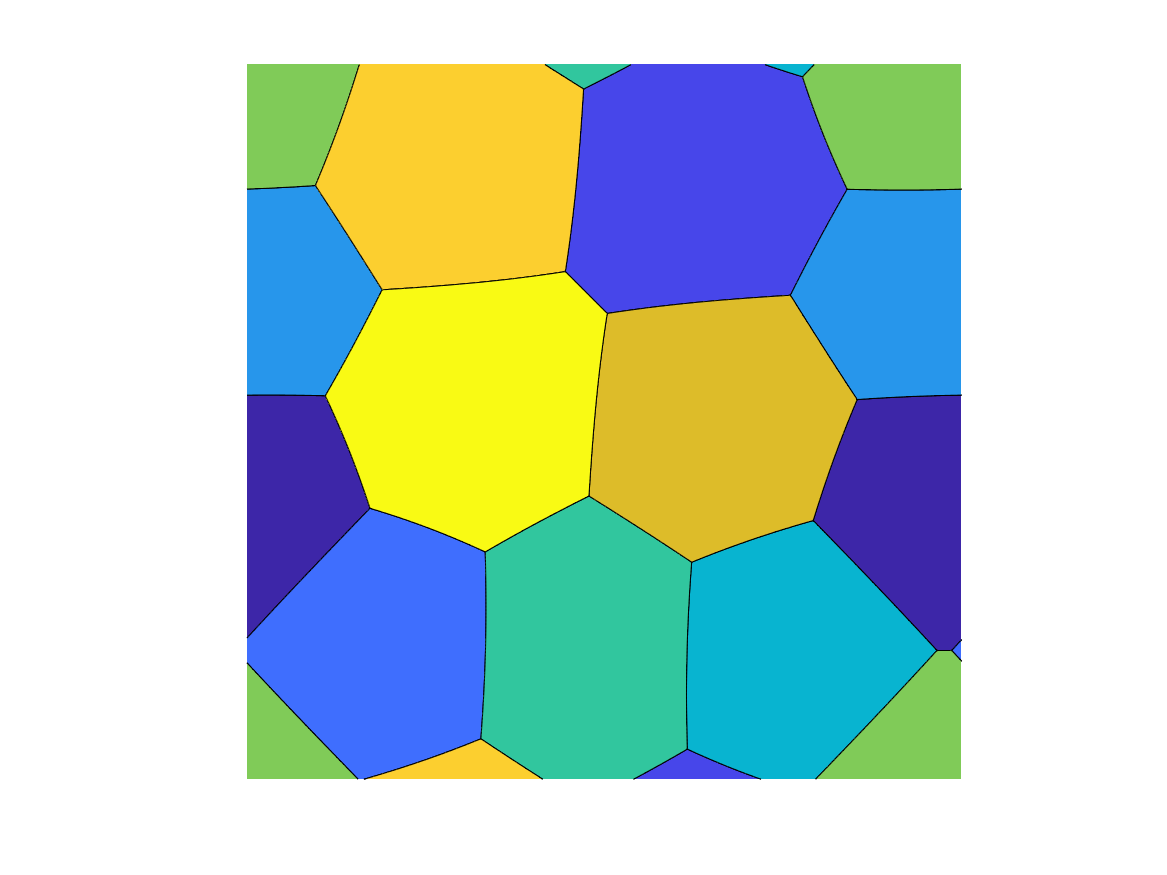}
		\includegraphics[width = 0.13\textwidth, clip, trim = 4cm 1cm 3cm 1cm]{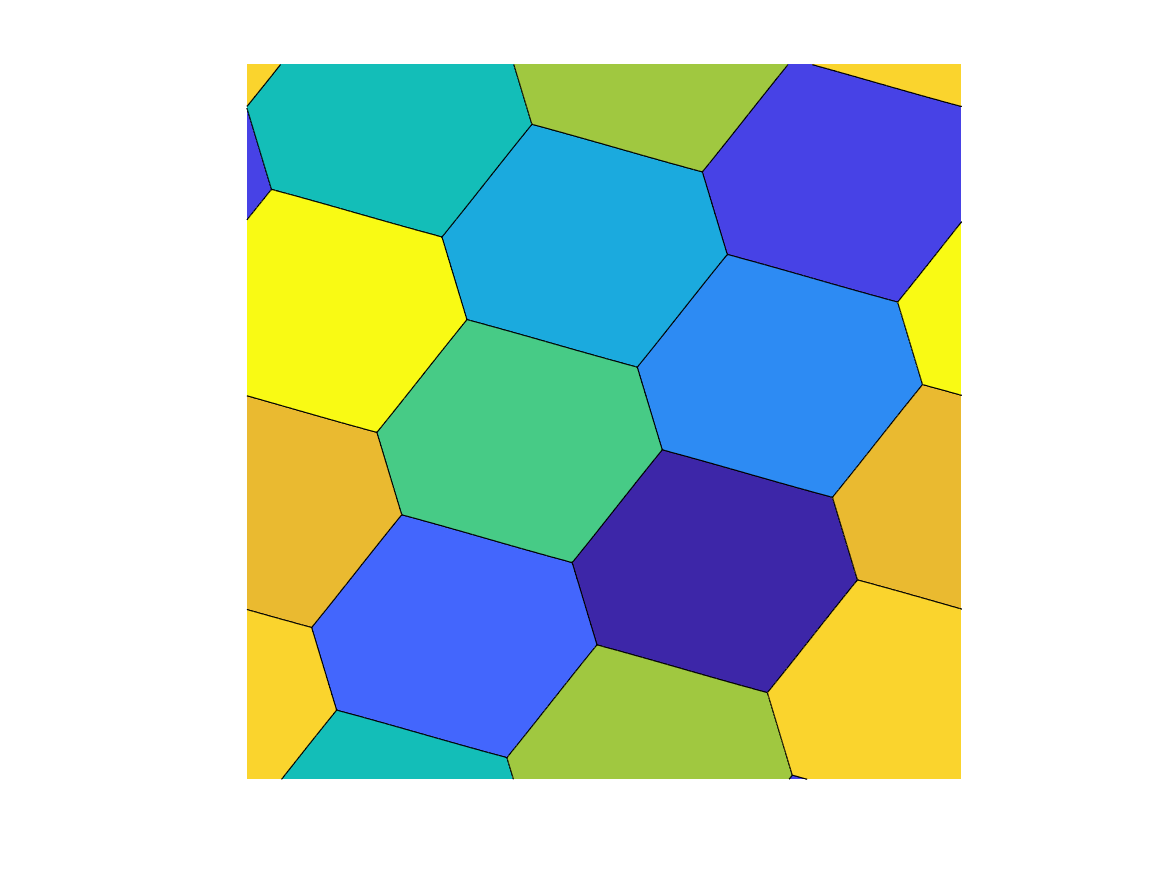}
		\includegraphics[width = 0.13\textwidth, clip, trim = 4cm 1cm 3cm 1cm]{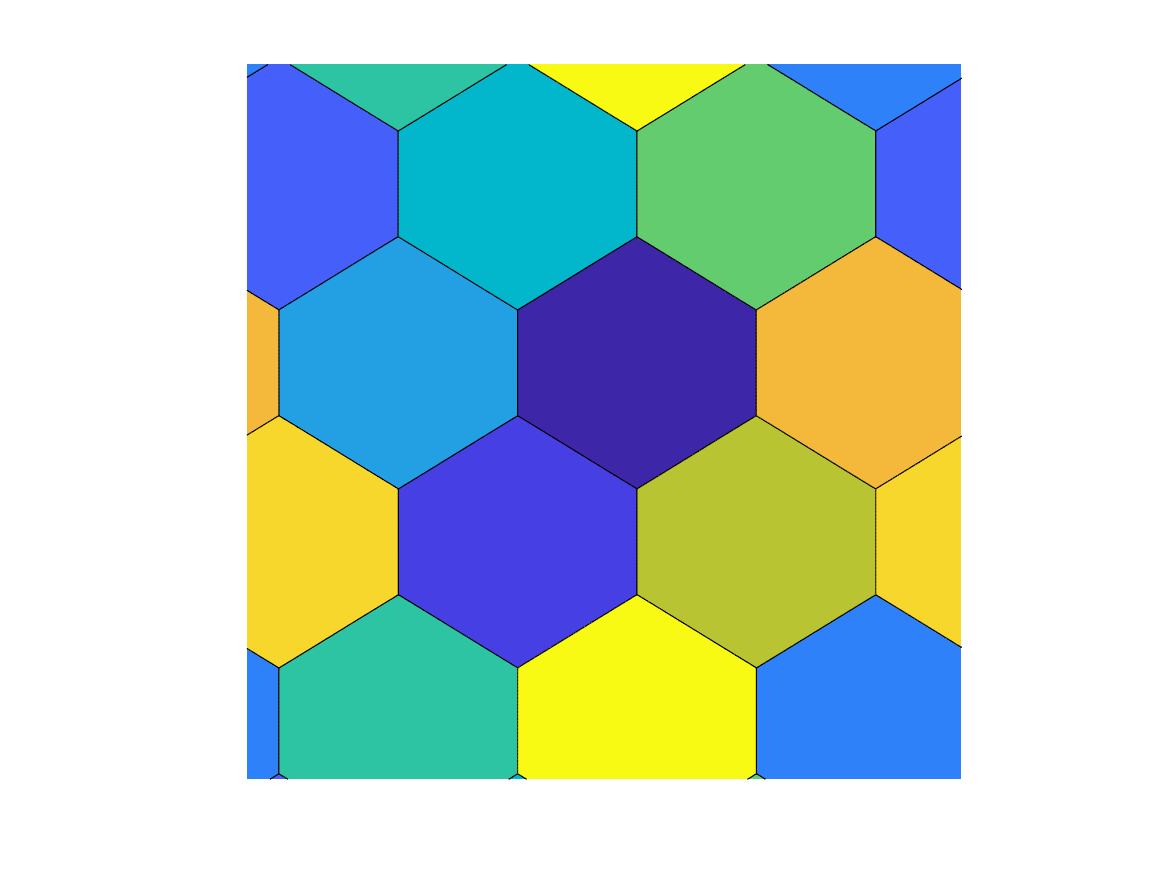}
		\includegraphics[width = 0.13\textwidth, clip, trim = 4cm 1cm 3cm 1cm]{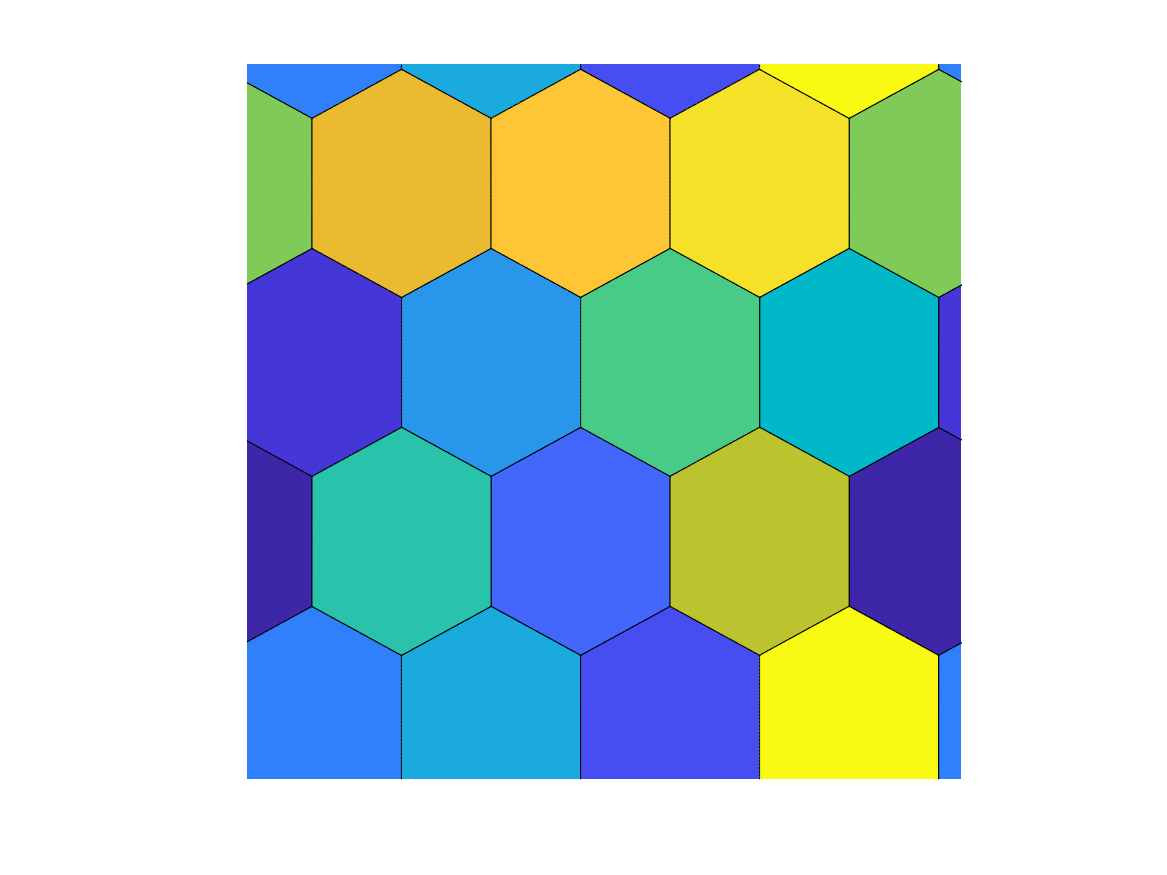}
		\includegraphics[width = 0.13\textwidth, clip, trim = 4cm 1cm 3cm 1cm]{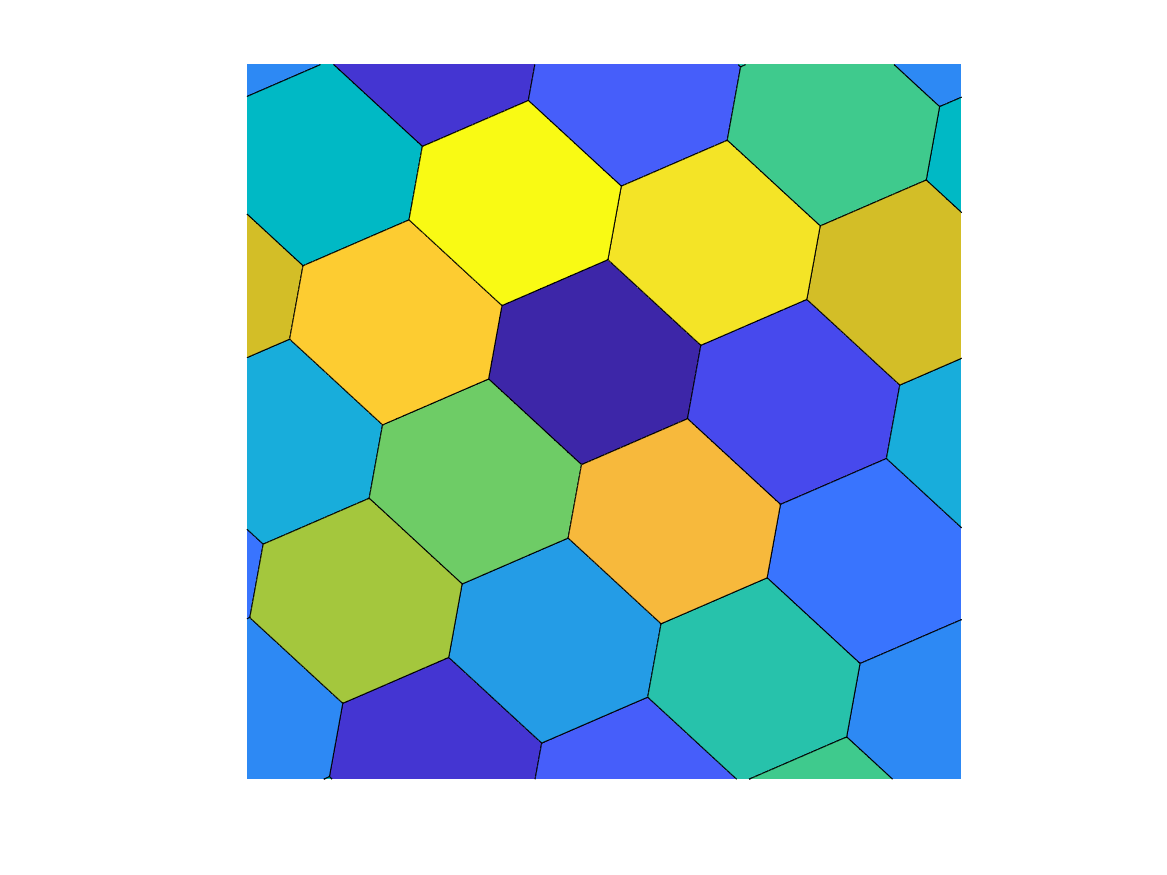}
		\includegraphics[width = 0.13\textwidth, clip, trim = 4cm 1cm 3cm 1cm]{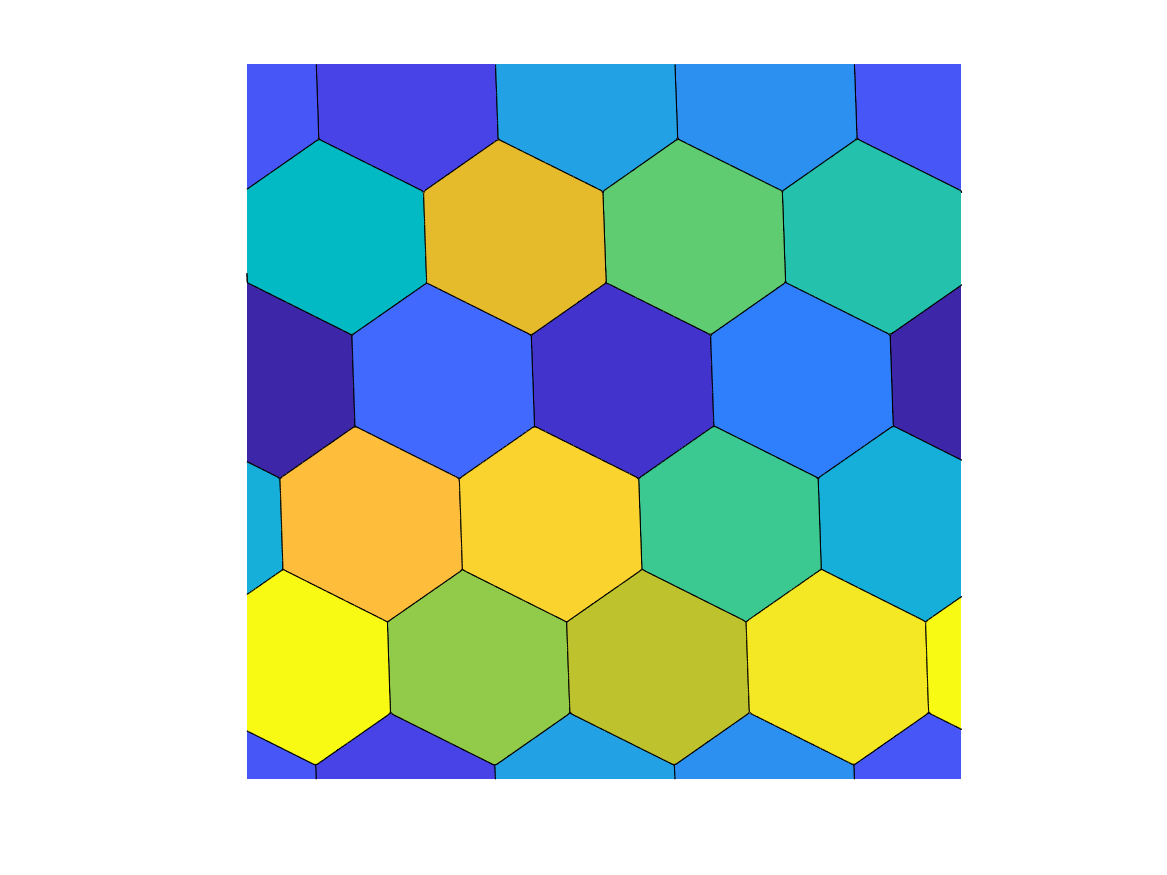}
		\includegraphics[width = 0.13\textwidth, clip, trim = 4cm 1cm 3cm 1cm]{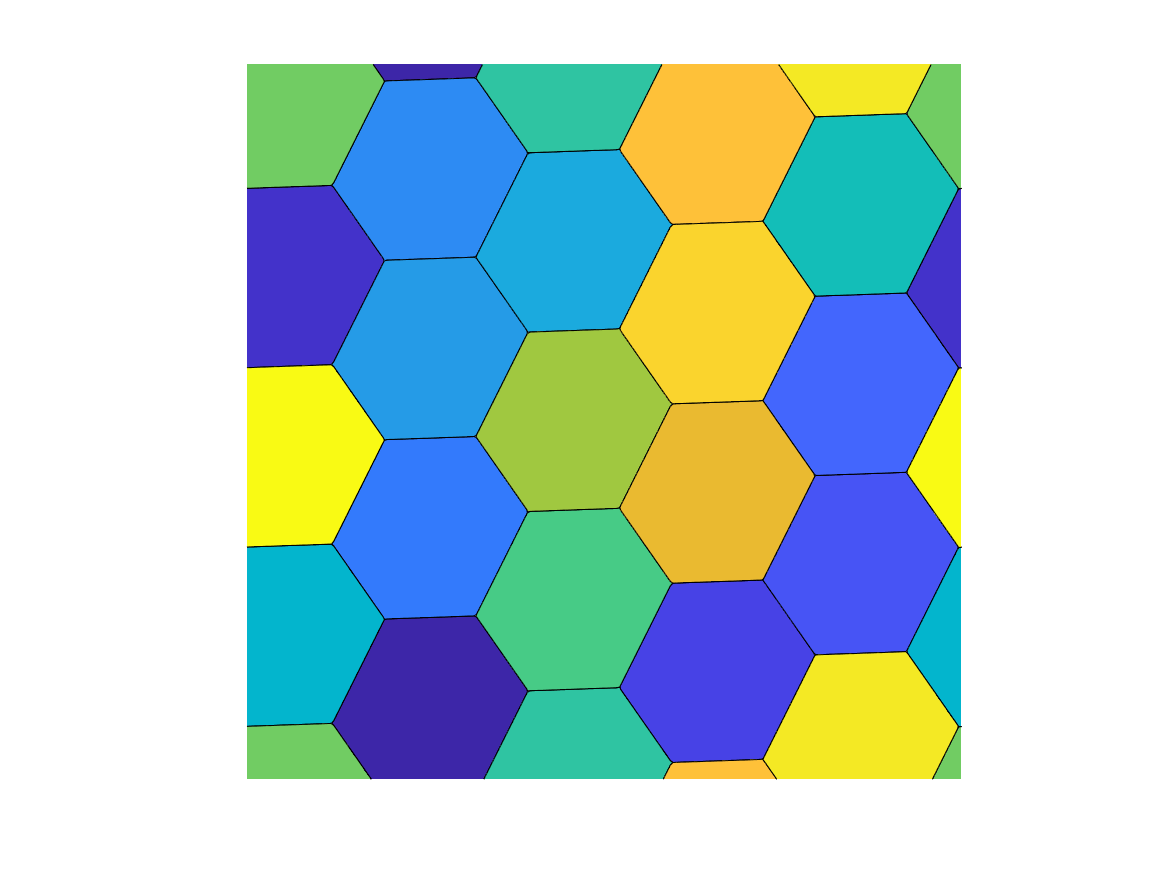}
		\includegraphics[width = 0.13\textwidth, clip, trim = 4cm 1cm 3cm 1cm]{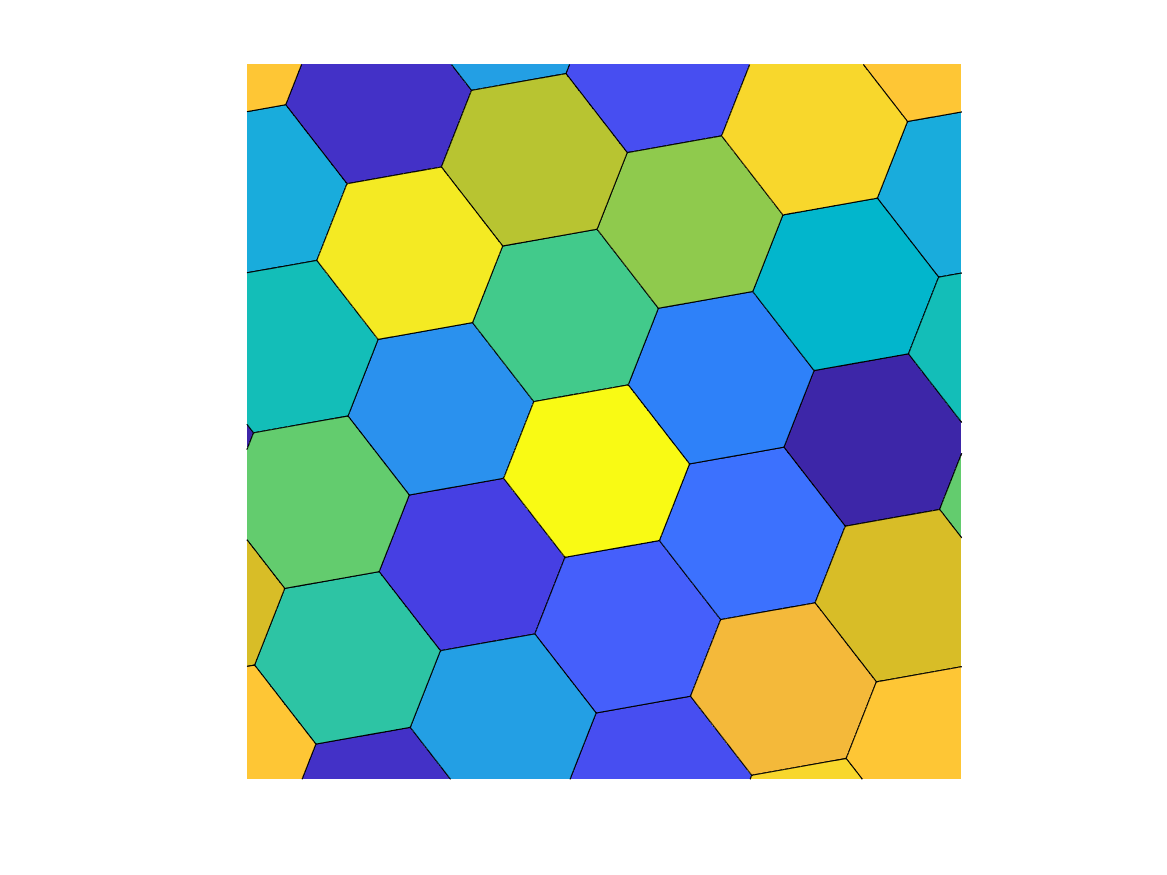}
		\includegraphics[width = 0.13\textwidth, clip, trim = 4cm 1cm 3cm 1cm]{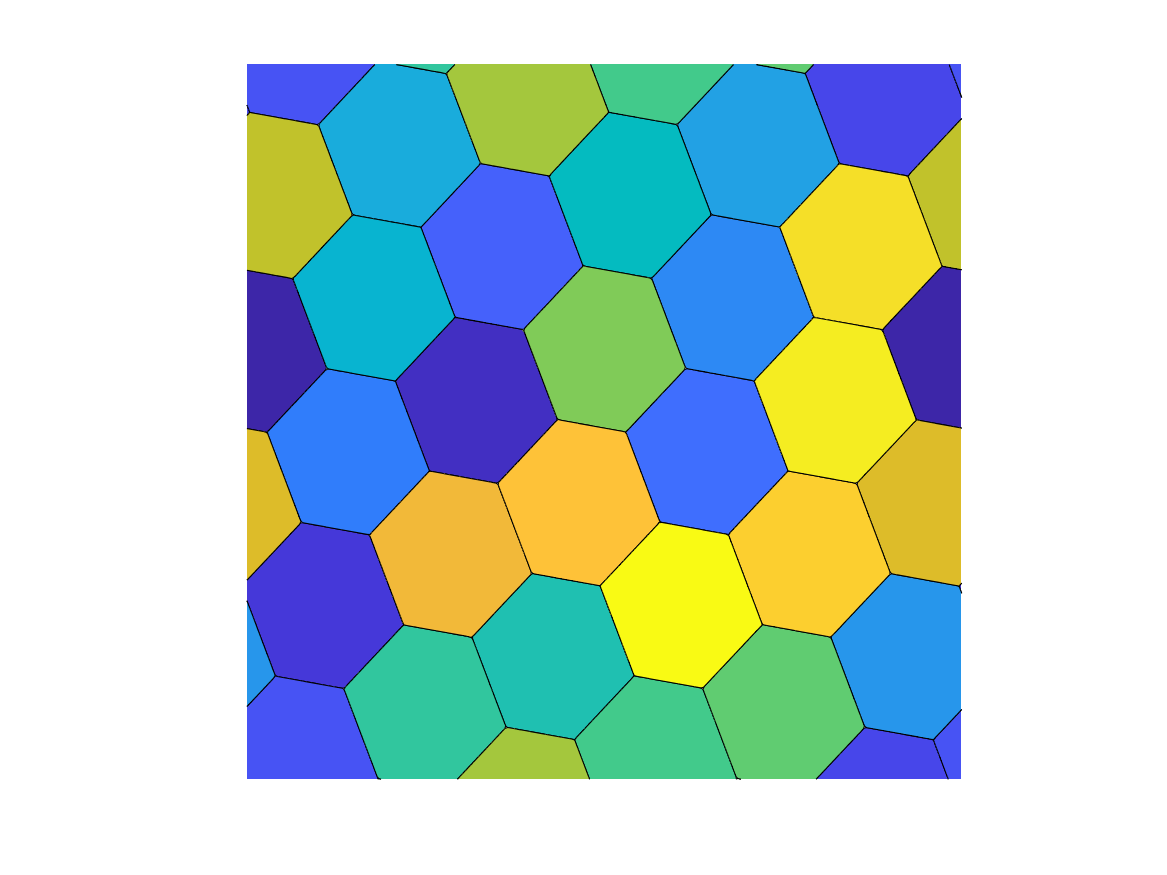}
		\caption{Numerical results of Algorithm~\ref{Alg:4step}  on a $256\times256$ discretized mesh with $\tau=0.1$, $k = 3-6$, $9-12$, $16$, $18$, $20$, $21$, $23$, $28$.} \label{fig:Alg_4stepkdt01}
	\end{figure}
	\begin{figure}[H]
		\centering
		\includegraphics[width = 0.13\textwidth, clip, trim = 4cm 1cm 3cm 1cm]{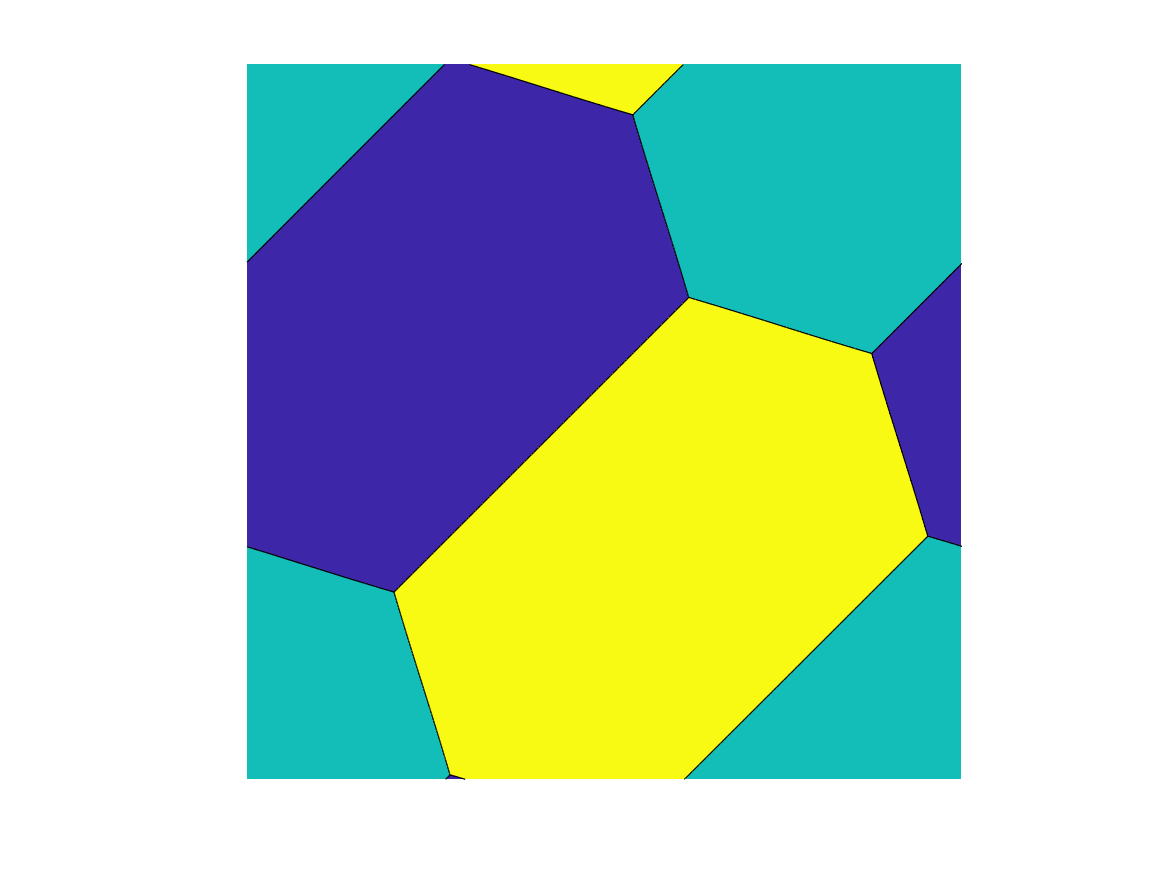}
		\includegraphics[width = 0.13\textwidth, clip, trim = 4cm 1cm 3cm 1cm]{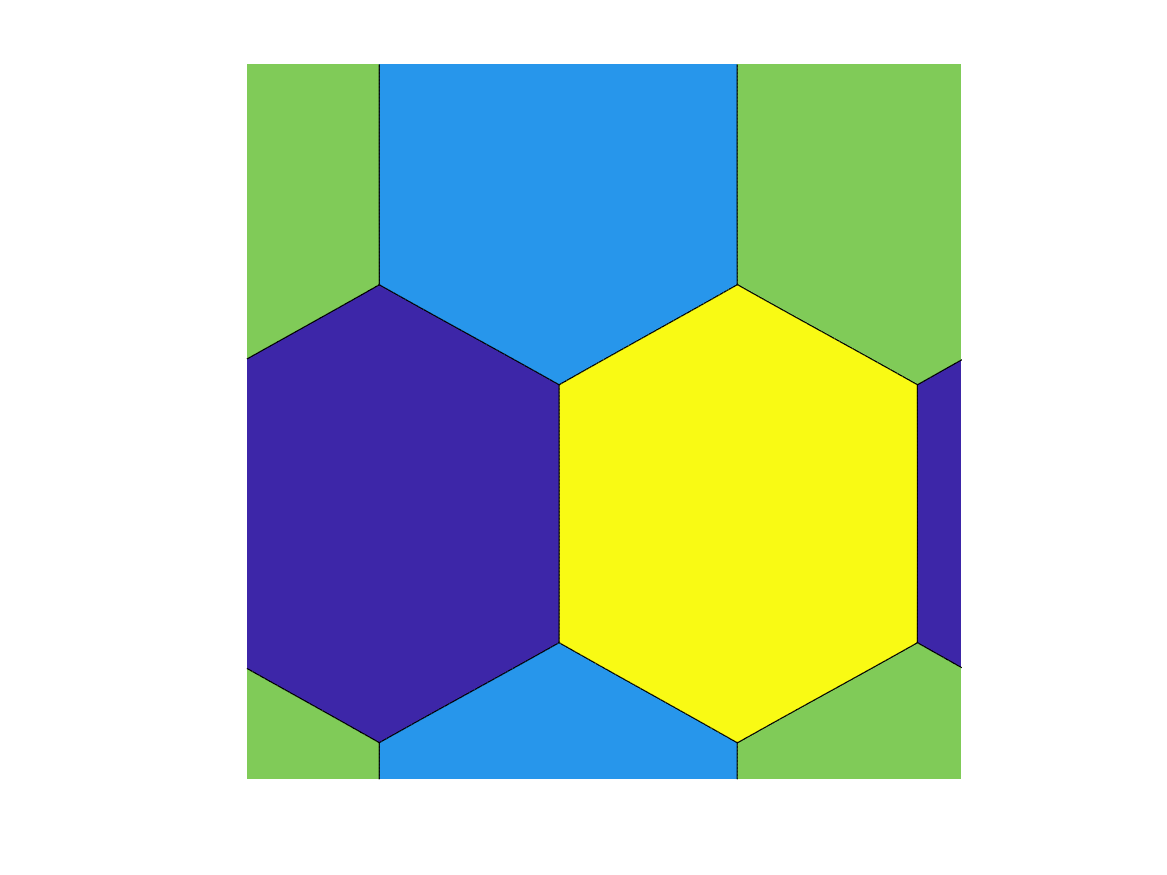}
		\includegraphics[width = 0.13\textwidth, clip, trim = 4cm 1cm 3cm 1cm]{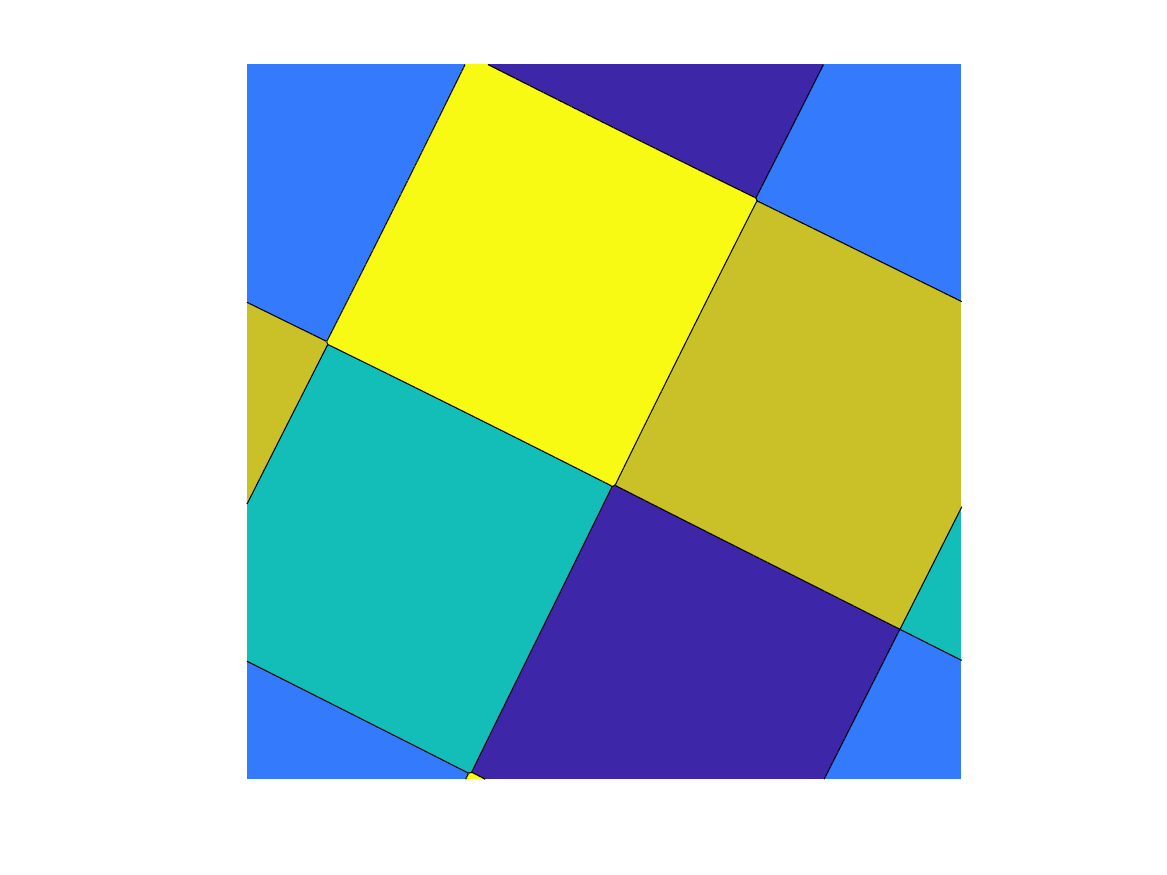}
		\includegraphics[width = 0.13\textwidth, clip, trim = 4cm 1cm 3cm 1cm]{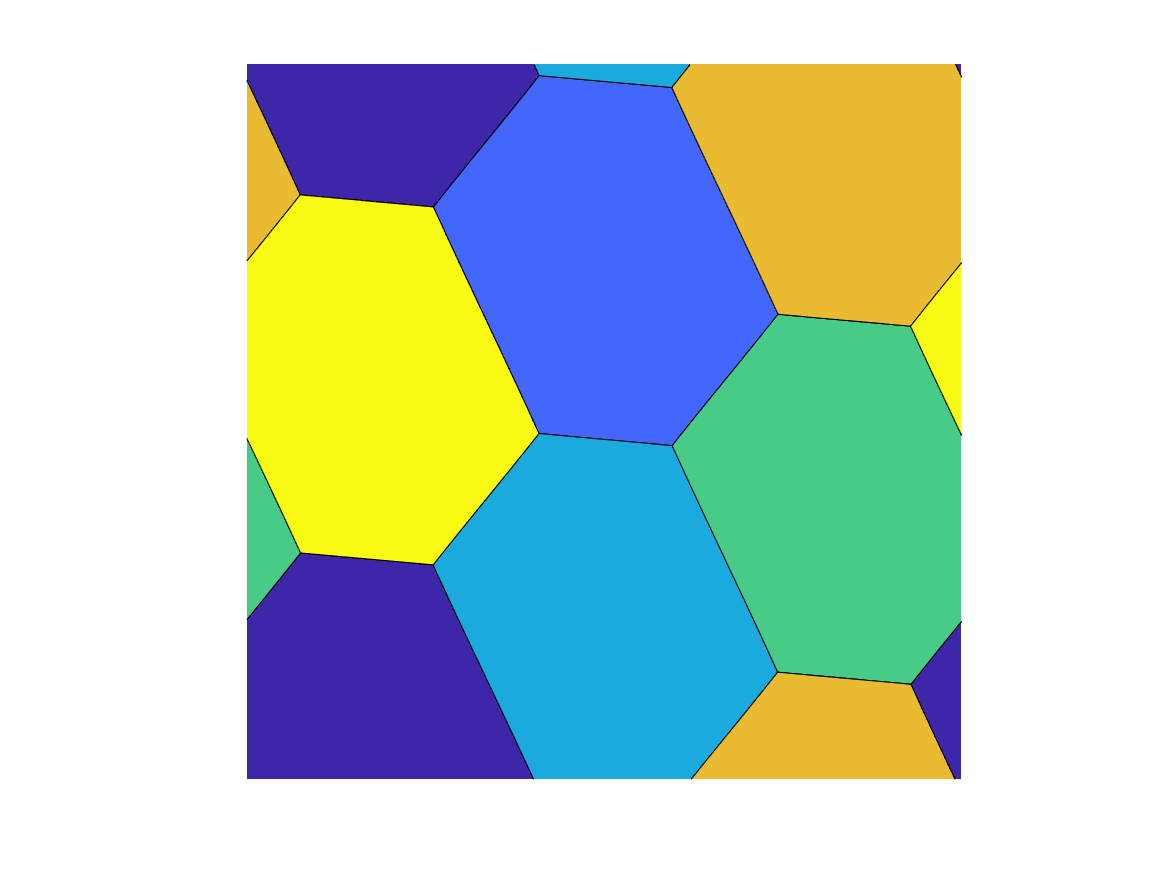}
		\includegraphics[width = 0.13\textwidth, clip, trim = 4cm 1cm 3cm 1cm]{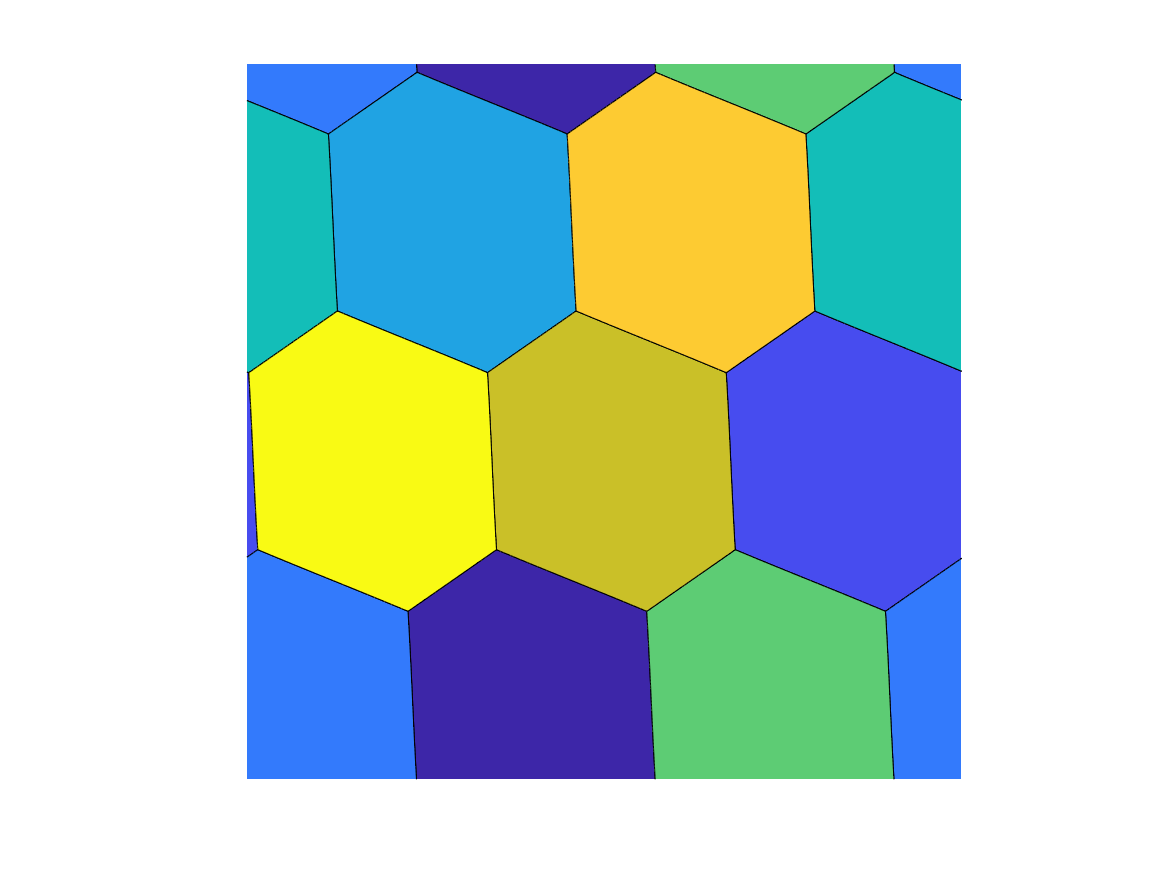}
		\includegraphics[width = 0.13\textwidth, clip, trim = 4cm 1cm 3cm 1cm]{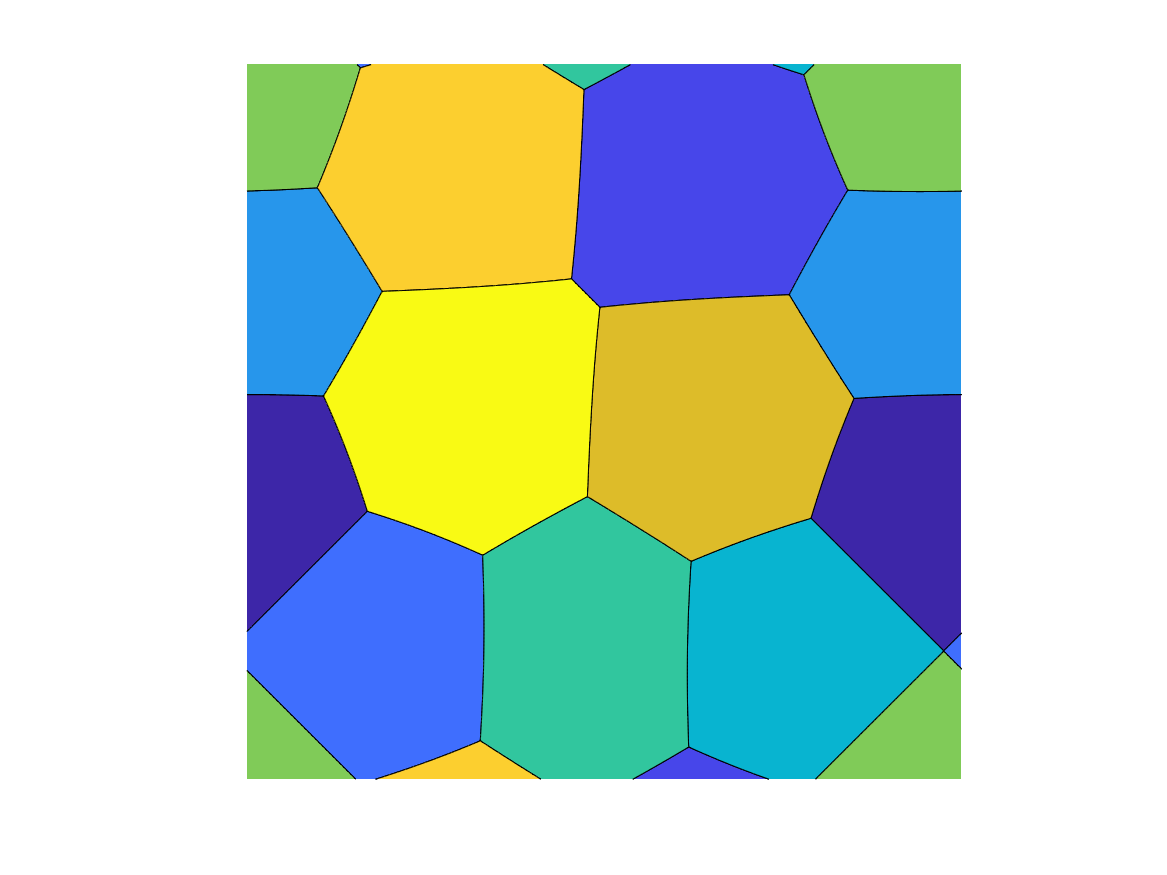}
		\includegraphics[width = 0.13\textwidth, clip, trim = 4cm 1cm 3cm 1cm]{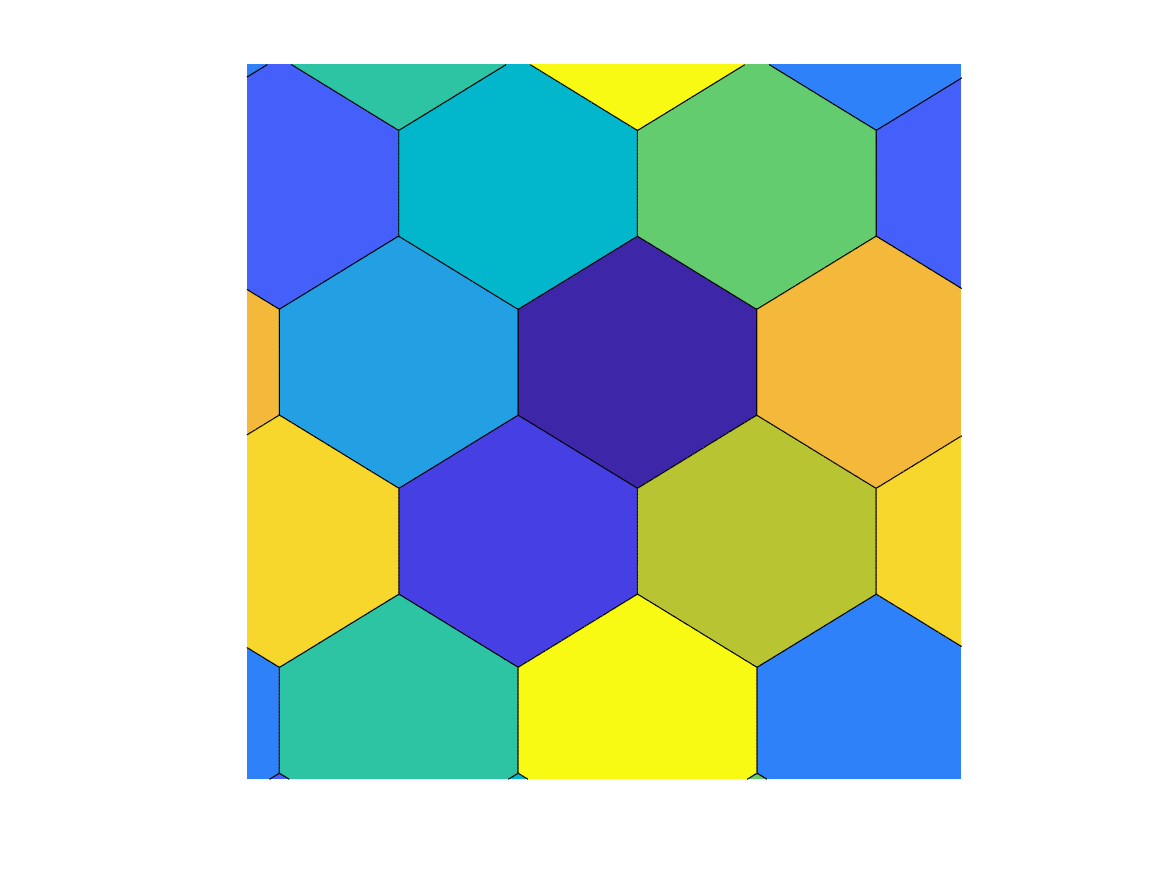}
		\includegraphics[width = 0.13\textwidth, clip, trim = 4cm 1cm 3cm 1cm]{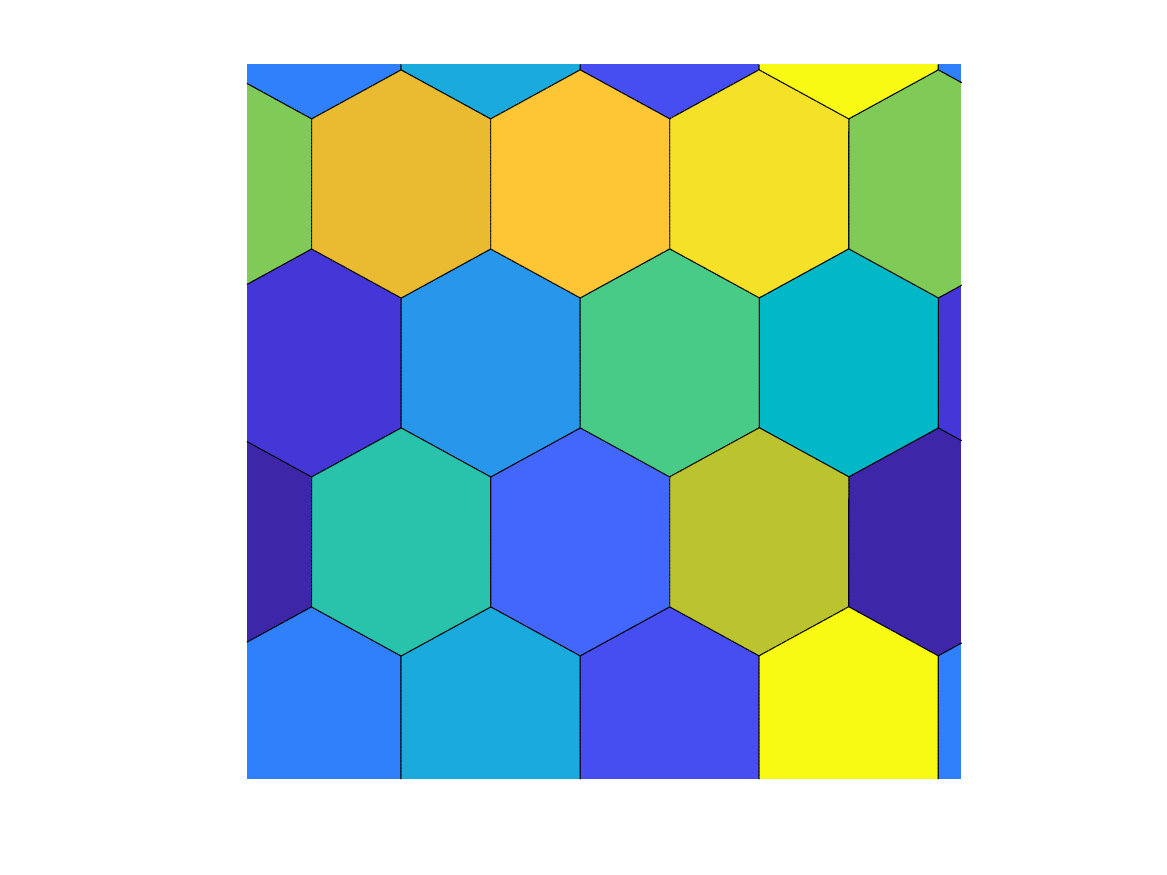}
		\includegraphics[width = 0.13\textwidth, clip, trim = 4cm 1cm 3cm 1cm]{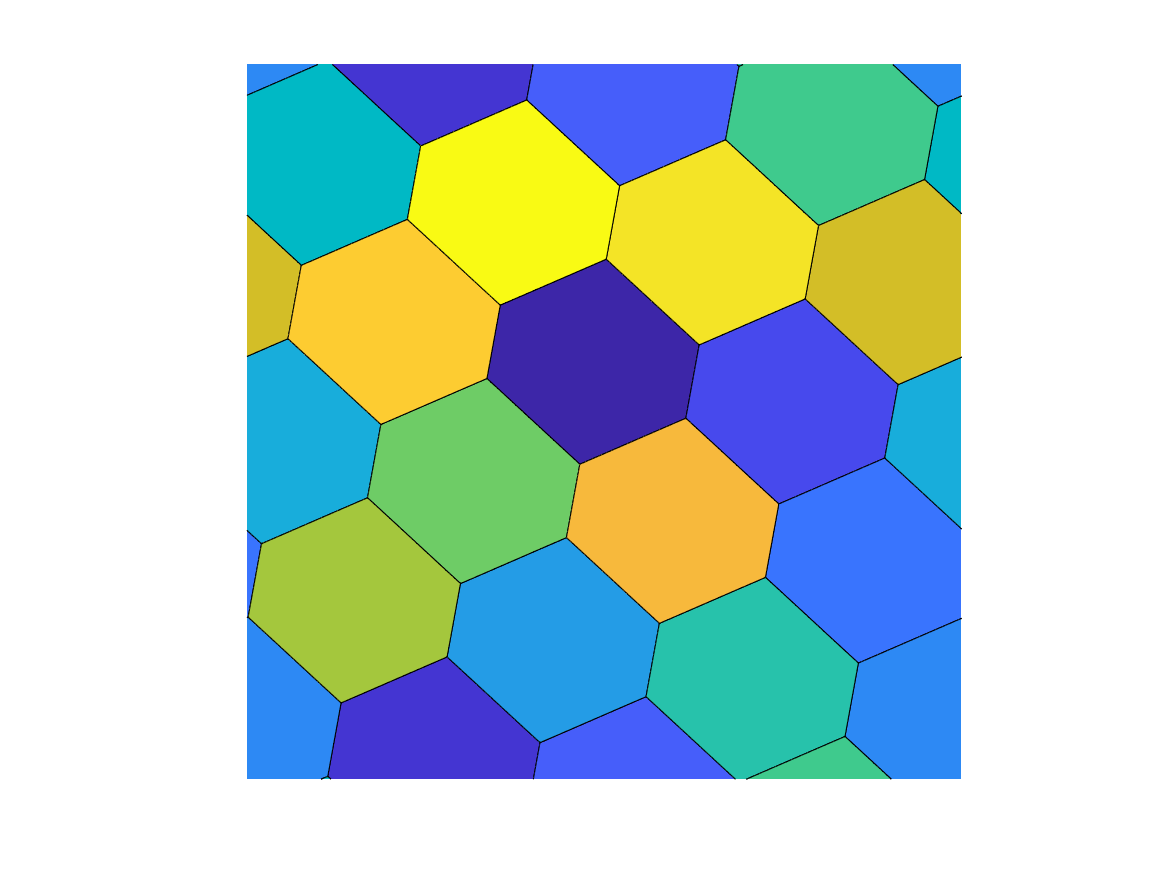}
		\includegraphics[width = 0.13\textwidth, clip, trim = 4cm 1cm 3cm 1cm]{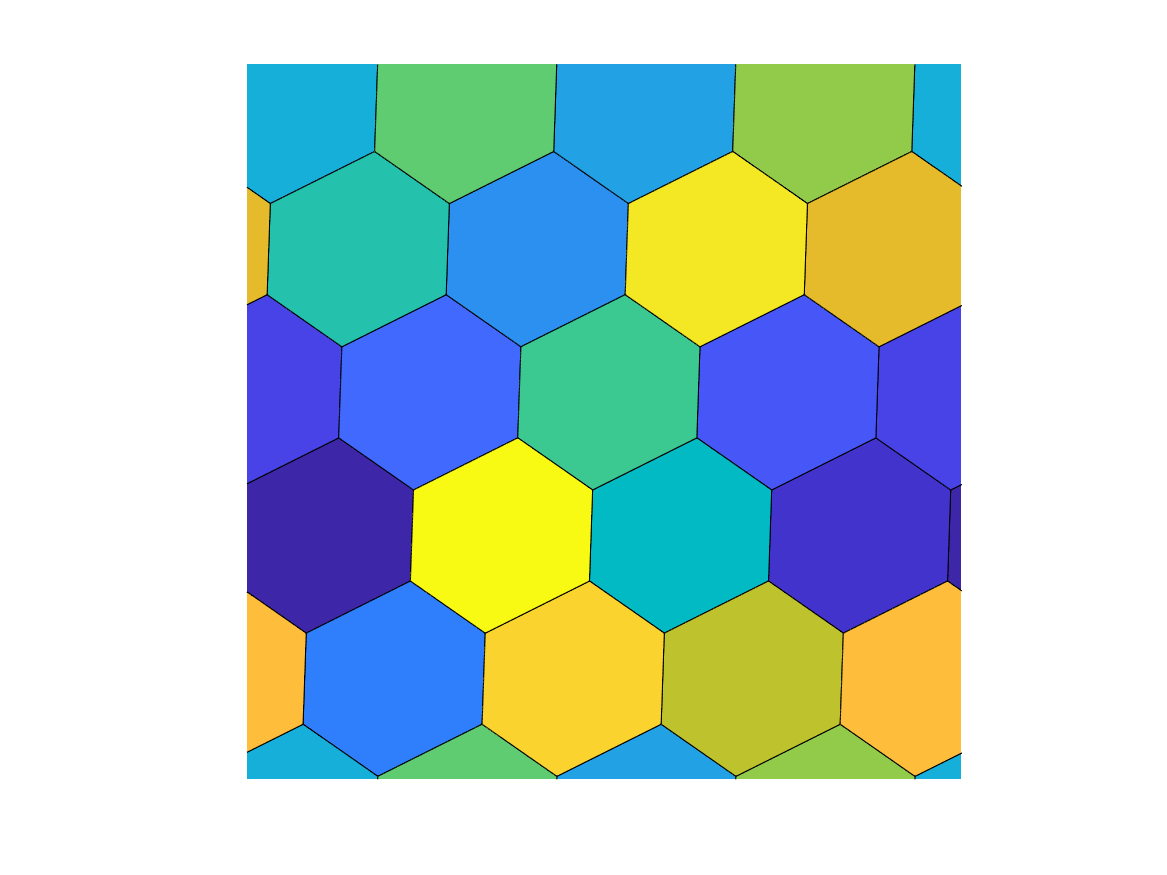}
		\includegraphics[width = 0.13\textwidth, clip, trim = 4cm 1cm 3cm 1cm]{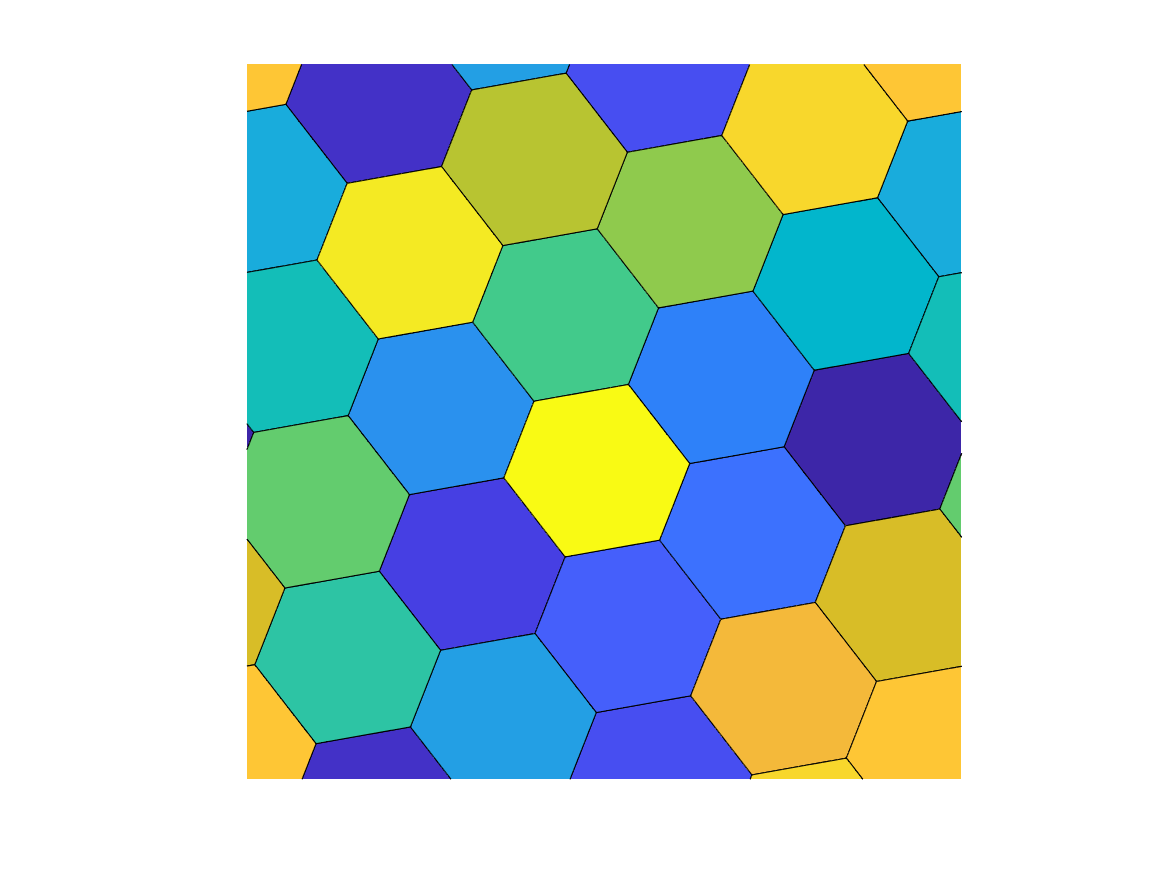}
		\includegraphics[width = 0.13\textwidth, clip, trim = 4cm 1cm 3cm 1cm]{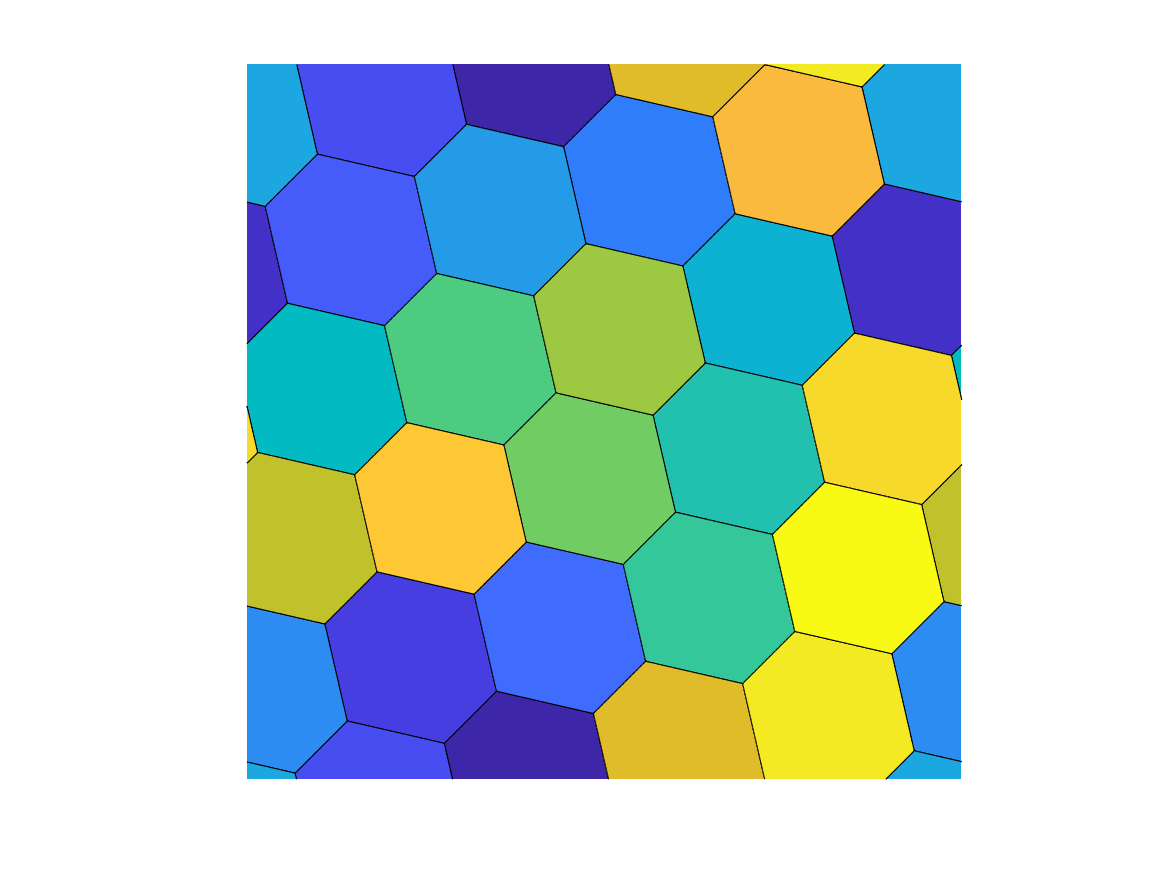}
		\includegraphics[width = 0.13\textwidth, clip, trim = 4cm 1cm 3cm 1cm]{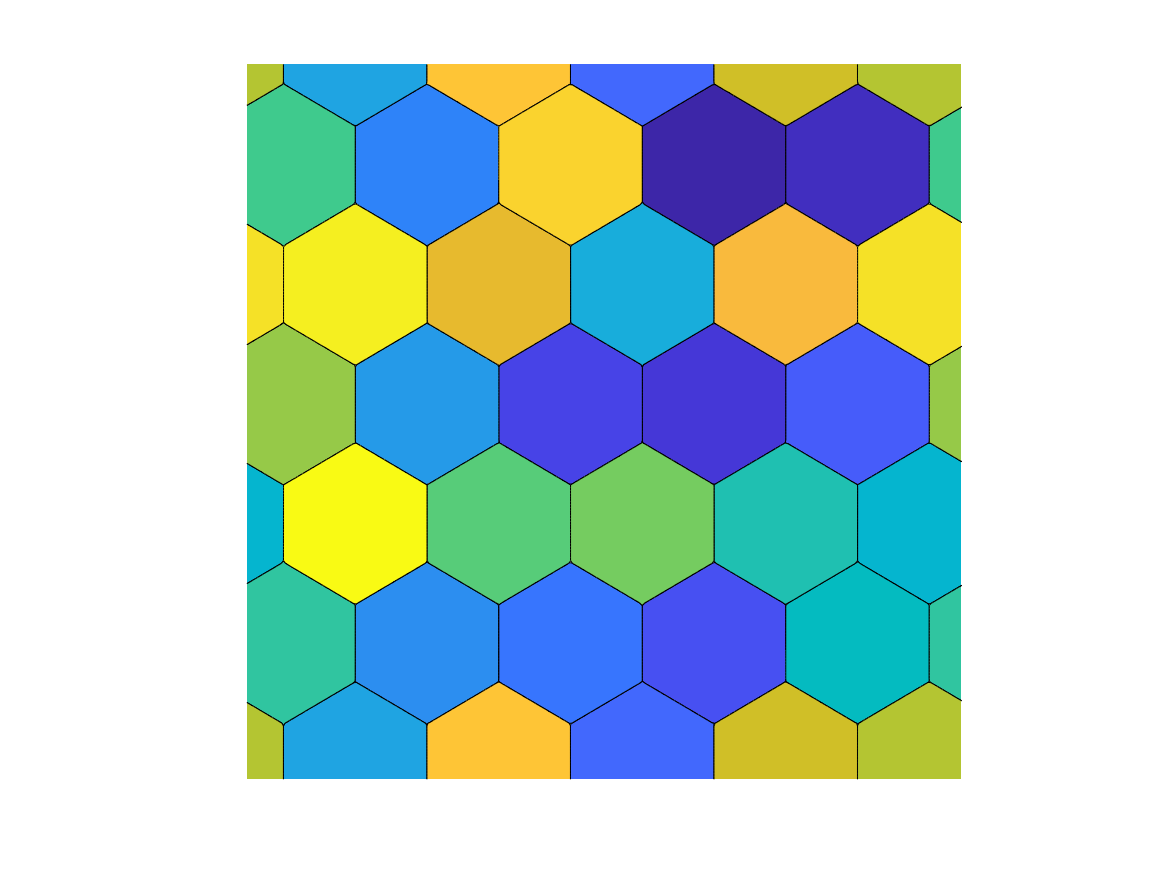}
		\includegraphics[width = 0.13\textwidth, clip, trim = 4cm 1cm 3cm 1cm]{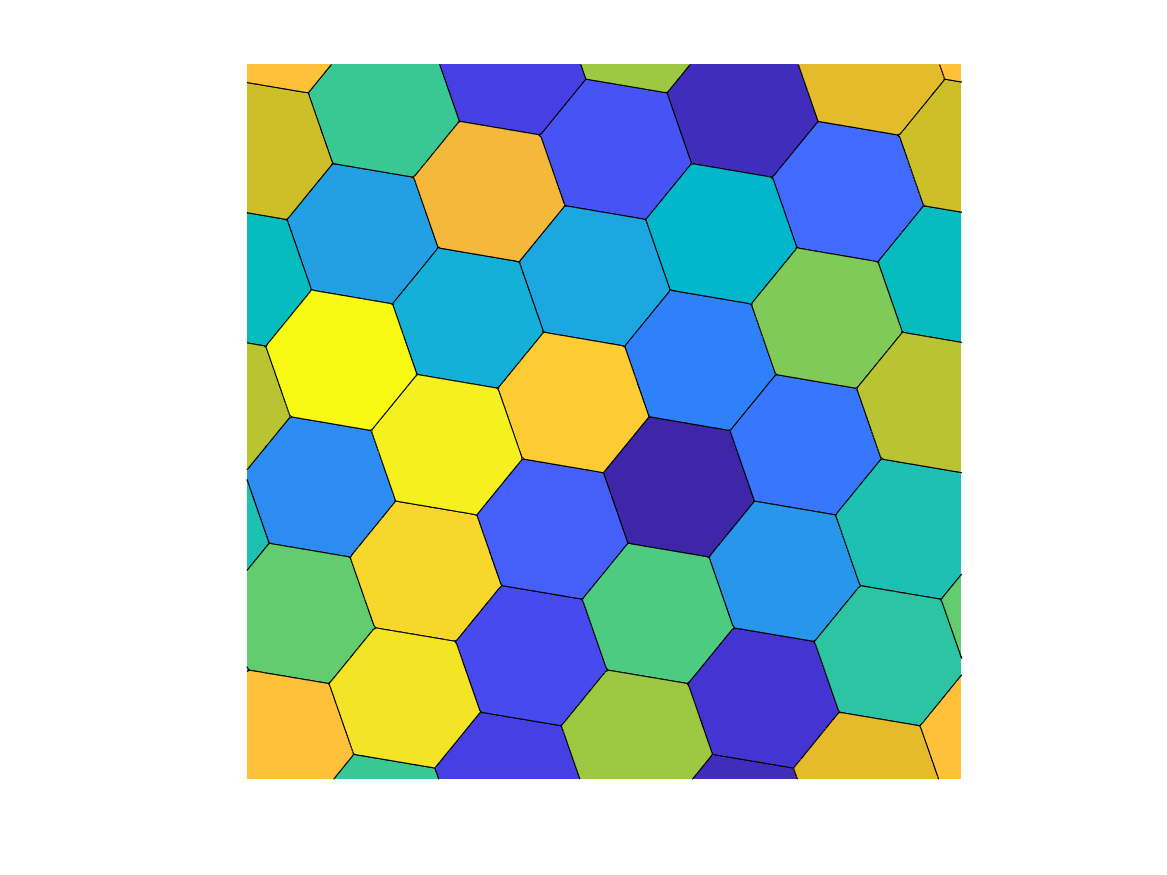}
		\caption{Numerical results of Algorithm~\ref{Alg:3step_Type1}  on a $256\times256$ discretized mesh with $\tau=0.1$, $k = 3-6$, $9$, $10$, $12$, $16$, $18$, $20$,  $23$, $24$, $30$, $34$.} \label{fig:Alg_3stepType1kdt01}
	\end{figure}
	\begin{figure}[ht]
		\centering
		\includegraphics[width = 0.3\textwidth, clip, trim = 8cm 3cm 6cm 2cm]{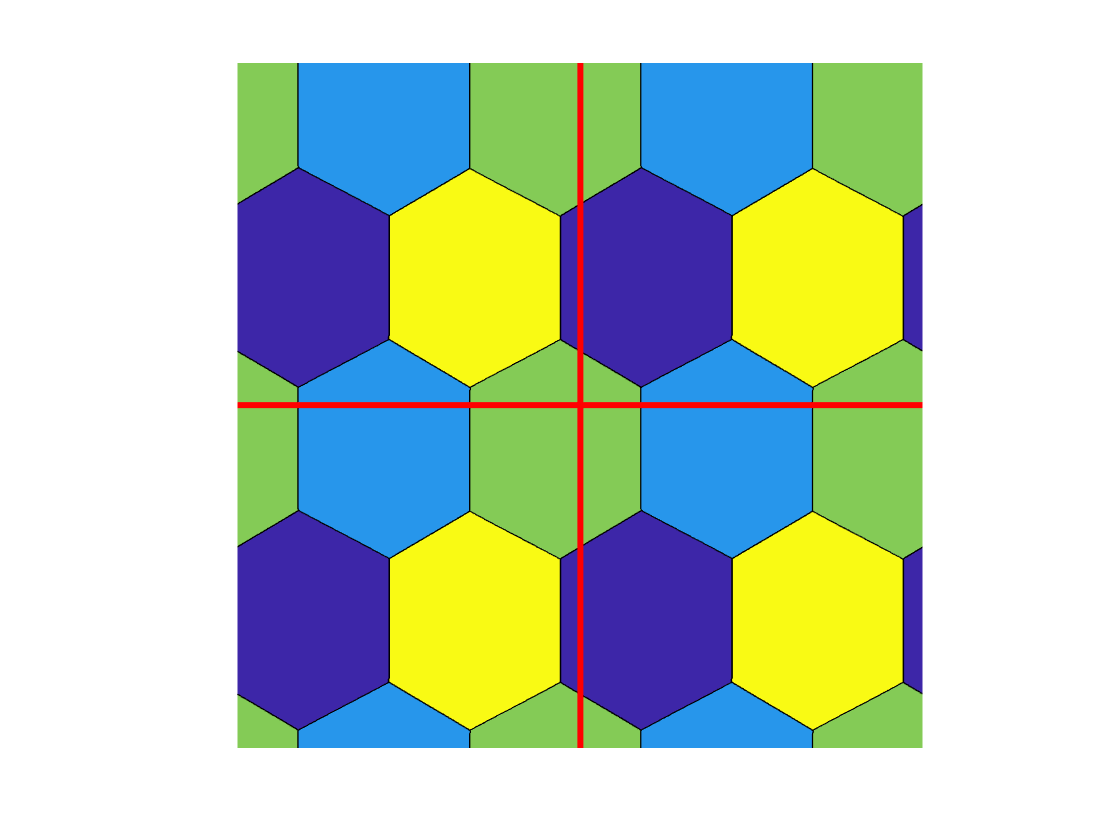}
		\includegraphics[width = 0.3\textwidth, clip, trim = 8cm 3cm 6cm 2cm]{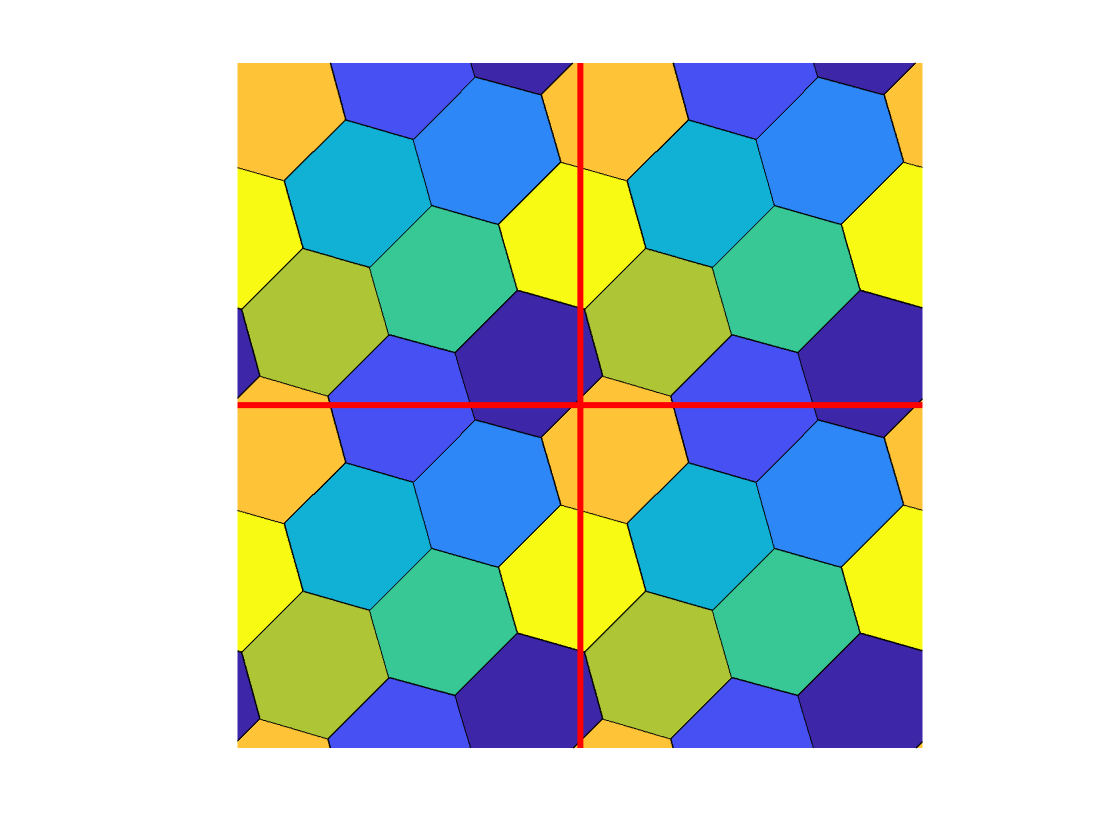}
		\includegraphics[width = 0.3\textwidth, clip, trim = 8cm 3cm 6cm 2cm]{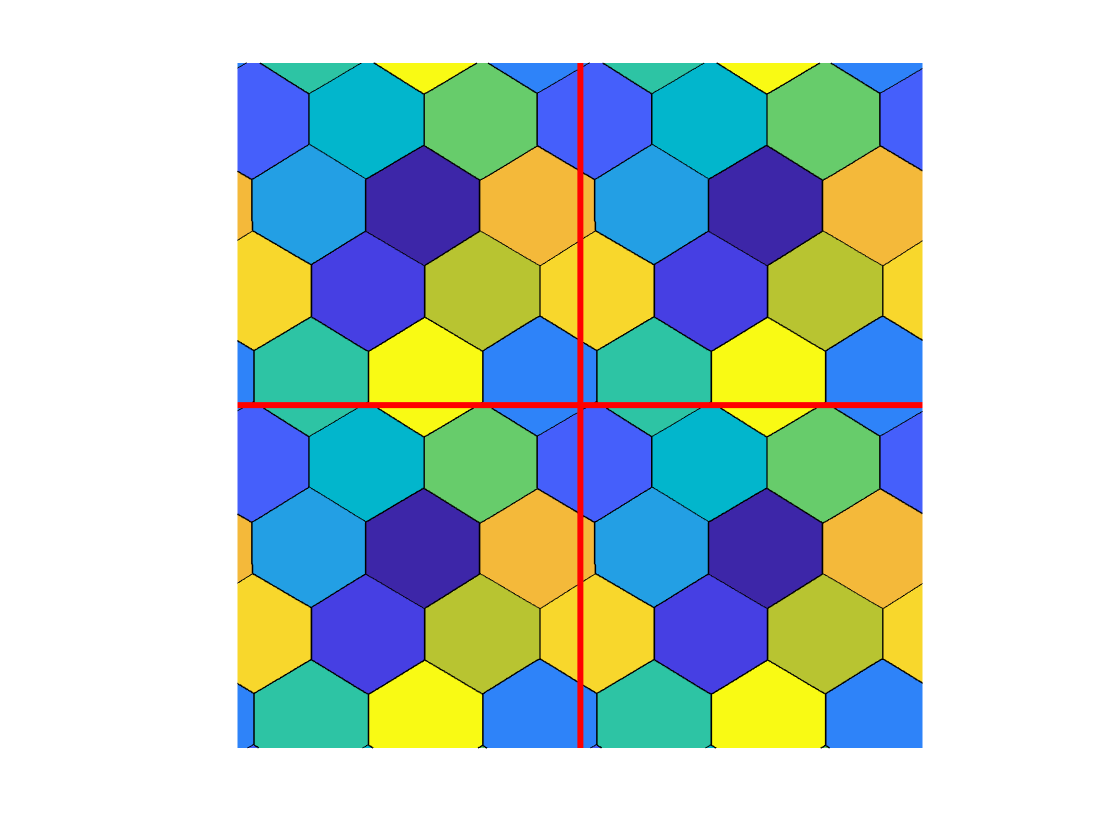}
		
		\caption{Periodic extensions of $k$-partitions implemented by Algorithm~\ref{Alg:3step_Type2} with $k = 4,8,12$ on a $512\times512$ discretized mesh with $\tau=0.01$.} \label{fig:Alg3_extension}
	\end{figure}
	\subsubsection{Partitions for 3-dimensional tori}
To test the feasibility and efficiency of the proposed algorithm for the optimal partition problem in a three-dimensional domain, we use a uniform mesh with \(128^3\) points to approximate the problem under periodic boundary conditions within the domain \([- \pi, \pi]^3\).  The numerical results obtained using Algorithm~\ref{Alg:4step} are presented in Figures~\ref{fig:Alg1m43D} and \ref{fig:Alg1m83D}. The structures depicted in these figures are consistent with those reported in \cite{Wang22},  demonstrating the algorithm's ability to effectively capture and reproduce the expected partitioning patterns. The results highlight the robustness of the algorithm in solving the three-dimensional optimal partition problem.
	\begin{figure}[H]
		\centering
		\includegraphics[width = 0.3\textwidth, trim = 3cm 8cm 6cm 5.5cm]{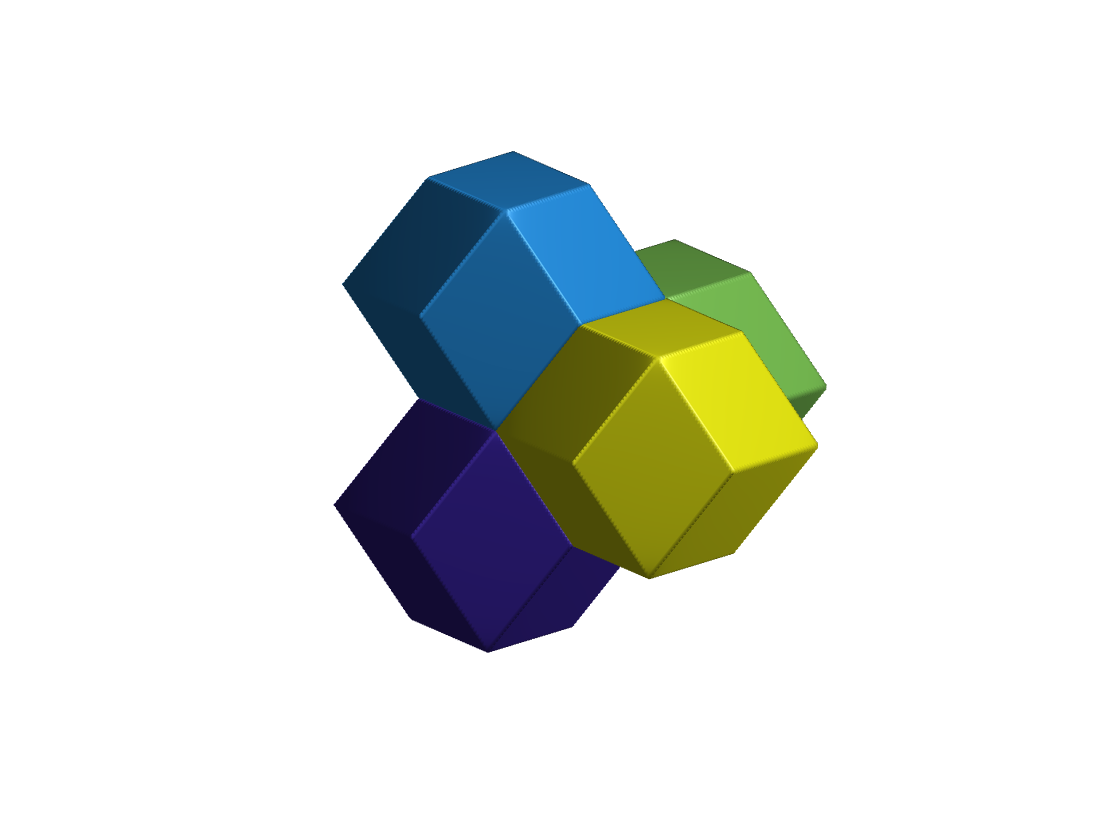} 
		\includegraphics[width = 0.3\textwidth,clip, trim = 3cm 6cm 6cm 2cm]{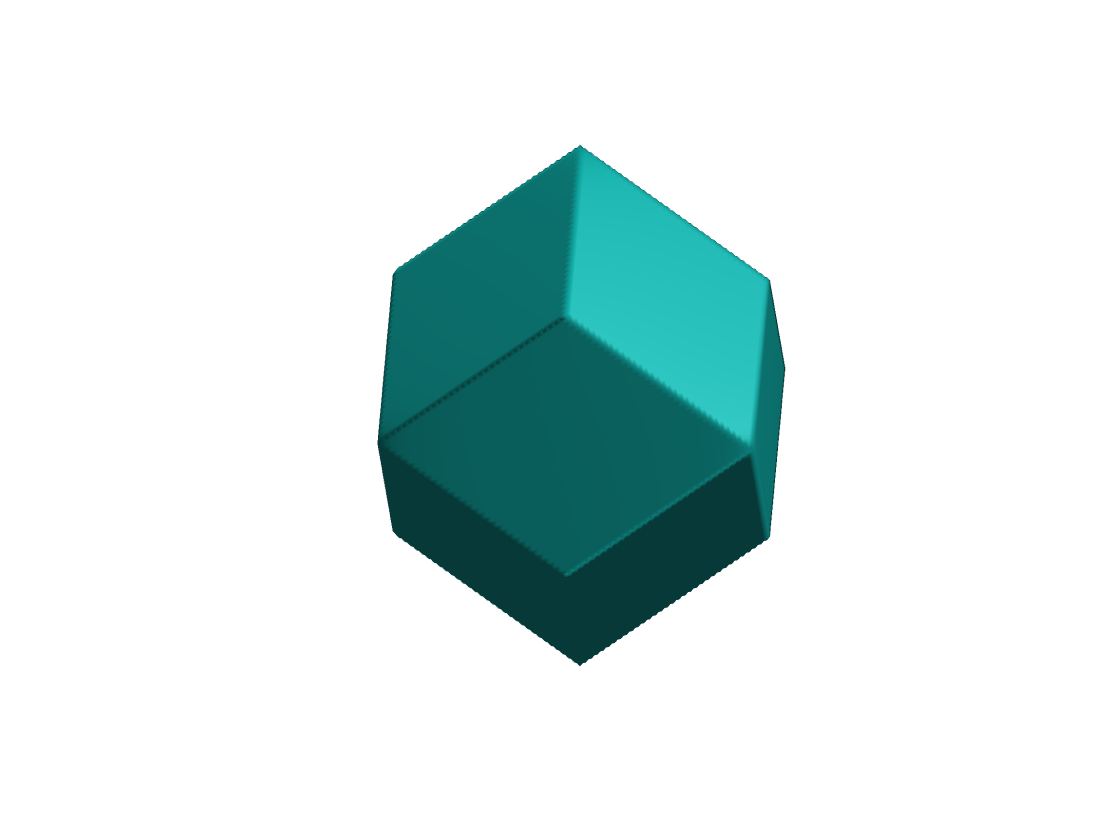} 
		\caption{Numerical results of  Algorithm~\ref{Alg:4step} on a $128\times 128\times 128$ discretized mesh in space with the time step size $\tau=\pi/16 $.  Left: A periodic extension on the $4$-partition in a 3-dimensional flat torus. Middle: The rhombic dodecahedron structure. } \label{fig:Alg1m43D}
	\end{figure}

	\begin{figure}[H]
		\centering
		\includegraphics[width = 0.2\textwidth,clip, trim = 8cm 5cm 7cm 6cm]{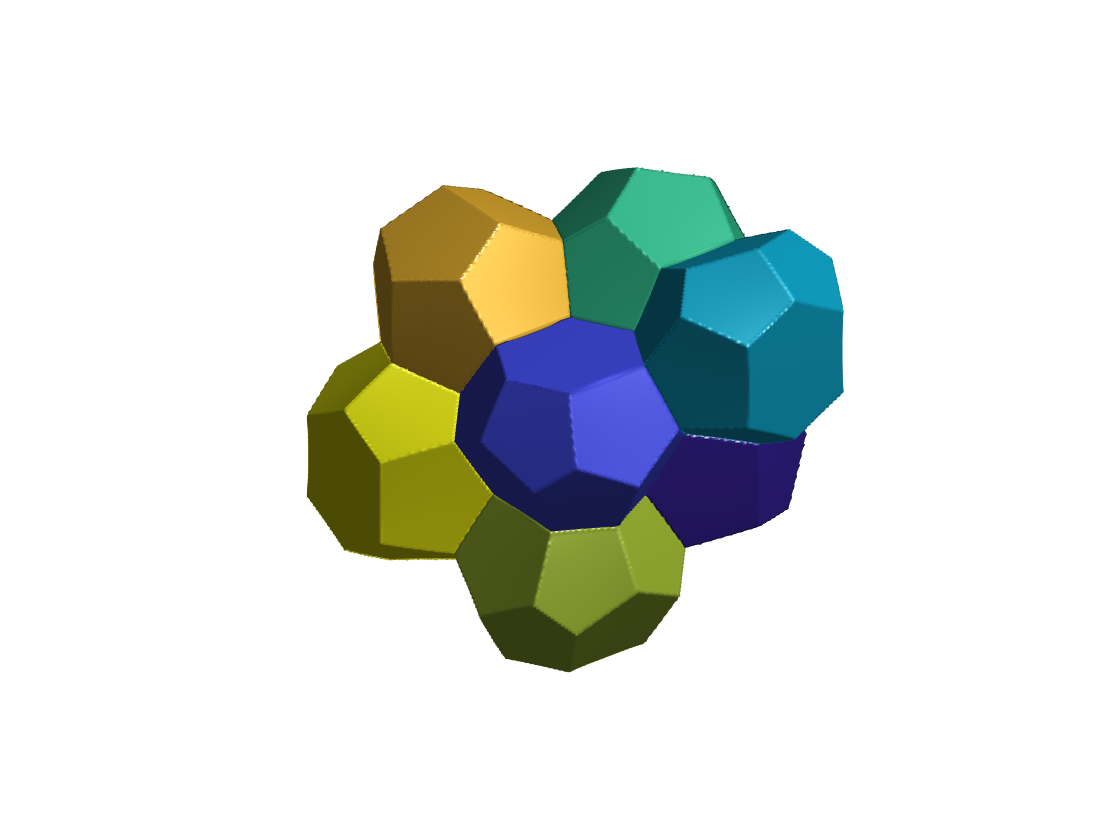}
		\includegraphics[width = 0.2 \textwidth,clip, trim = 5cm 3cm 7cm 1cm]{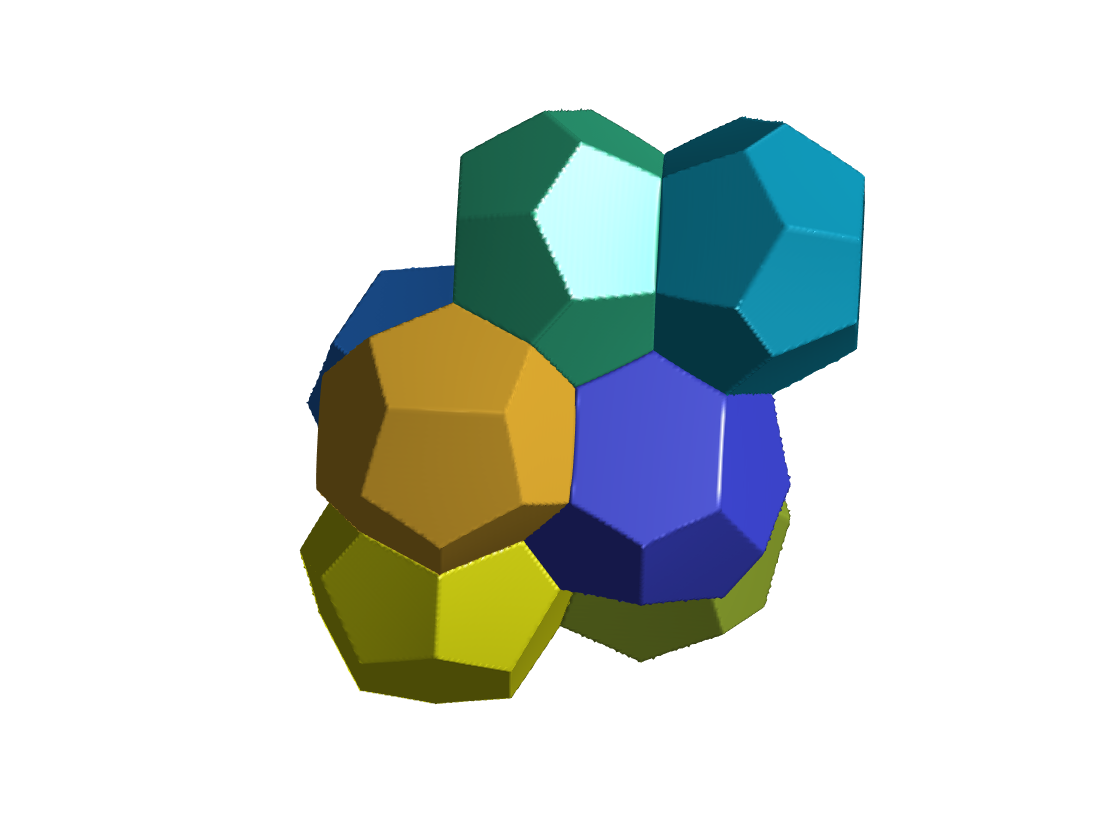}
		\includegraphics[width = 0.2 \textwidth,clip, trim = 5cm 2cm 5cm 1cm]{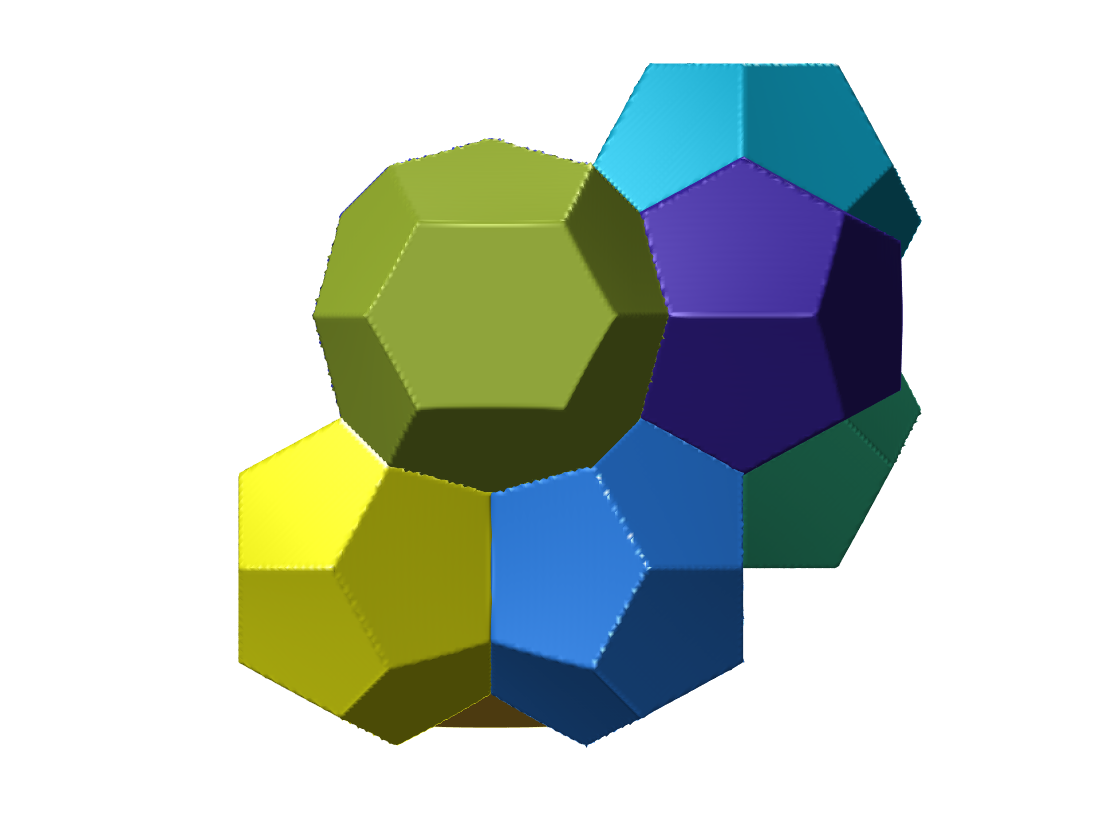}
		\includegraphics[width = 0.2 \textwidth,clip, trim = 6cm 3cm 5cm 1cm]{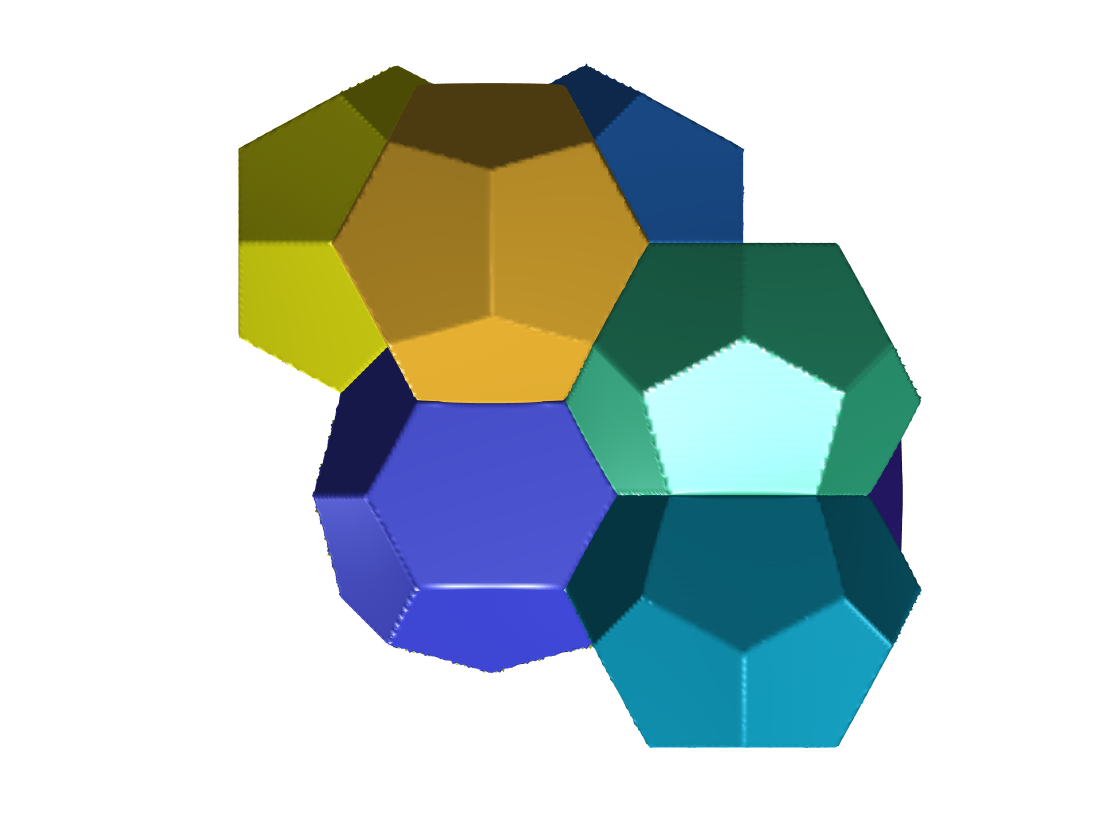}
		\includegraphics[width = 0.2\textwidth,clip, trim = 8.5cm 5.2cm 10cm 6cm]{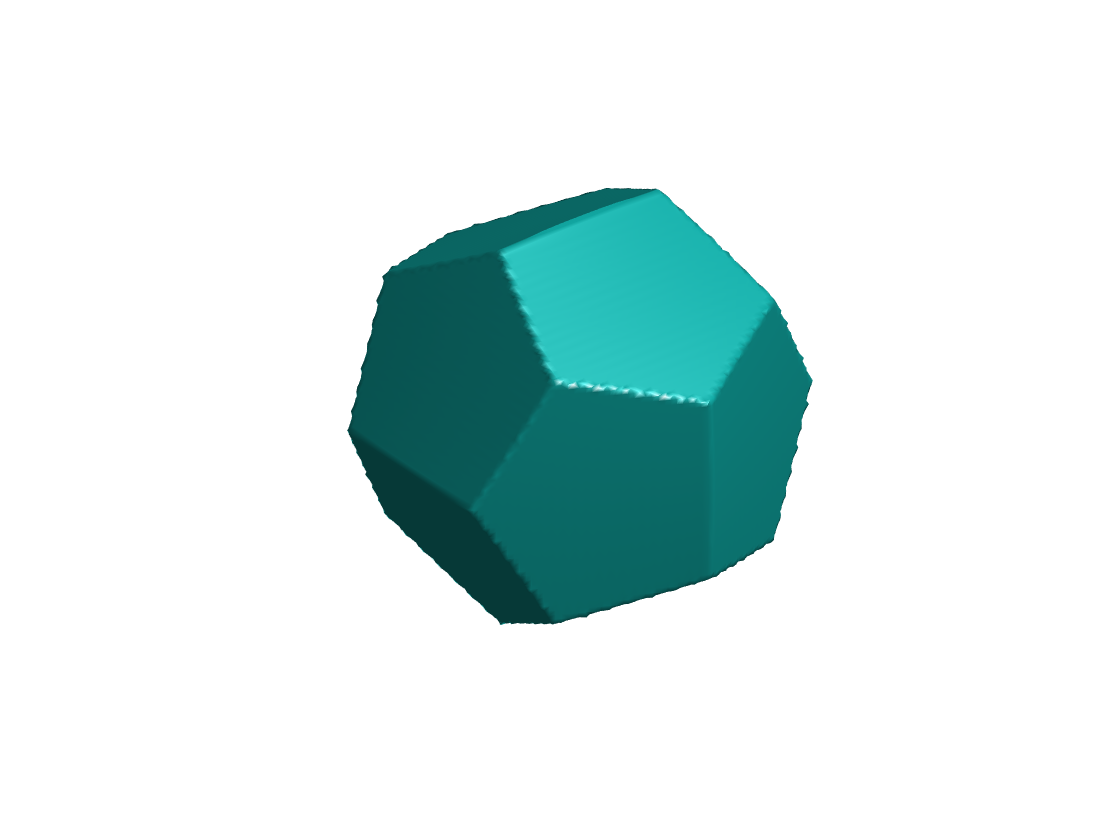}
		\includegraphics[width = 0.2 \textwidth,clip, trim = 7cm 3cm 7cm 1cm]{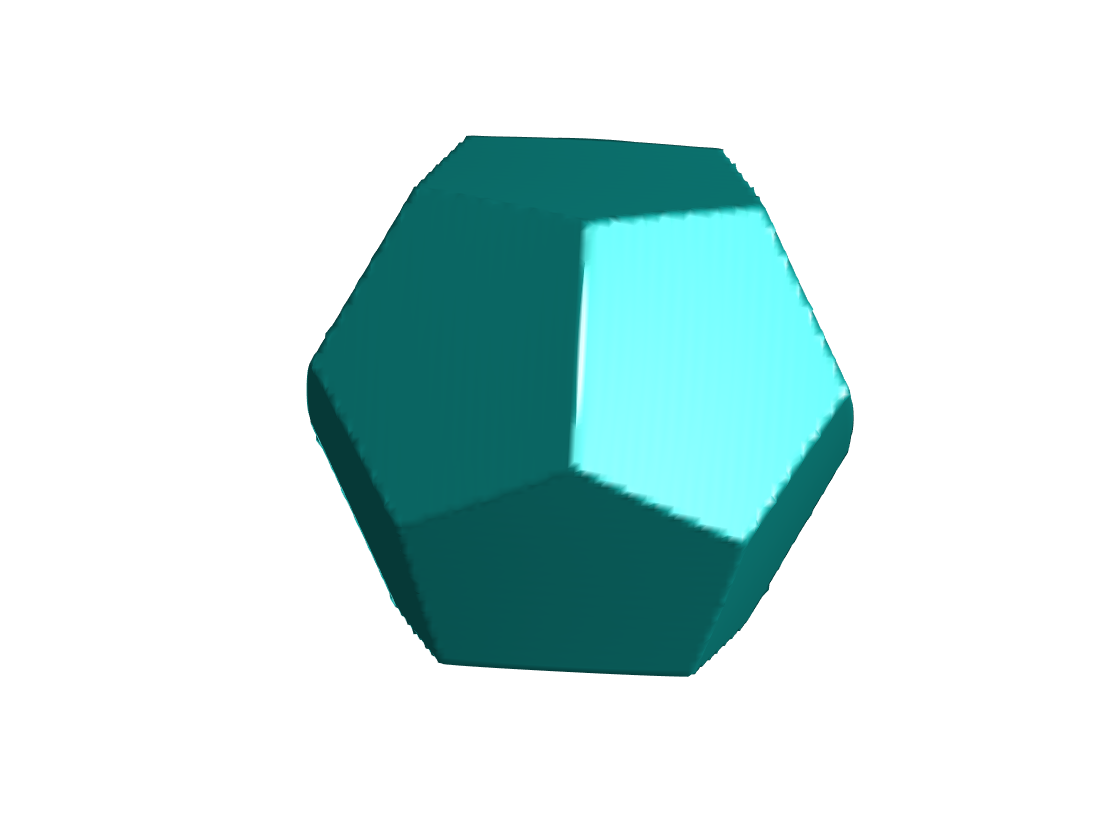}
		\includegraphics[width = 0.2 \textwidth,clip, trim = 5cm 2cm 5cm 1cm]{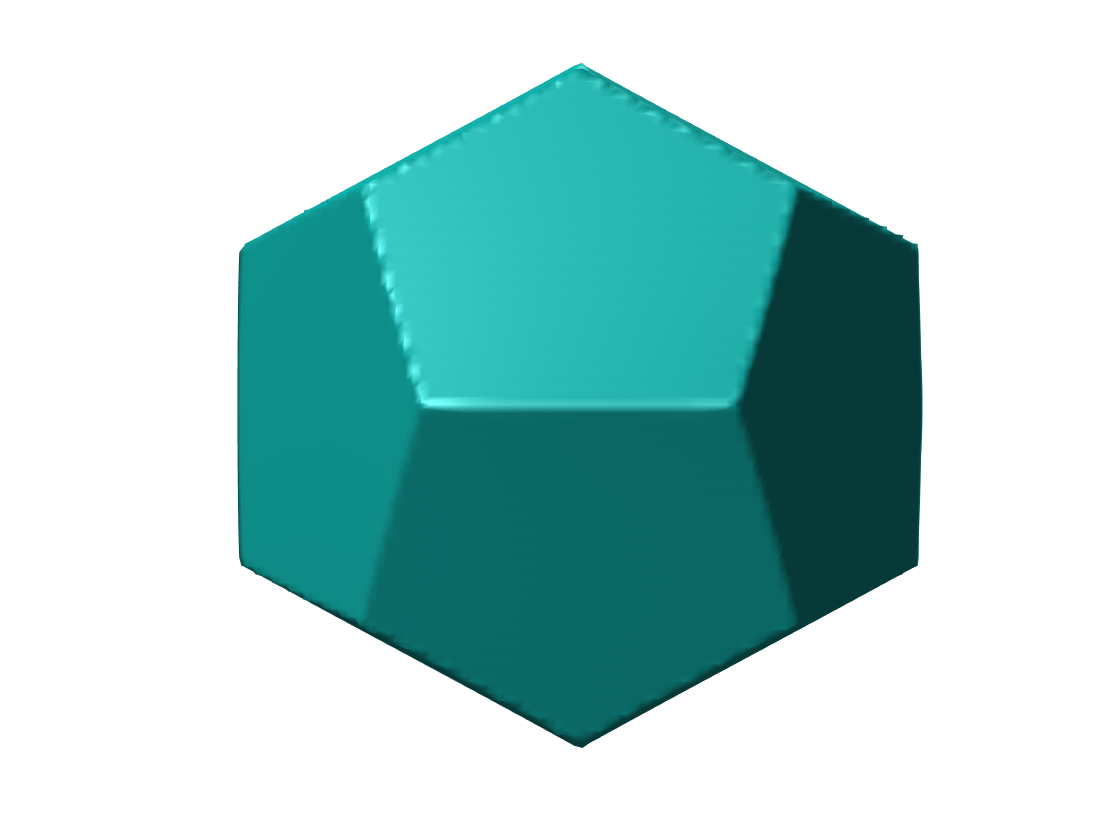}
		\includegraphics[width = 0.2 \textwidth,clip, trim = 5cm 2cm 5cm 0.5cm]{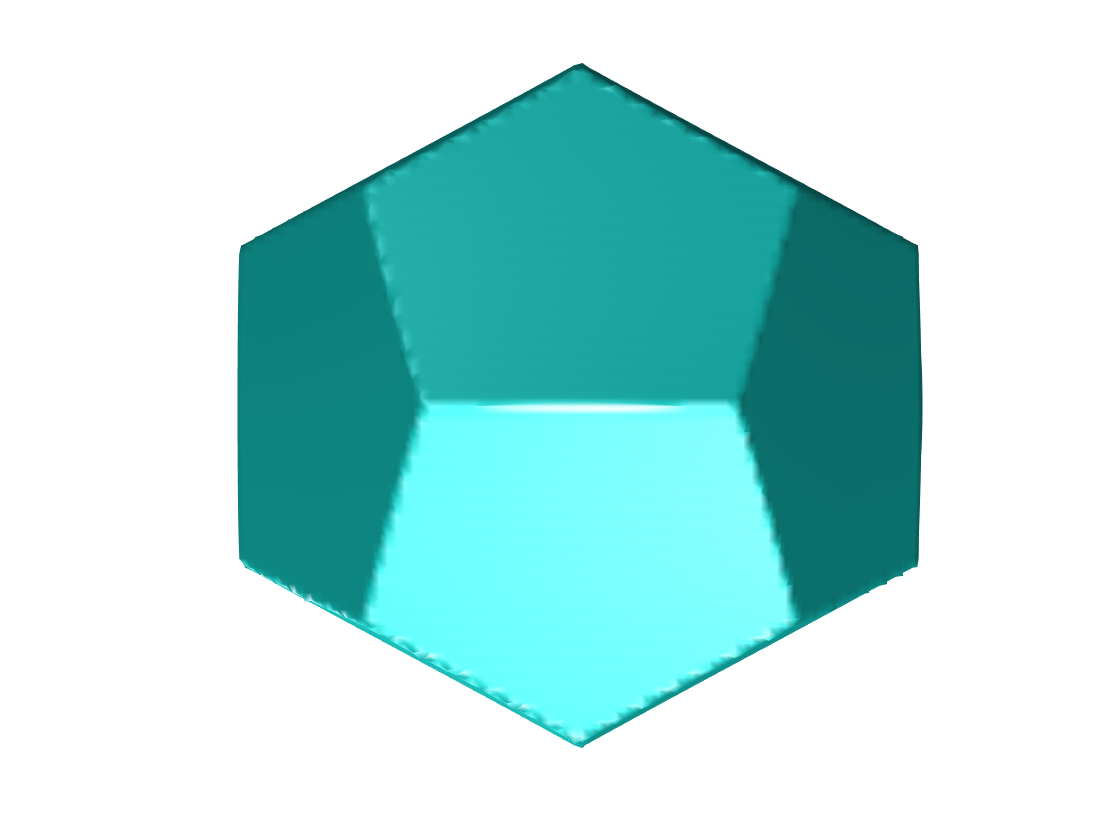}
		\includegraphics[width = 0.2\textwidth,clip, trim = 8.5cm 5.2cm 9cm 5cm]{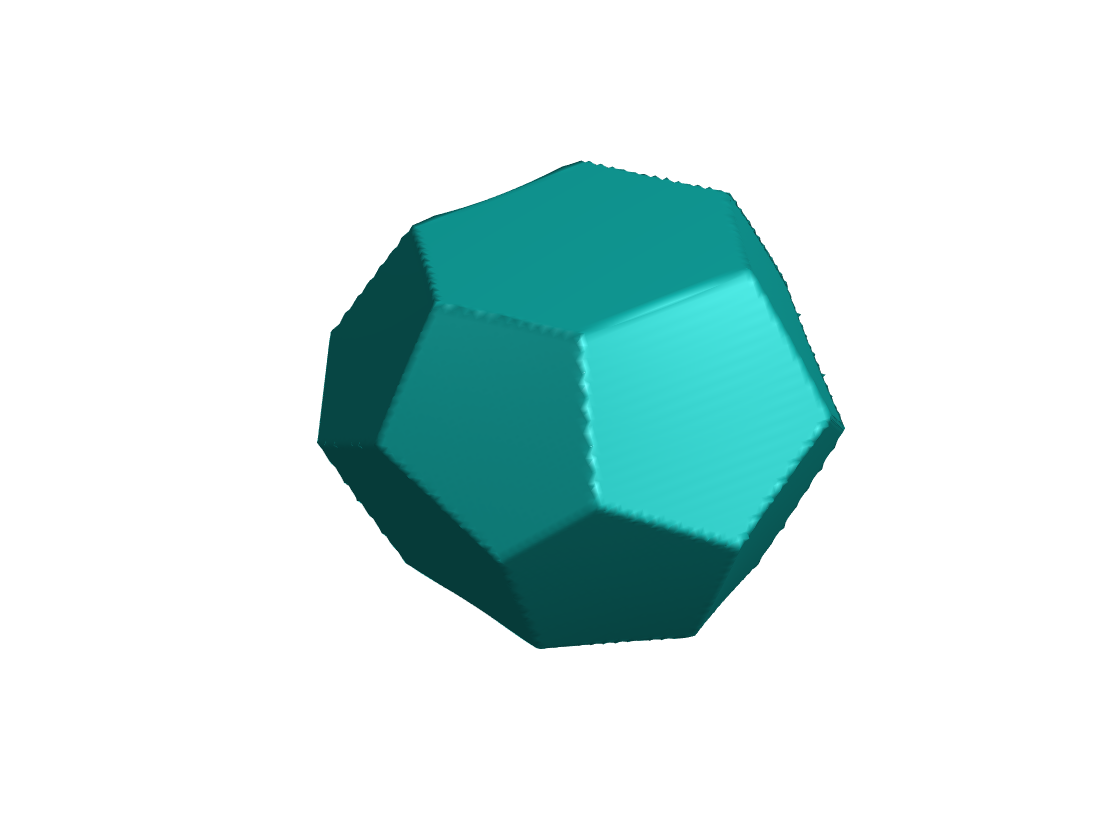}
		\includegraphics[width = 0.2 \textwidth,clip, trim = 7cm 5cm 8cm 3cm]{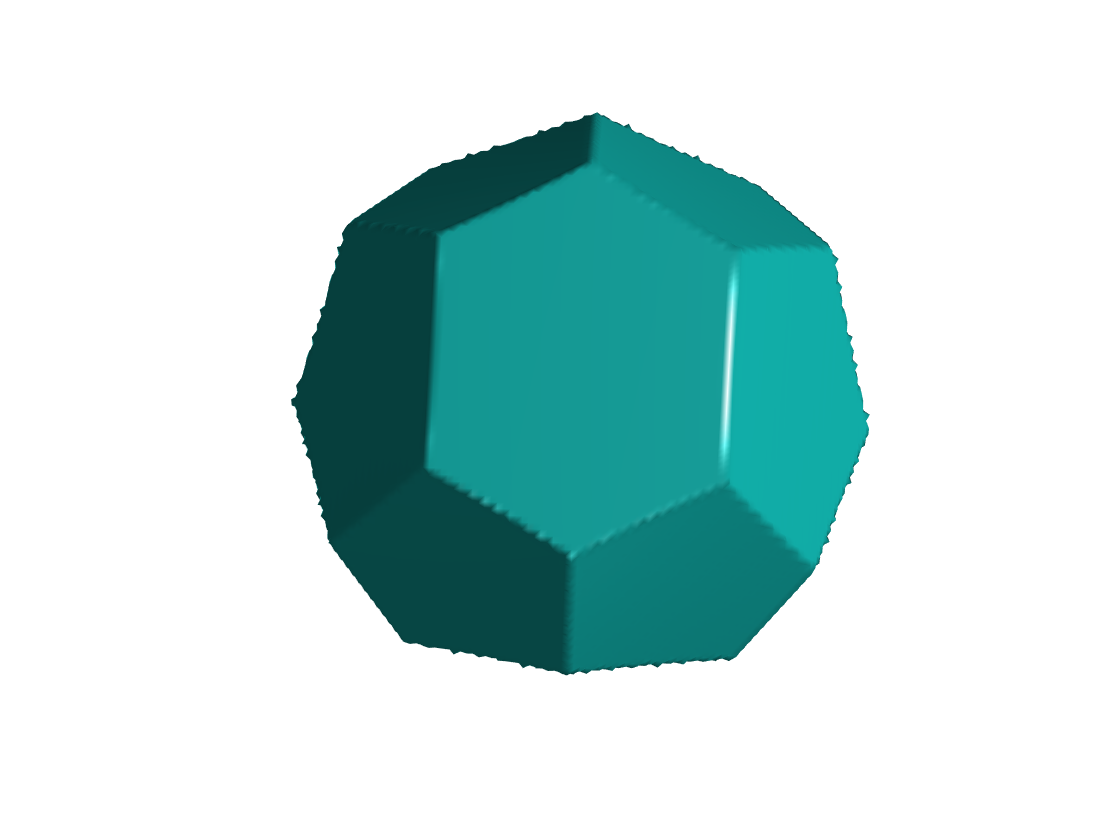}
		\includegraphics[width = 0.2 \textwidth,clip, trim = 5cm 3cm 5cm 1cm]{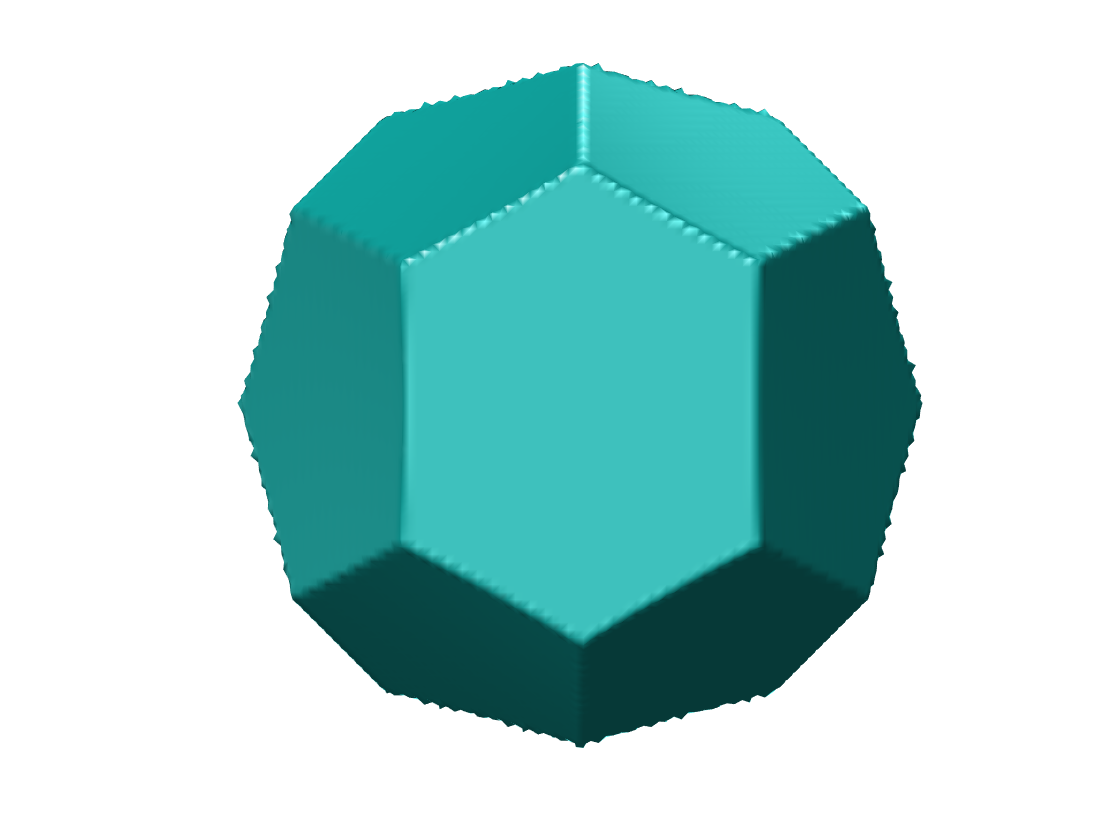}
		\includegraphics[width = 0.2 \textwidth,clip, trim = 5cm 3cm 5cm 1cm]{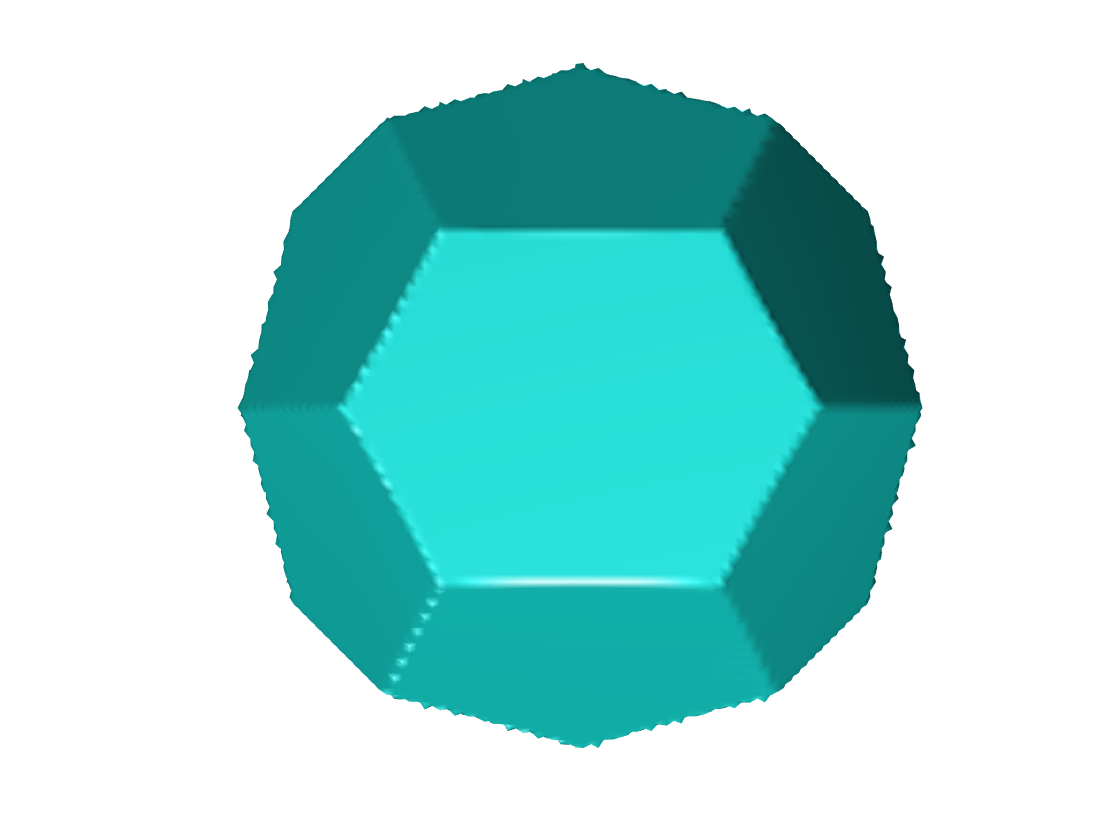}
		\caption{Numerical results of Algorithm~\ref{Alg:4step} on a $128\times 128\times 128$ discretized mesh in space with the time step size  $\tau=\pi/16$: First row: 8-partition of a periodic cube shown from four different views. Second row: Shape consists of 12 pentagonal faces. Third row:  Shape consists of 2 hexagonal surrounded by 12 pentagonal faces.}\label{fig:Alg1m83D}
	\end{figure}
	\subsection{Partitions with Dirichlet boundary conditions} 
 In this example, we explore the feasibility of Algorithms~\ref{Alg:4step}-\ref{Alg:3step_Type2} for solving the system described by \eqref{gradient_sys}-\eqref{norm:4} with Dirichlet boundary conditions. The diffusion step within these algorithms employs the standard compact difference method with a fast implementation, as detailed in \cite{ZhuJuZhao}, ensuring both accuracy and efficiency in implementing the diffusion step. We investigate the performance of the algorithms on partitions within both the square domain \([- \pi, \pi]^2\) and more general, arbitrary domains. This evaluation provides insights into the algorithms' adaptability and effectiveness across different geometric configurations. 
 \subsubsection{Square domain}
Figures~\ref{fig:EvolAlg_4stepdt01Dir} and \ref{fig:EvolAlg_3stepT1dt01Dir} illustrate the evolution of the partition for \(k=4\) and \(k=5\) with \(\tau = 0.1\), showing that the results of Algorithms~\ref{Alg:4step} and \ref{Alg:3step_Type1} change rapidly within the first few dozen iterations.
These figures highlight the effectiveness and unconditional stability of the algorithms. The results are consistent with those reported in \cite{Wang22}, confirming the reliability and accuracy of these methods for solving \eqref{gradient_sys}-\eqref{norm:4} with Dirichlet boundary conditions. Results for Algorithm~\ref{Alg:3step_Type2} are omitted, as its final partitions are identical to those shown in Figures~\ref{fig:EvolAlg_4stepdt01Dir} and \ref{fig:EvolAlg_3stepT1dt01Dir}.

	\begin{figure}[H]
		\centering
		\begin{tabular}{c|c|c|c|c|c|c|c}
			\hline 
			initial & 20 & 40  & 60 & 80  & 100& 120  & 164 \\
			\includegraphics[width = 0.09\textwidth, clip, trim = 4cm 1cm 3cm 1cm]{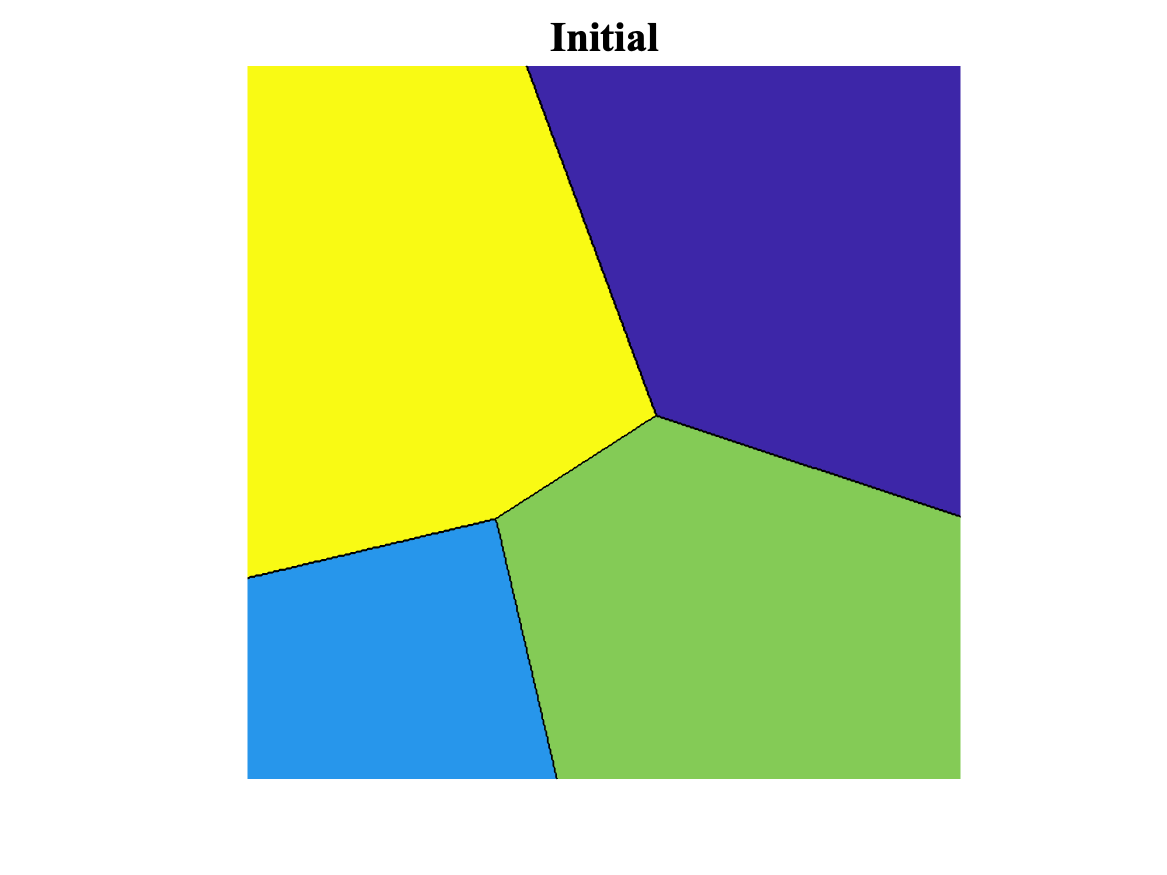}&  
			\includegraphics[width = 0.09\textwidth, clip, trim = 4cm 1cm 3cm 1cm]{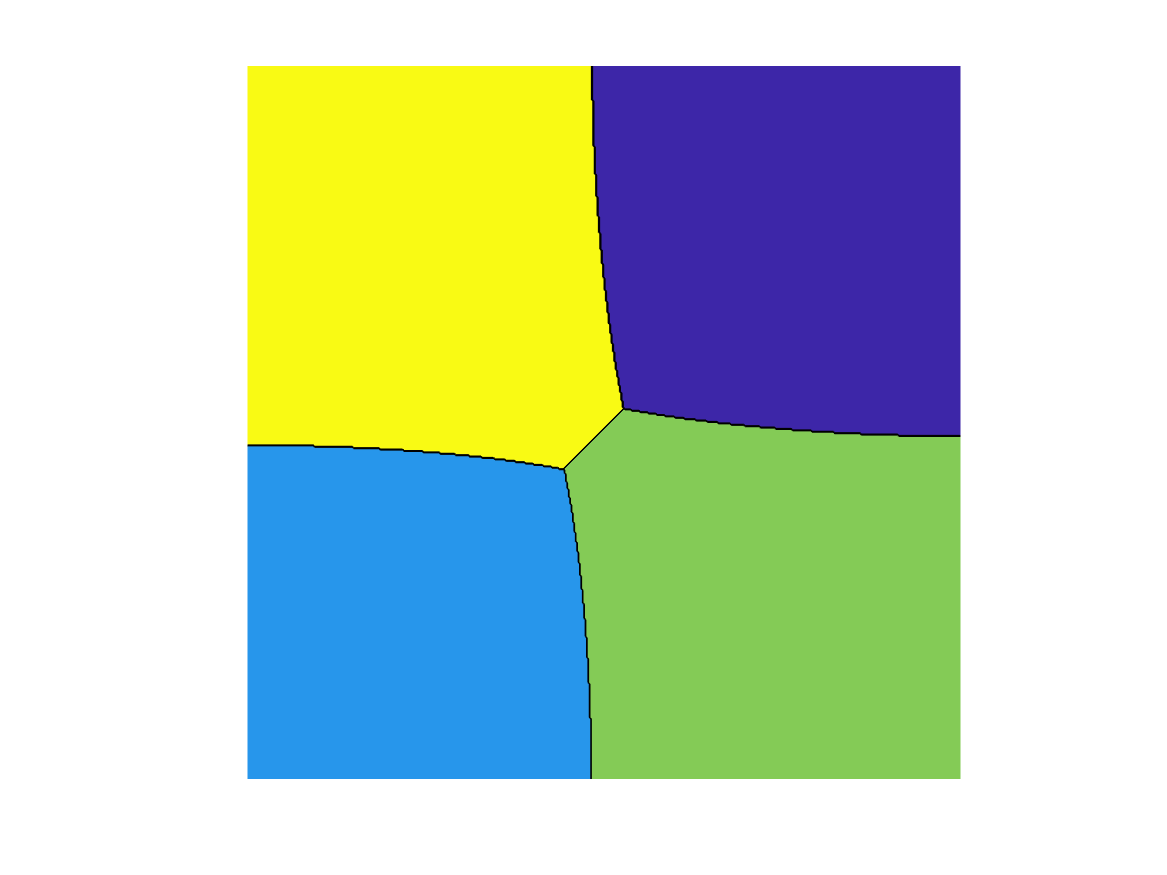}&
			\includegraphics[width = 0.09\textwidth, clip, trim = 4cm 1cm 3cm 1cm]{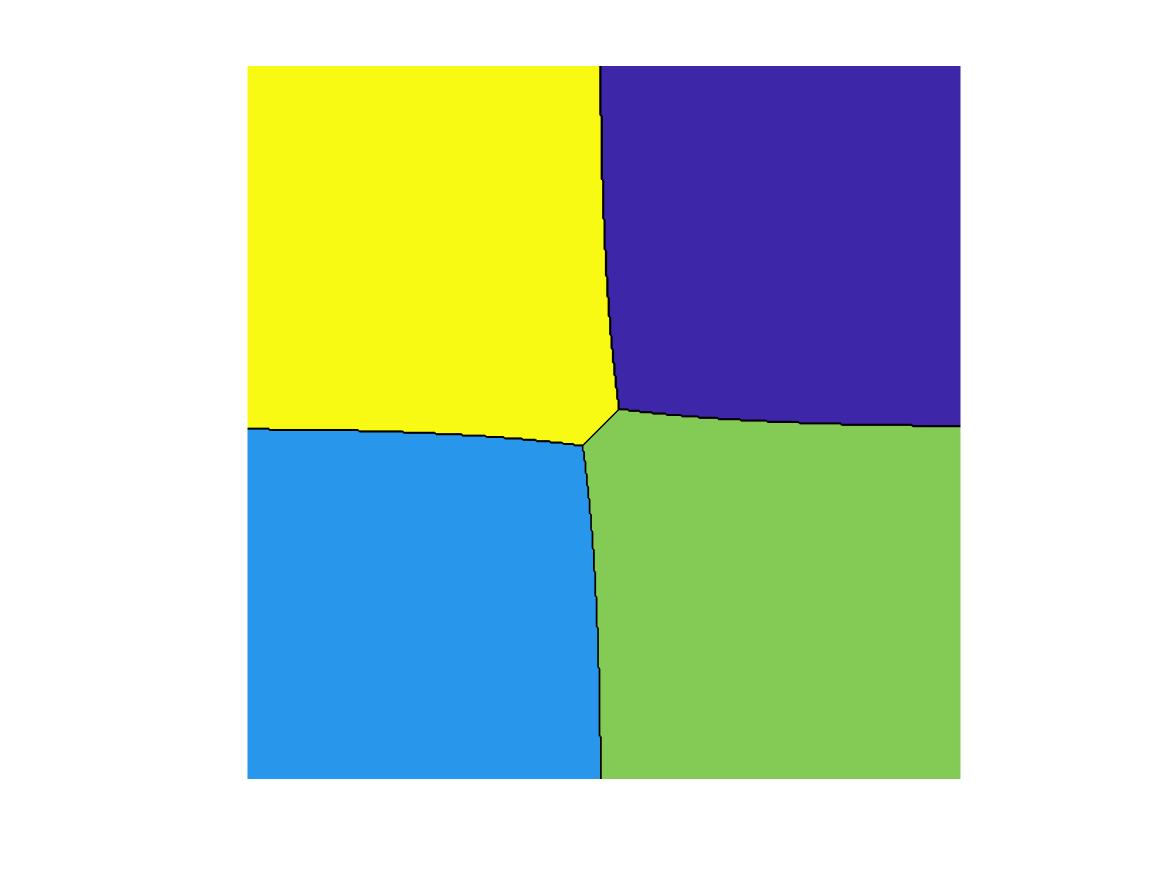}&
			\includegraphics[width = 0.09\textwidth, clip, trim = 4cm 1cm 3cm 1cm]{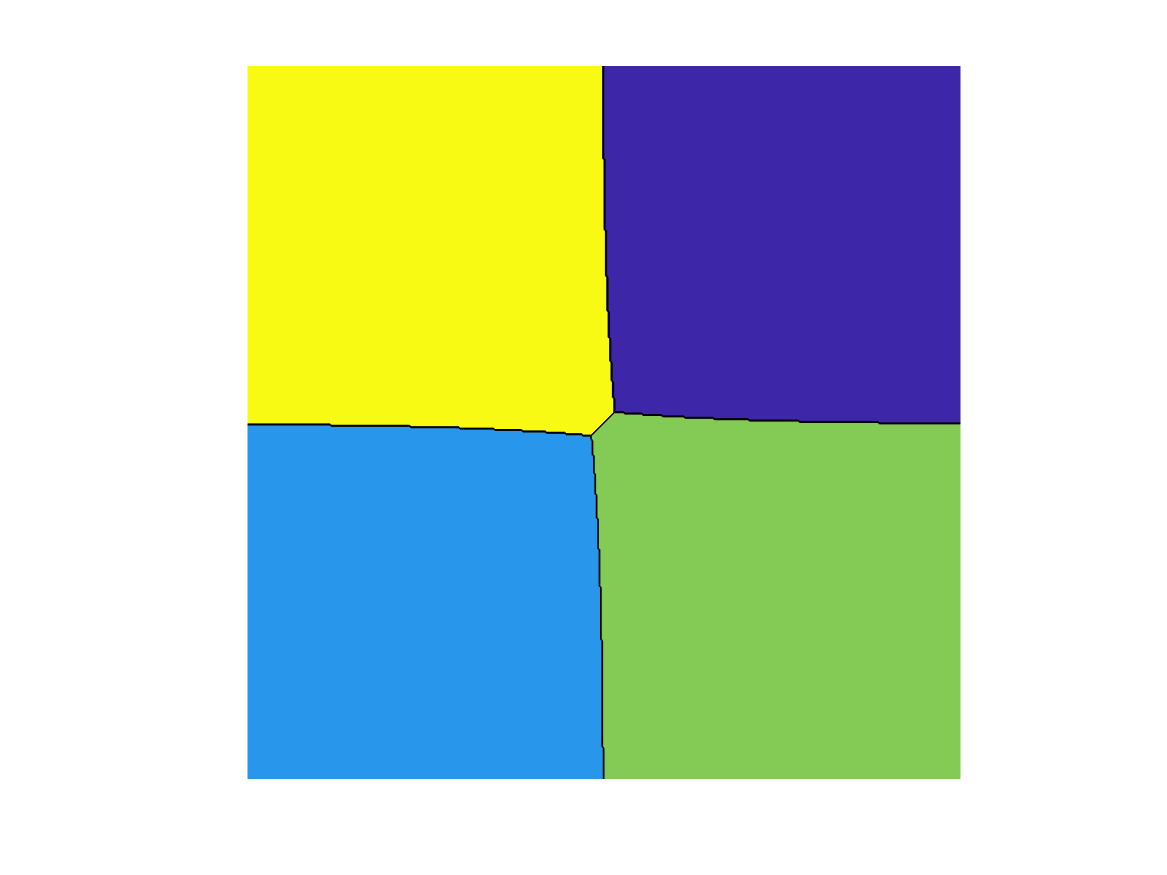}&
			\includegraphics[width = 0.09\textwidth, clip, trim = 4cm 1cm 3cm 1cm]{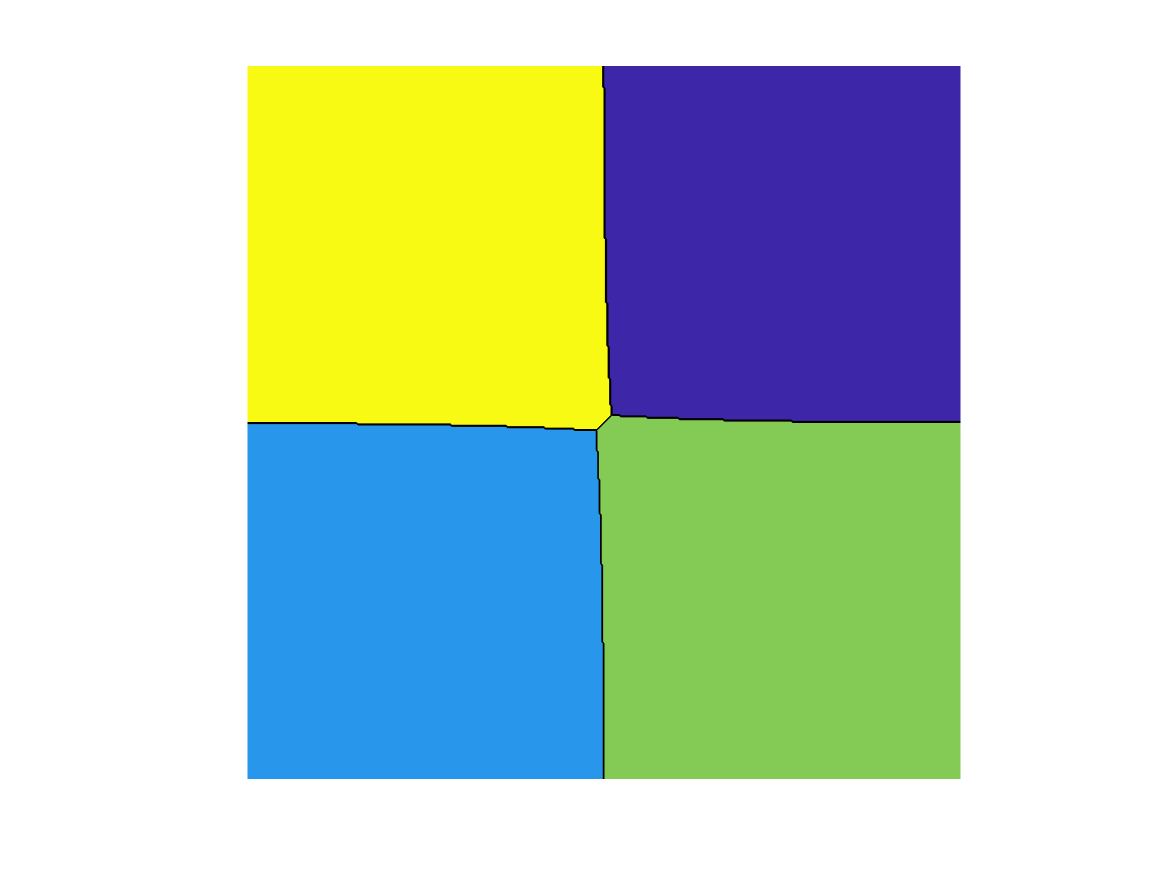}&
			\includegraphics[width = 0.09\textwidth, clip, trim = 4cm 1cm 3cm 1cm]{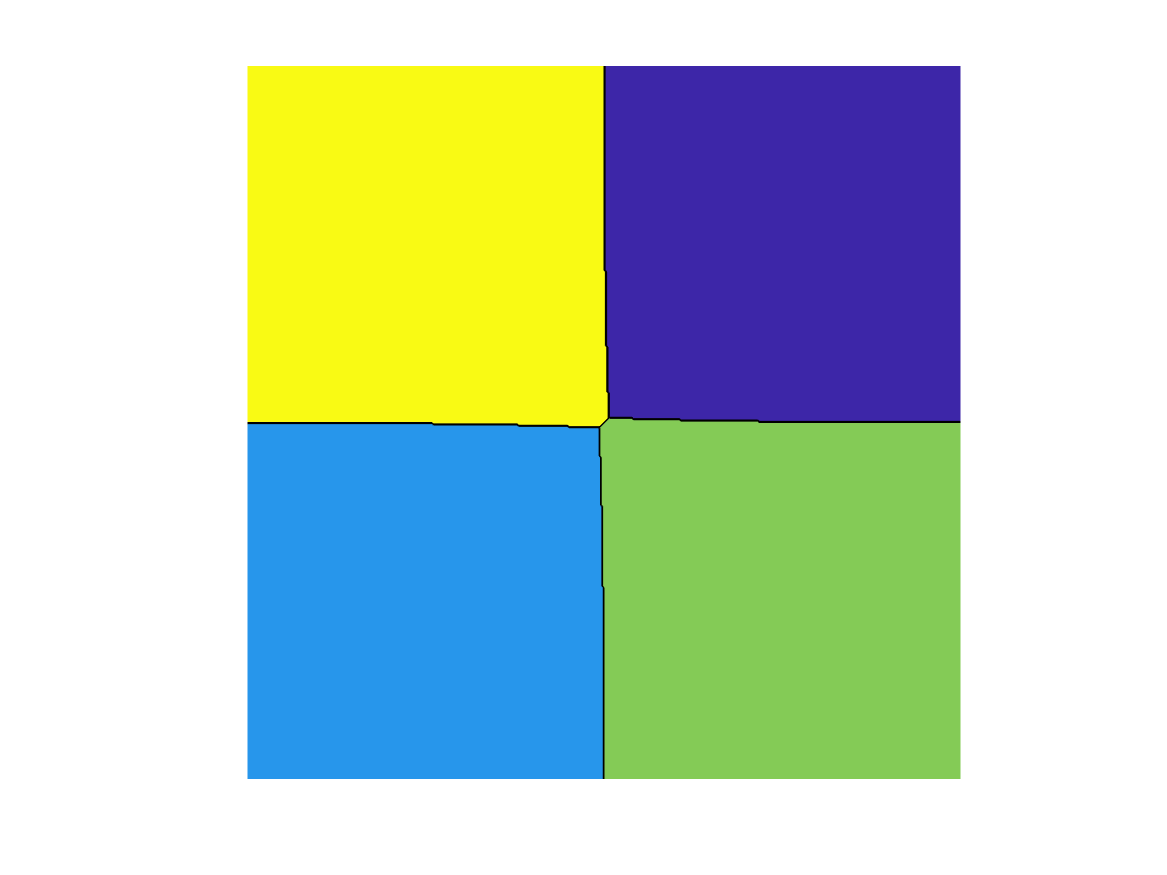}&
			\includegraphics[width = 0.09\textwidth, clip, trim = 4cm 1cm 3cm 1cm]{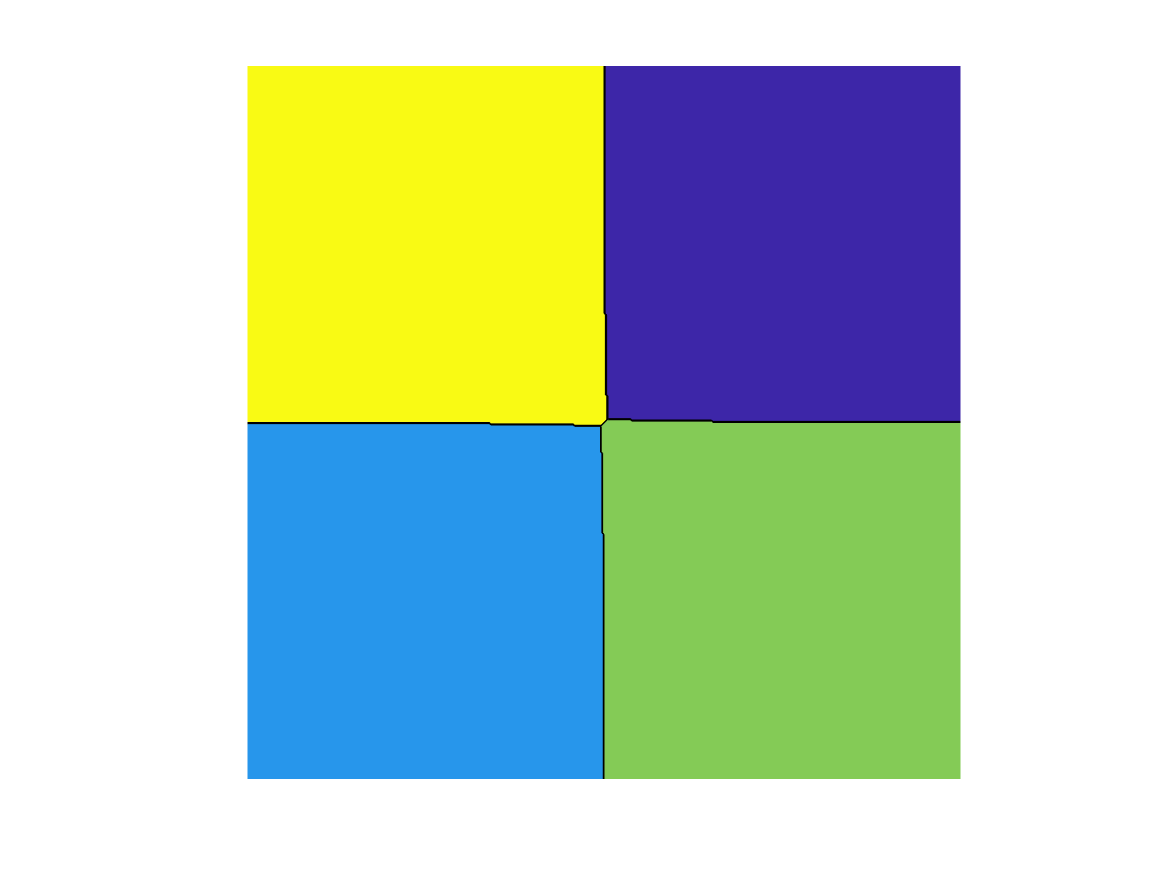}&
			\includegraphics[width = 0.09\textwidth, clip, trim = 4cm 1cm 3cm 1cm]{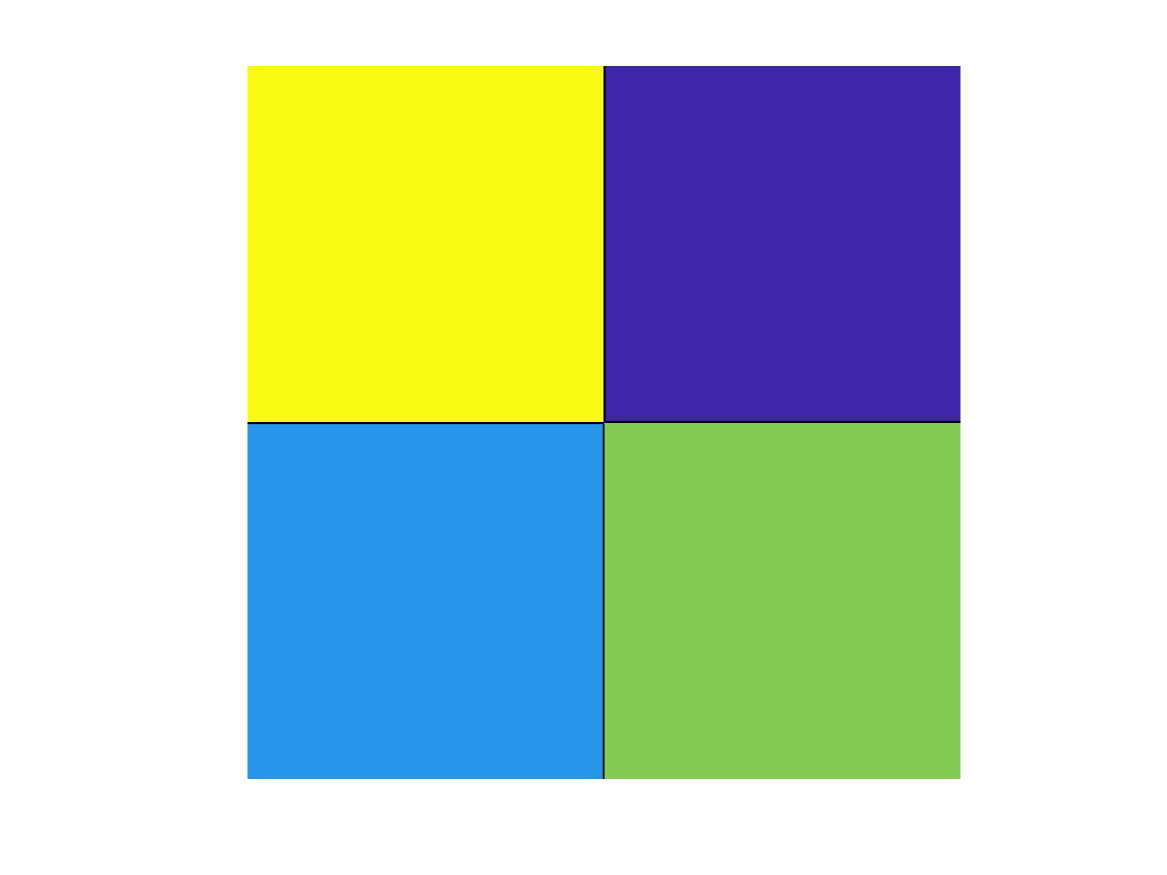} \\ 
			\hline
			\hline 
			initial & 50 &100  &150 & 200  & 250&300  & 389\\
			\includegraphics[width = 0.09\textwidth, clip, trim = 4cm 1cm 3cm 1cm]{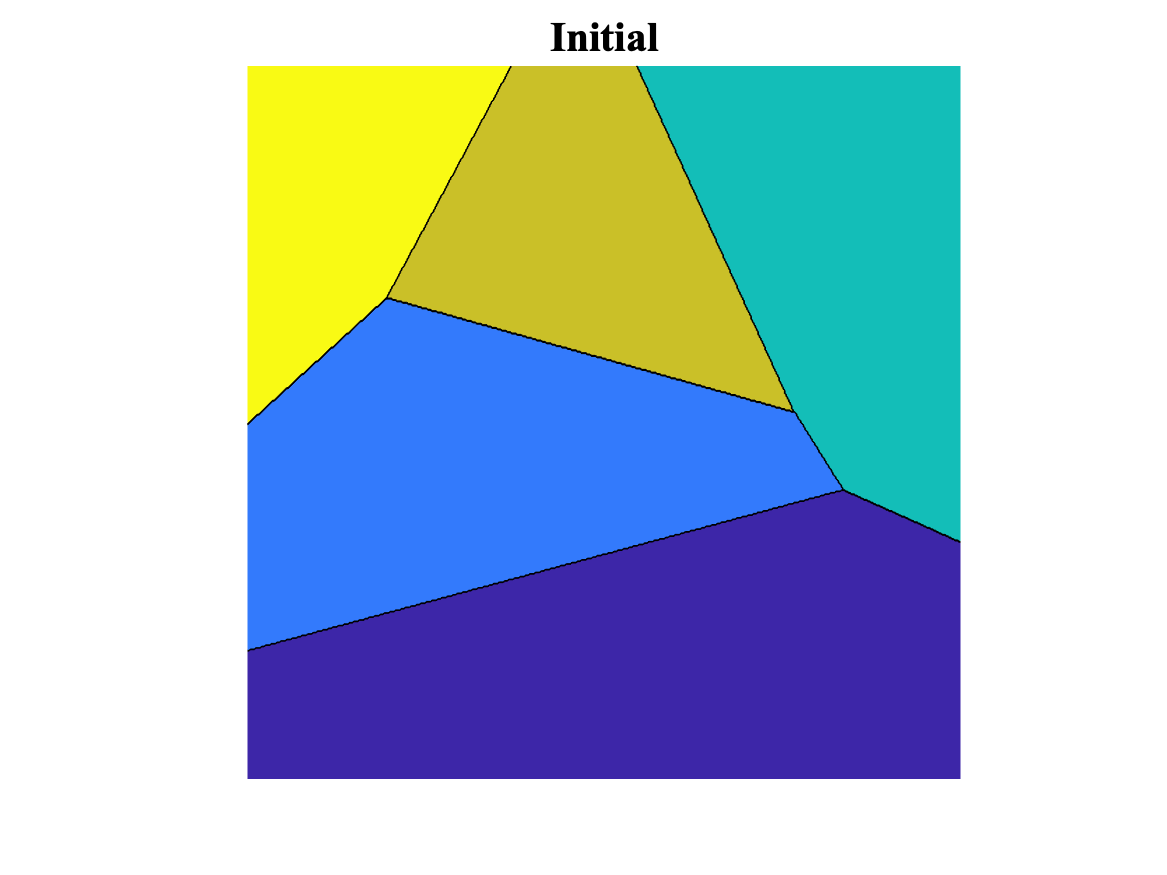}&  
			\includegraphics[width = 0.09\textwidth, clip, trim = 4cm 1cm 3cm 1cm]{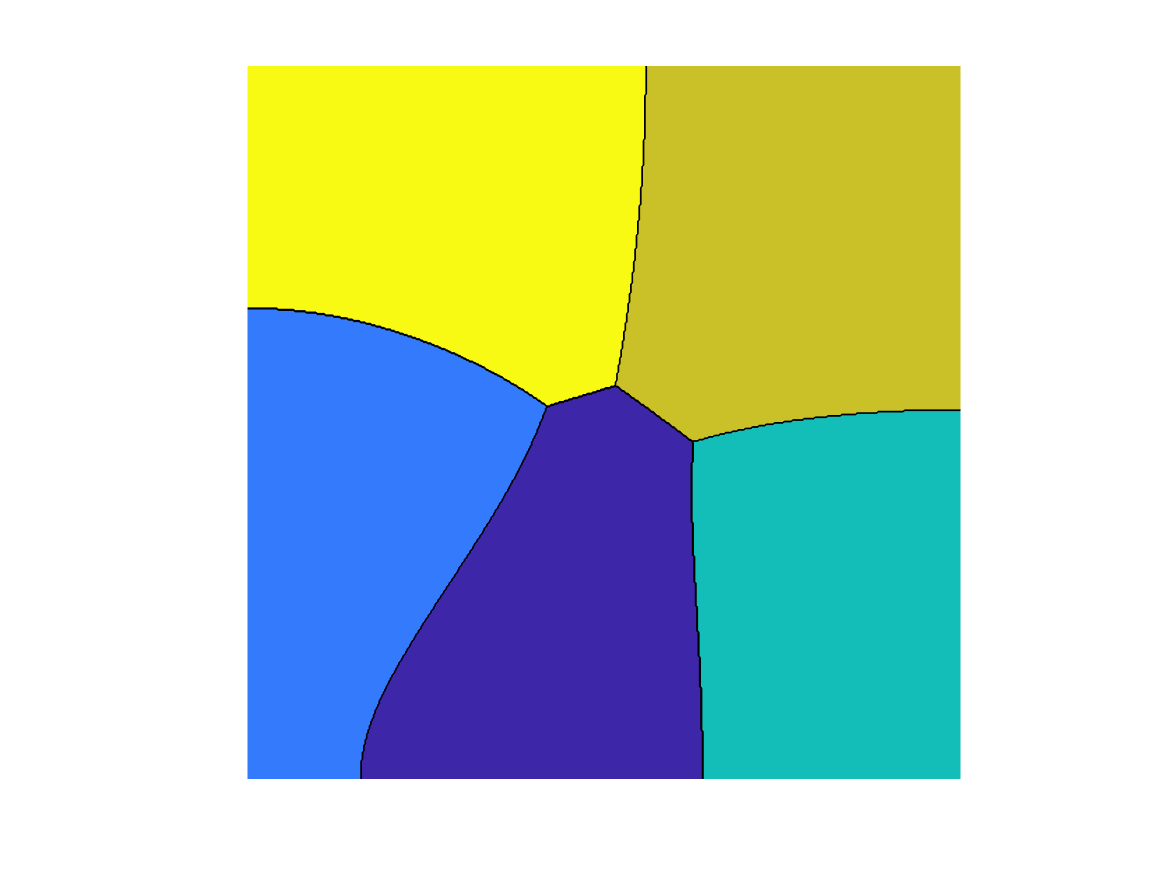}&
			\includegraphics[width = 0.09\textwidth, clip, trim = 4cm 1cm 3cm 1cm]{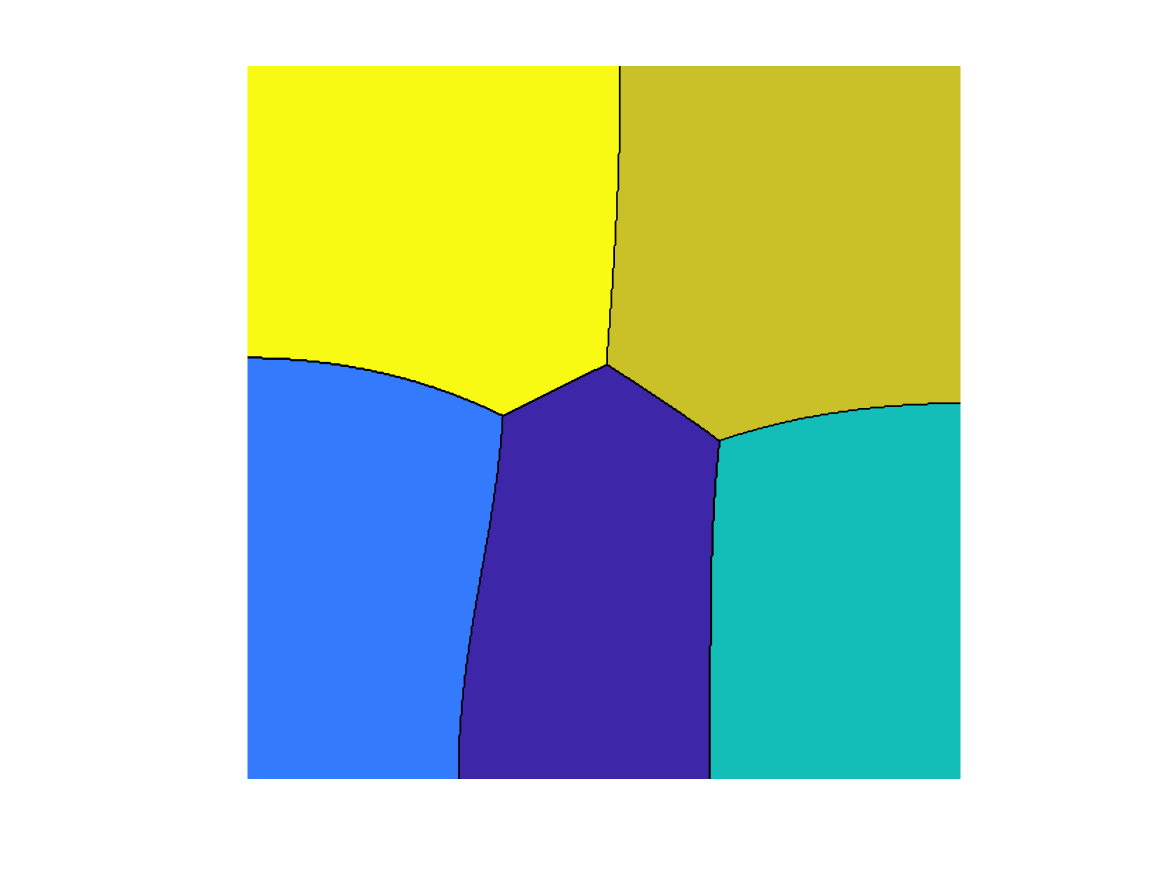}&
			\includegraphics[width = 0.09\textwidth, clip, trim = 4cm 1cm 3cm 1cm]{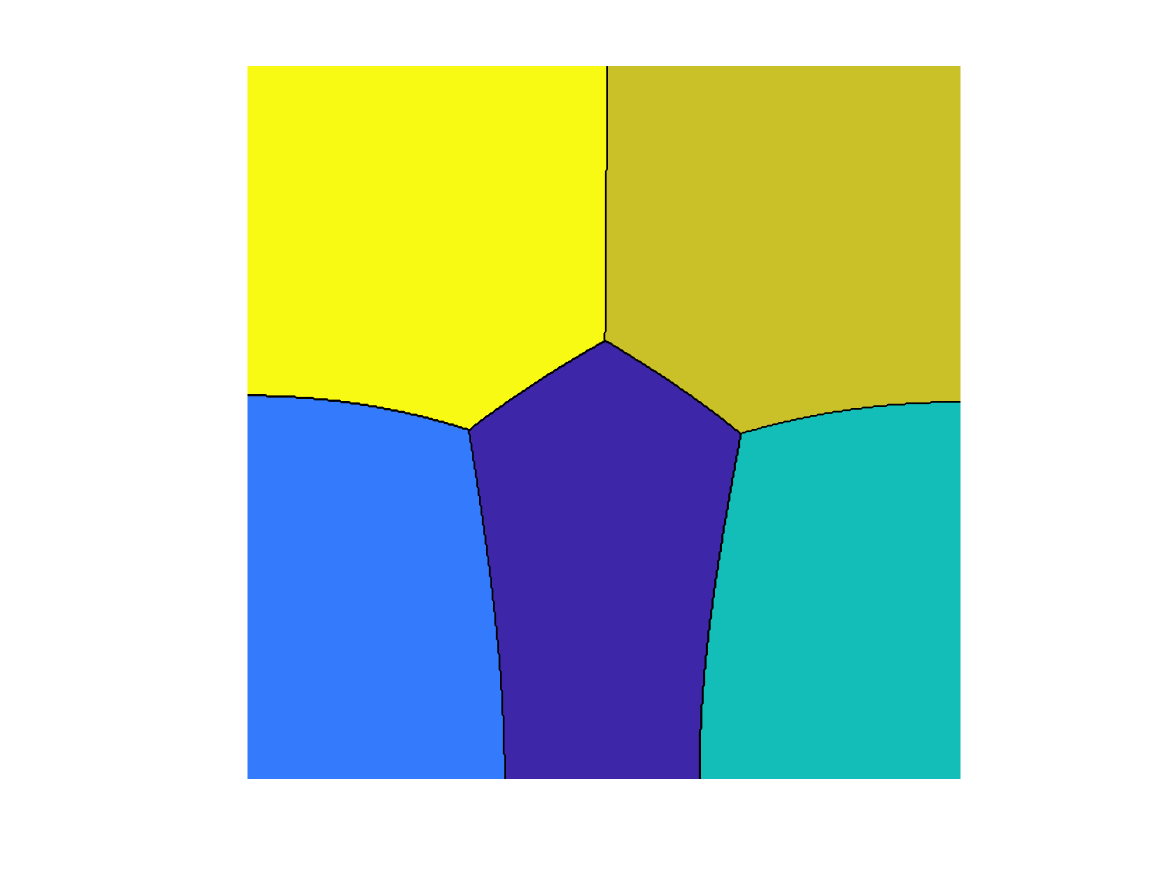}&
			\includegraphics[width = 0.09\textwidth, clip, trim = 4cm 1cm 3cm 1cm]{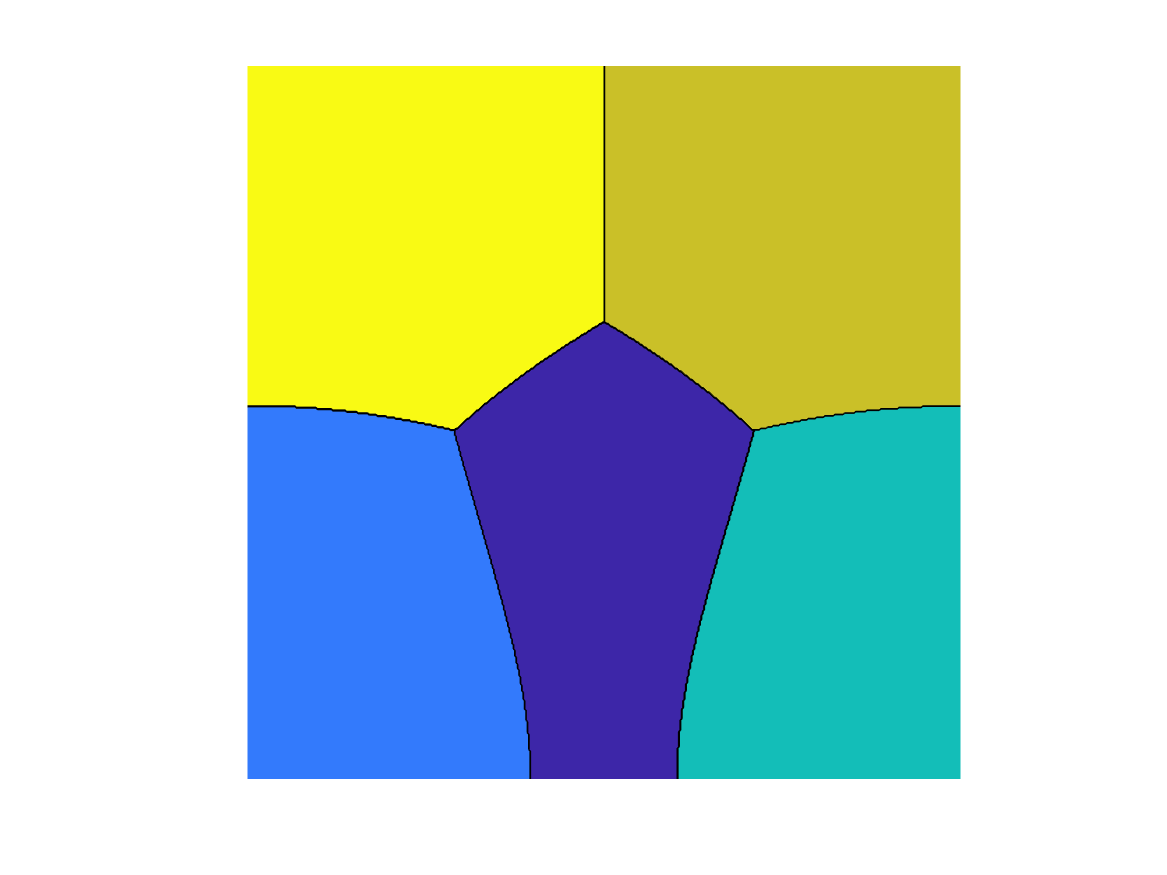}&
			\includegraphics[width = 0.09\textwidth, clip, trim = 4cm 1cm 3cm 1cm]{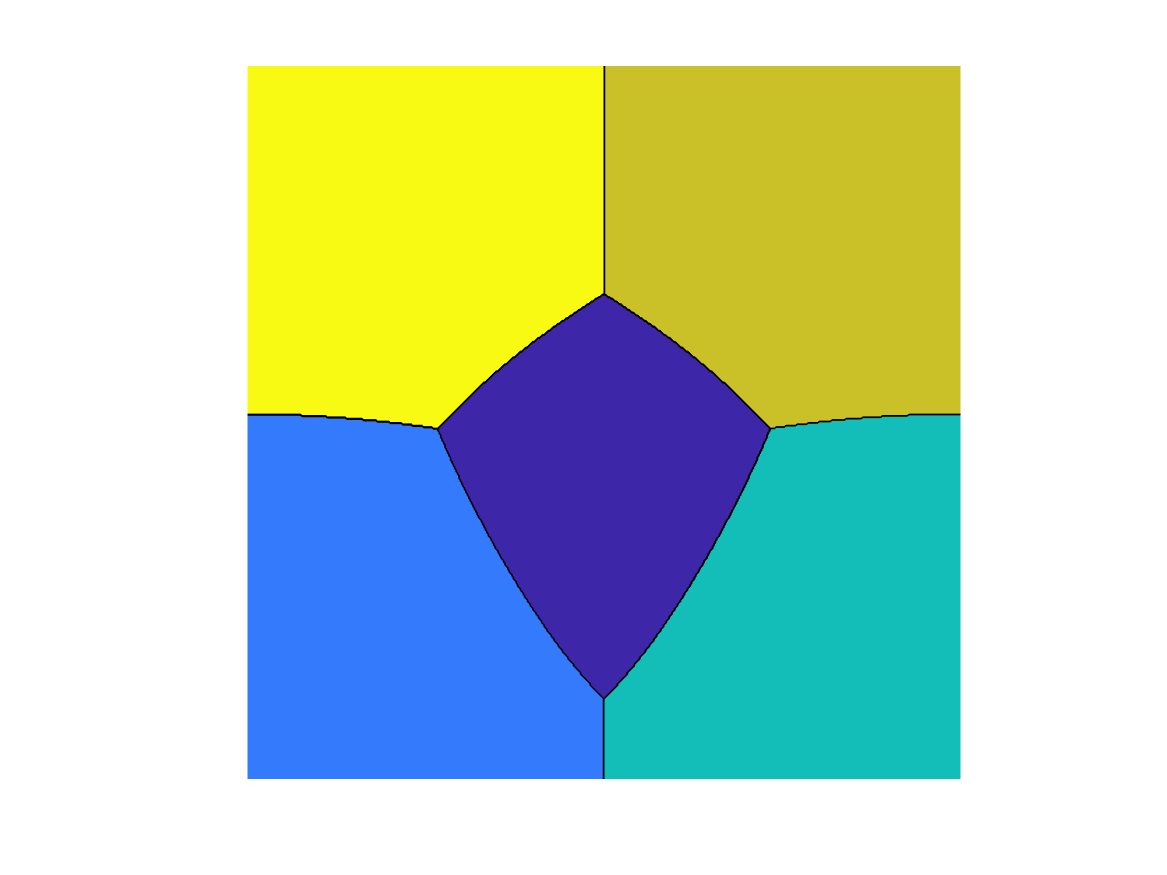}&
			\includegraphics[width = 0.09\textwidth, clip, trim = 4cm 1cm 3cm 1cm]{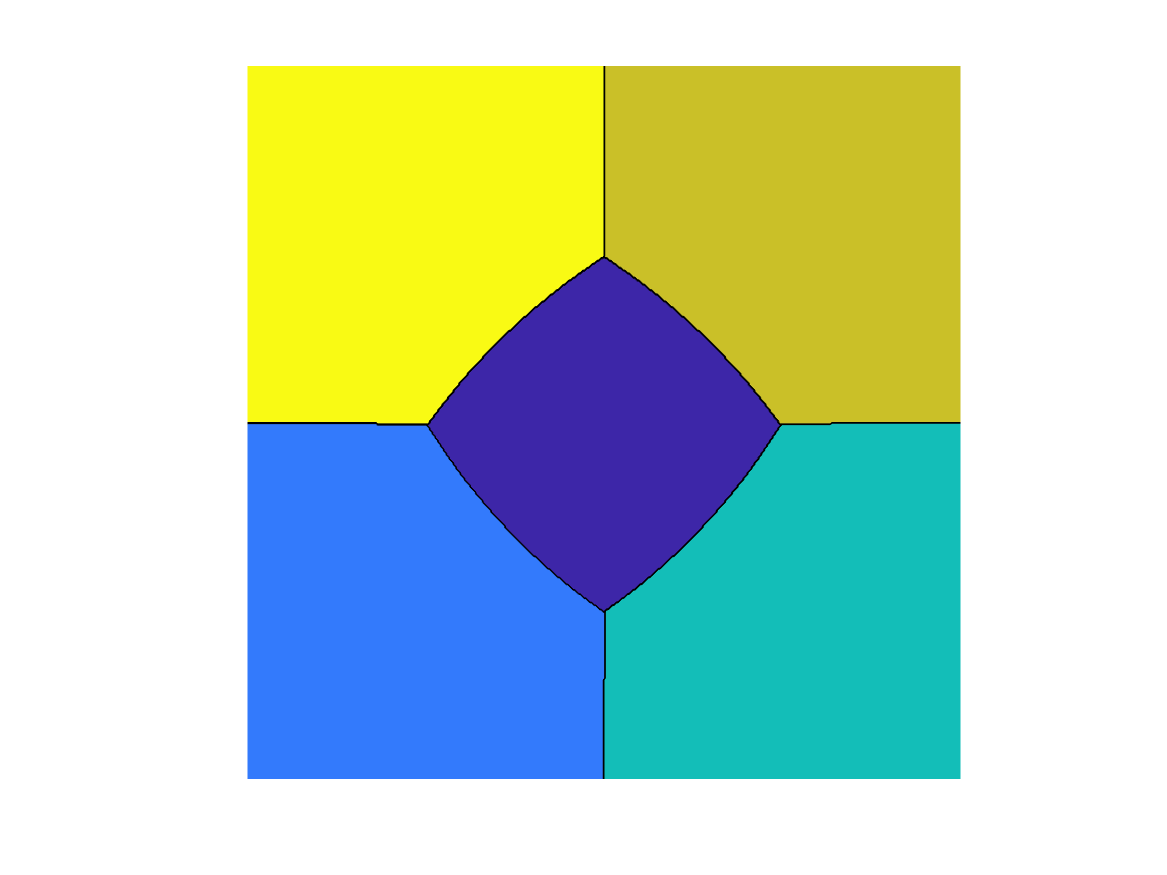}&
			\includegraphics[width = 0.09\textwidth, clip, trim = 4cm 1cm 3cm 1cm]{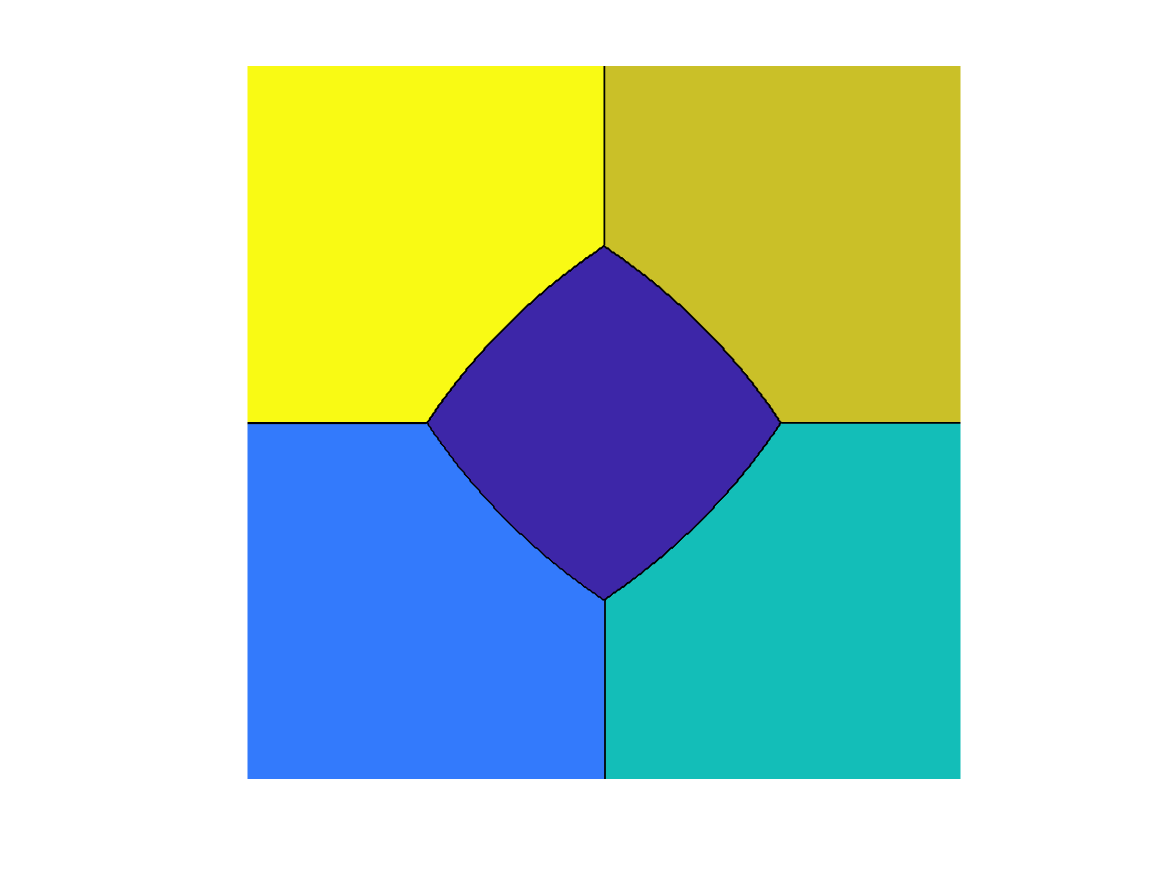} \\ 
			\hline
		\end{tabular}
		\caption{Snapshots of Algorithm~\ref{Alg:4step} with Dirichlet boundary conditions at different iterations on a $512\times512$ discretized mesh with a time step size of $\tau=0.1$. The first row corresponds to $k=4$, while the second row corresponds to $k=5$.} \label{fig:EvolAlg_4stepdt01Dir}
	\end{figure}
	
	\begin{figure}[H]
		\centering
		\begin{tabular}{c|c|c|c|c|c|c|c}
			\hline 
			initial & 20 & 40  & 60 & 80  & 100& 120  & 128 \\
			\includegraphics[width = 0.09\textwidth, clip, trim = 4cm 1cm 3cm 1cm]{figures_Dbnd/Dirinitialm4.png}&  
			\includegraphics[width = 0.09\textwidth, clip, trim = 4cm 1cm 3cm 1cm]{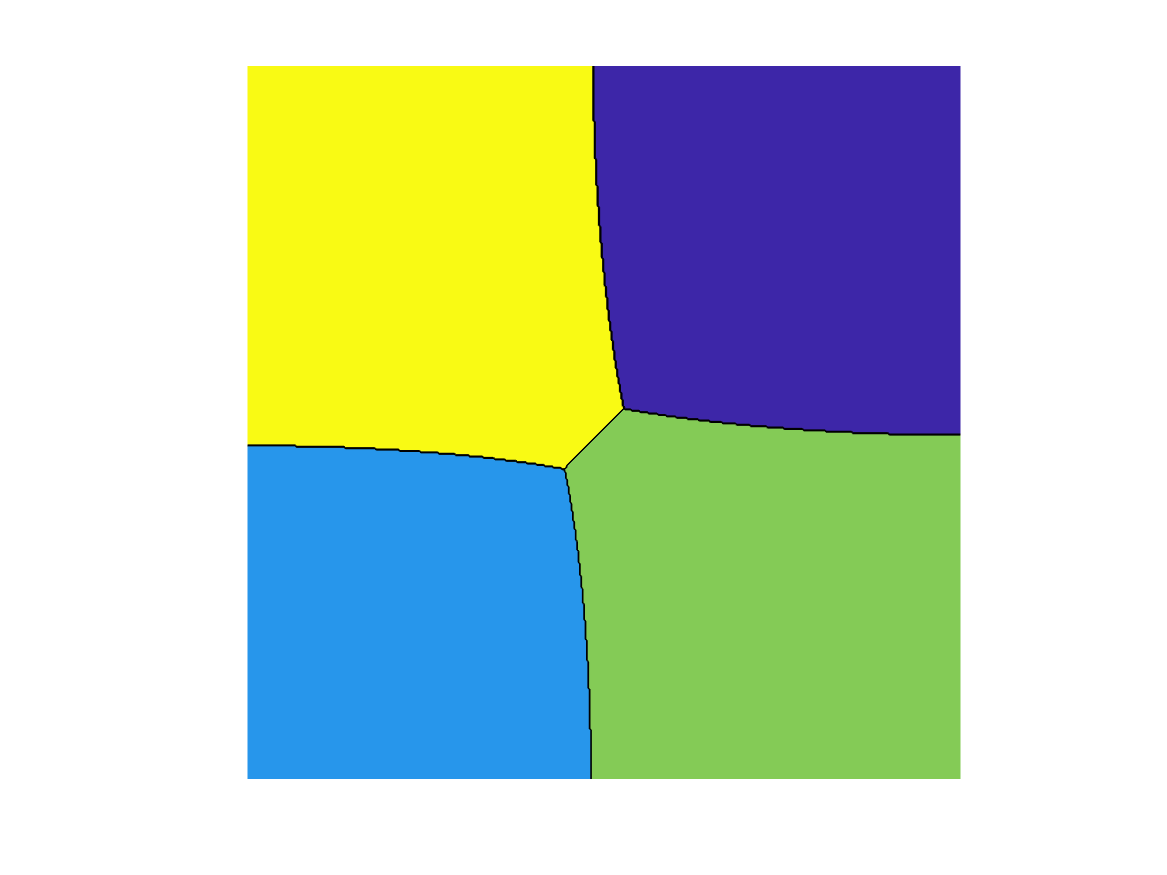}&
			\includegraphics[width = 0.09\textwidth, clip, trim = 4cm 1cm 3cm 1cm]{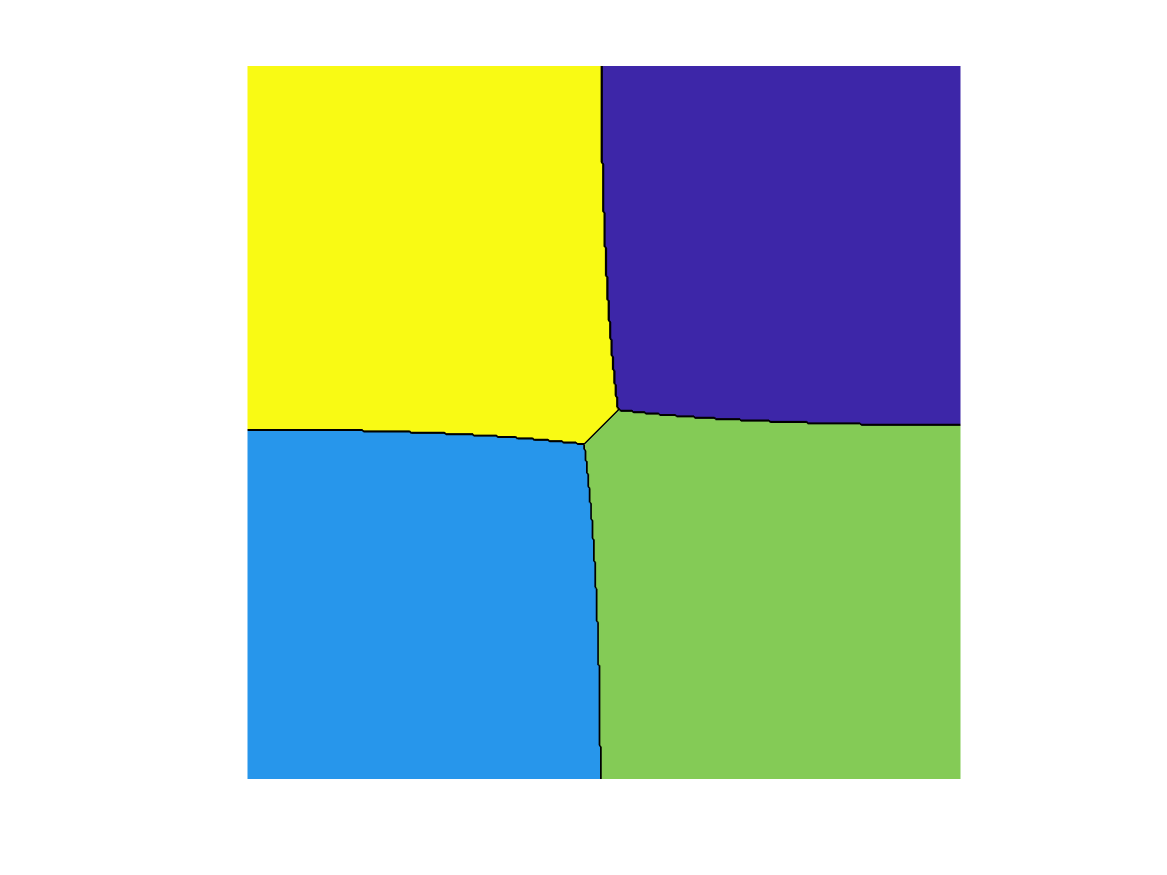}&
			\includegraphics[width = 0.09\textwidth, clip, trim = 4cm 1cm 3cm 1cm]{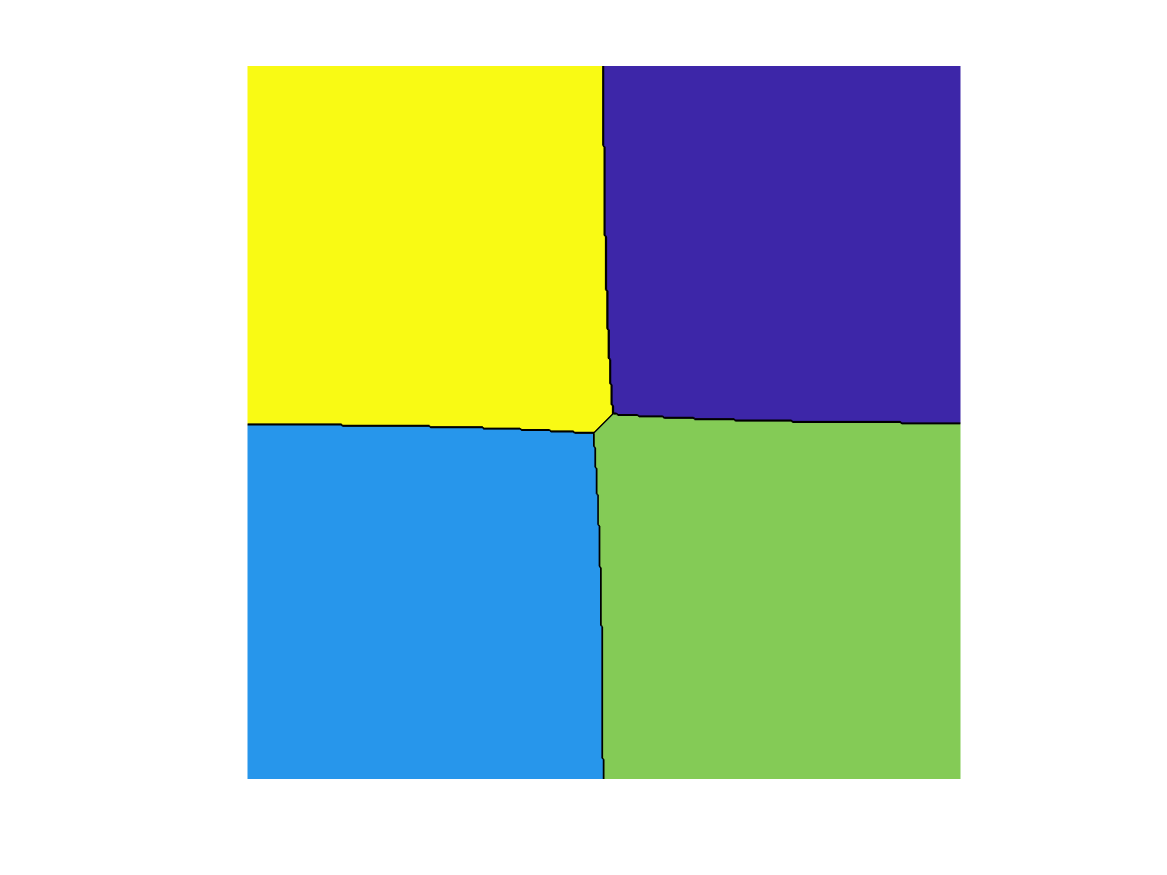}&
			\includegraphics[width = 0.09\textwidth, clip, trim = 4cm 1cm 3cm 1cm]{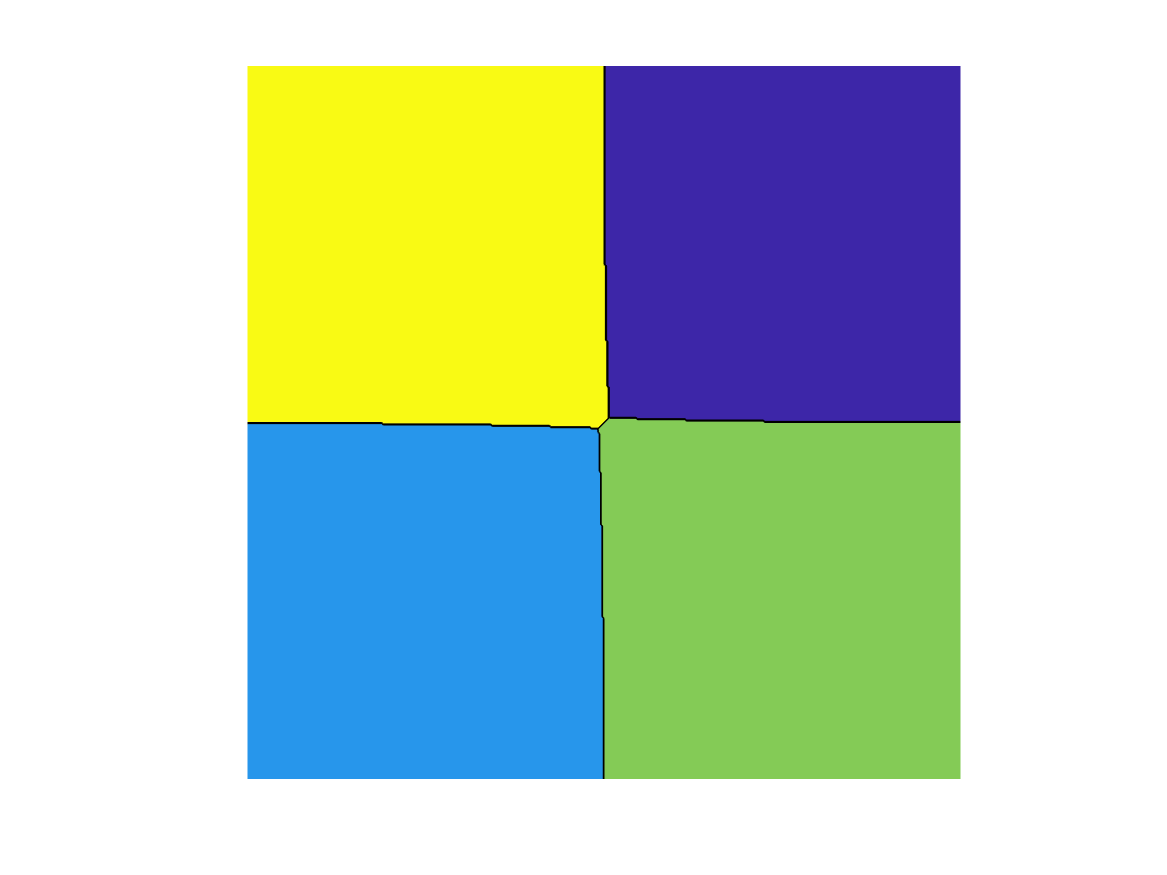}&
			\includegraphics[width = 0.09\textwidth, clip, trim = 4cm 1cm 3cm 1cm]{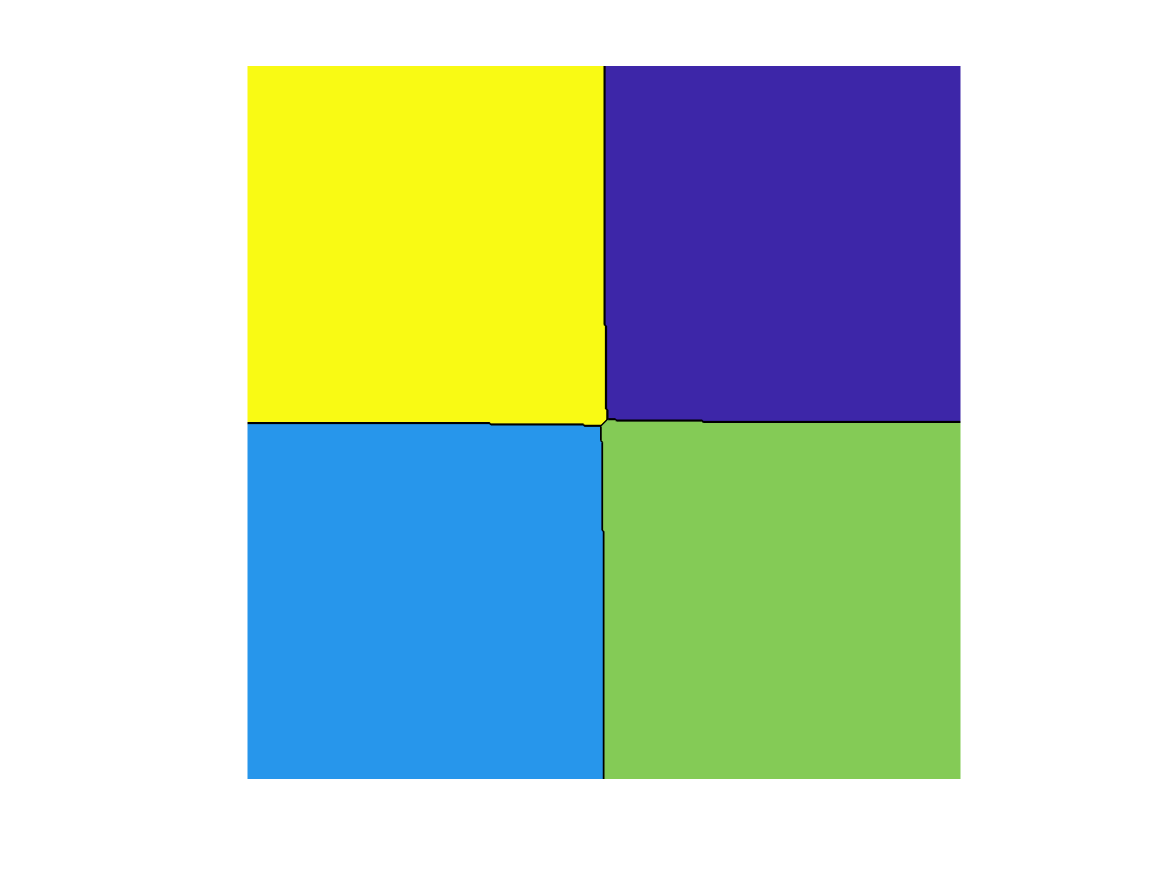}&
			\includegraphics[width = 0.09\textwidth, clip, trim = 4cm 1cm 3cm 1cm]{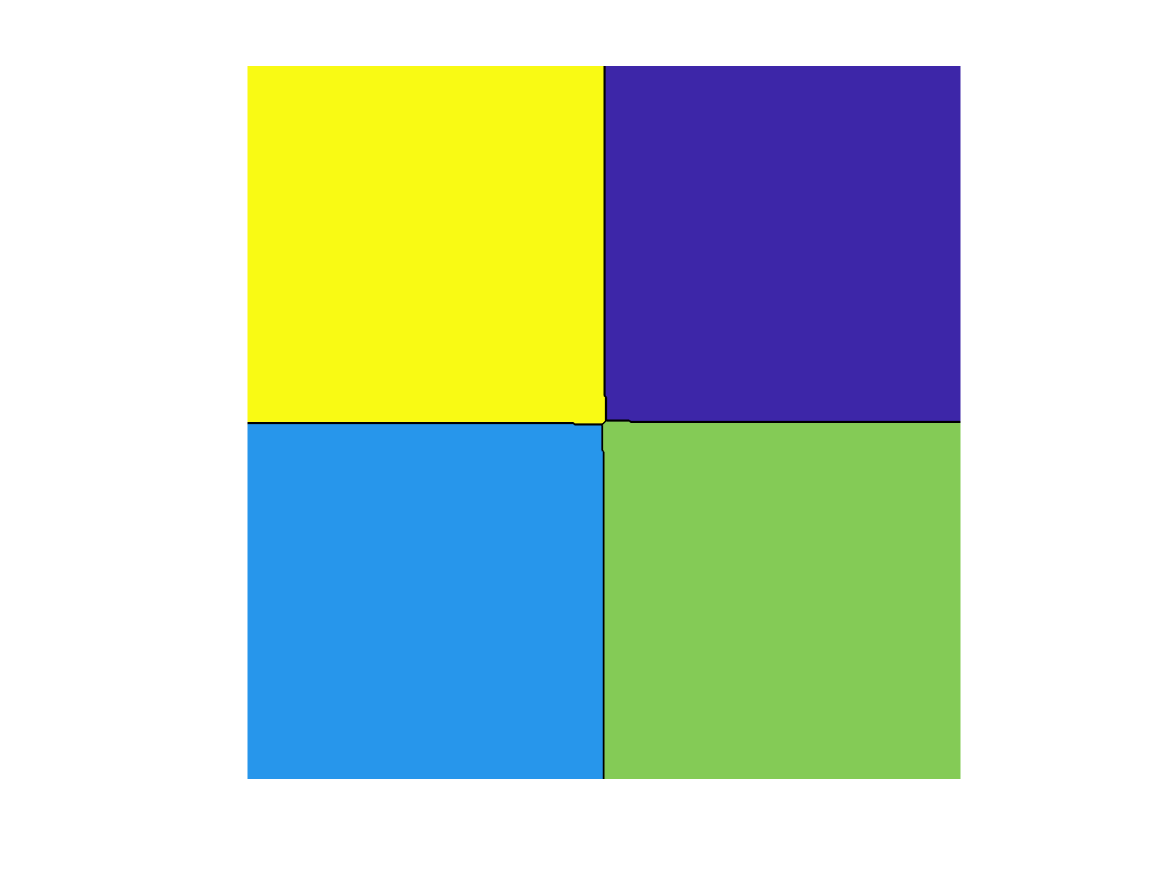}&
			\includegraphics[width = 0.09\textwidth, clip, trim = 4cm 1cm 3cm 1cm]{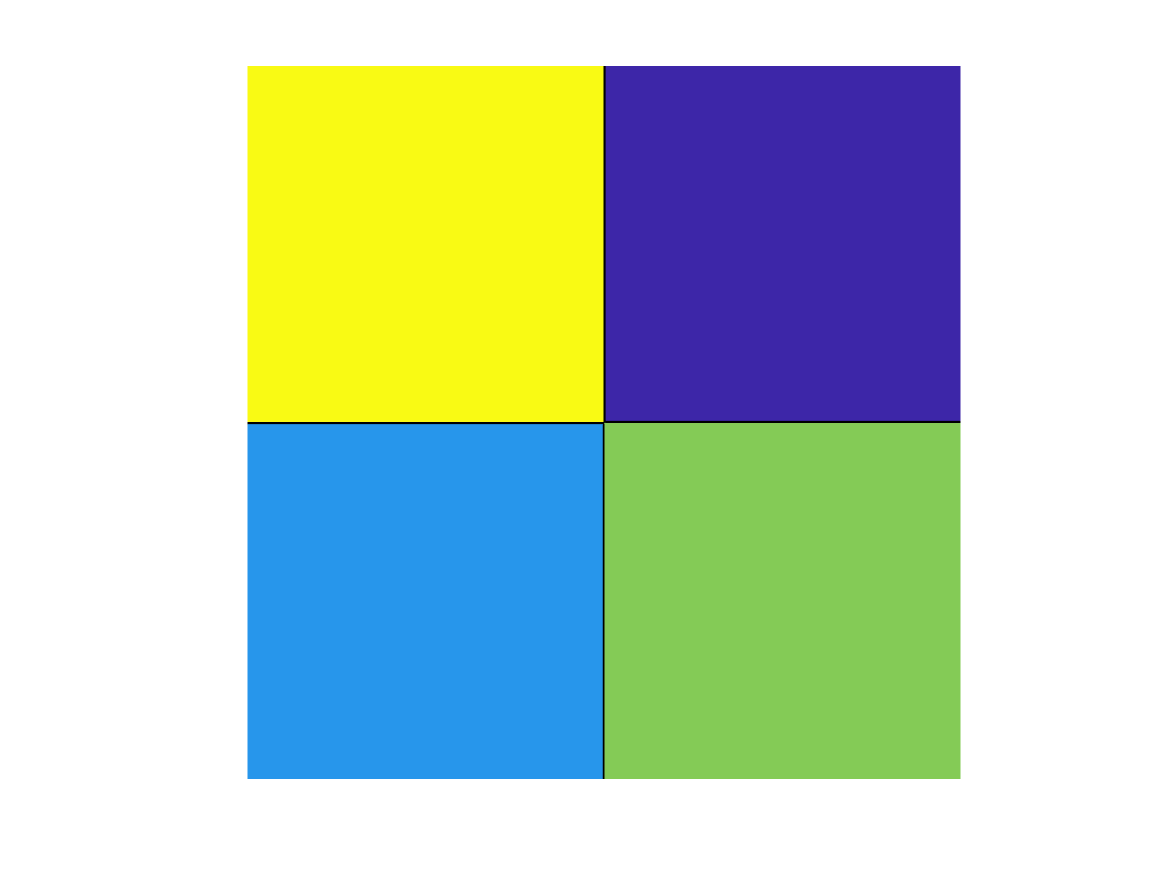} \\ 
			\hline
			\hline 
			initial & 50 &100  &150 & 200  & 250&300  & 415\\
			\includegraphics[width = 0.09\textwidth, clip, trim = 4cm 1cm 3cm 1cm]{figures_Dbnd/Dirinitialm5.png}&  
			\includegraphics[width = 0.09\textwidth, clip, trim = 4cm 1cm 3cm 1cm]{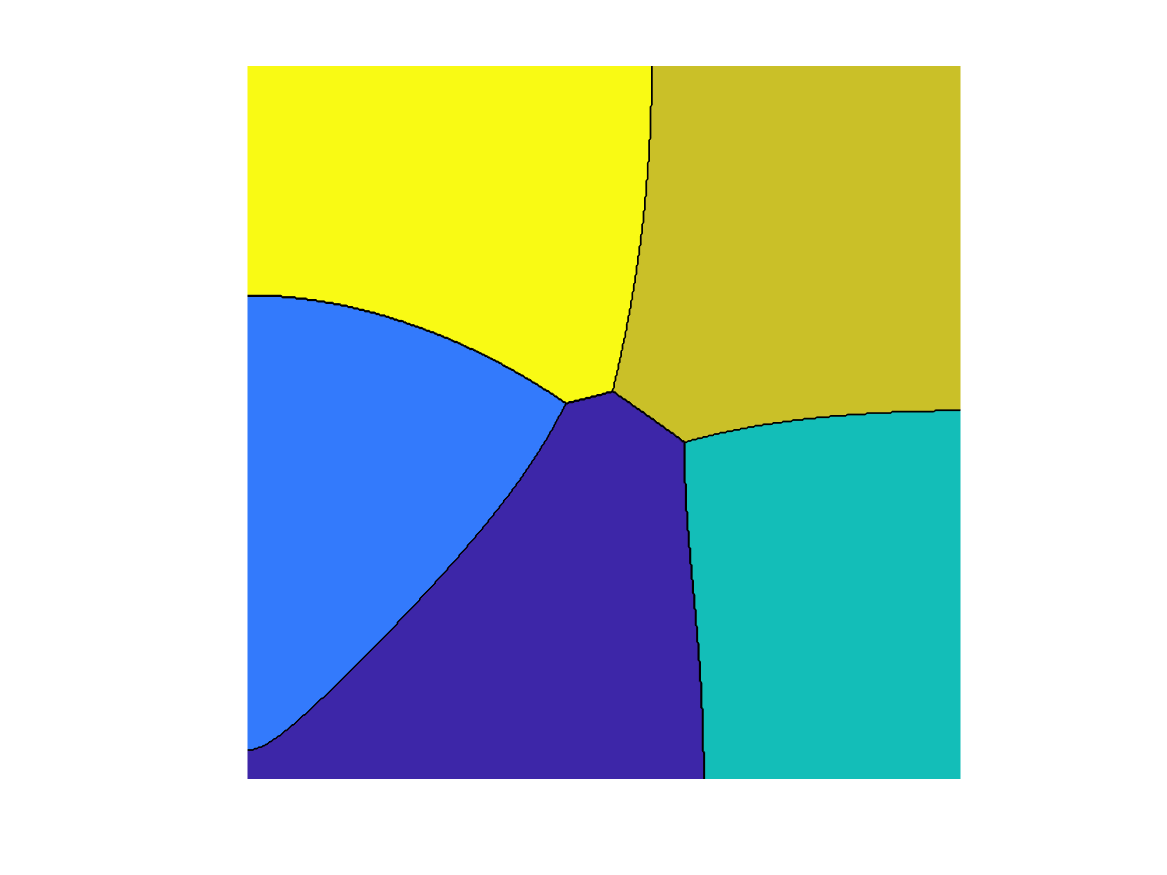}&
			\includegraphics[width = 0.09\textwidth, clip, trim = 4cm 1cm 3cm 1cm]{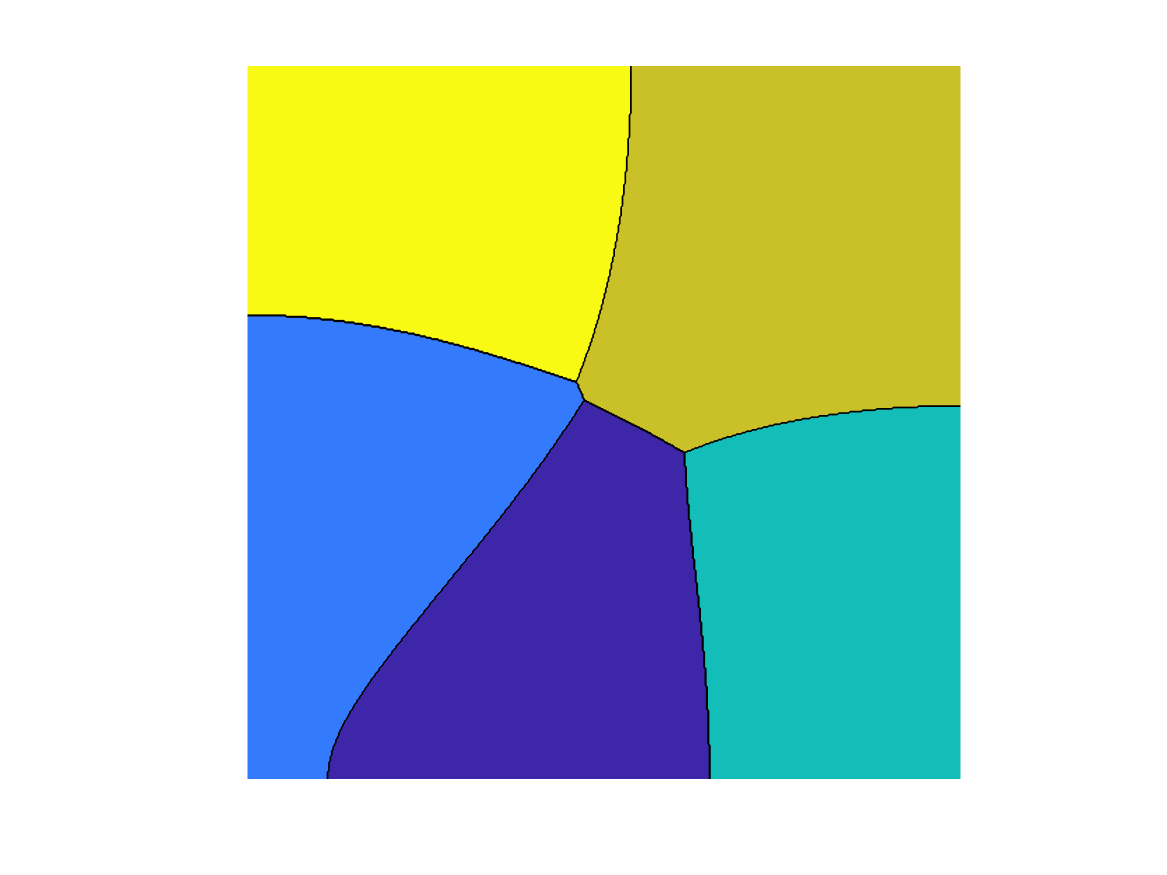}&
			\includegraphics[width = 0.09\textwidth, clip, trim = 4cm 1cm 3cm 1cm]{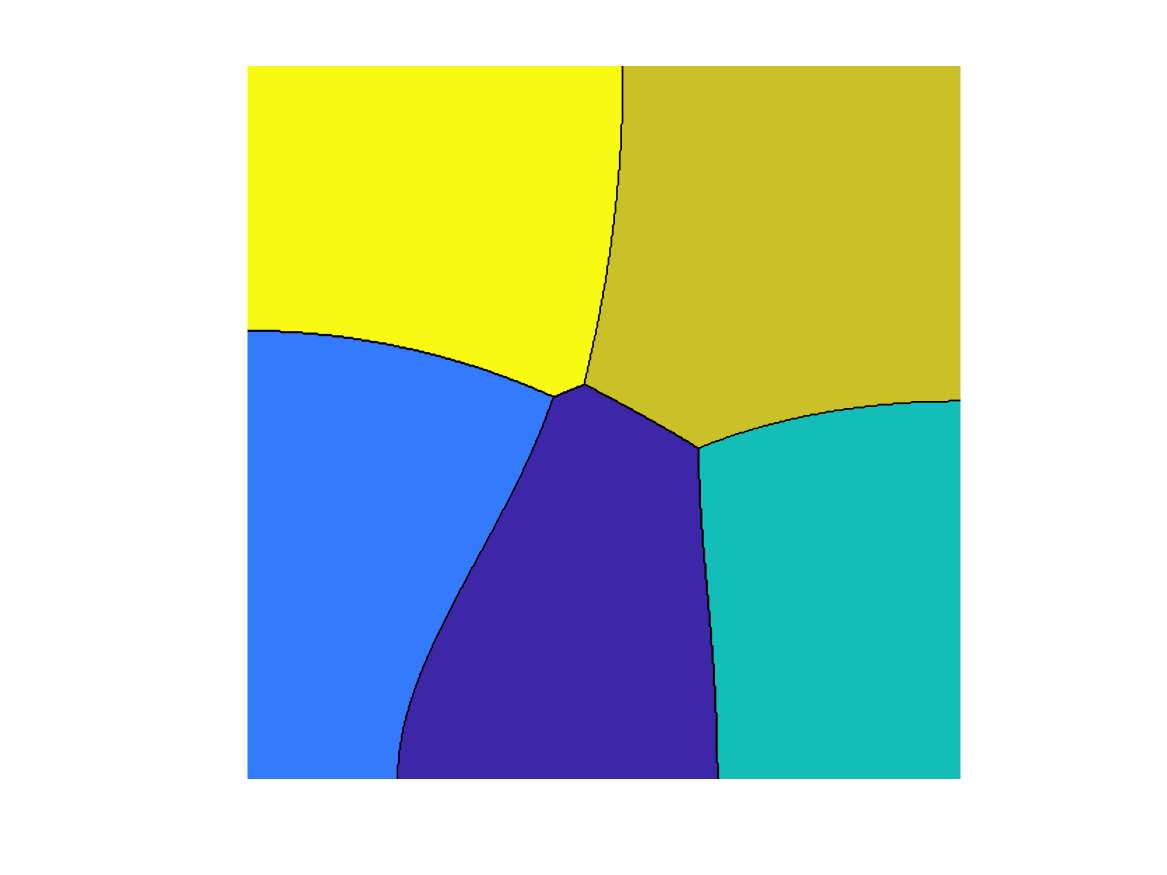}&
			\includegraphics[width = 0.09\textwidth, clip, trim = 4cm 1cm 3cm 1cm]{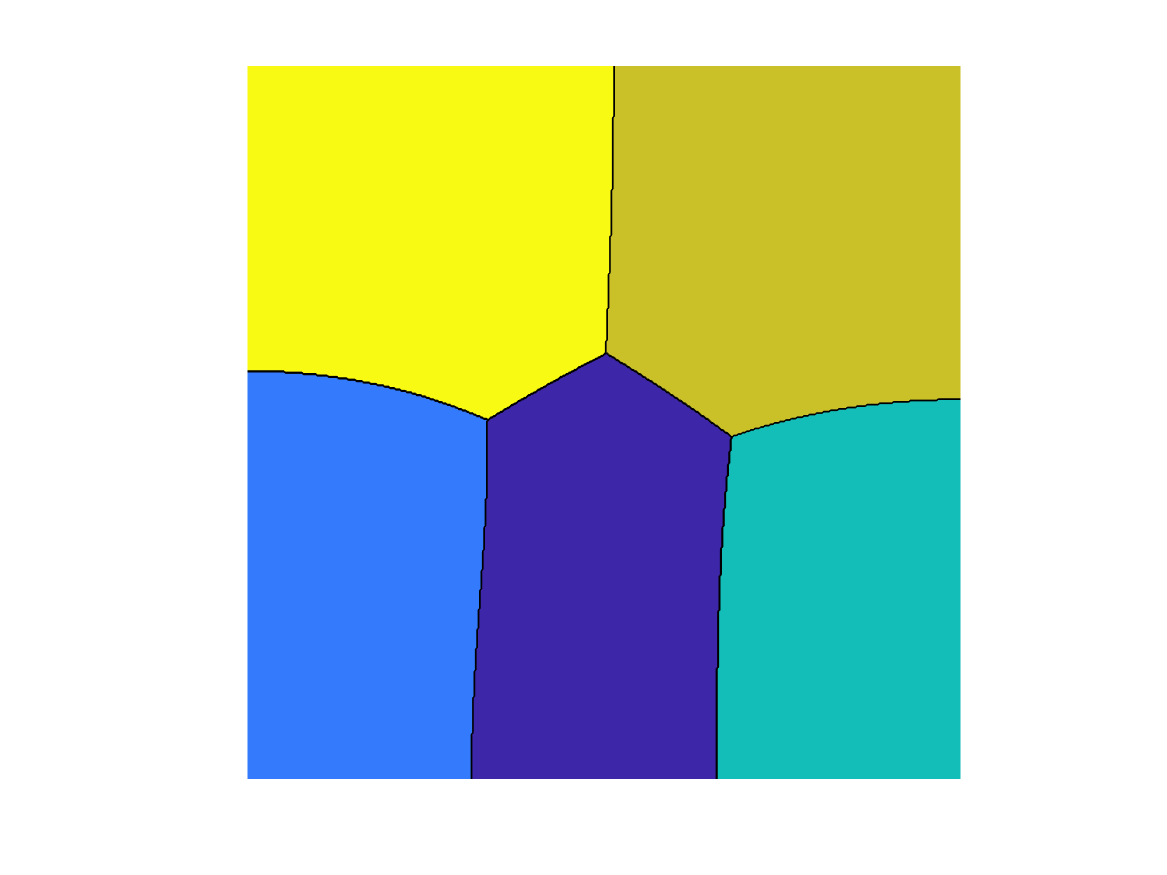}&
			\includegraphics[width = 0.09\textwidth, clip, trim = 4cm 1cm 3cm 1cm]{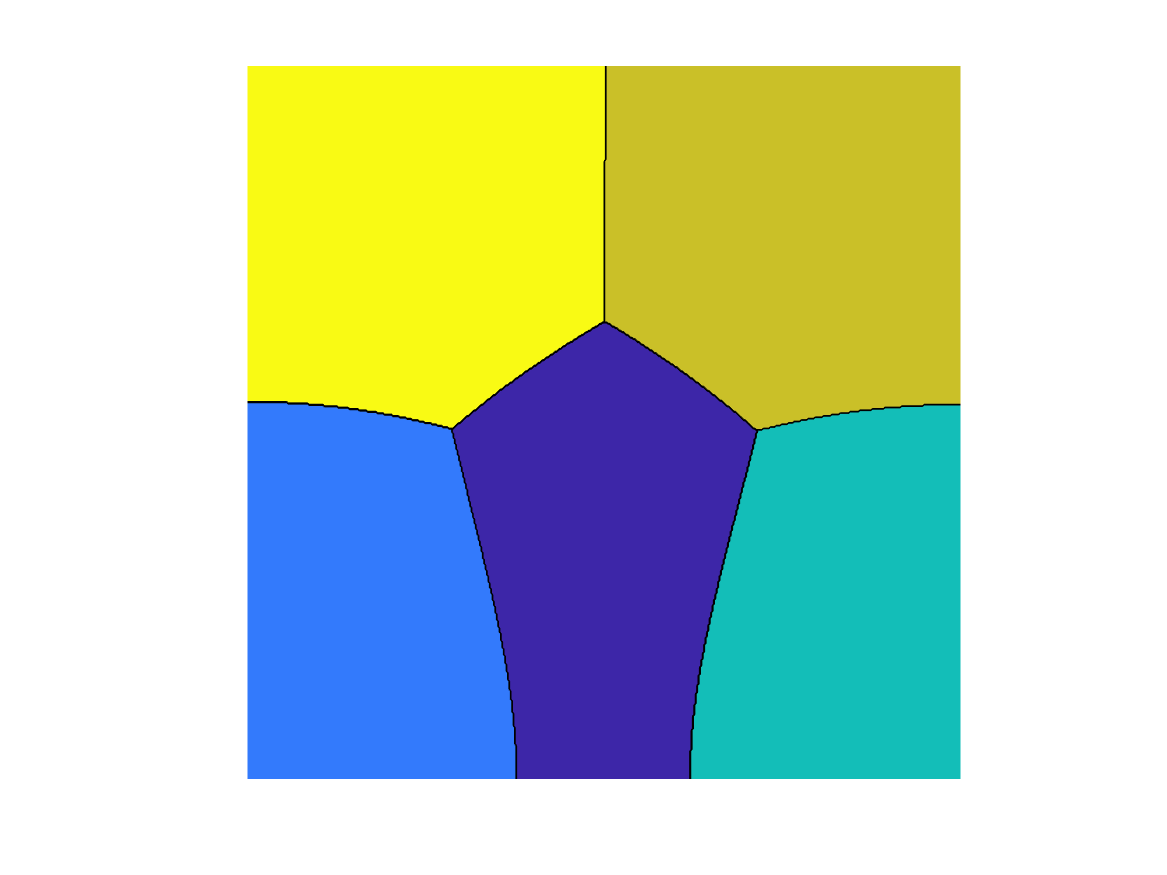}&
			\includegraphics[width = 0.09\textwidth, clip, trim = 4cm 1cm 3cm 1cm]{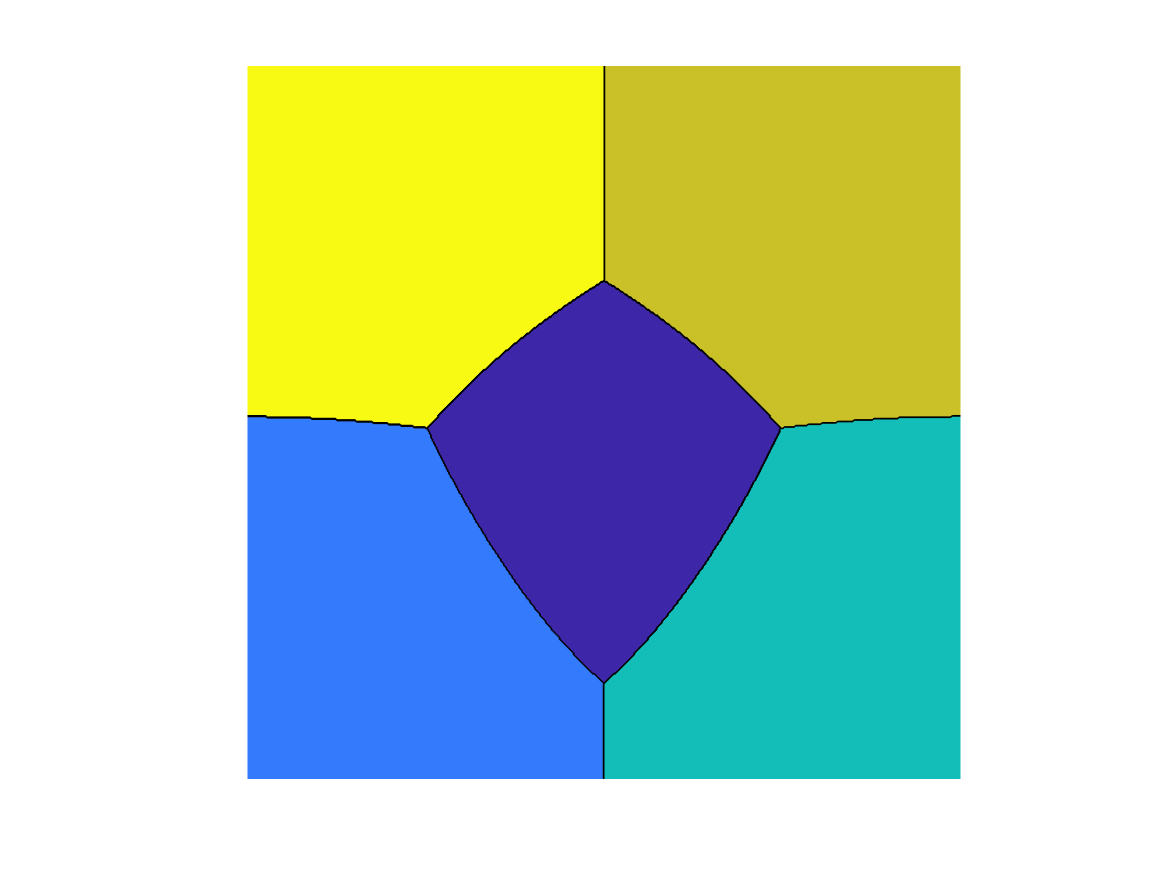}&
			\includegraphics[width = 0.09\textwidth, clip, trim = 4cm 1cm 3cm 1cm]{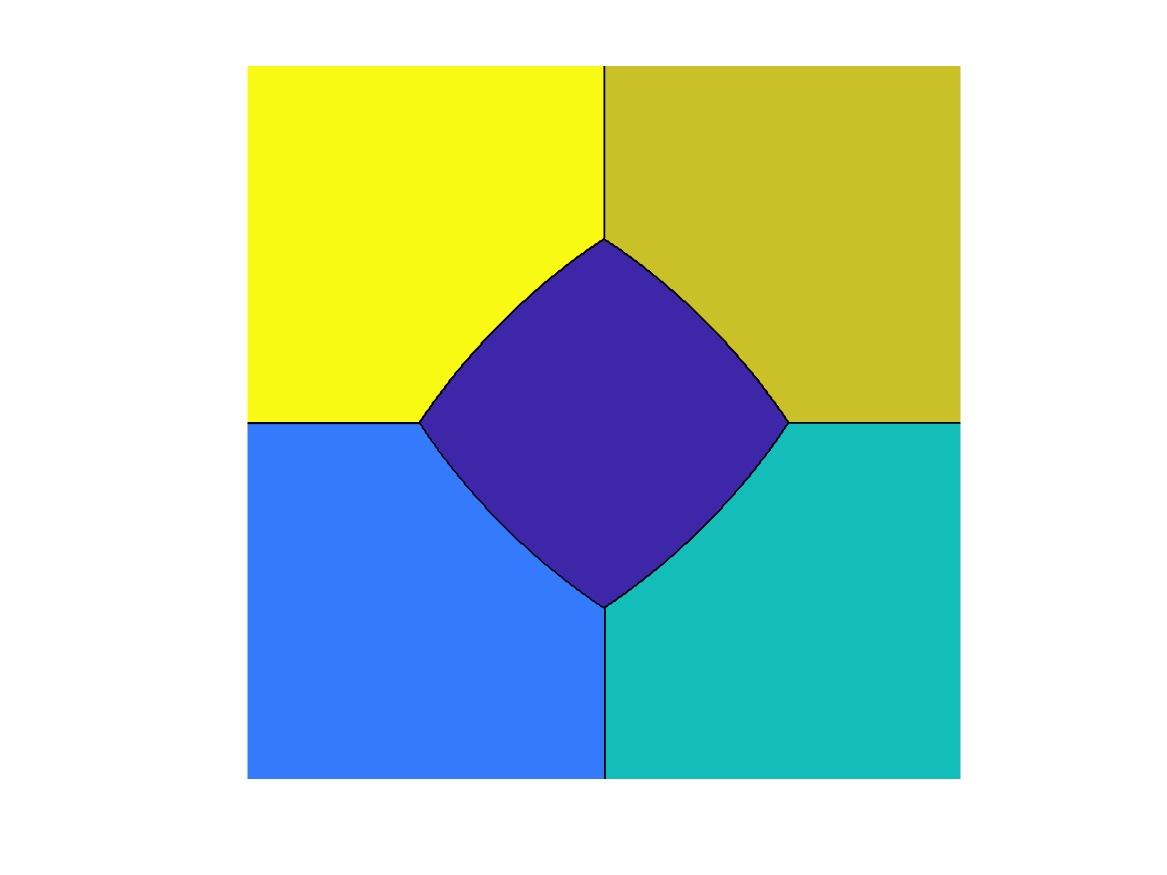} \\ 
			\hline
		\end{tabular}
		\caption{Snapshots of Algorithm~\ref{Alg:3step_Type1} with Dirichlet boundary conditions at different iterations on a $512\times512$ discretized mesh with a time step size of $\tau=0.1$. The first row corresponds to $k=4$, while the second row corresponds to $k=5$.} \label{fig:EvolAlg_3stepT1dt01Dir}
	\end{figure}
	\subsubsection{Arbitrary domains}  
For solving the gradient system \eqref{gradient_sys}-\eqref{norm:4} with Dirichlet boundary conditions on an arbitrary domain \(\Omega_{\text{int}}\), the algorithms discussed in this paper remain applicable by extending \(\Omega_{\text{int}}\) to a larger, regular domain \(\Omega\) (i.e., \(\Omega_{\text{int}} \subsetneq \Omega\)). Specifically, an additional step needs to be introduced following the diffusion step of the algorithms:
\[
\tilde{u}^{n+1} = \chi_{\Omega_{\text{int}}} \tilde{u}^{n+1},
\]
where the indicator function \( \chi_{\Omega_{\text{int}}} \) is defined in \eqref{u0_guess}.
 The results depicted in Figures~\ref{fig:5AbiDAlg1Alg1} and \ref{fig:AbiDAlg1Alg2Alg3} are obtained with $\Omega=[-\pi,\pi]^2$ using a standard compact difference method \cite{ZhuJuZhao} on a $512 \times 512$ discretized mesh. To prevent the partition from diffusing outside \(\Omega_{\text{int}}\), we set \(\tau=[1/128, 1/64, 1/32, 1/16]\) for the initial four iterations and then adopt \(\tau=1/8\) for subsequent iterations.
 The results in Figure~\ref{fig:5AbiDAlg1Alg1} validate the robust performance of Algorithm~\ref{Alg:4step} across domains with diverse shapes and topologies. Additionally, Figure~\ref{fig:AbiDAlg1Alg2Alg3} illustrates that Algorithms~\ref{Alg:4step}-\ref{Alg:3step_Type2} converge to the same partition when initialized under identical conditions, highlighting their consistency and reliability.
	\begin{figure}[H]
		\centering
		\includegraphics[width = 0.1\textwidth, clip, trim = 7cm 5.5cm 6cm 4cm]{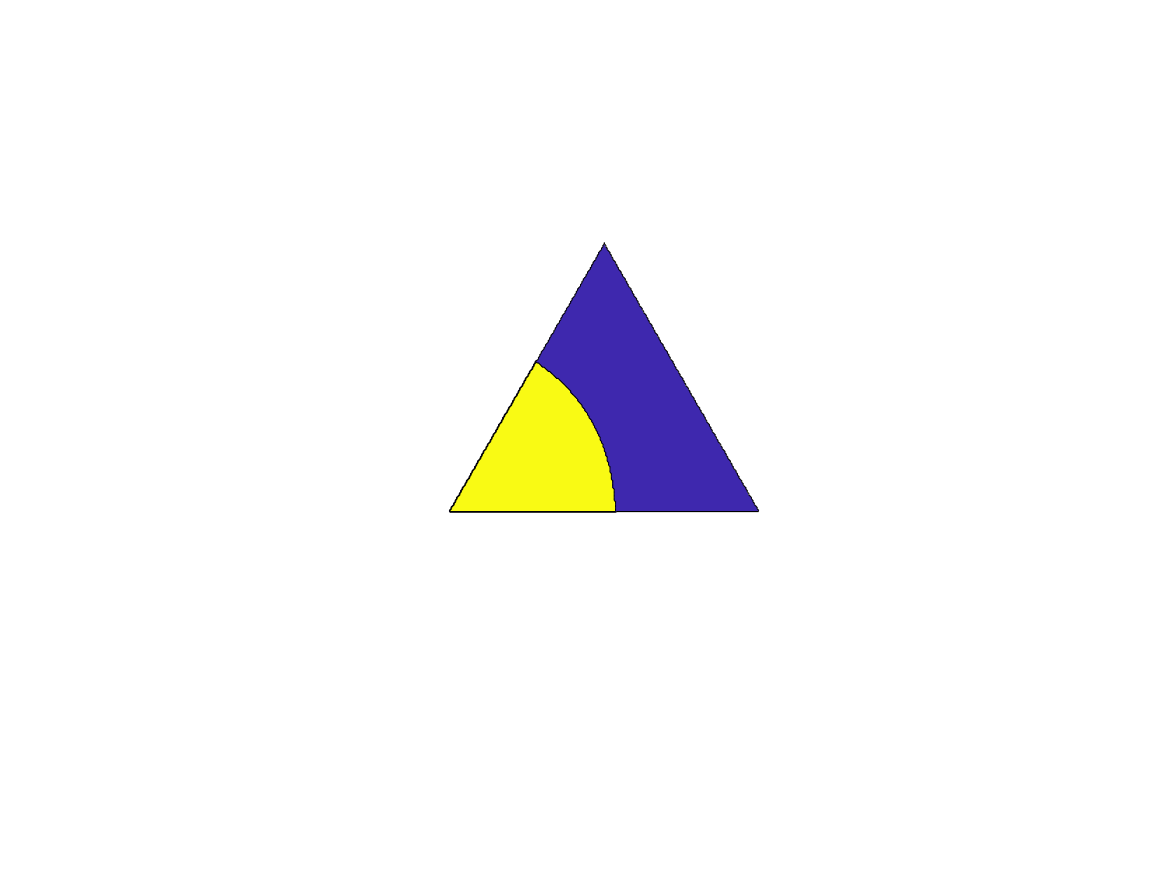}
		\includegraphics[width = 0.1\textwidth, clip, trim = 7cm 5.5cm 6cm 4cm]{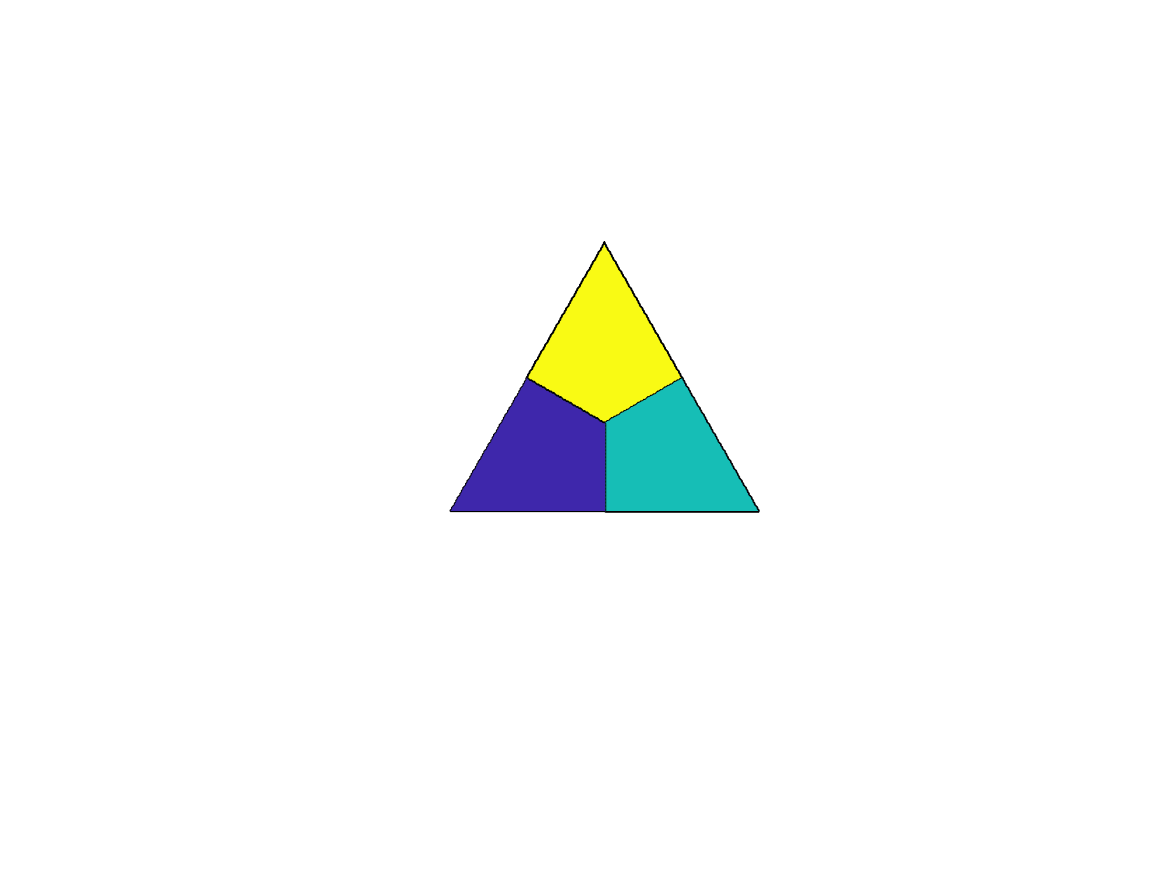}
		\includegraphics[width = 0.1\textwidth, clip, trim = 7cm 5.5cm 6cm 4cm]{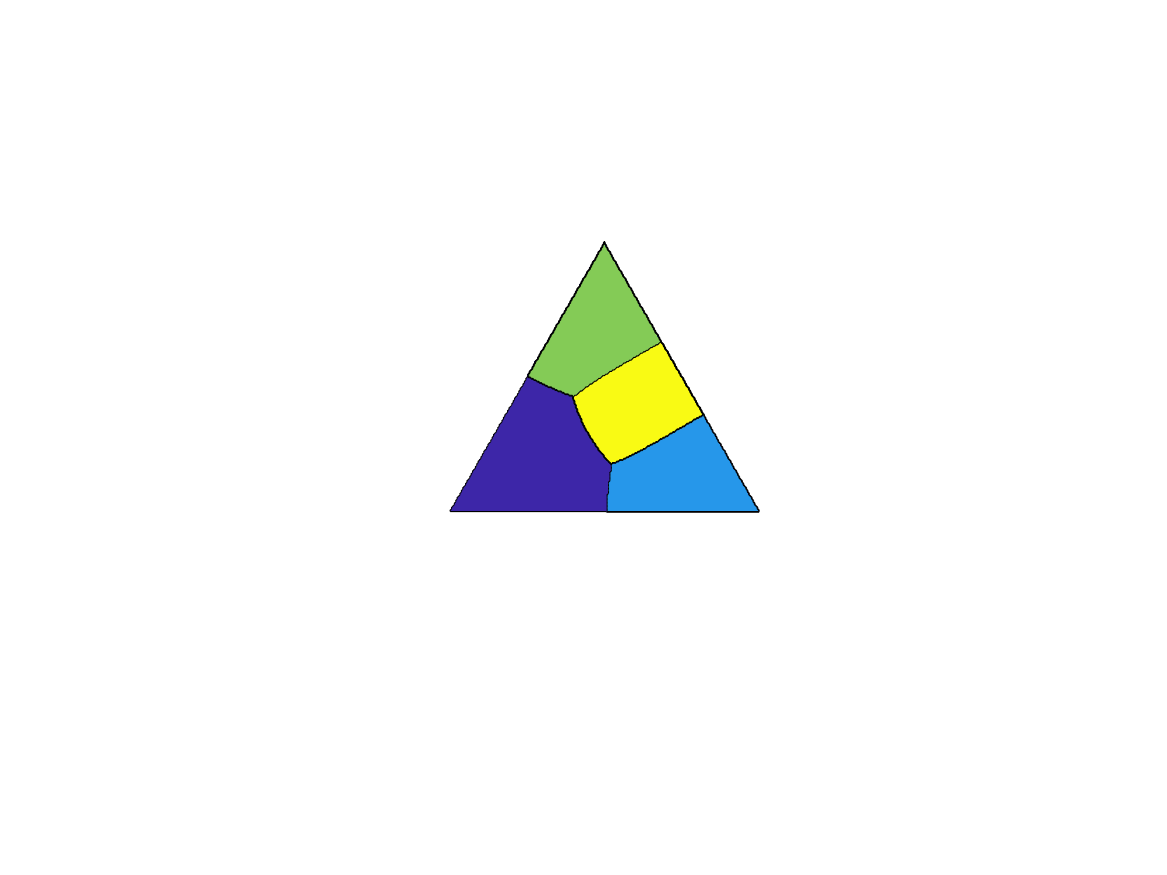}
		\includegraphics[width = 0.1\textwidth, clip, trim = 7cm 5.5cm 6cm 4cm]{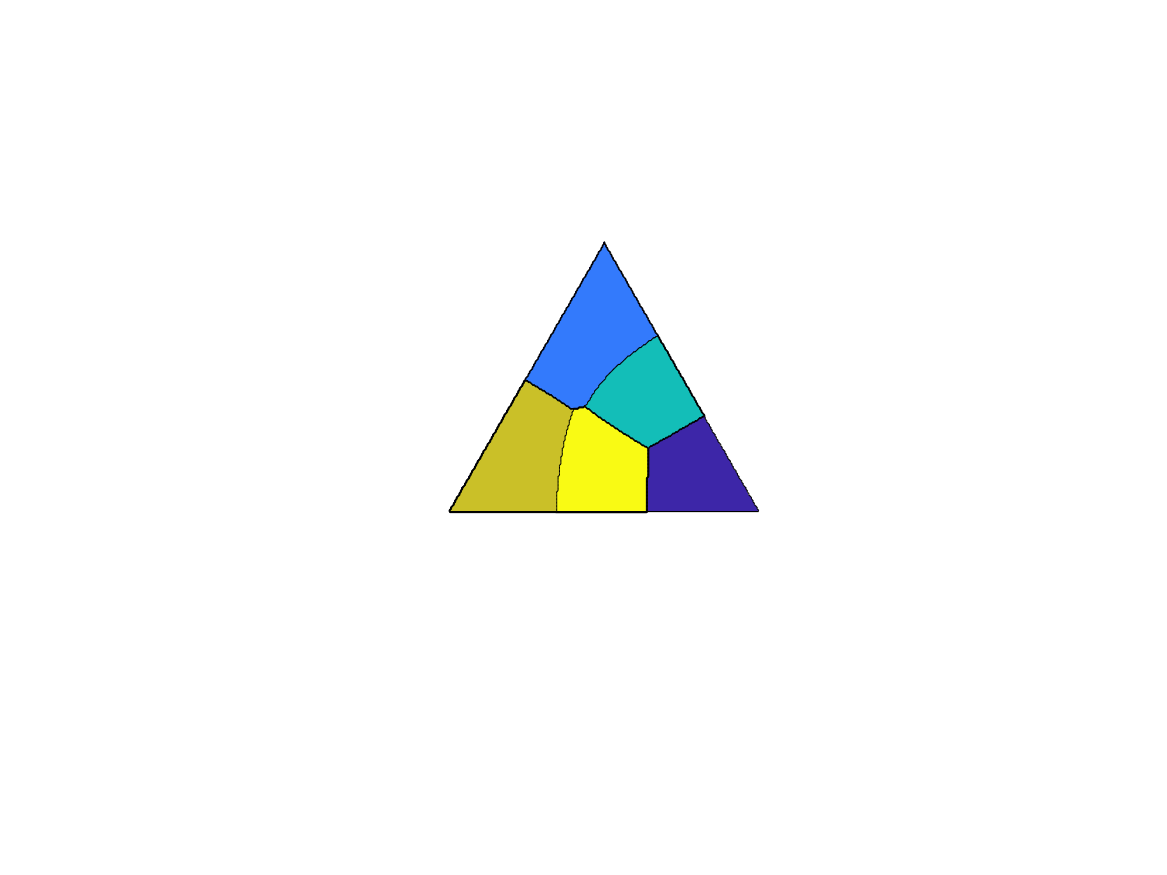}
		\includegraphics[width = 0.1\textwidth, clip, trim = 7cm 5.5cm 6cm 4cm]{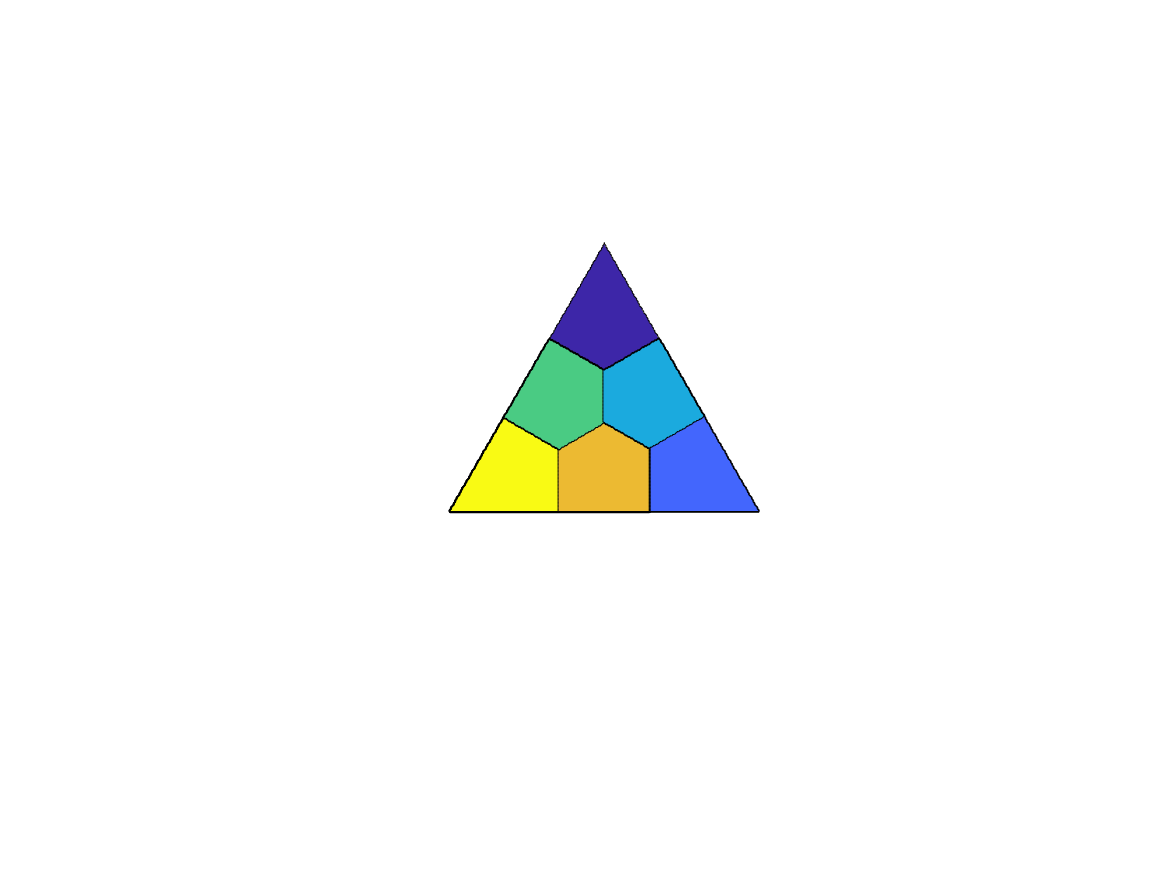}
		\includegraphics[width = 0.1\textwidth, clip, trim = 7cm 5.5cm 6cm 4cm]{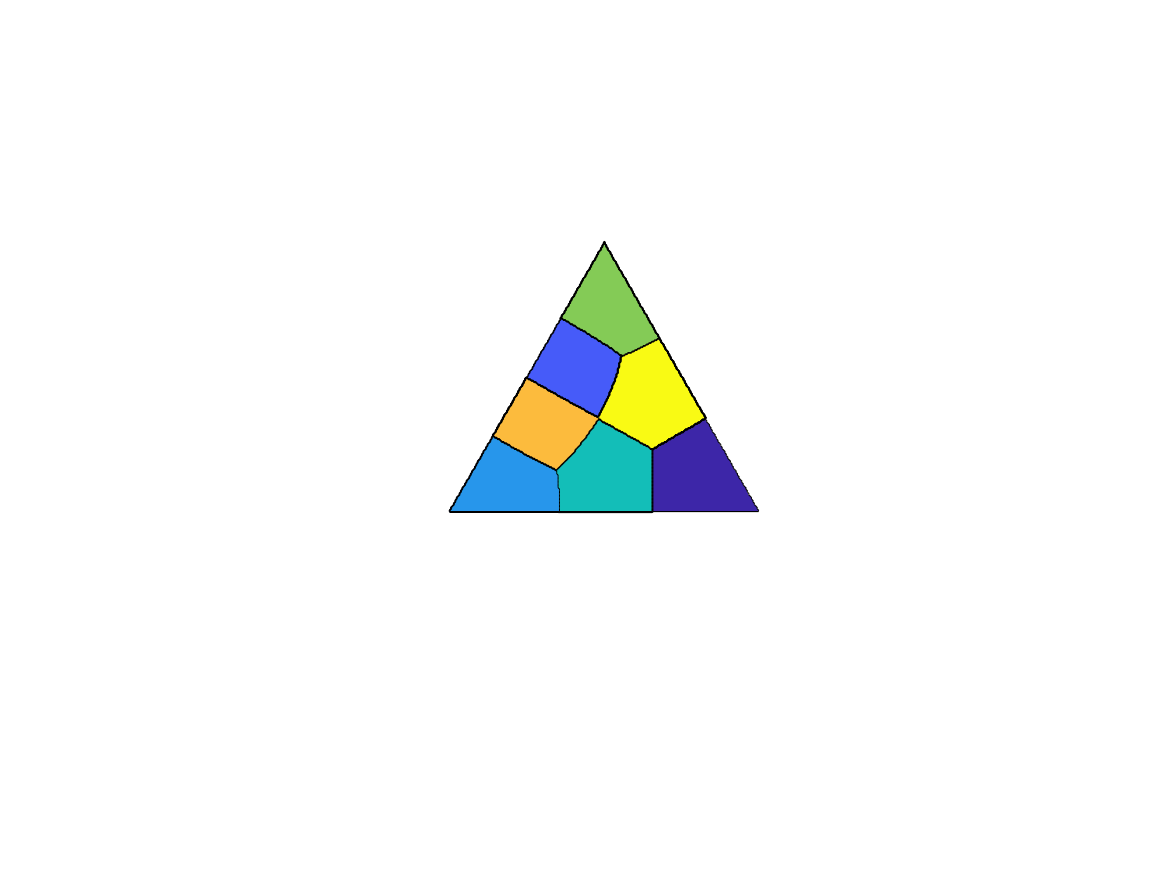}
		\includegraphics[width = 0.1\textwidth, clip, trim = 7cm 5.5cm 6cm 4cm]{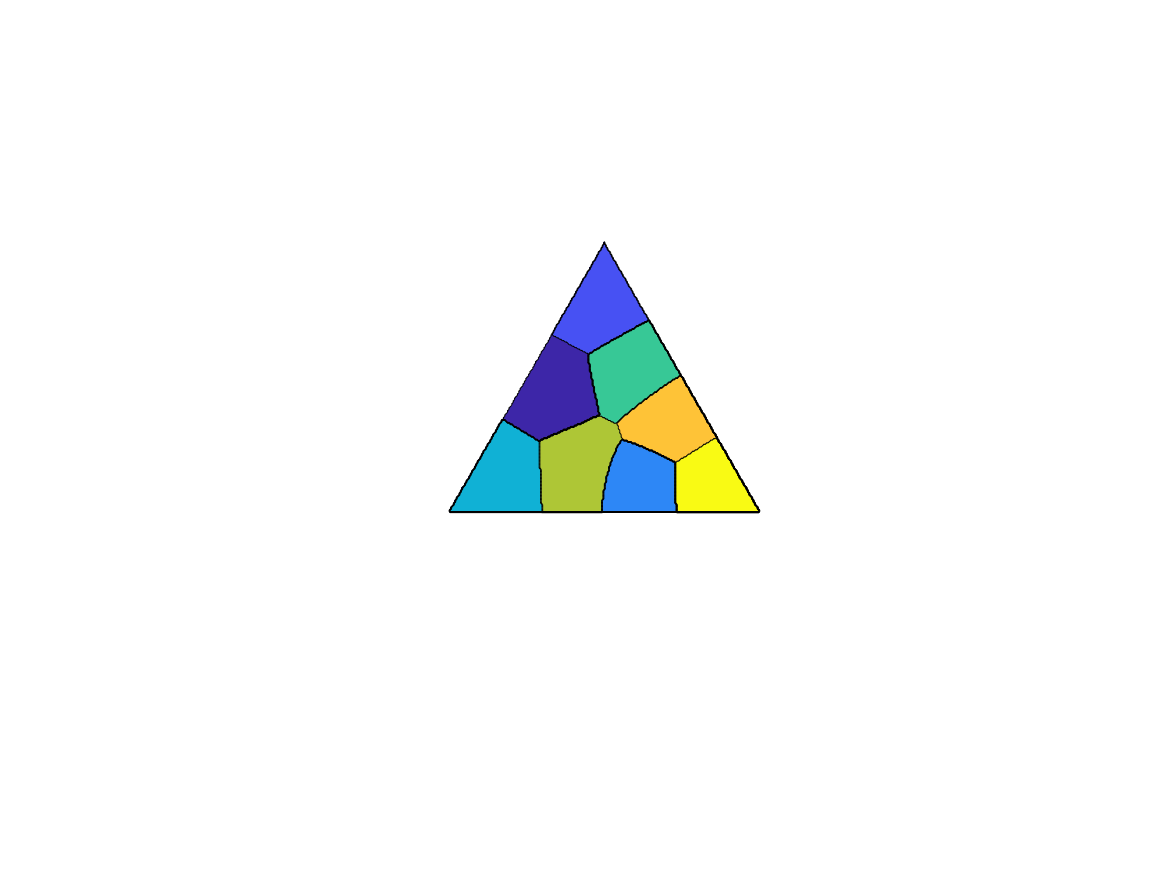}
		\includegraphics[width = 0.1\textwidth, clip, trim = 7cm 5.5cm 6cm 4cm]{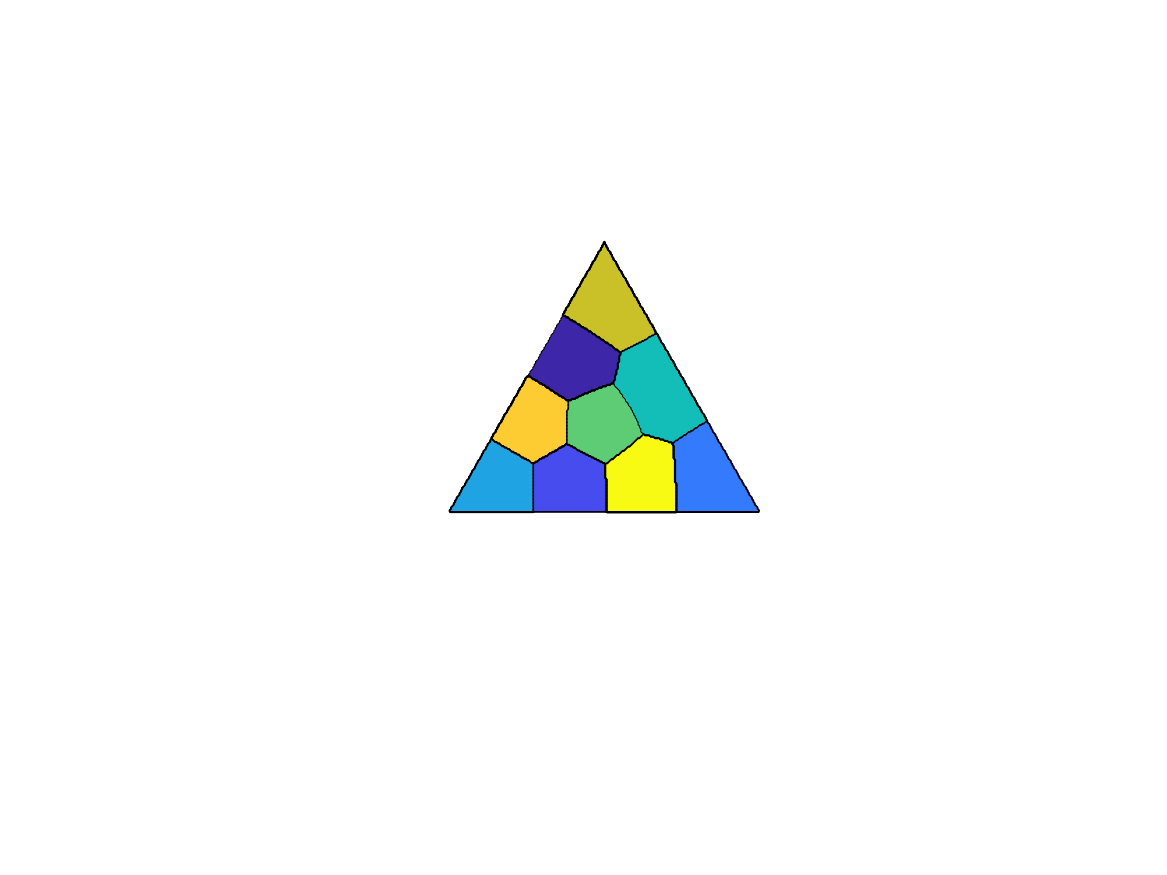}
		\includegraphics[width = 0.1\textwidth, clip, trim = 7cm 5.5cm 6cm 4cm]{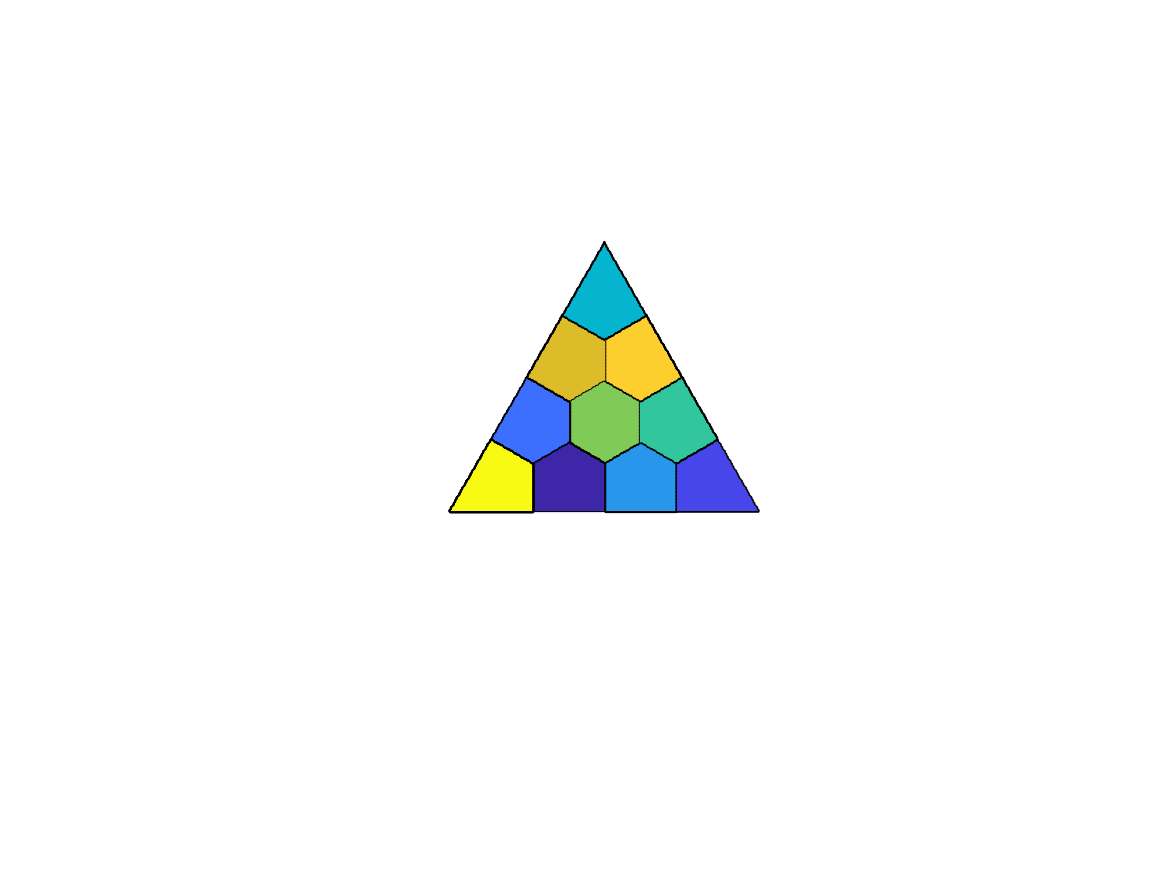}
		\smallskip
		\includegraphics[width = 0.1\textwidth, clip, trim = 7cm 5cm 6cm 4cm]{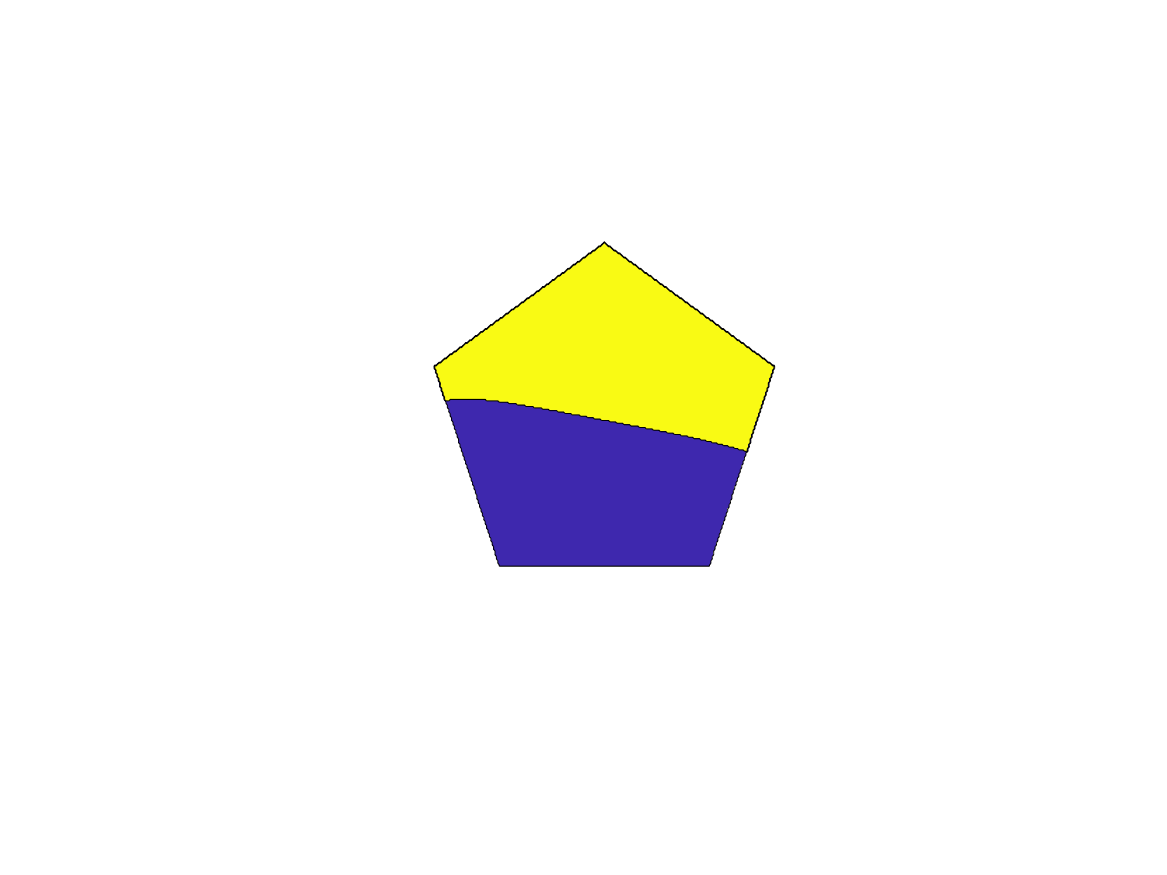}
		\includegraphics[width = 0.1\textwidth, clip, trim = 7cm 5cm 6cm 4cm]{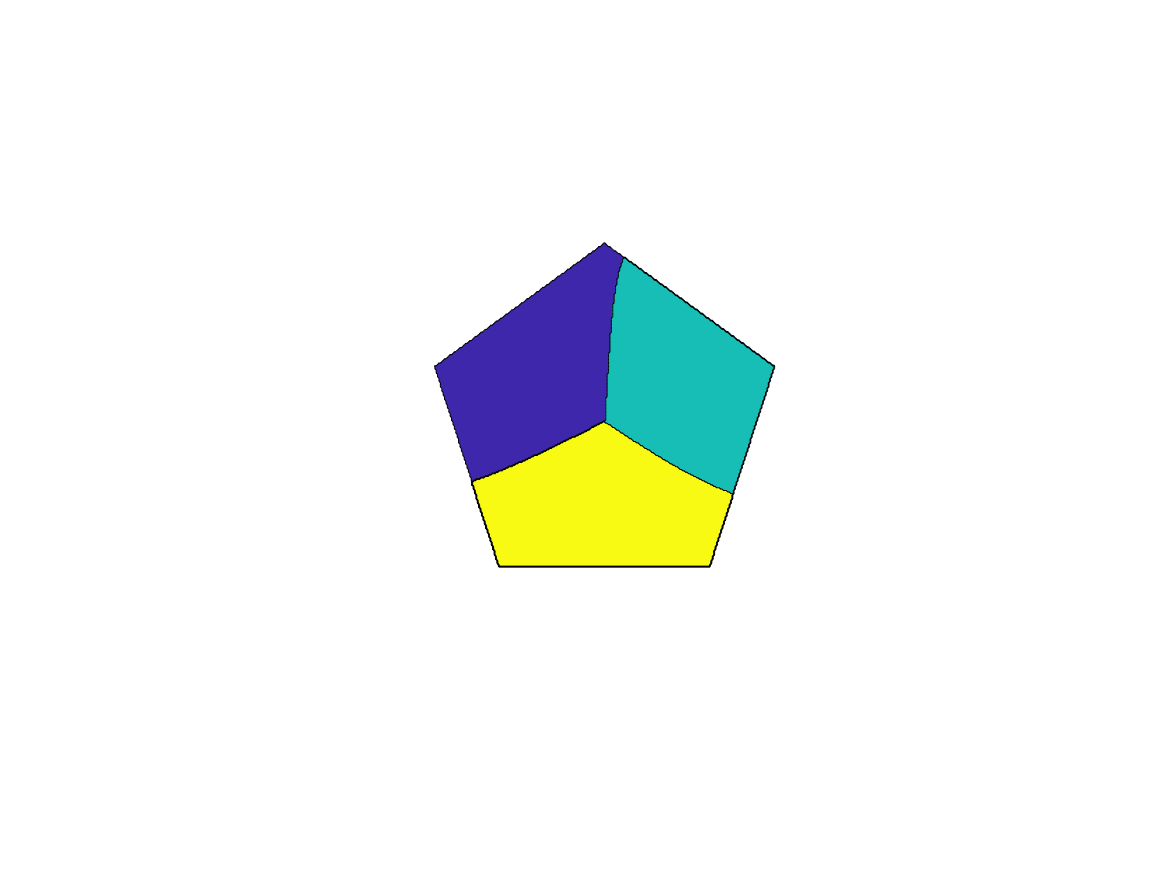}
		\includegraphics[width = 0.1\textwidth, clip, trim = 7cm 5cm 6cm 4cm]{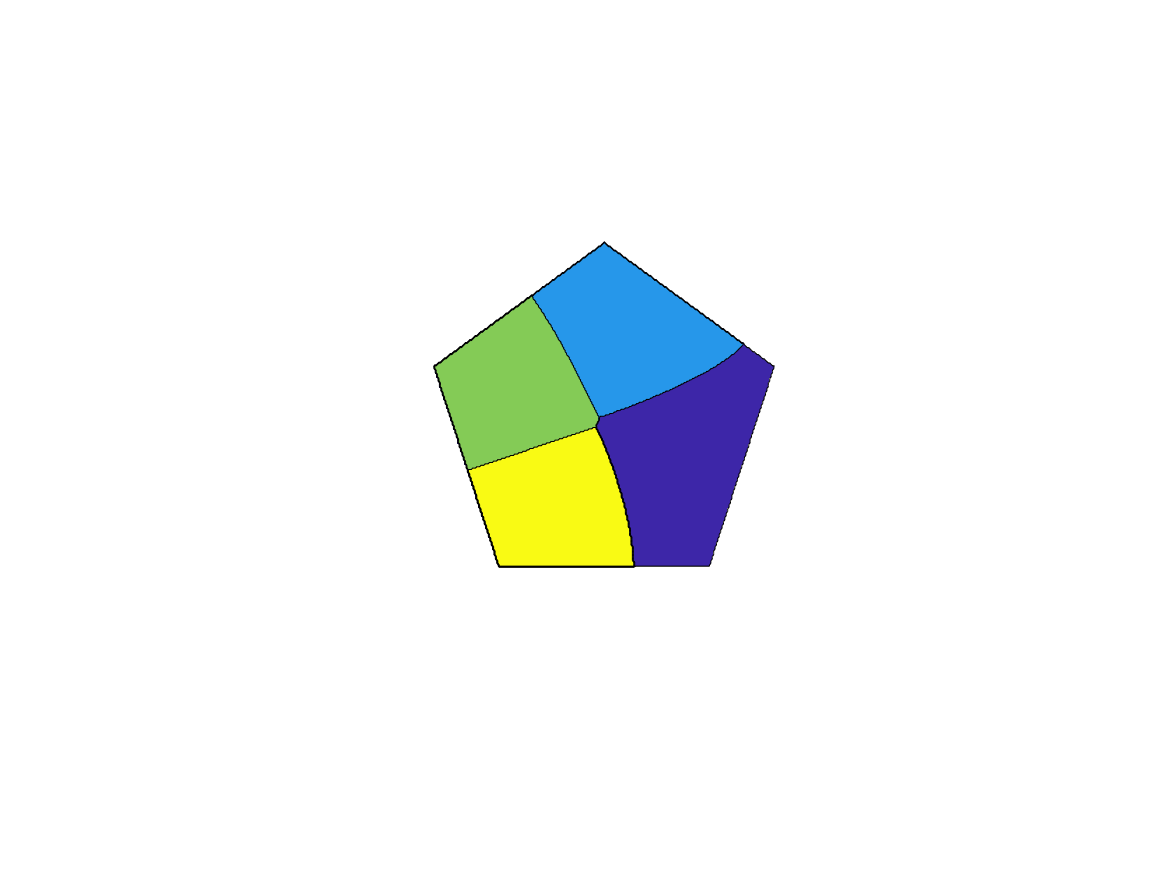}
		\includegraphics[width = 0.1\textwidth, clip, trim = 7cm 5cm 6cm 4cm]{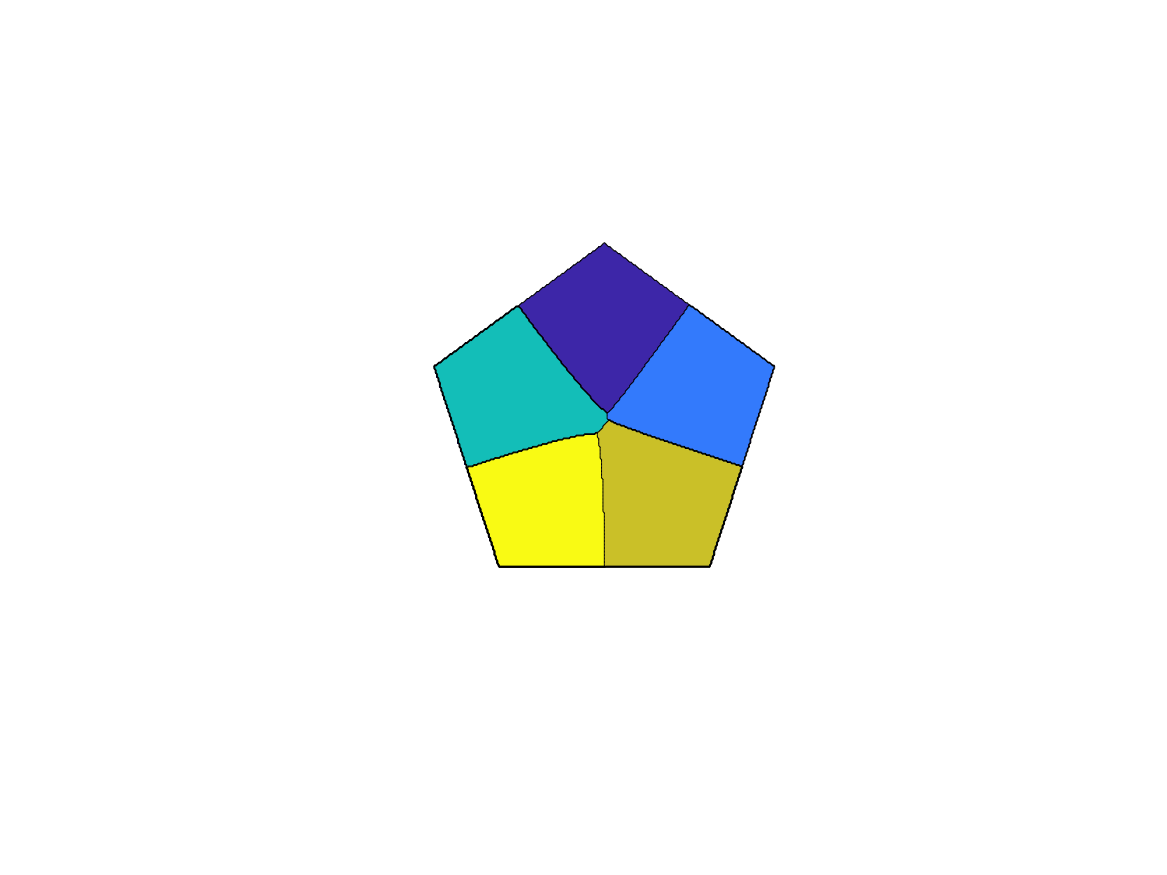}
		\includegraphics[width = 0.1\textwidth, clip, trim = 7cm 5cm 6cm 4cm]{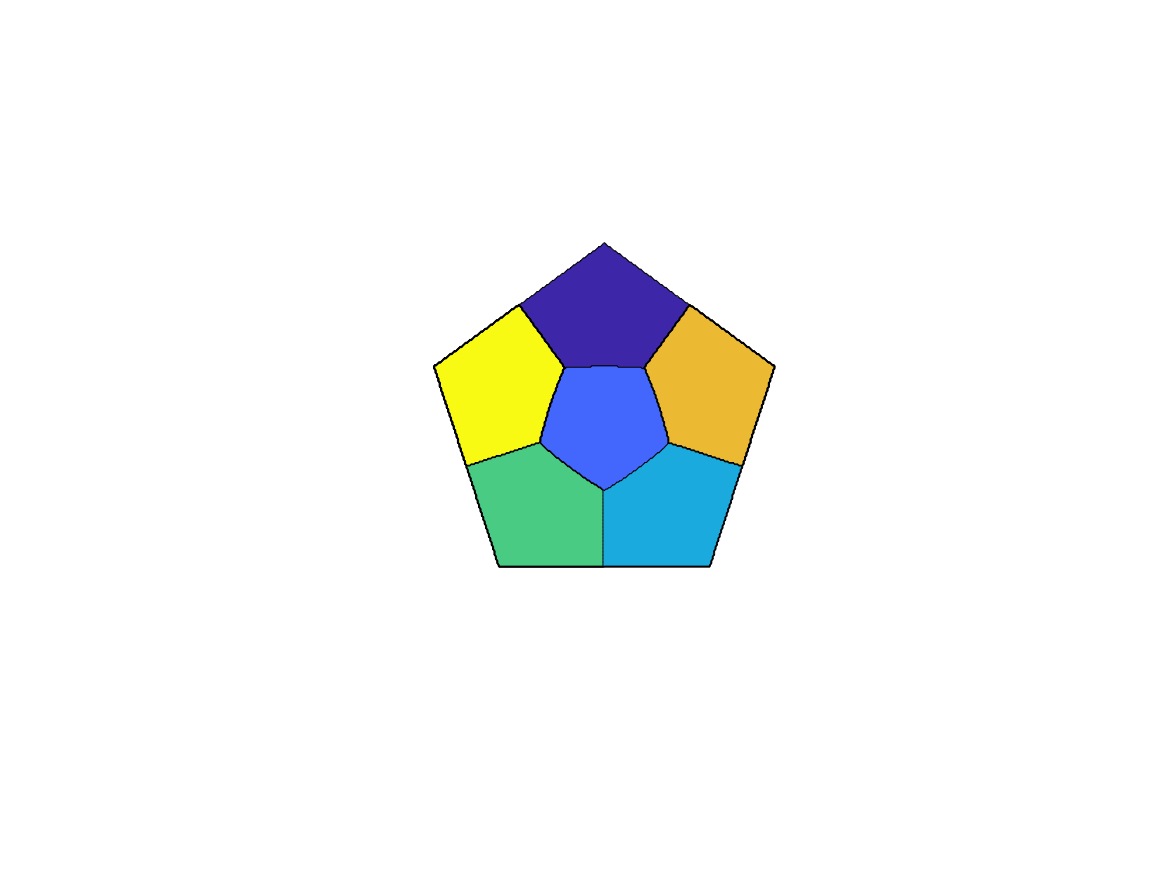}
		\includegraphics[width = 0.1\textwidth, clip, trim = 7cm 5cm 6cm 4cm]{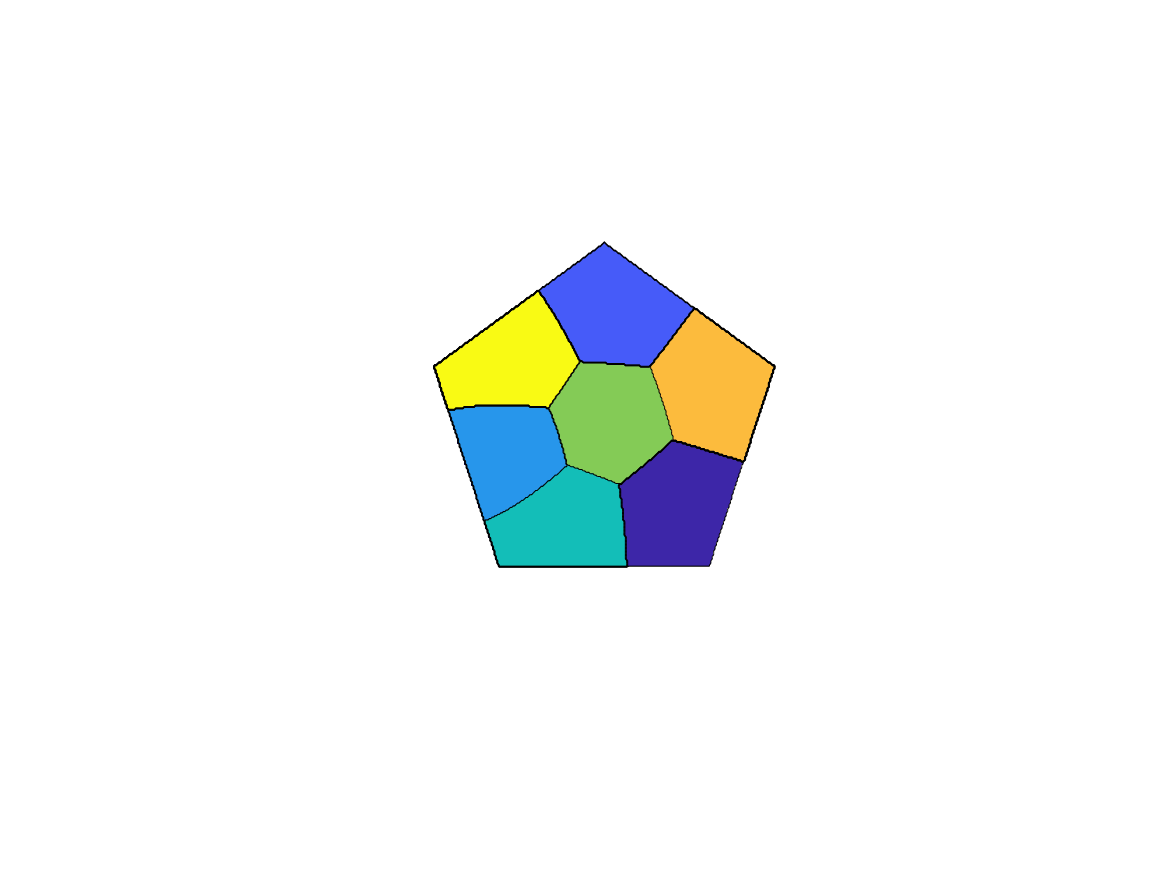}
		\includegraphics[width = 0.1\textwidth, clip, trim = 7cm 5cm 6cm 4cm]{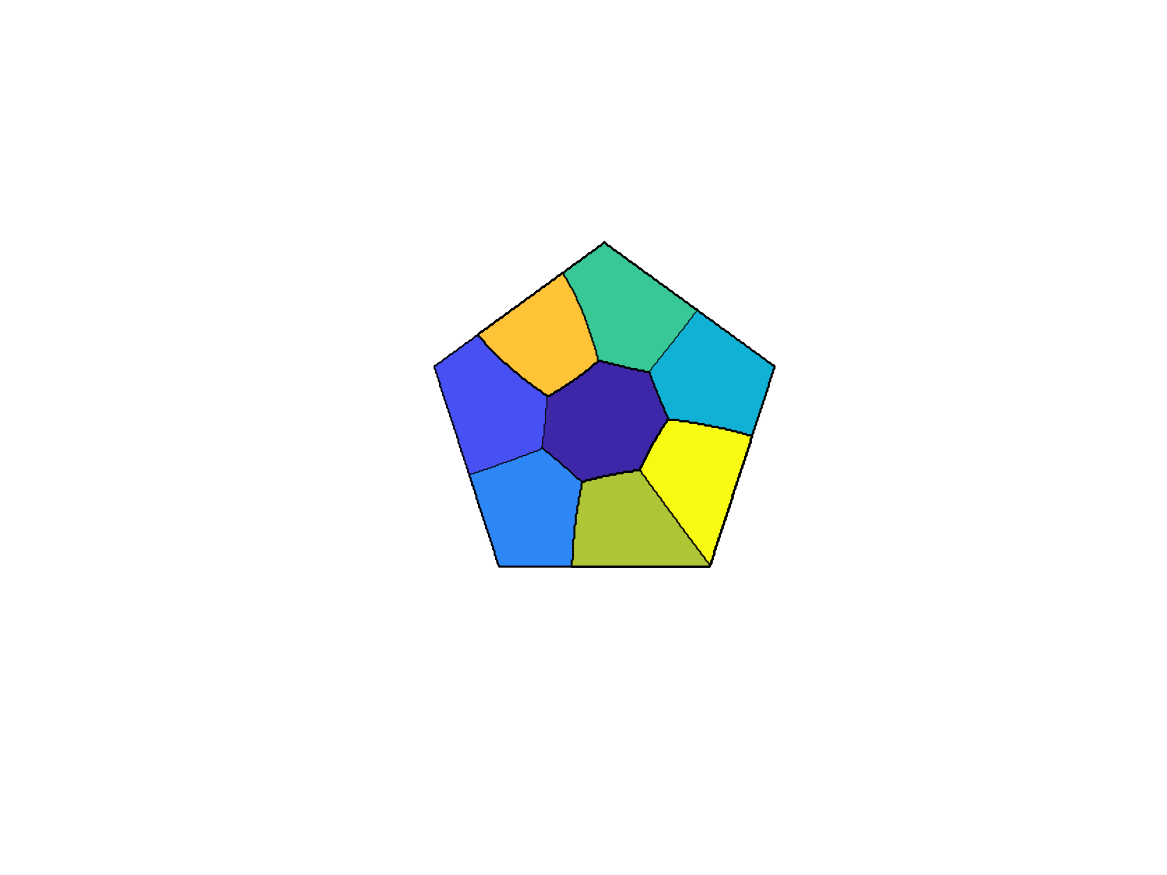}
		\includegraphics[width = 0.1\textwidth, clip, trim = 7cm 5cm 6cm 4cm]{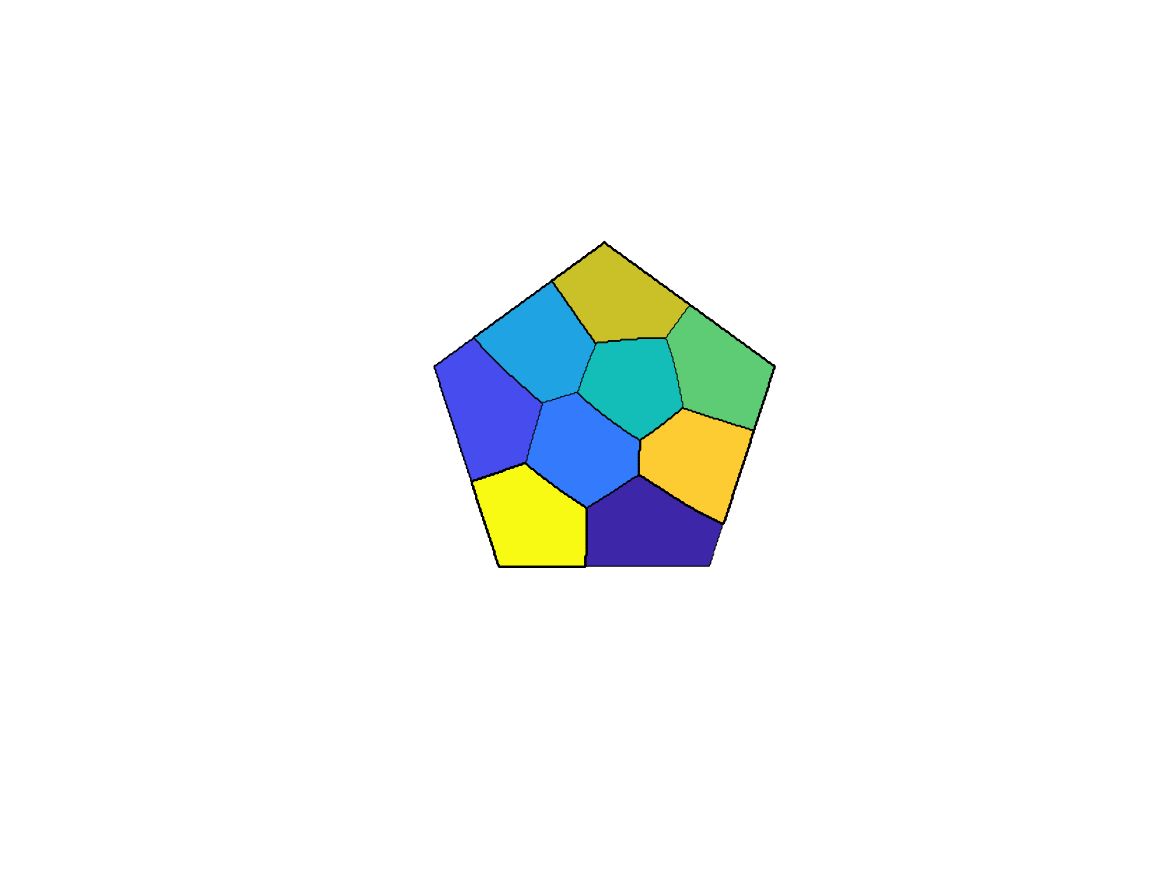}
		\includegraphics[width = 0.1\textwidth, clip, trim = 7cm 5cm 6cm 4cm]{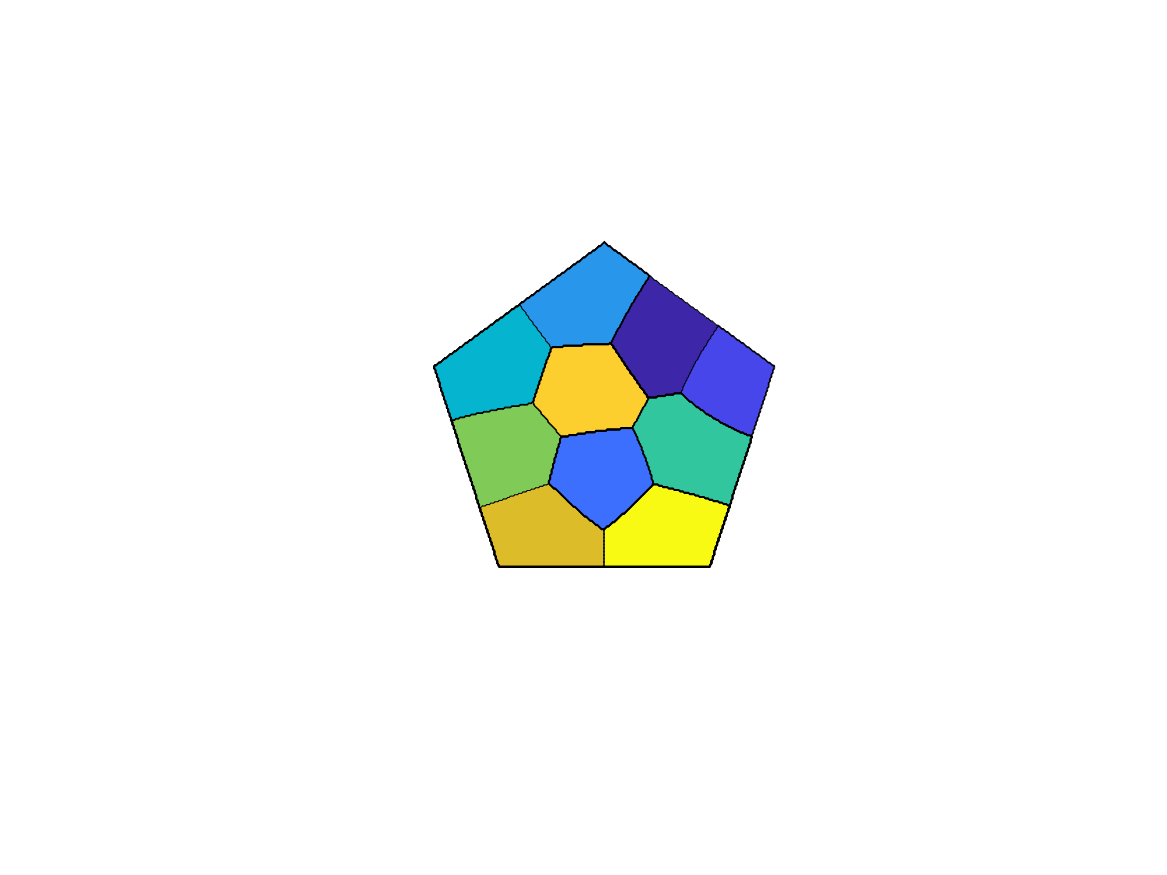}
		\smallskip
		\includegraphics[width = 0.1\textwidth, clip, trim = 7cm 4.5cm 6cm 4cm]{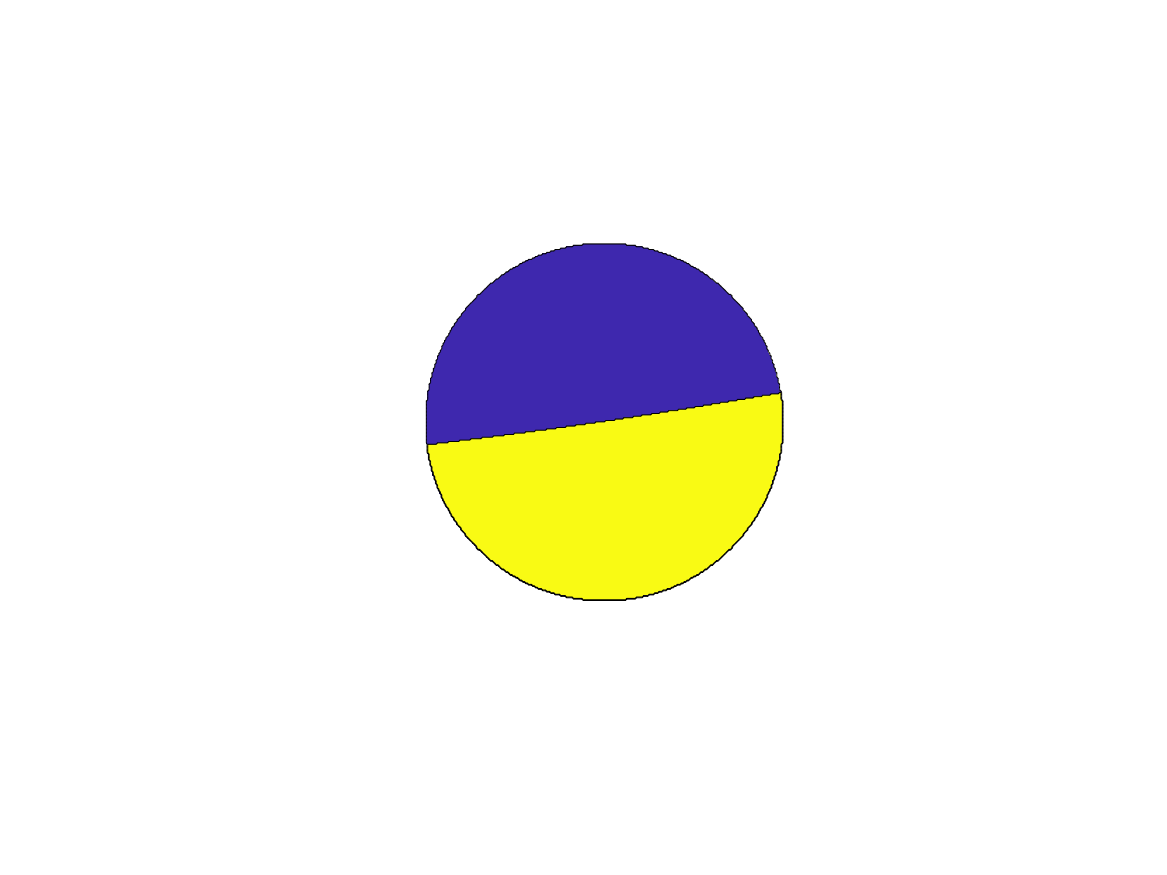}
		\includegraphics[width = 0.1\textwidth, clip, trim = 7cm 4.5cm 6cm 4cm]{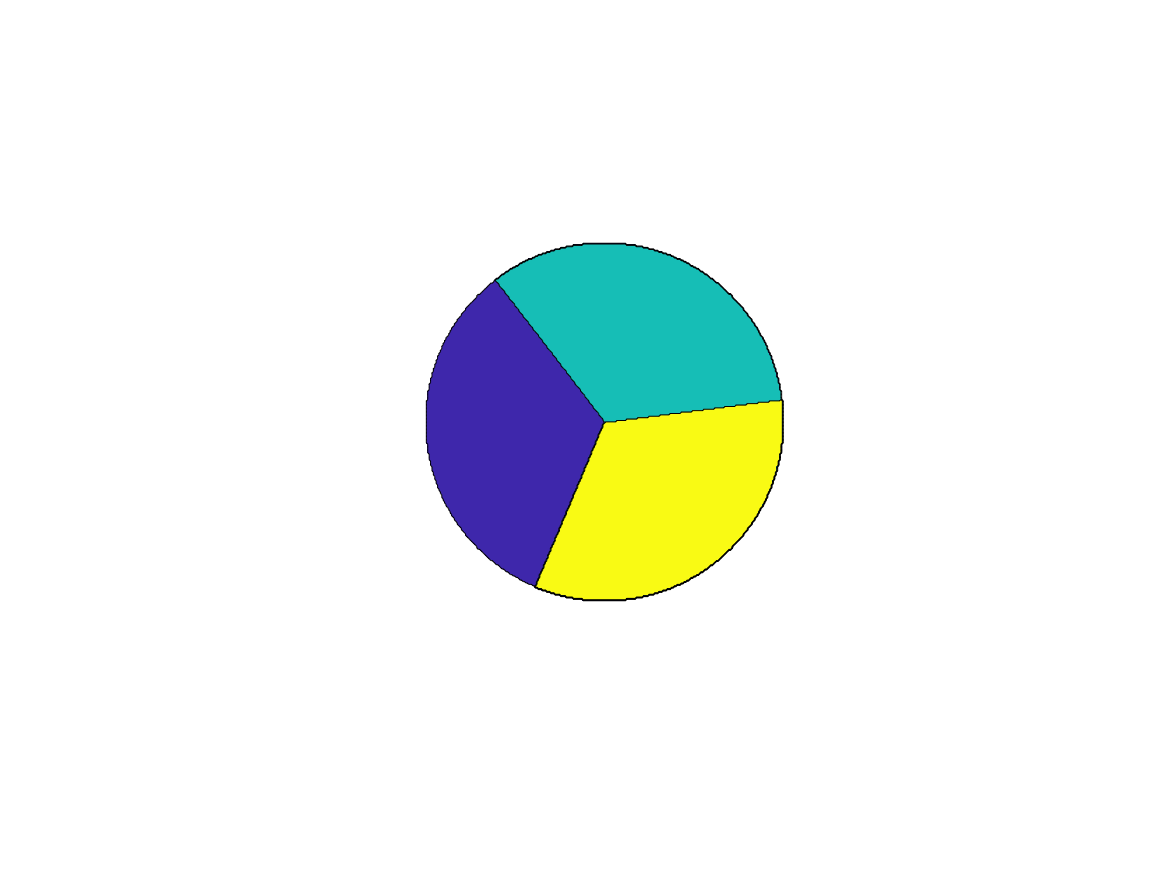}
		\includegraphics[width = 0.1\textwidth, clip, trim = 7cm 4.5cm 6cm 4cm]{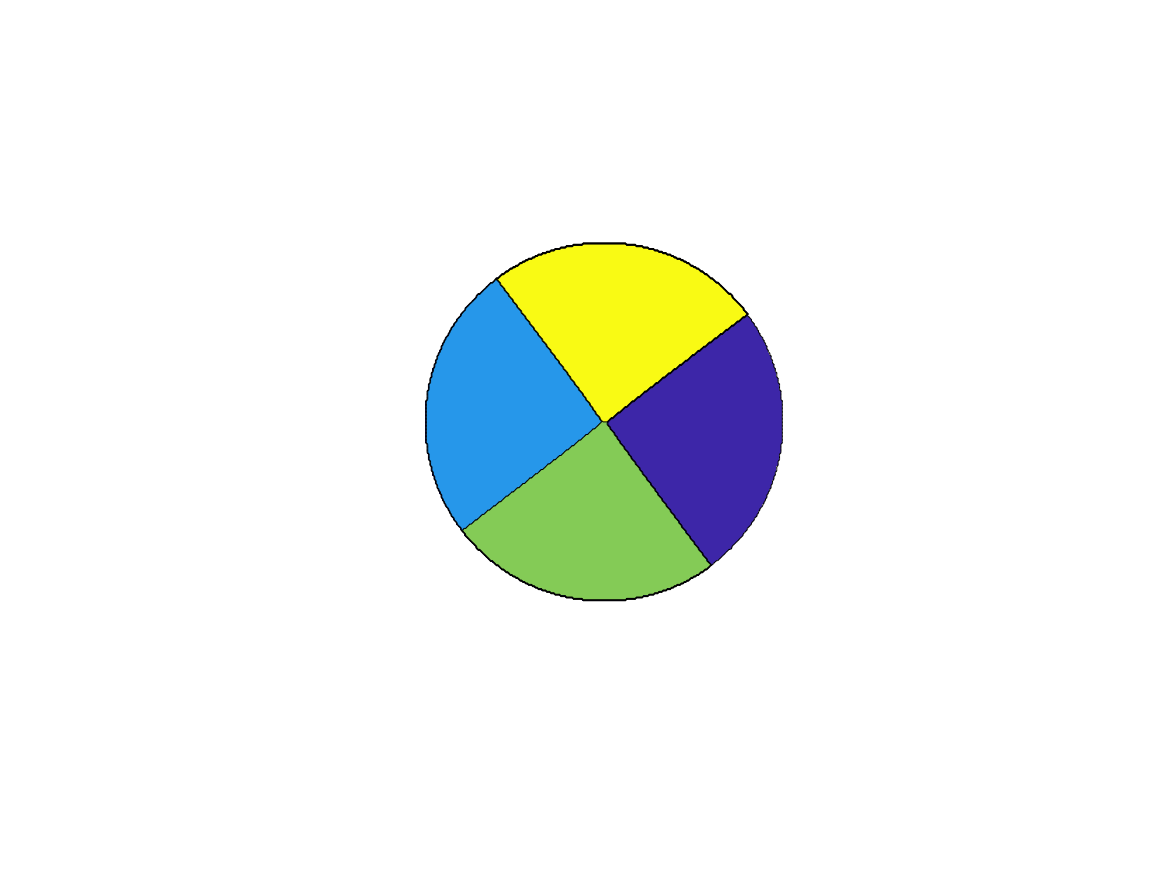}
		\includegraphics[width = 0.1\textwidth, clip, trim = 7cm 4.5cm 6cm 4cm]{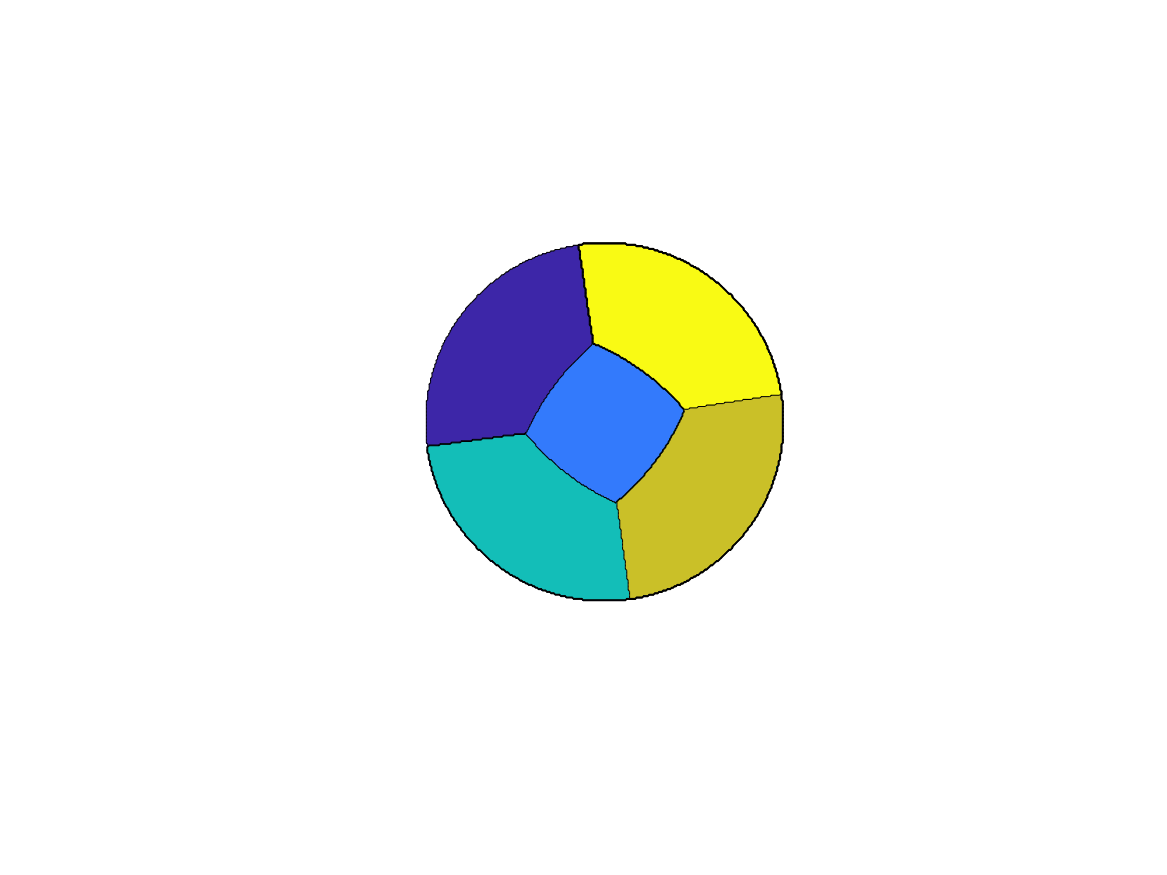}
		\includegraphics[width = 0.1\textwidth, clip, trim = 7cm 4.5cm 6cm 4cm]{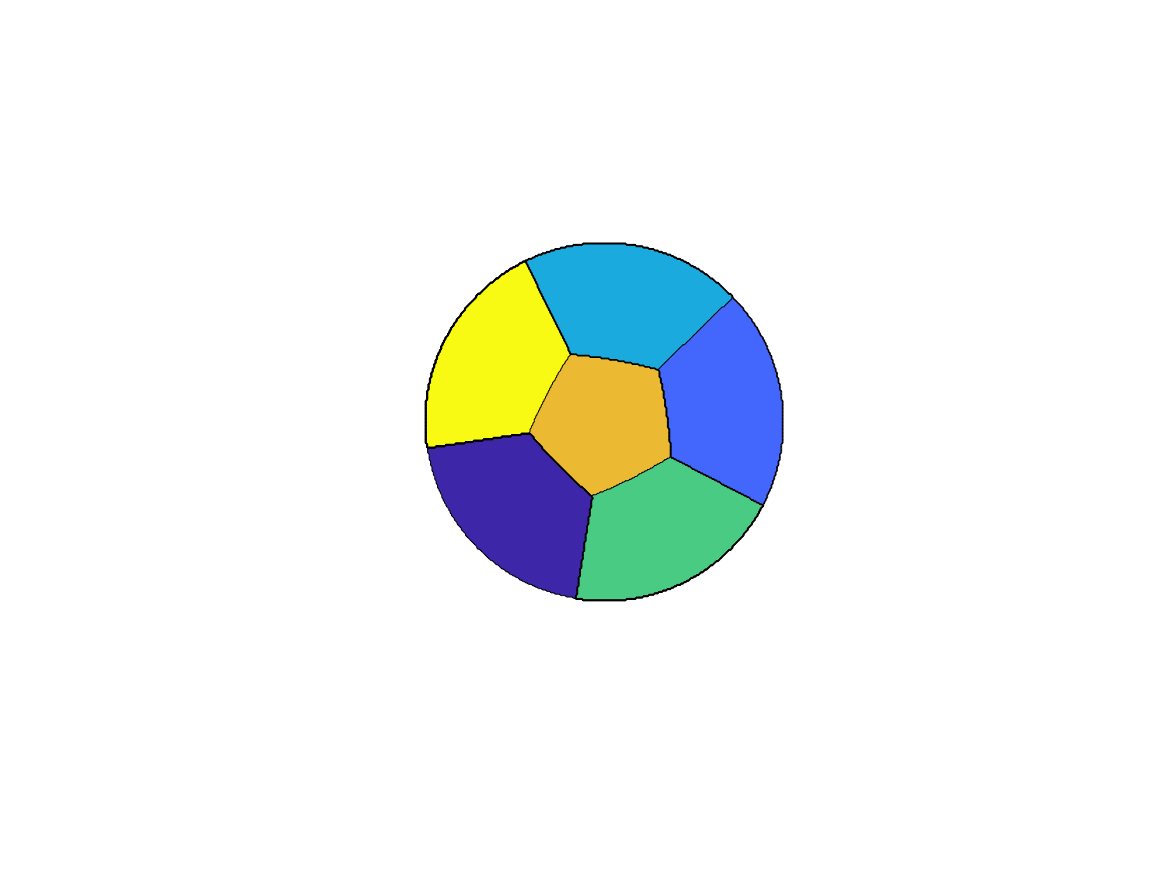}
		\includegraphics[width = 0.1\textwidth, clip, trim = 7cm 4.5cm 6cm 4cm]{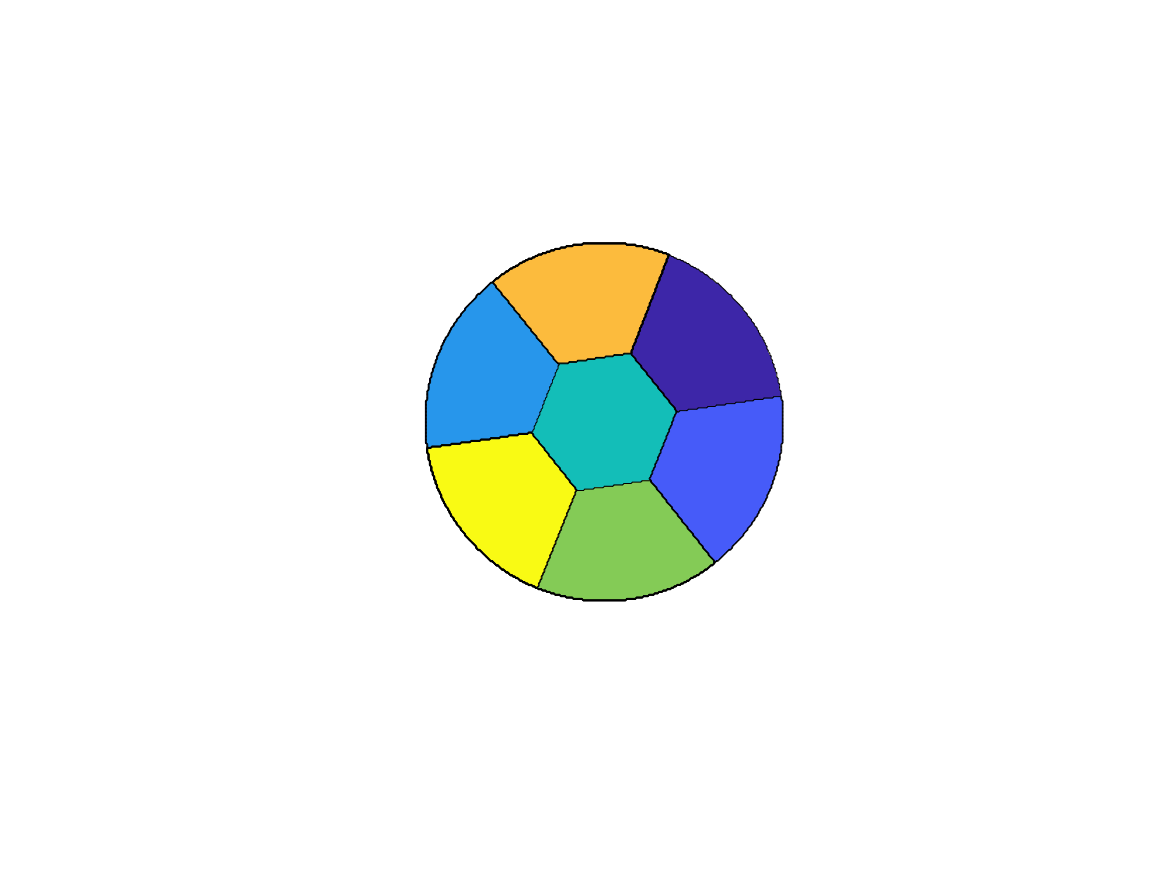}
		\includegraphics[width = 0.1\textwidth, clip, trim = 7cm 4.5cm 6cm 4cm]{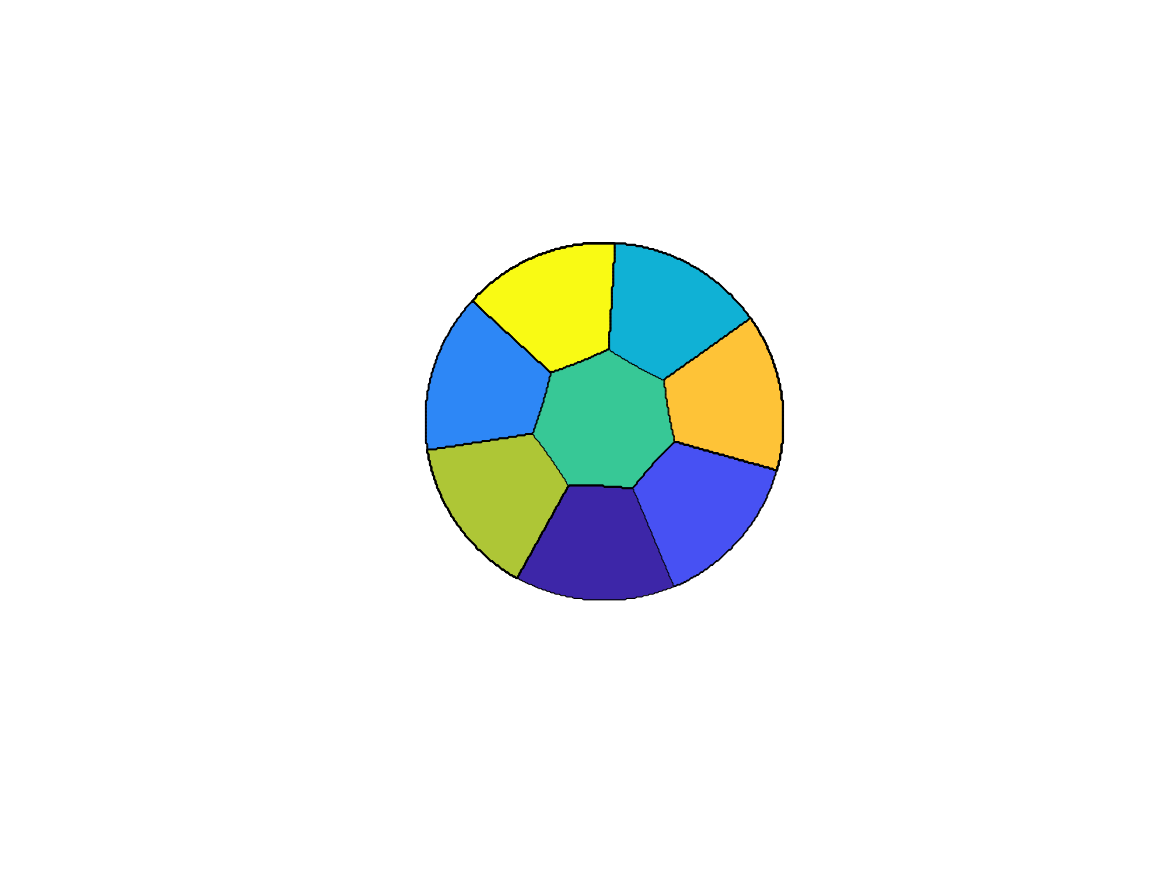}
		\includegraphics[width = 0.1\textwidth, clip, trim = 7cm 4.5cm 6cm 4cm]{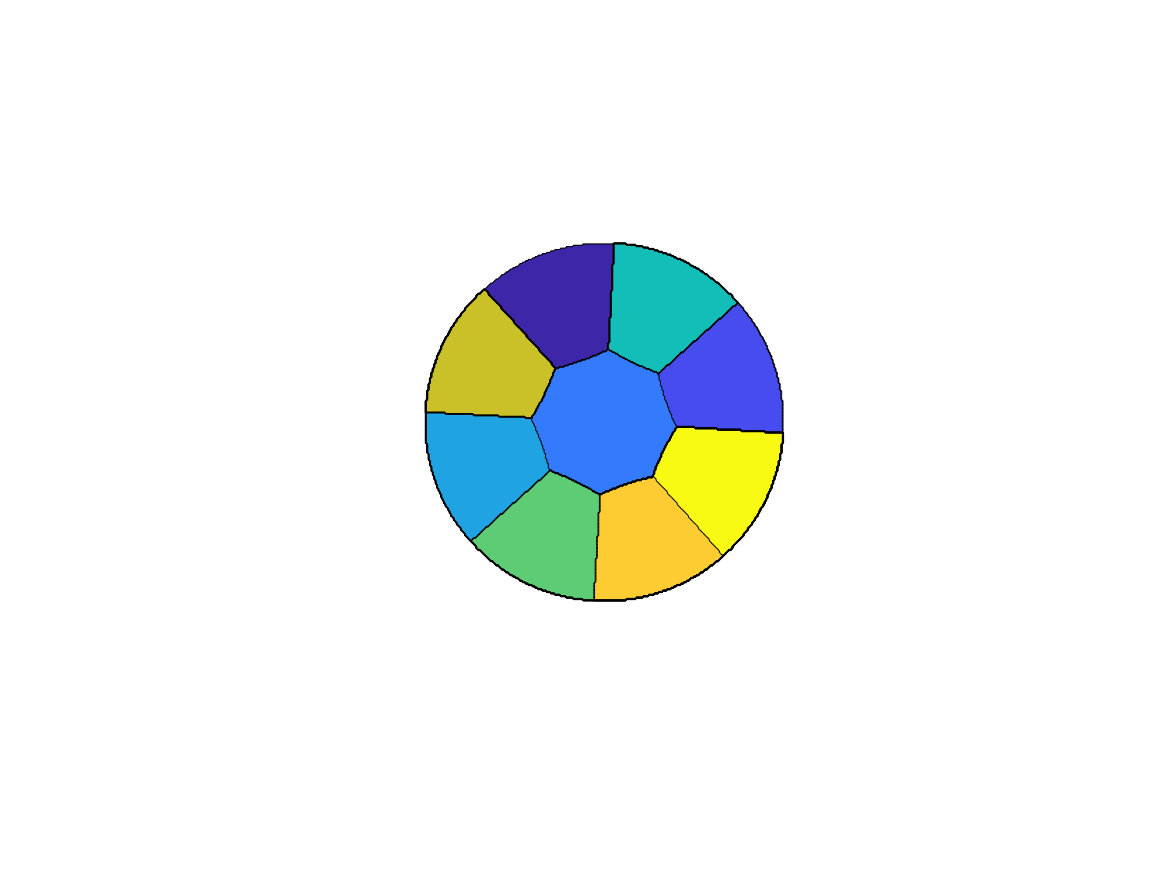}
		\includegraphics[width = 0.1\textwidth, clip, trim = 7cm 4.5cm 6cm 4cm]{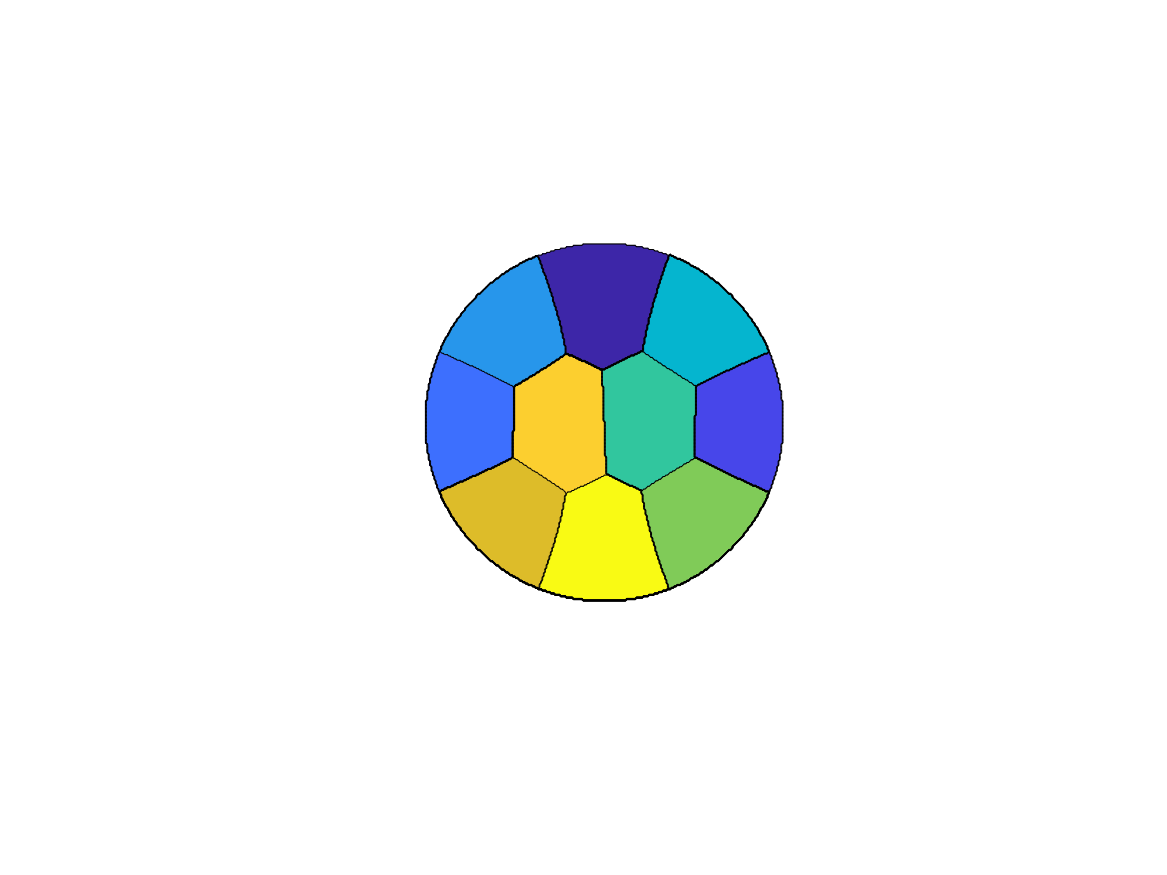}
		\smallskip
		\includegraphics[width = 0.1\textwidth, clip, trim = 6.5cm 4cm 5.5cm 3cm]{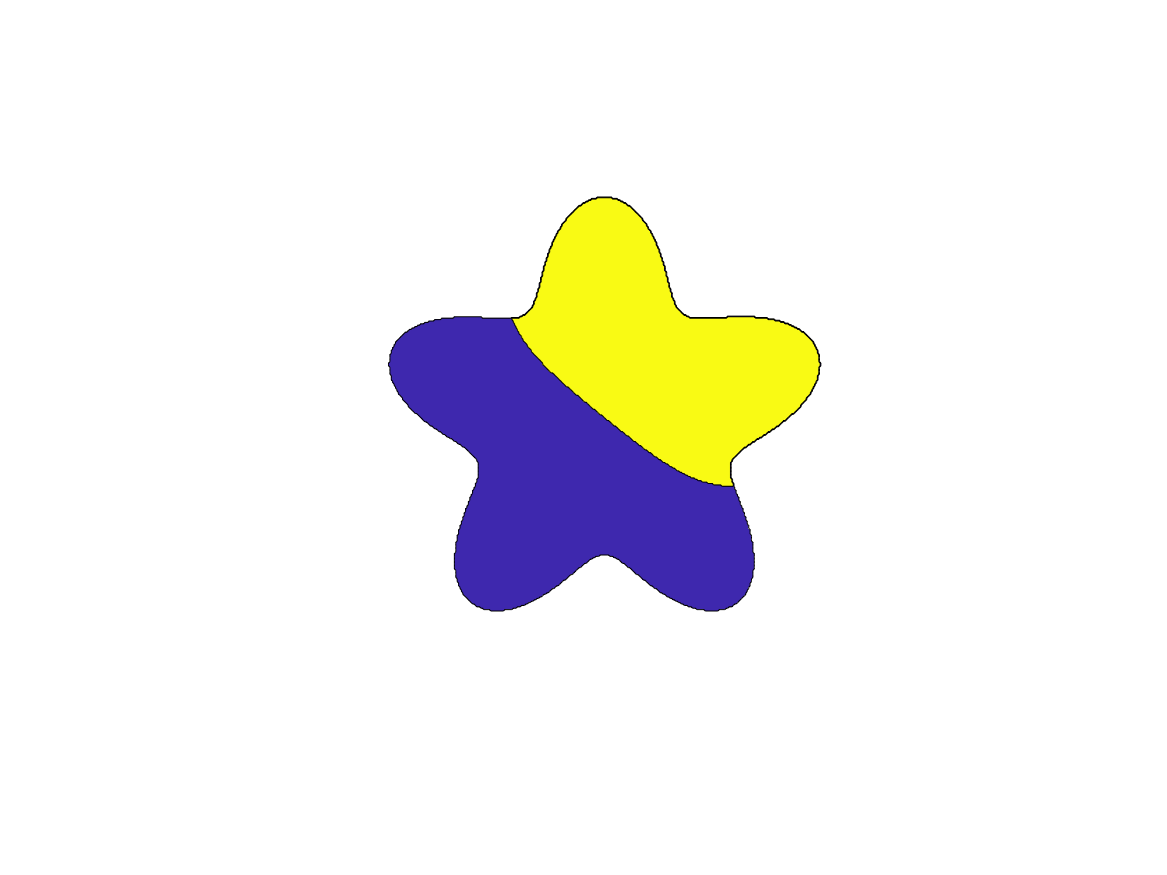}
		\includegraphics[width = 0.1\textwidth, clip, trim = 6.5cm 4cm 5.5cm 3cm]{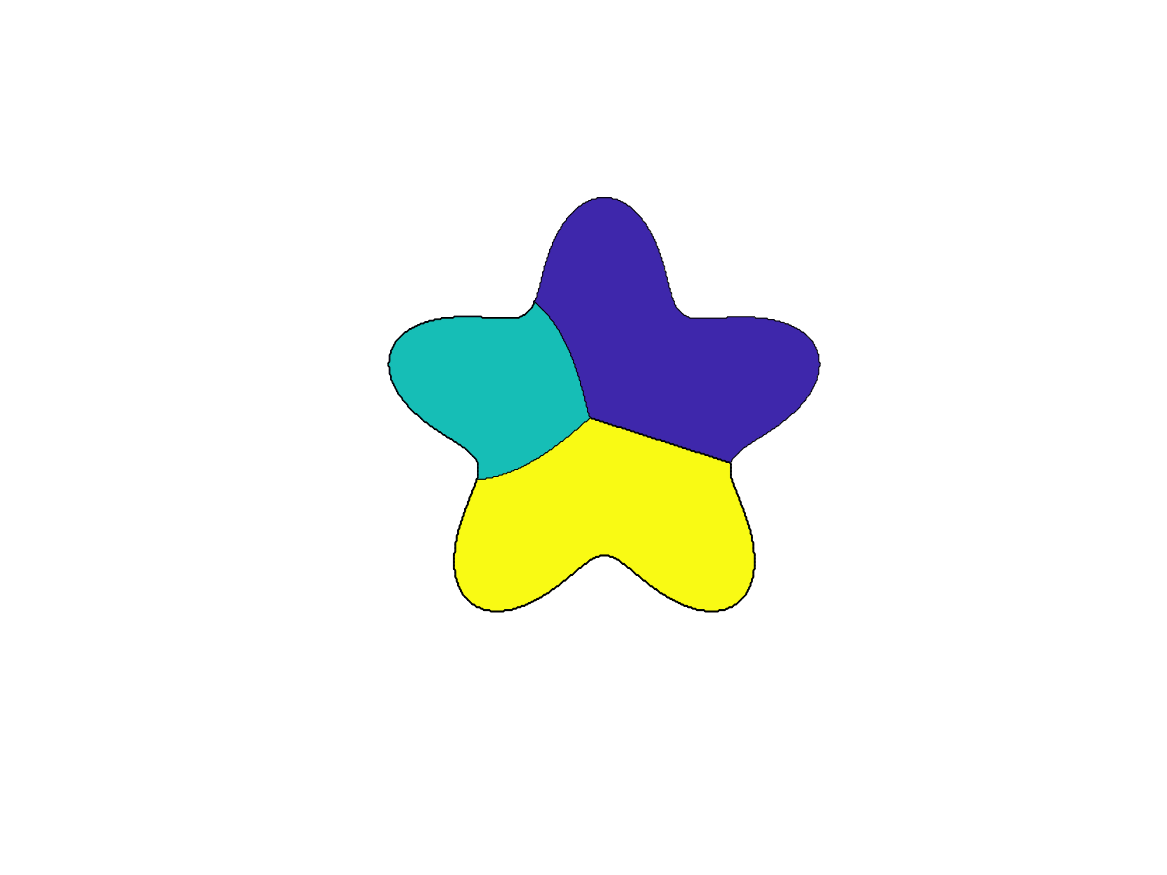}
		\includegraphics[width = 0.1\textwidth, clip, trim = 6.5cm 4cm 5.5cm 3cm]{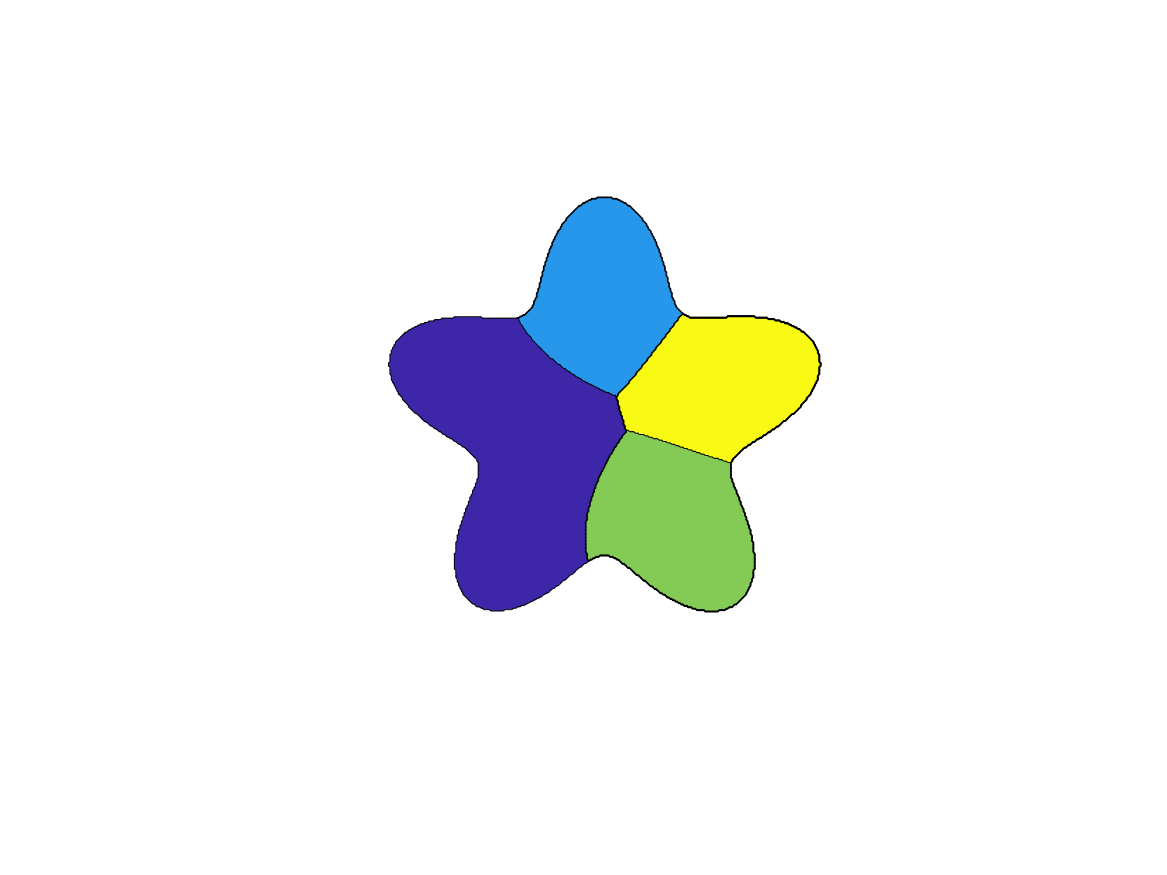}
		\includegraphics[width = 0.1\textwidth, clip, trim = 6.5cm 4cm 5.5cm 3cm]{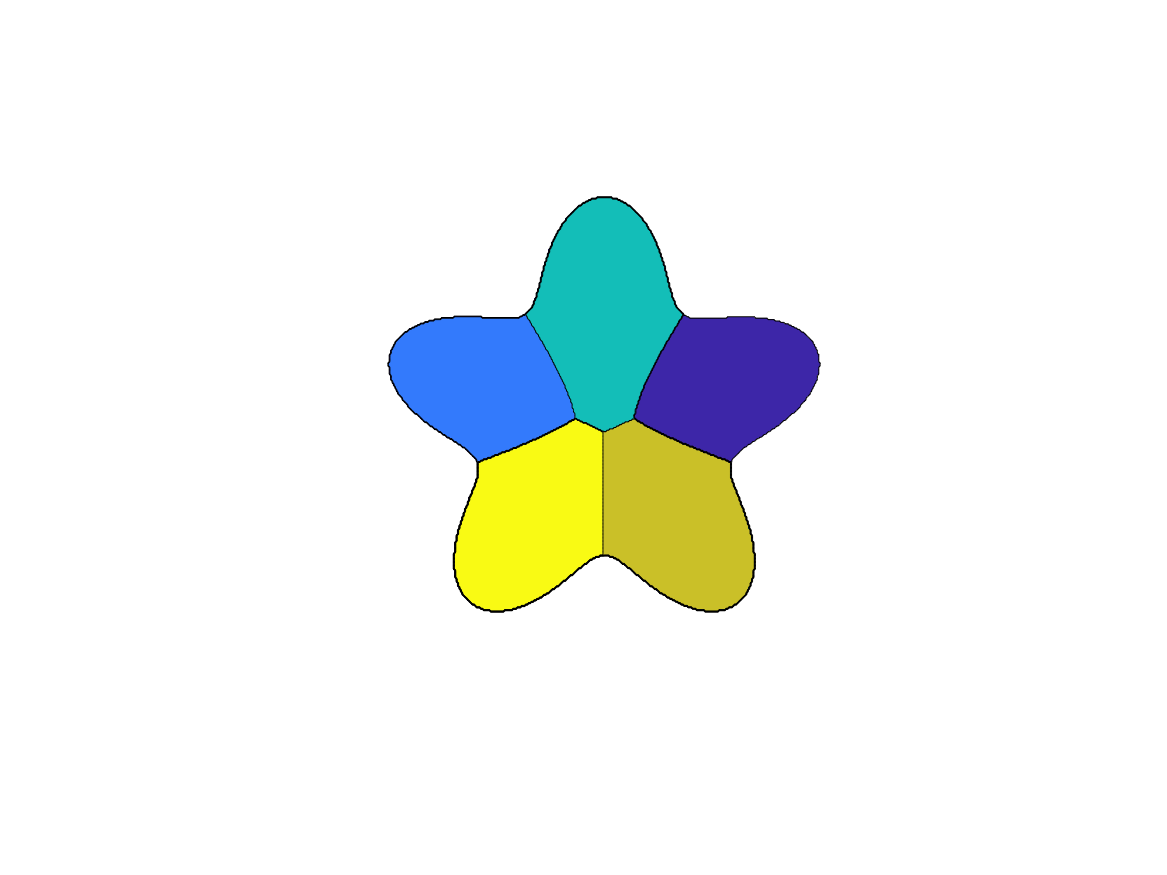}
		\includegraphics[width = 0.1\textwidth, clip, trim = 6.5cm 4cm 5.5cm 3cm]{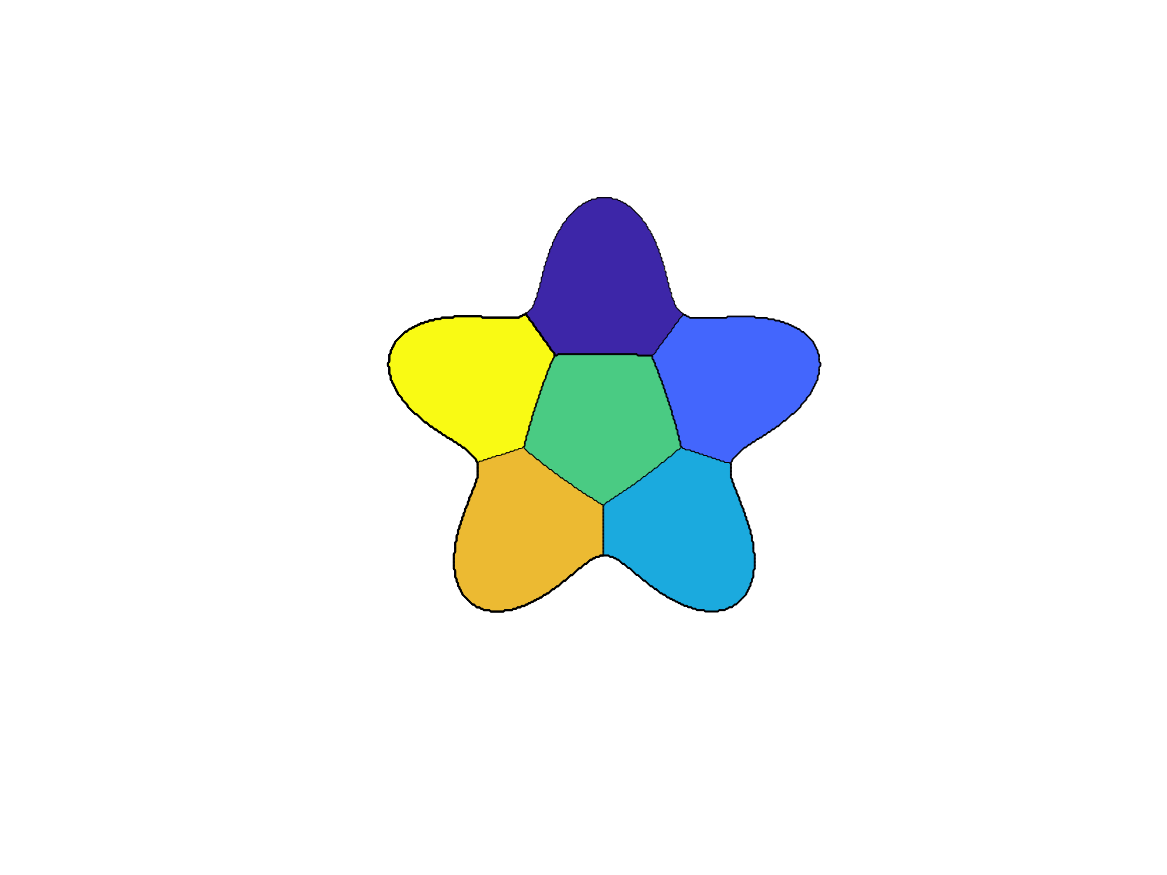}
		\includegraphics[width = 0.1\textwidth, clip, trim = 6.5cm 4cm 5.5cm 3cm]{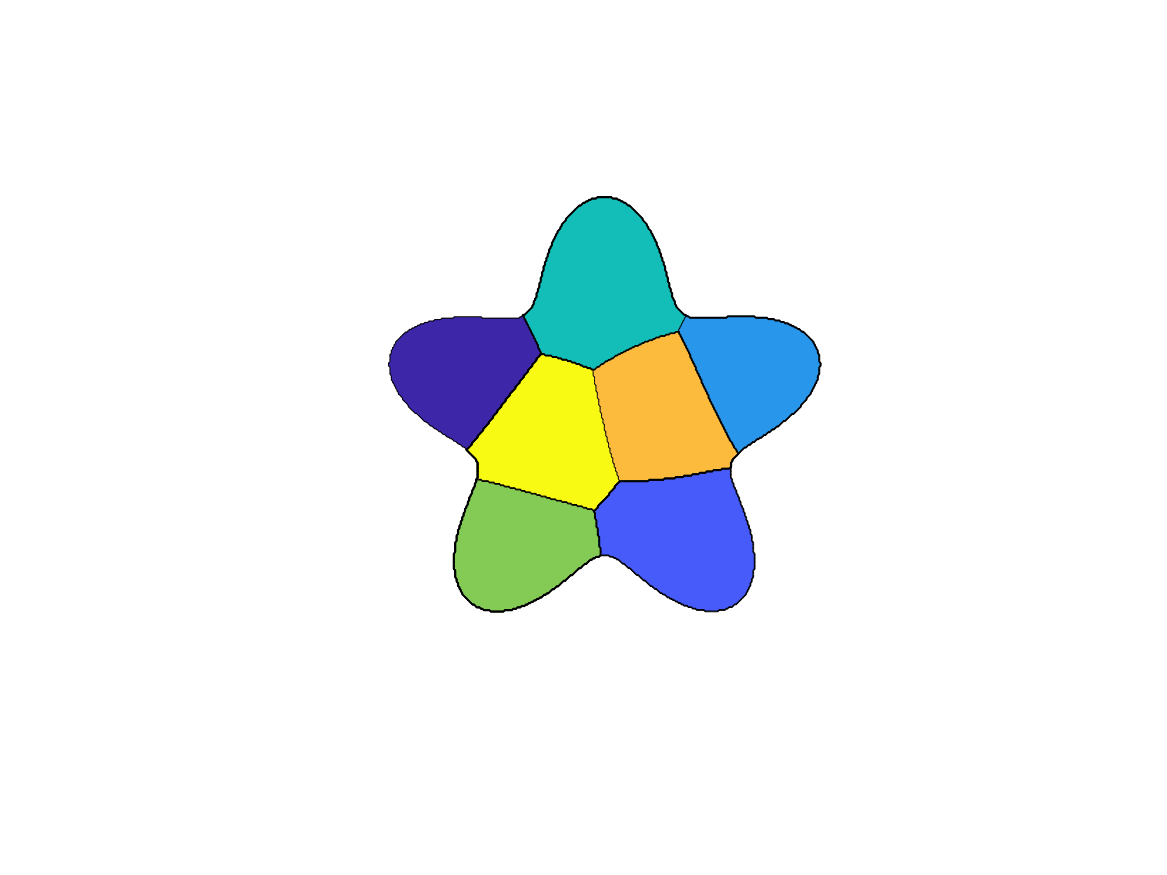}
		\includegraphics[width = 0.1\textwidth, clip, trim = 6.5cm 4cm 5.5cm 3cm]{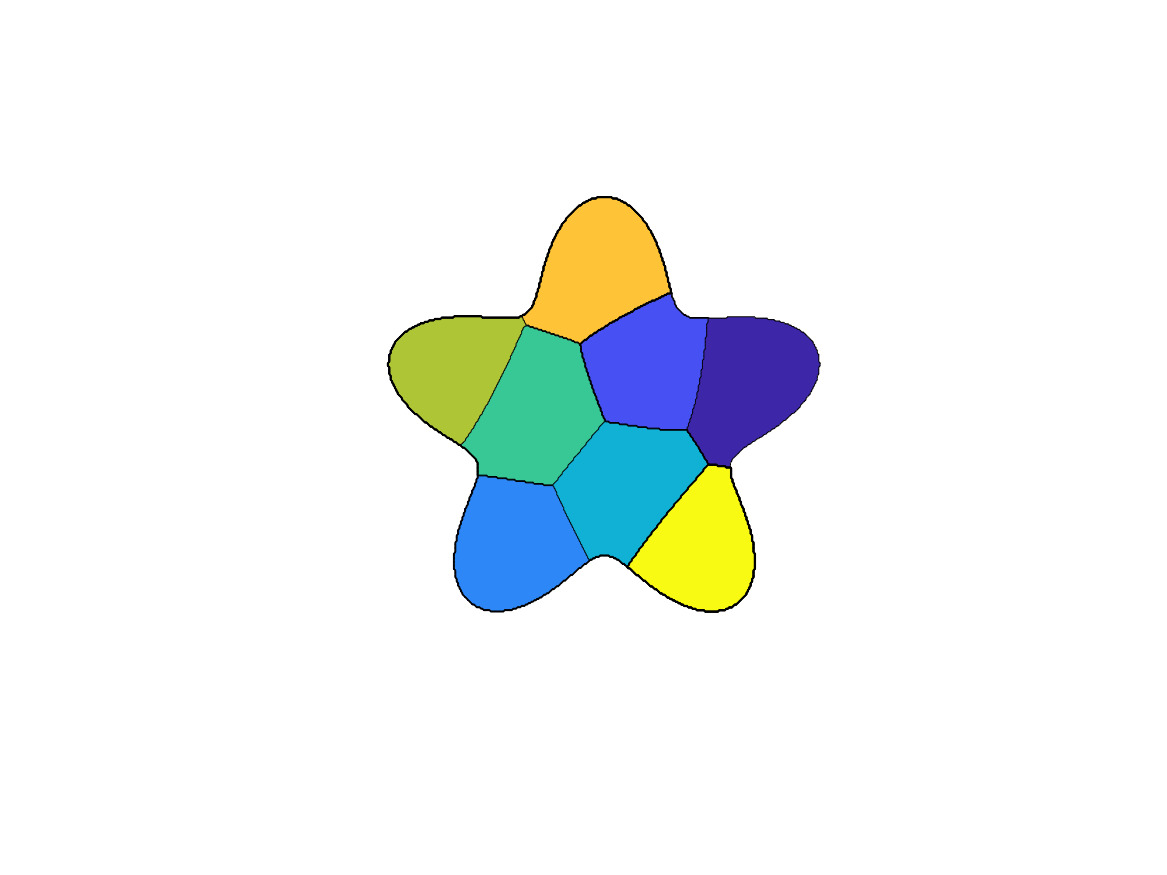}
		\includegraphics[width = 0.1\textwidth, clip, trim = 6.5cm 4cm 5.5cm 3cm]{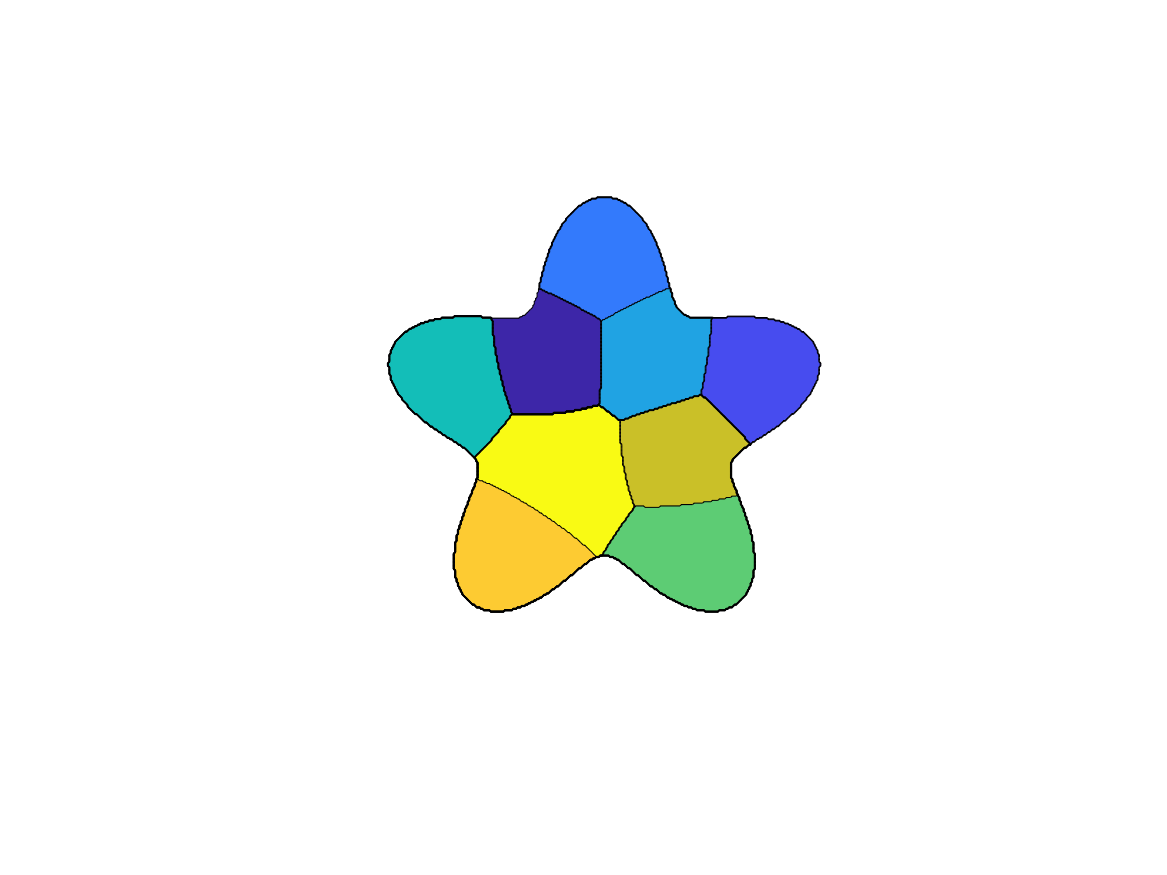}
		\includegraphics[width = 0.1\textwidth, clip, trim = 6.5cm 4cm 5.5cm 3cm]{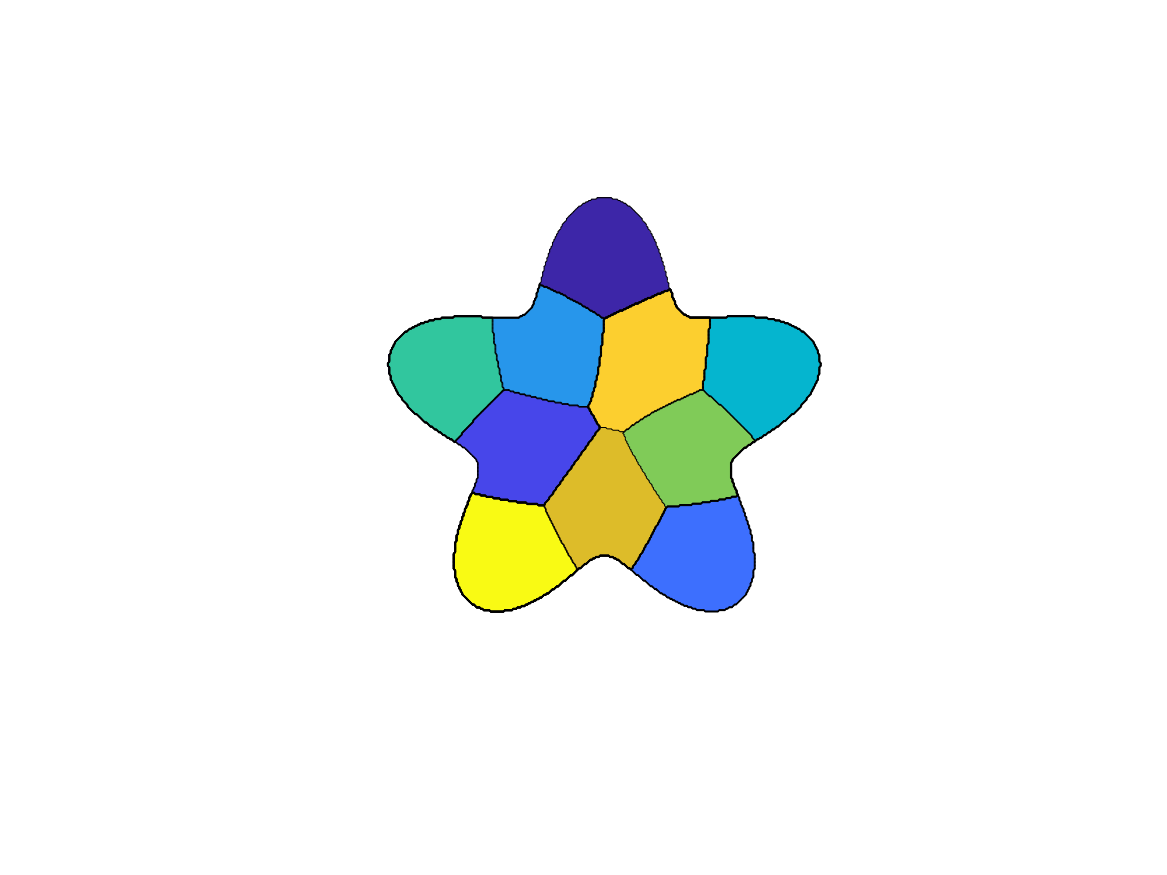}
		\smallskip
		\includegraphics[width = 0.1\textwidth, clip, trim =3cm 1.3cm 3cm 1.3cm]{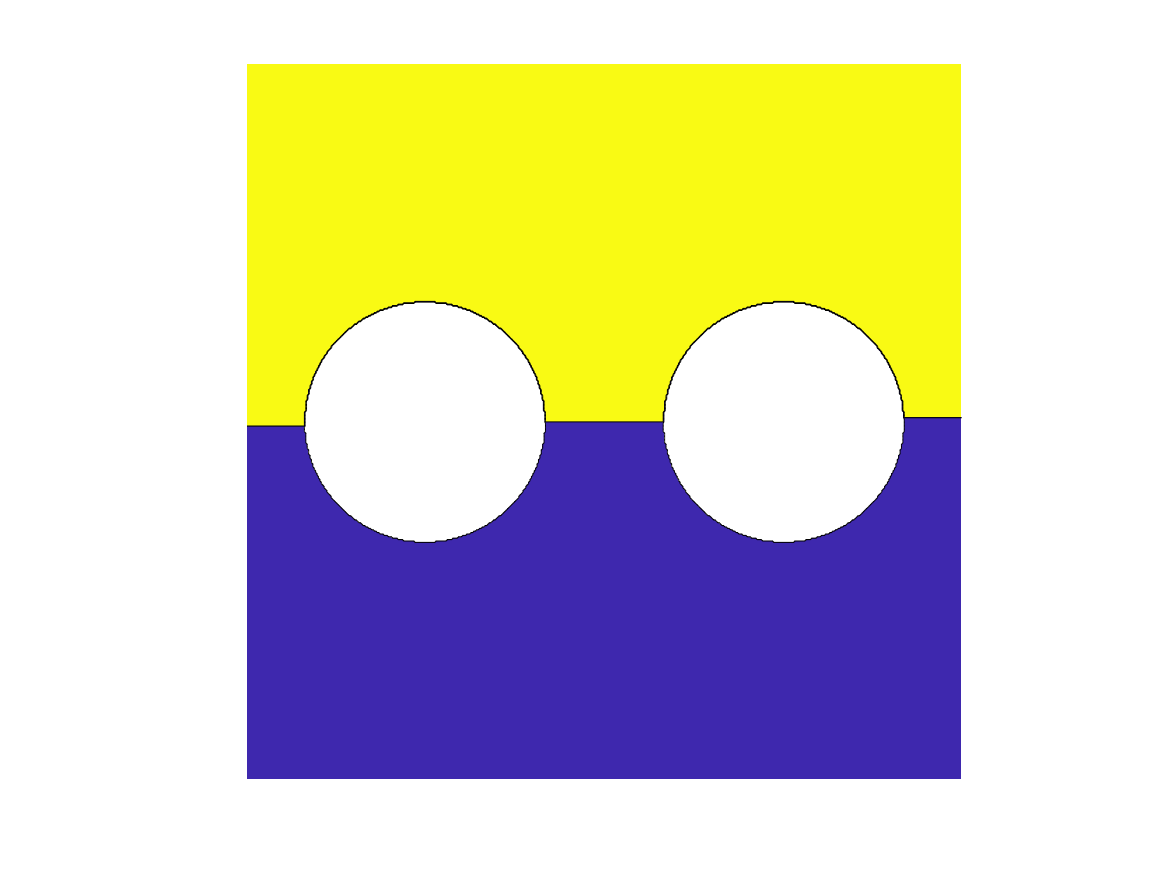}
		\includegraphics[width = 0.1\textwidth, clip, trim =3cm 1.3cm 3cm 1.3cm]{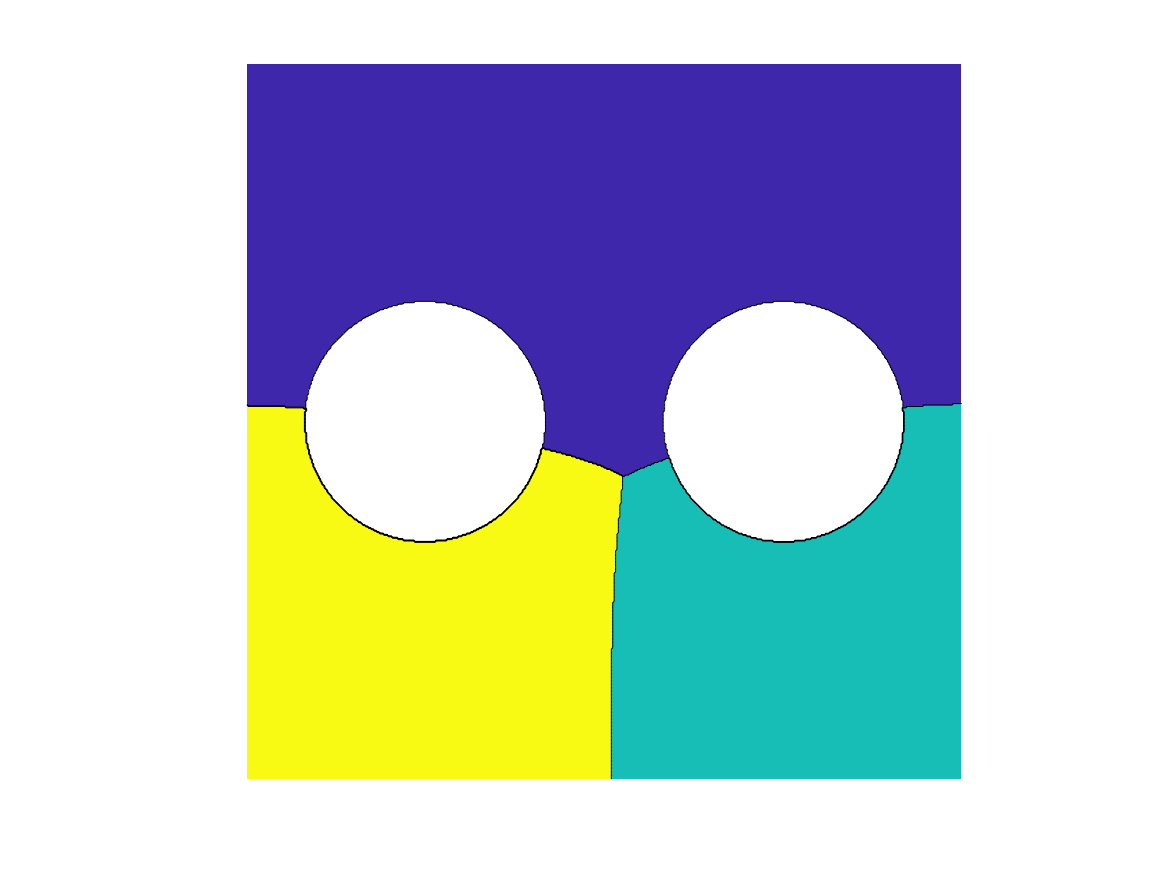}
		\includegraphics[width = 0.1\textwidth, clip, trim =3cm 1.3cm 3cm 1.3cm]{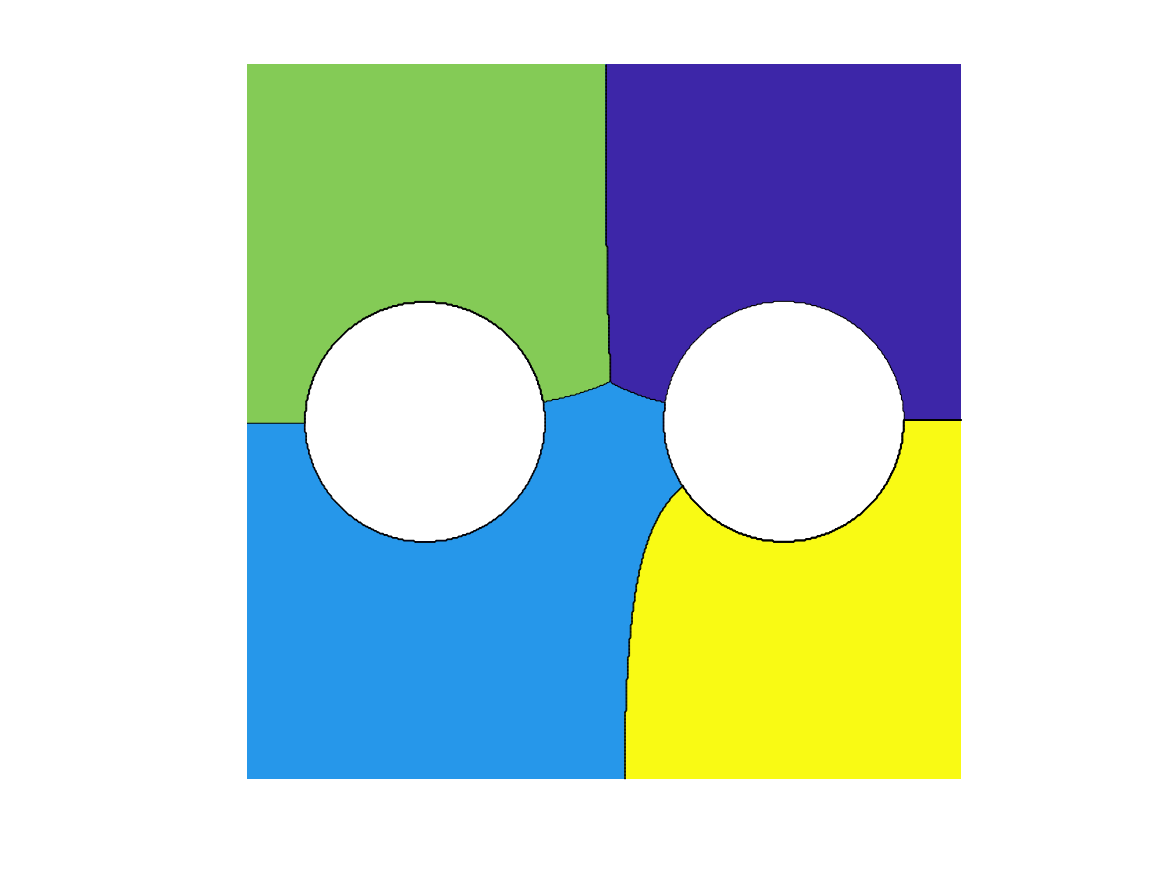}
		\includegraphics[width = 0.1\textwidth, clip, trim =3cm 1.3cm 3cm 1.3cm]{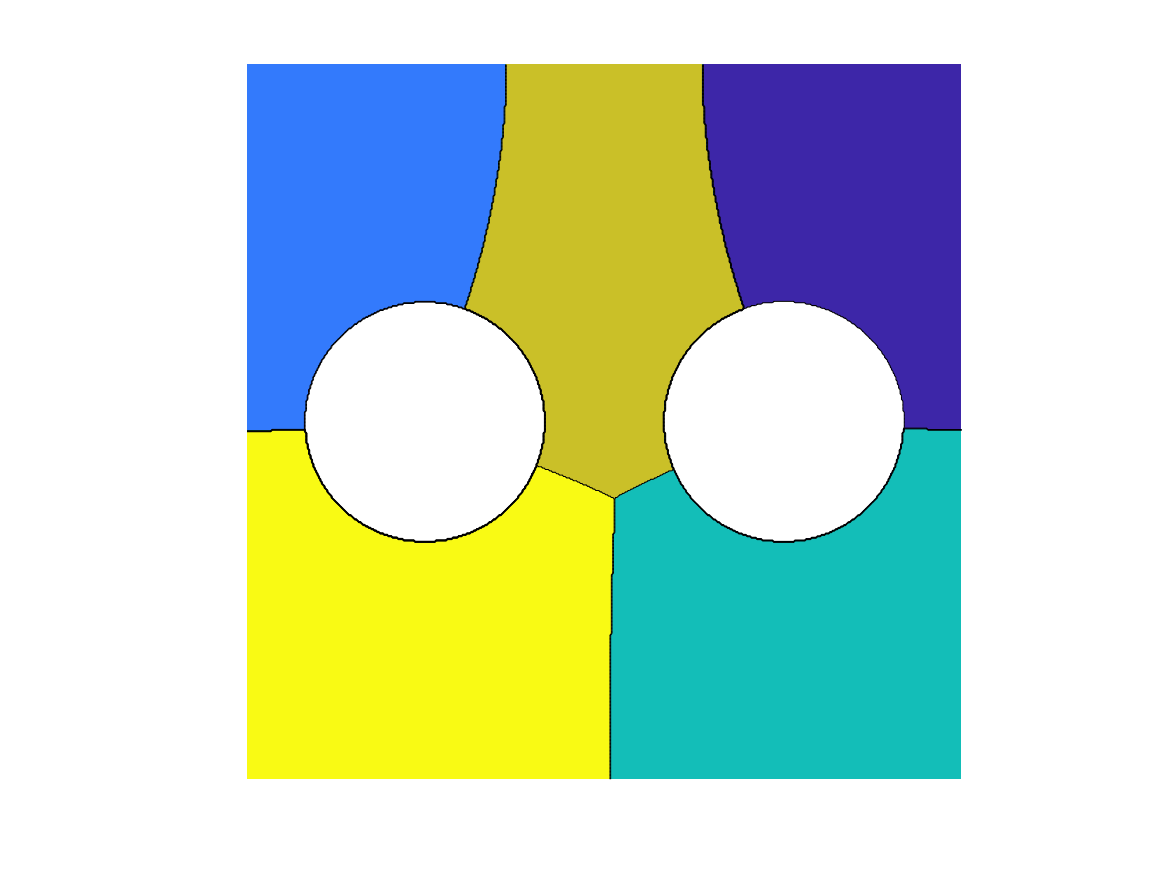}
		\includegraphics[width = 0.1\textwidth, clip, trim =3cm 1.3cm 3cm 1.3cm]{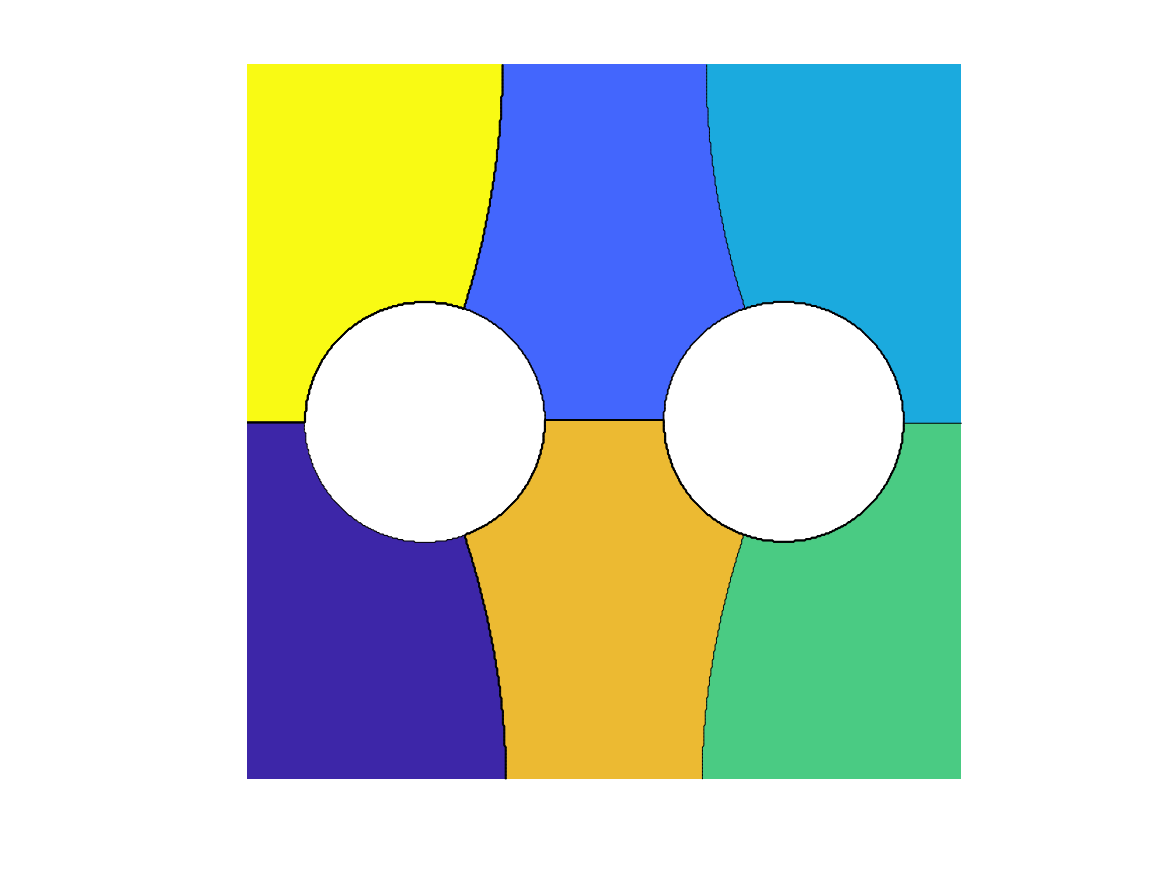}
		\includegraphics[width = 0.1\textwidth, clip, trim =3cm 1.3cm 3cm 1.3cm]{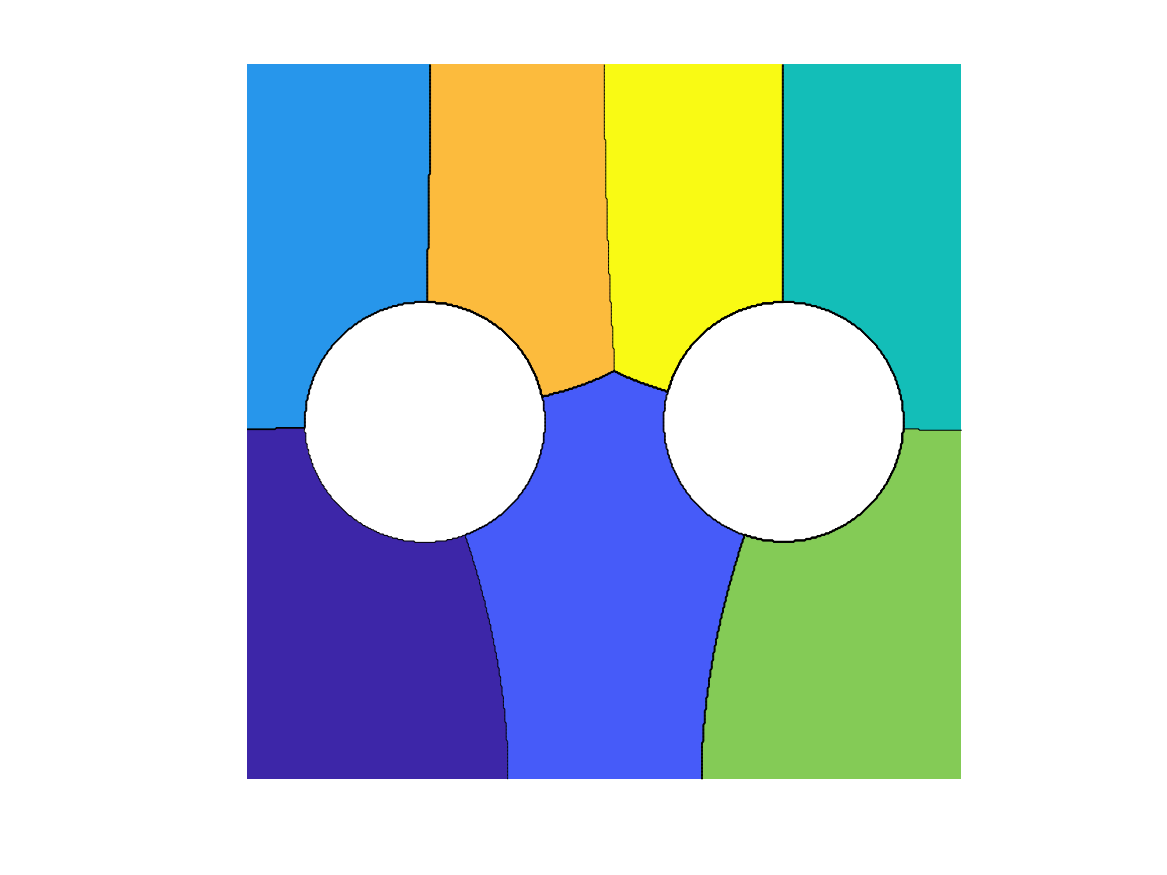}
		\includegraphics[width = 0.1\textwidth, clip, trim =3cm 1.3cm 3cm 1.3cm]{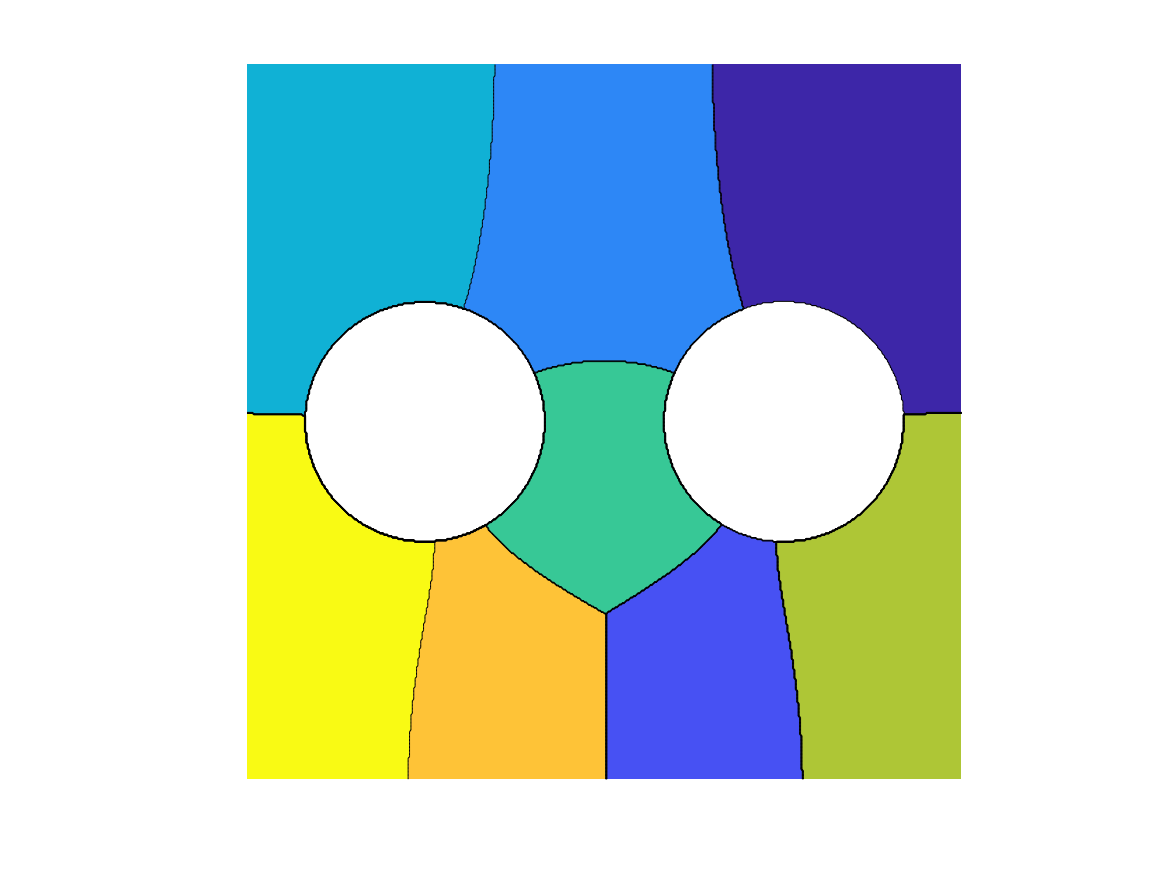}
		\includegraphics[width = 0.1\textwidth, clip, trim =3cm 1.3cm 3cm 1.3cm]{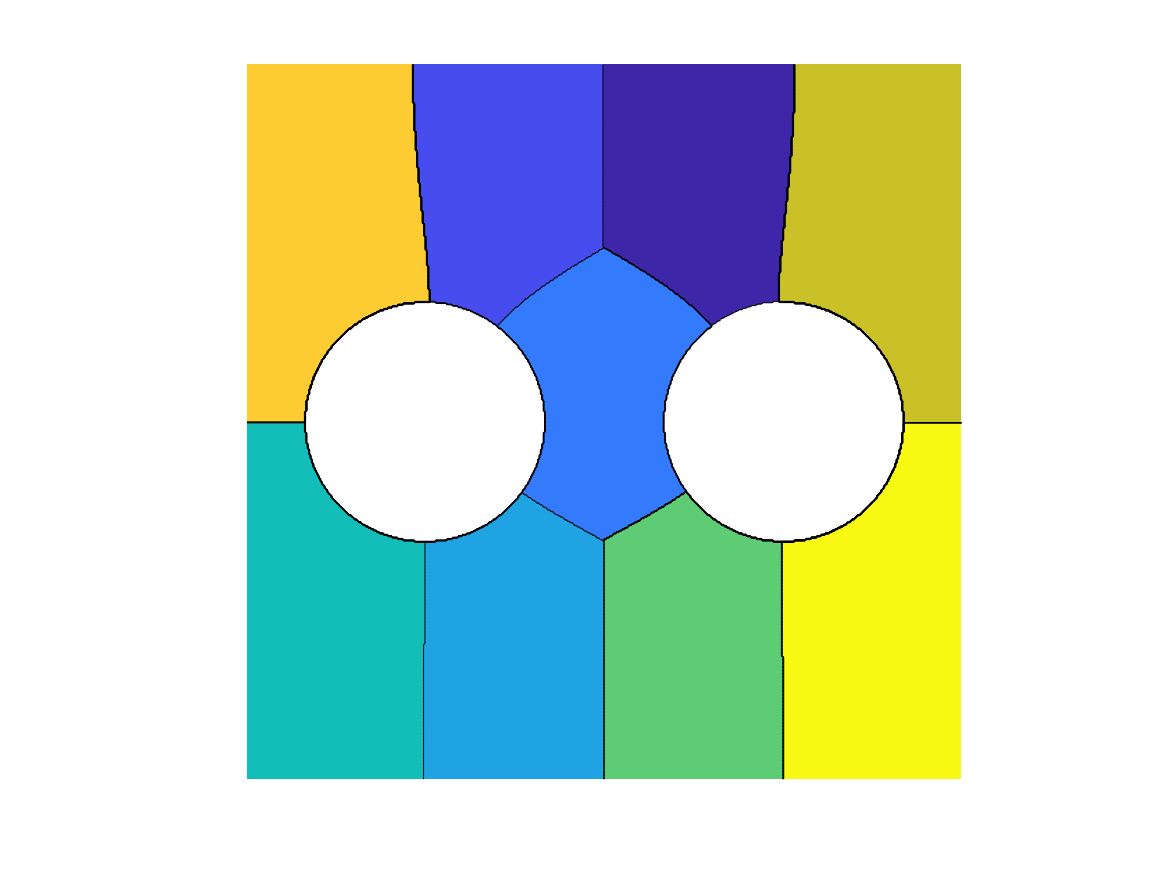}
		\includegraphics[width = 0.1\textwidth, clip, trim =3cm 1.3cm 3cm 1.3cm]{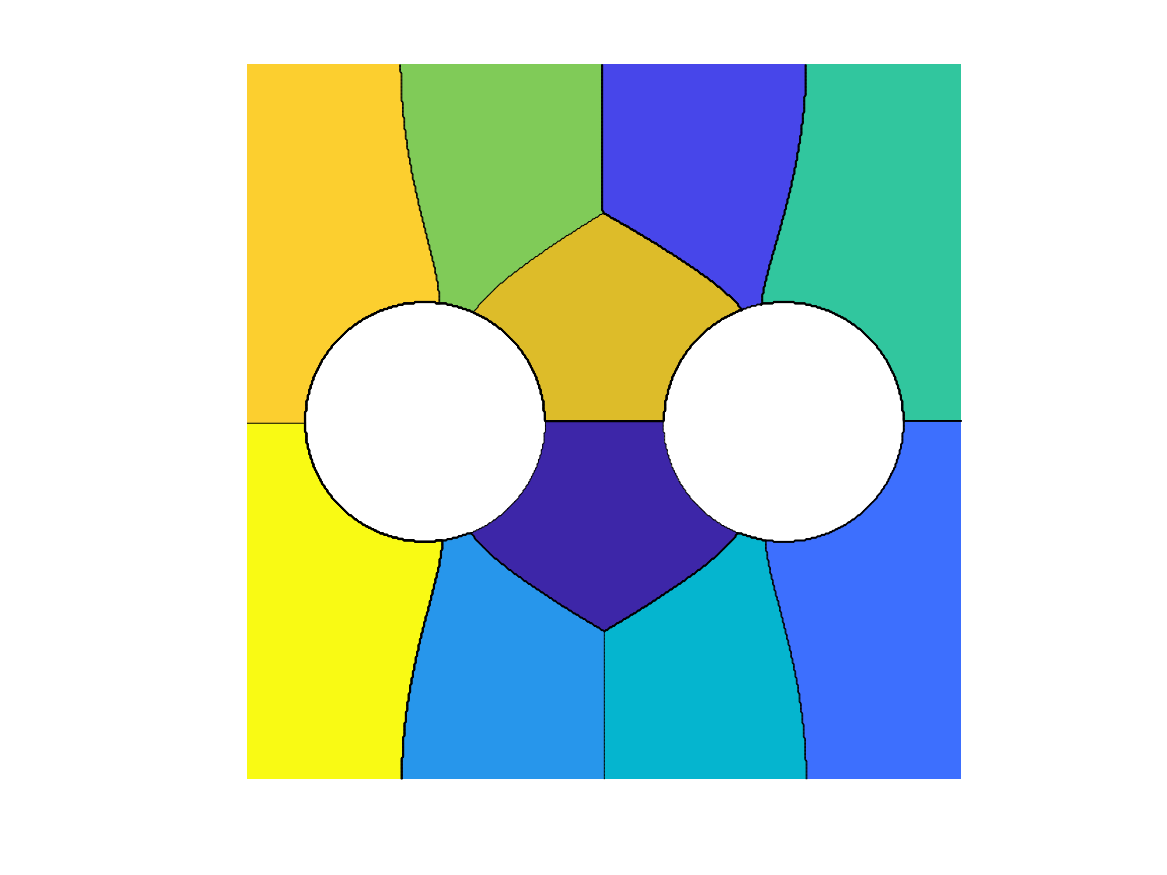}
		\medskip
		\caption{Computed partition by Algorithm~\ref{Alg:4step} from 2 to 10 in different domains: Regular triangle, Regular pentagon, Disk, Five-fold star and Square with two holes. }\label{fig:5AbiDAlg1Alg1}
	\end{figure}
	\begin{figure}[ht]
		\centering
\begin{table}[H]
    \centering
    \renewcommand{\arraystretch}{1.5}
    \begin{tabular}{|m{2cm}|m{2.2cm}|m{2.2cm}|m{2cm}|m{2cm}|}
        \hline
        \diagbox[innerwidth=2cm]{Algorithm}{Domain} &  \centering Three-fold star & \centering  Regular octagon & \centering  Sector &  {\centering \quad \, Ellipse} \\
        \hline
				\centering Initial condition& 
				\centering  \includegraphics[width=0.12\textwidth, clip, trim=6.5cm 3.5cm 5.5cm 3cm]{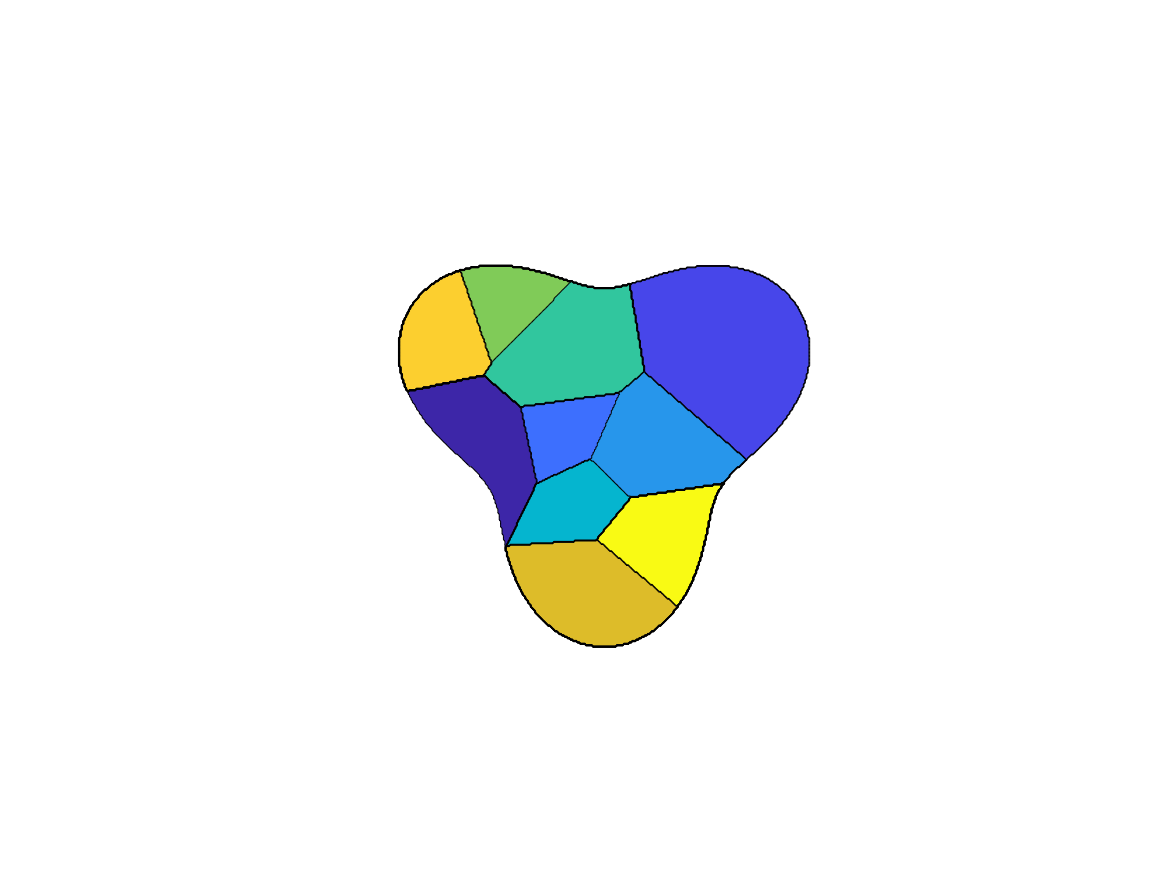} &
				\centering  \includegraphics[width=0.12\textwidth, clip, trim=6cm 4.5cm 6.5cm 3cm]{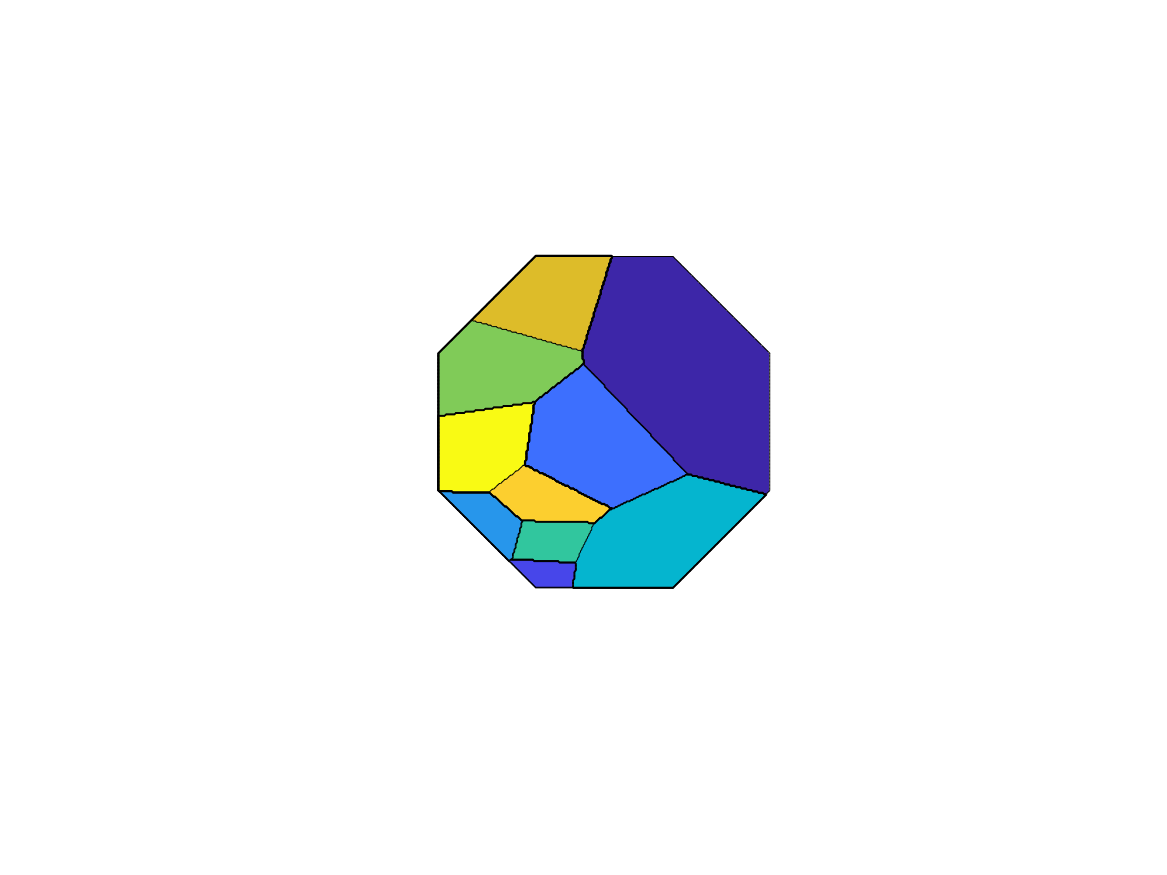} &
				\centering \includegraphics[width=0.12\textwidth, clip, trim=3.5cm 1cm 5.5cm 5cm]{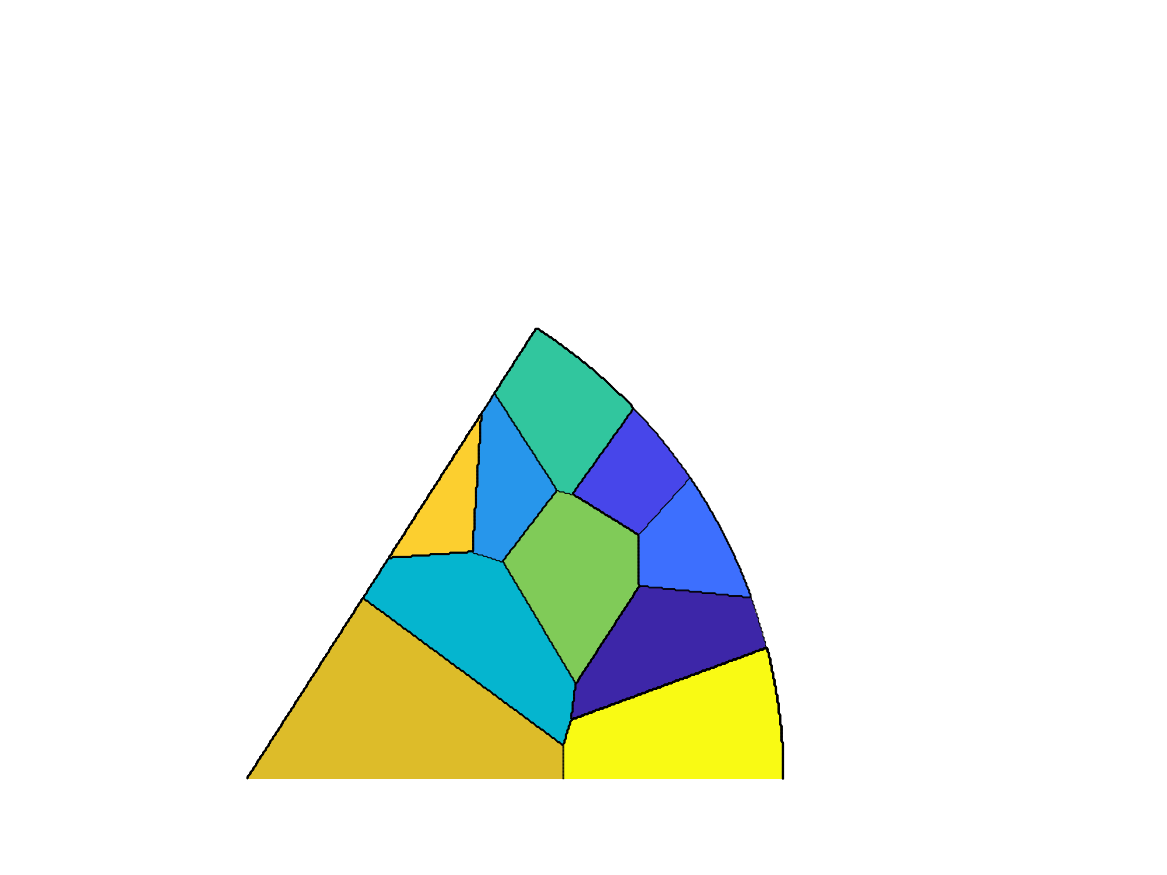} &
				{\centering  \includegraphics[width=0.12\textwidth, clip, trim=4cm 4.5cm 4cm 4cm]{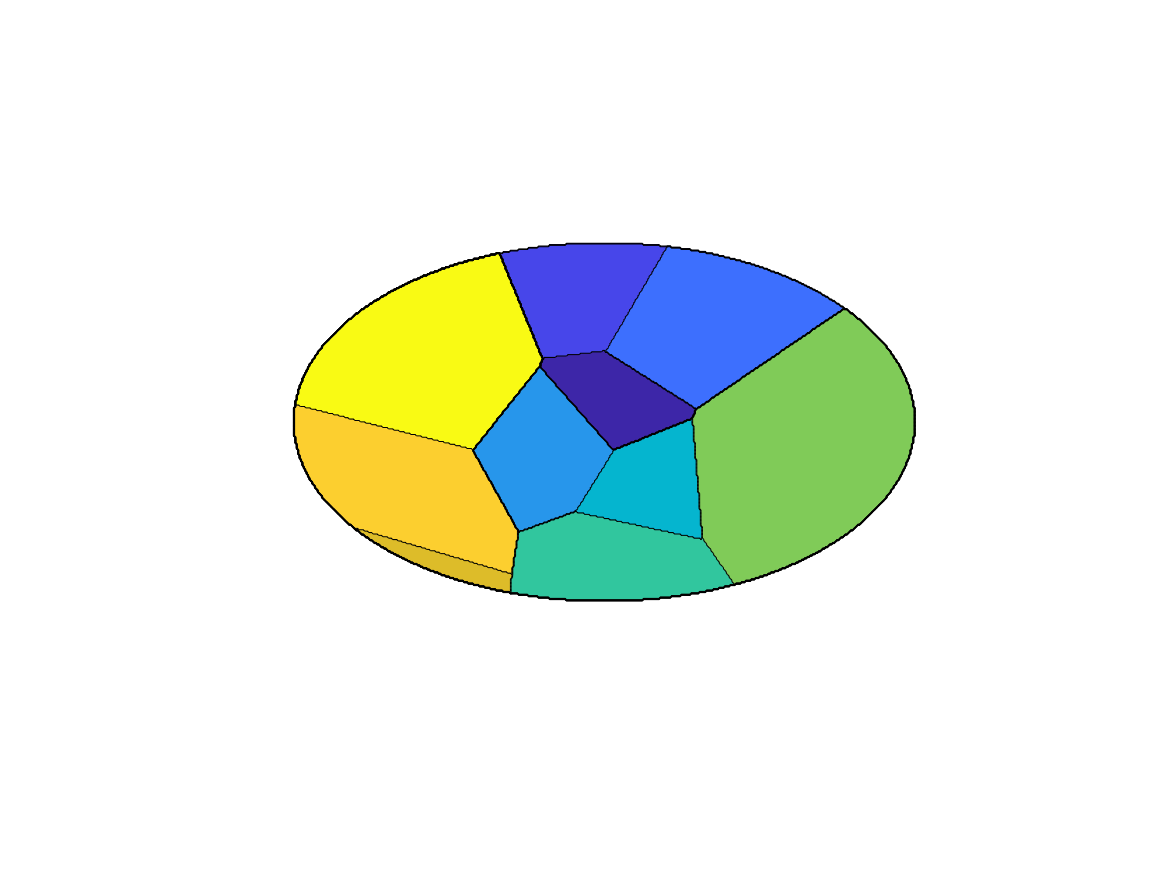}} \\
				\hline
				\centering   Algorithm~\ref{Alg:4step} & 
				\centering  \includegraphics[width=0.12\textwidth, clip, trim=6.5cm 3.5cm 5.5cm 3cm]{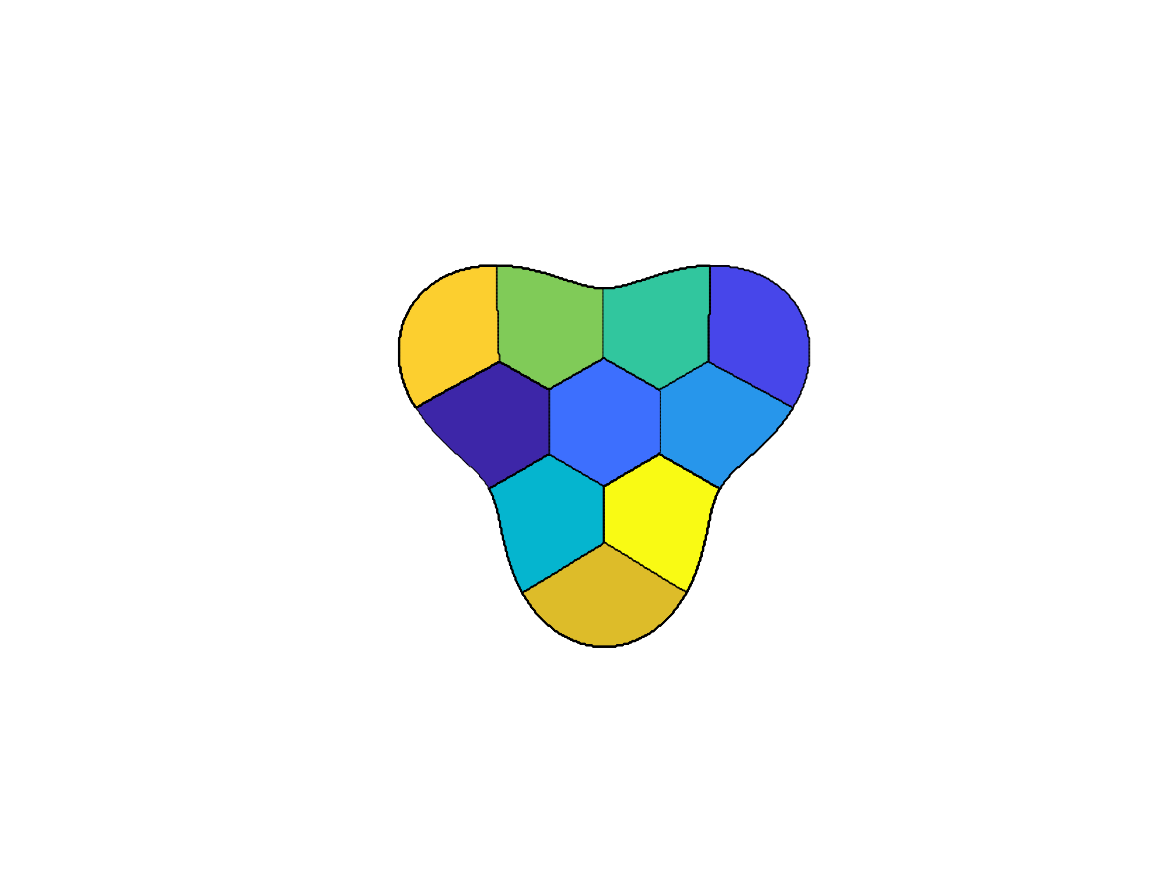} &
				\centering \includegraphics[width=0.12\textwidth, clip, trim=6cm 4.5cm 6.5cm 3cm]{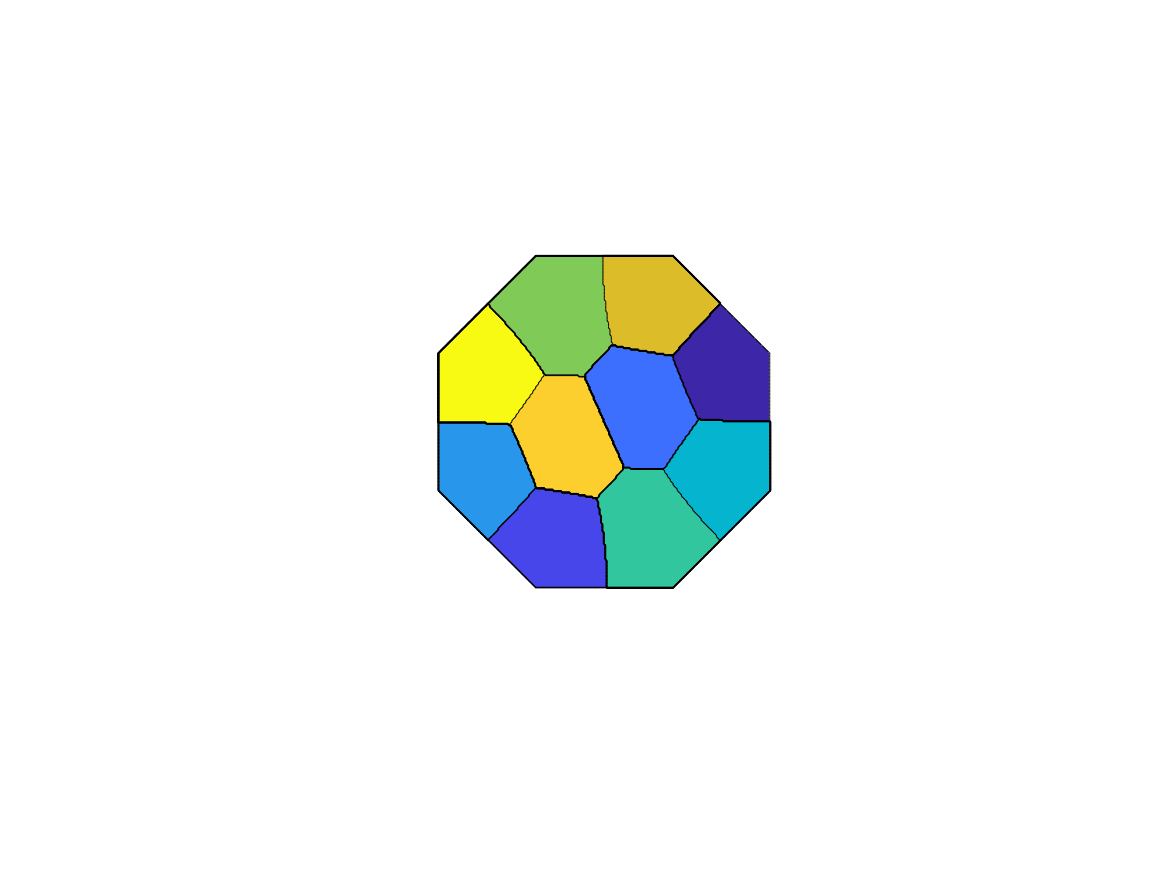} &
				\centering \includegraphics[width=0.12\textwidth, clip, trim=3.5cm 1cm 5.5cm 5cm]{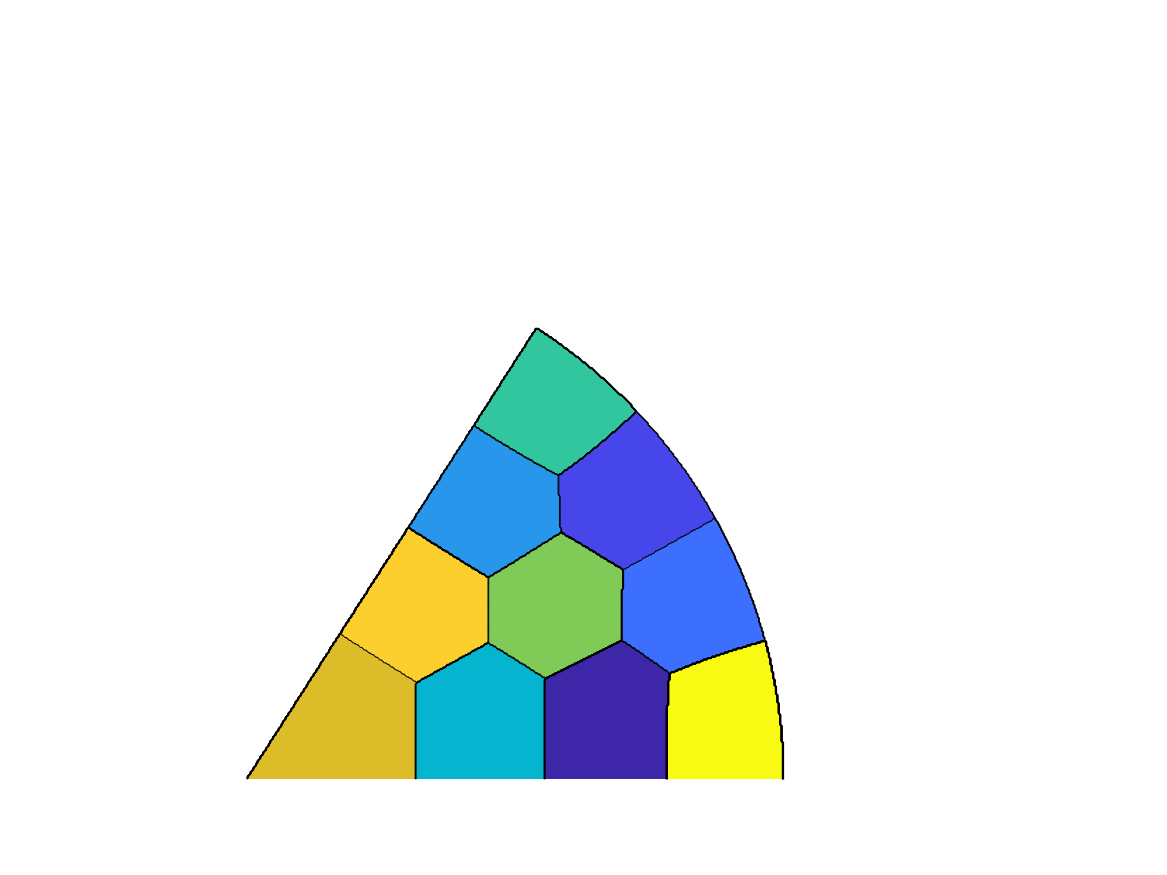} &
				\includegraphics[width=0.12\textwidth, clip, trim=4cm 4.5cm 4cm 4cm]{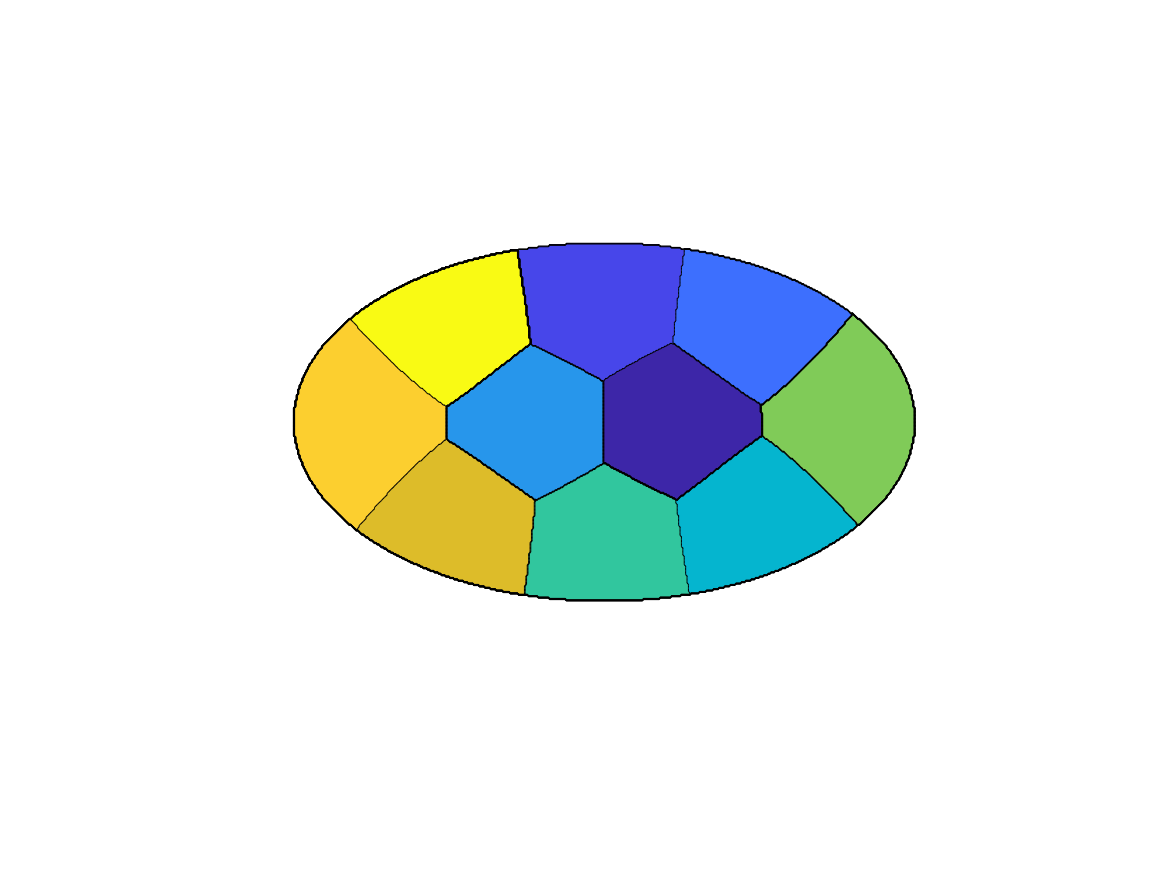} \\
				\hline
				\centering  Algorithm~\ref{Alg:3step_Type1} & 
				\centering  \includegraphics[width=0.12\textwidth, clip, trim=6.5cm 3.5cm 5.5cm 3cm]{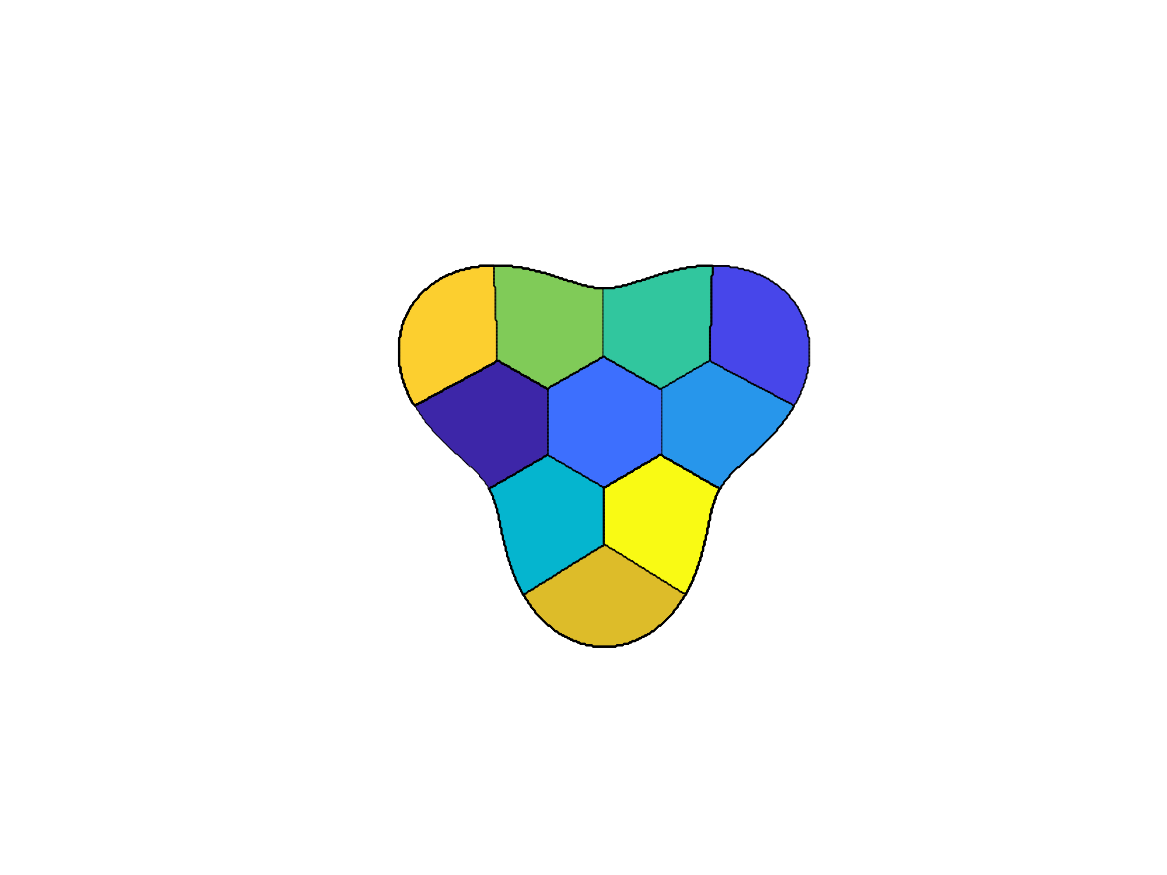} &
				\centering  \includegraphics[width=0.12\textwidth, clip, trim=6cm 4.5cm 6.5cm 3cm]{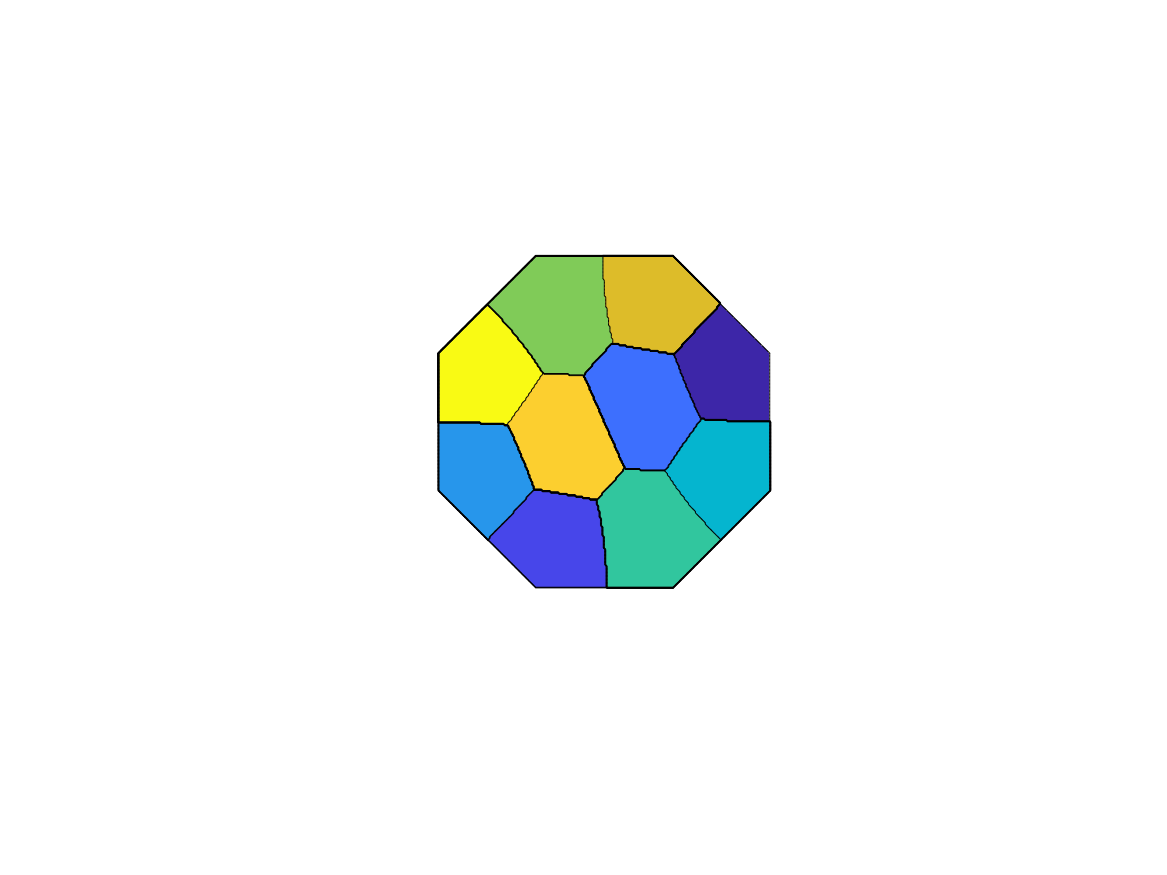} &
				\centering  \includegraphics[width=0.12\textwidth, clip, trim=3.5cm 1cm 5.5cm 5cm]{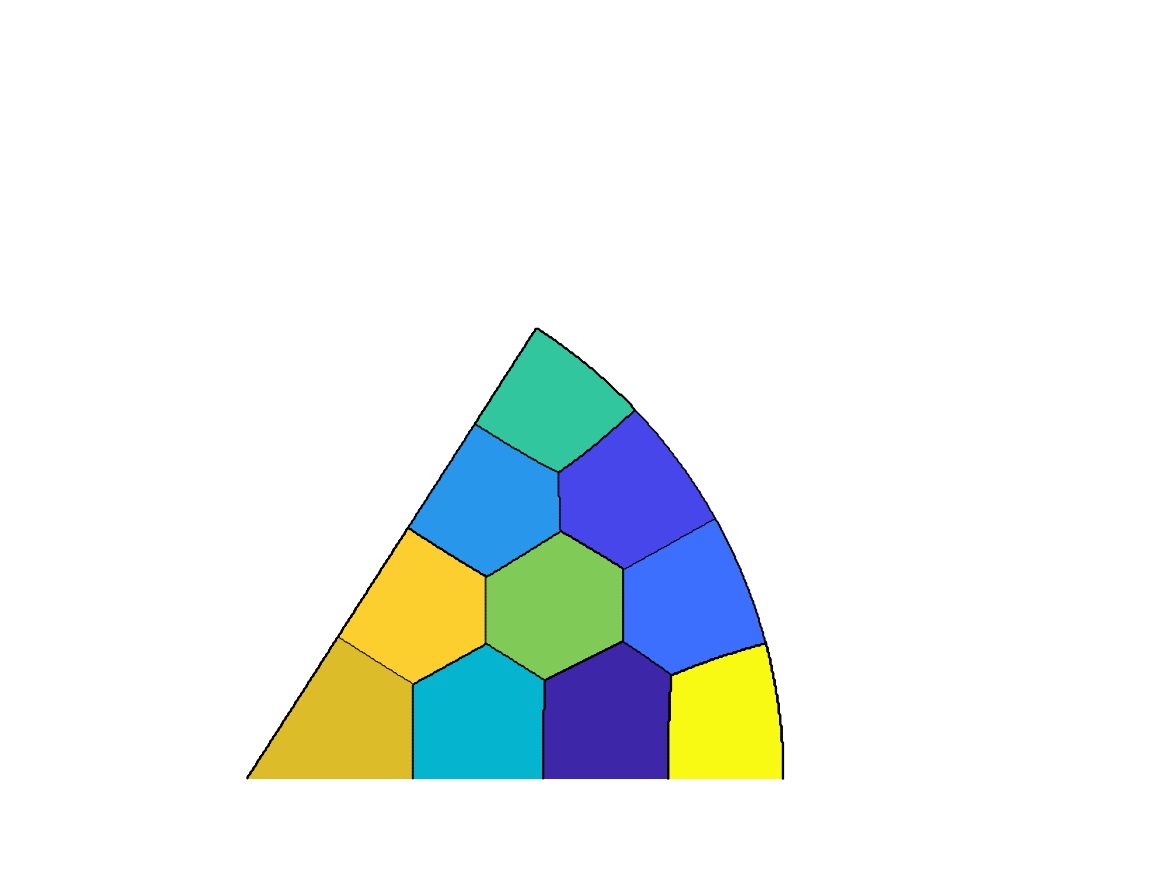} &
				\includegraphics[width=0.12\textwidth, clip, trim=4cm 4.5cm 4cm 4cm]{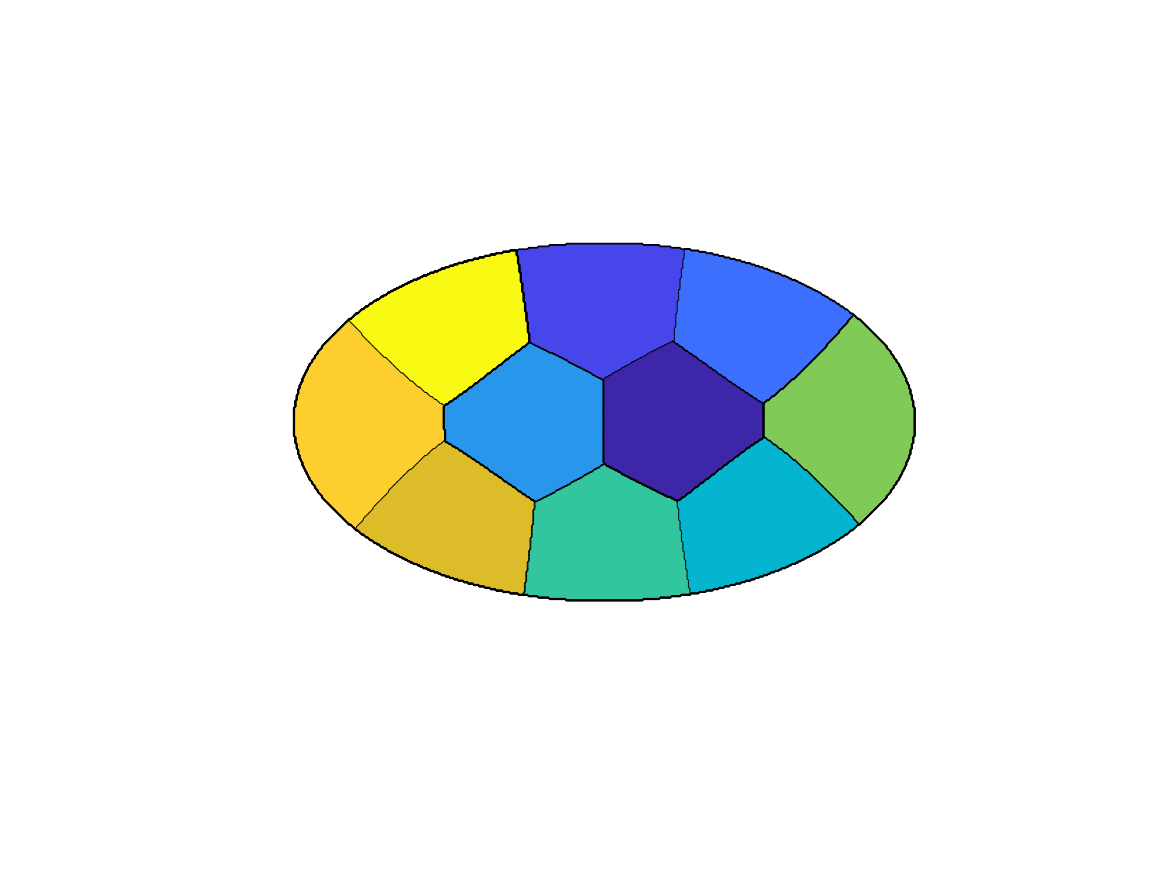} \\            
				\hline
				\centering Algorithm~\ref{Alg:3step_Type2} & 
				\centering  \includegraphics[width=0.12\textwidth, clip, trim=6.5cm 3.5cm 5.5cm 3cm]{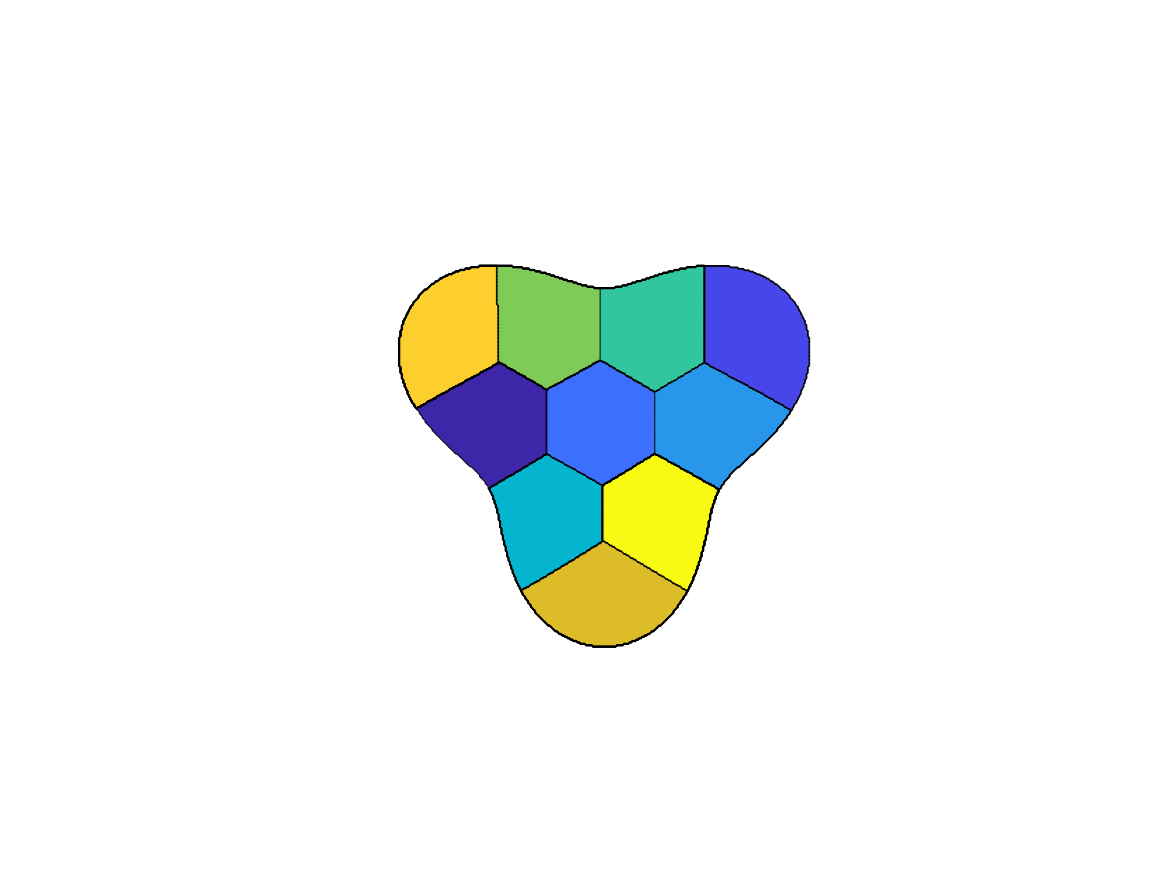} &
				\centering  \includegraphics[width=0.12\textwidth, clip, trim=6cm 4.5cm 6.5cm 3cm]{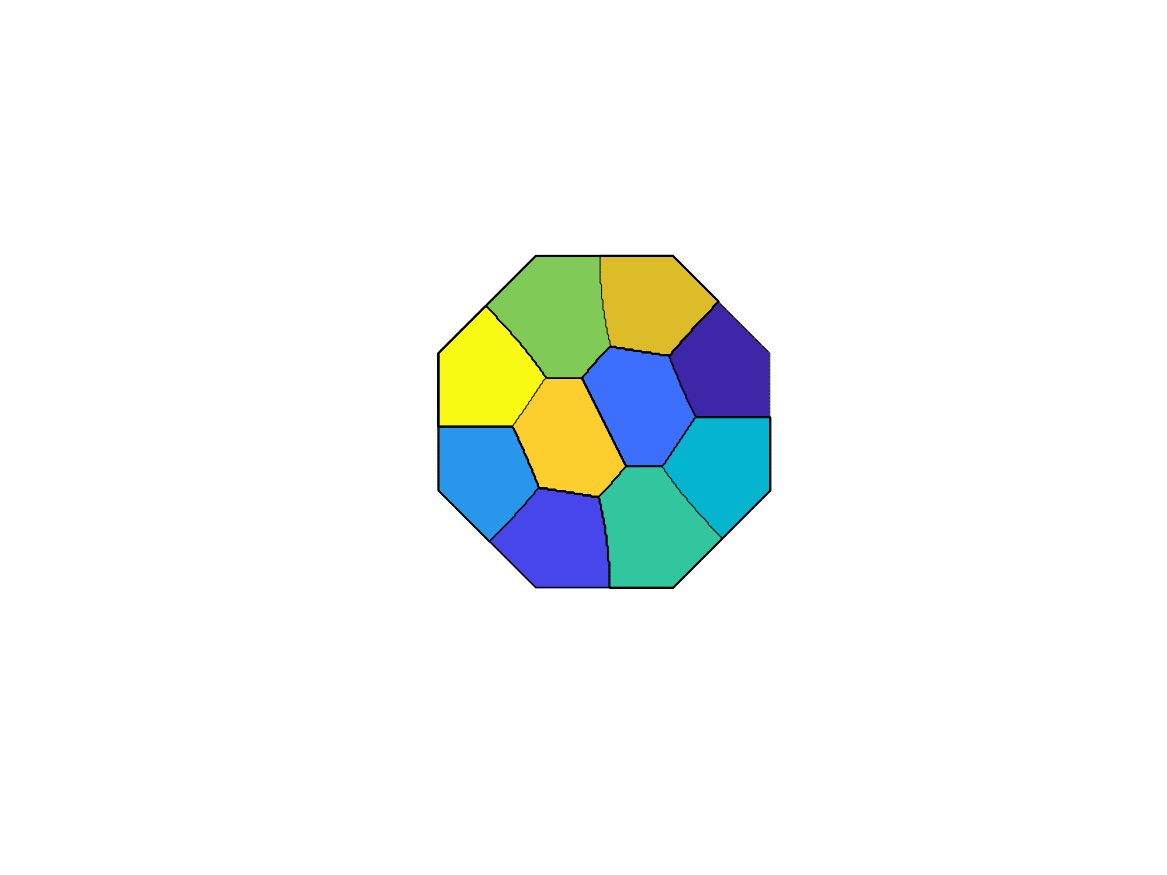} &
				\centering  \includegraphics[width=0.12\textwidth, clip, trim=3.5cm 1cm 5.5cm 5cm]{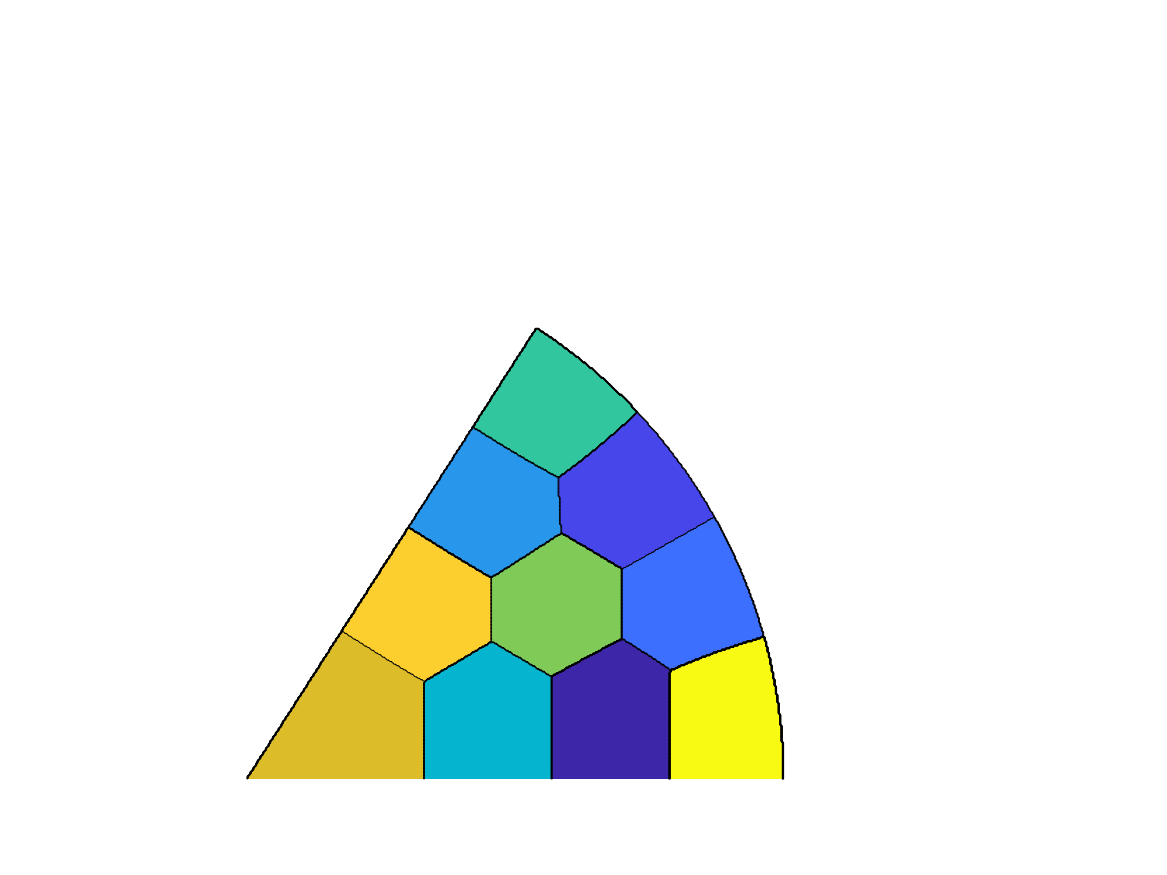} &
				\includegraphics[width=0.12\textwidth, clip, trim=4cm 4.5cm 4cm 4cm]{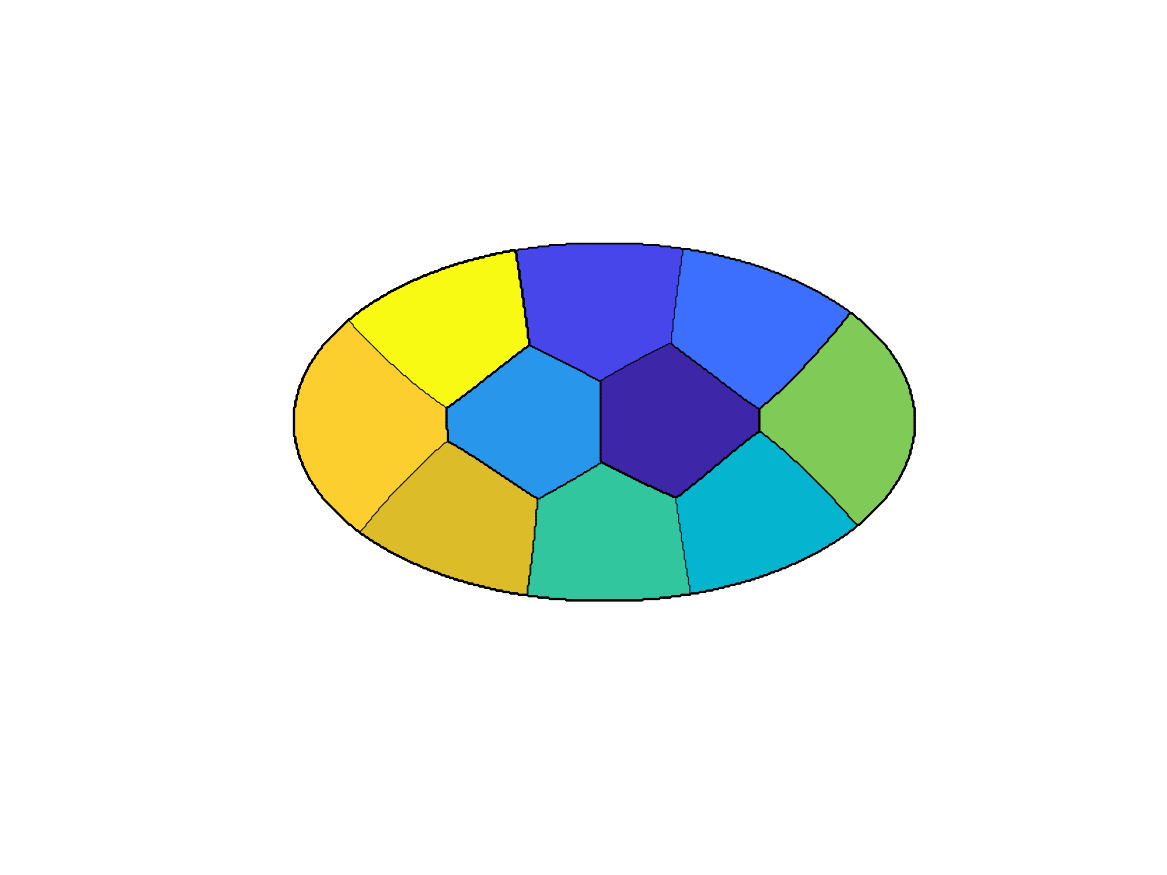} \\
				\hline
			\end{tabular}
		\end{table}
		\caption{Results of different algorithms for the 10-partition problem across various domains.}\label{fig:AbiDAlg1Alg2Alg3}
	\end{figure}

	\section{Concluding remarks}\label{sec:con}
	
	In this paper,  we proposed  numerical schemes for the optimal partition problem which is a minimization problem with multiple constraints. The key difficulty of this problem is how to  develop numerical schemes which can  satisfy  the constraint of orthogonality exactly for different parts and also admit the properties of $L^2$ norm conserving and positivity-preserving. The key idea  of our approach is rewrite the optimal partition problem into a gradient flow in which the physical constraints are enforced exactly by Lagrange multipliers. The key question is how to solve the gradient flow efficiently with low computational cost.
	
	Taking advantage of recently proposed Lagrange multiplier approach in \cite{ChSh22,ChSh_II22}, we construct two classes of constraint-preserving schemes by using the operator-splitting approach. The first class of  numerical schemes can be  linear which only need to solve Poisson equation and satisfy the physical constraints exactly at each time step. The other class of schemes can satisfy a discrete energy dissipative law which need to solve Poisson equation plus a nonlinear algebraic equation at each time step. 
	
	Numerical experiments indicate that our new schemes are very effective and accurate for the optimal partition problem that we tested. We need to mention that our numerical approach can also be generalized to other constrained minimization problems. 
	
\section*{Acknowledgement}

Q. Cheng is partially supported by National Natural Science Foundation of China (Grant No. 12301522). D. Wang is partially supported by National Natural Science Foundation of China (Grant No. 12101524), Guangdong Basic and Applied Basic Research Foundation (Grant No. 2023A1515012199), Shenzhen Science and Technology Innovation Program (Grant No. JCYJ20220530143803007, RCYX20221008092843046), Guangdong Provincial Key Laboratory of Mathematical Foundations for Artificial Intelligence (2023B1212010001), and Hetao Shenzhen-Hong Kong Science and Technology Innovation Cooperation Zone Project (No.HZQSWS-KCCYB-2024016).

  \section*{References}
	\bibliographystyle{elsarticle-harv}
%	\bibliography{dpLagrange}

\end{document}